\documentclass[a4paper,10pt]{article}
\usepackage[DIV=13]{typearea}

\usepackage{fontspec}
\usepackage{microtype}
\usepackage{csquotes}

\usepackage{tabularx}
\usepackage{booktabs}

\usepackage{amsmath,amsfonts,amssymb,mathtools}
\usepackage{bbm}
\usepackage[mathcal]{euscript}

\usepackage{quiver}
\usepackage{tikz-cd}
\usepackage{color}

\usepackage[url=false,backend=biber,style=ieee-alphabetic,dashed=true,giveninits=false,maxbibnames=999,maxalphanames=4]{biblatex}
\renewbibmacro{in:}{}

\DeclareFieldFormat[article,inbook,incollection,inproceedings,patent,thesis,unpublished]{title}{\emph{#1}\isdot}
\DeclareFieldFormat{journaltitle}{#1\nopunct}
\DeclareFieldFormat*[incollection,book]{booktitle}{#1\isdot}
\DeclareFieldFormat*[incollection,book]{series}{#1\isdot}
\DeclareFieldFormat[article,inbook,incollection,inproceedings,patent,thesis,unpublished]{volume}{vol.\ #1\isdot}

\usepackage{enumitem}
\setlist[enumerate]{label=\textup{(\arabic*)},nosep}

\usepackage{amsthm}
\usepackage[colorlinks=true,linkcolor=red!70!black,citecolor=cyan!25!blue,urlcolor=red!70!black,unicode,bookmarksnumbered]{hyperref}
\usepackage{xurl}
\hypersetup{breaklinks=true}
\usepackage[capitalize,compress]{cleveref}

\newtheorem{theorem}{Theorem}[section]
\newtheorem*{theorem*}{Theorem}
\newtheorem*{proposition*}{Proposition}
\newtheorem*{corollary*}{Corollary}
\newtheorem{maintheorem}{Theorem}

\usepackage{aliascnt}
\newaliascnt{lemma}{theorem}
\newtheorem{lemma}[lemma]{Lemma}
\aliascntresetthe{lemma}
\newaliascnt{corollary}{theorem}
\newtheorem{corollary}[corollary]{Corollary}
\aliascntresetthe{corollary}
\newaliascnt{proposition}{theorem}
\newtheorem{proposition}[proposition]{Proposition}
\aliascntresetthe{proposition}
\newaliascnt{question}{theorem}
\newtheorem{question}[question]{Question}
\aliascntresetthe{question}

\theoremstyle{definition}
\newtheorem*{definition*}{Definition}
\newaliascnt{definition}{theorem}
\newtheorem{definition}[definition]{Definition}
\aliascntresetthe{definition}
\newaliascnt{remark}{theorem}
\newtheorem{remark}[remark]{Remark}
\aliascntresetthe{remark}
\newaliascnt{example}{theorem}
\newtheorem{example}[example]{Example}
\aliascntresetthe{example}
\newaliascnt{construction}{theorem}
\newtheorem{construction}[construction]{Construction}
\aliascntresetthe{construction}
\newaliascnt{assumption}{theorem}
\newtheorem{assumption}[assumption]{Assumption}
\aliascntresetthe{assumption}

\newcommand{\coev}{\operatorname{coev}}
\newcommand{\gas}{\operatorname{gas}}
\newcommand{\VNuc}{\operatorname{VNuc}}
\newcommand{\solid}{\operatorname{solid}}
\newcommand{\an}{\operatorname{an}}
\newcommand{\KU}{\operatorname{KU}}
\newcommand{\cof}{\operatorname{cof}}
\newcommand{\const}{\operatorname{const}}

\newcommand{\A}{\mathcal{A}}

\newcommand{\C}{\mathcal{C}}
\newcommand{\D}{\mathcal{D}}
\newcommand{\E}{\mathcal{E}}
\newcommand{\F}{\mathcal{F}}

\renewcommand{\H}{\mathcal{H}}
\newcommand{\I}{\mathcal{I}}

\newcommand{\K}{\mathcal{K}}

\newcommand{\M}{\mathcal{M}}
\newcommand{\N}{\mathcal{N}}
\renewcommand{\O}{\mathcal{O}}
\renewcommand{\P}{\mathcal{P}}

\let\Section\S
\renewcommand{\S}{\mathcal{S}}

\newcommand{\U}{\mathcal{U}}

\newcommand{\bbC}{\mathbb{C}}

\newcommand{\bbE}{\mathbb{E}}
\newcommand{\bbF}{\mathbb{F}}

\newcommand{\bbN}{\mathbb{N}}

\newcommand{\bbQ}{\mathbb{Q}}

\newcommand{\bbS}{\mathbb{S}}

\newcommand{\bbZ}{\mathbb{Z}}

\newcommand{\fib}{\operatorname{fib}}

\newcommand{\id}{\operatorname{id}}
\newcommand{\colim}{\operatornamewithlimits{colim}}
\newcommand{\Hom}{\operatorname{Hom}}
\newcommand{\iHom}{\underline{\operatorname{Hom}}}
\newcommand{\End}{\operatorname{End}}
\newcommand{\iEnd}{\underline{\operatorname{End}}}

\newcommand{\Fun}{\operatorname{Fun}}
\newcommand{\FunL}{\operatorname{Fun}^\mathrm{L}}

\newcommand{\Gr}{\operatorname{Gr}}
\newcommand{\ev}{\operatorname{ev}}

\newcommand{\Nm}{\operatorname{Nm}}
\newcommand{\Tot}{\operatorname{Tot}}

\newcommand{\Spf}{\operatorname{Spf}}

\newcommand{\THH}{\operatorname{THH}}

\newcommand{\Frob}{\operatorname{Frob}}

\newcommand{\HH}{\operatorname{HH}}

\newcommand{\Cat}{\mathrm{Cat}}
\newcommand{\Mod}{\mathrm{Mod}}
\newcommand{\Grp}{\mathrm{Grp}}

\newcommand{\Shv}{\mathrm{Shv}}
\newcommand{\An}{\mathrm{An}}

\newcommand{\Alg}{\mathrm{Alg}}
\newcommand{\CAlg}{\mathrm{CAlg}}
\newcommand{\Sp}{\mathrm{Sp}}

\renewcommand{\Pr}{\mathrm{Pr}}
\newcommand{\PrL}{\mathrm{Pr}^\mathrm{L}}
\newcommand{\Ind}{\mathrm{Ind}}

\newcommand{\QCoh}{\mathrm{QCoh}}
\newcommand{\Perf}{\mathrm{Perf}}

\newcommand{\coMod}{\mathrm{coMod}}

\newcommand{\Nuc}{\mathrm{Nuc}}
\newcommand{\Calk}{\mathrm{Calk}}

\newcommand{\Free}{\mathrm{Free}}

\newcommand{\op}{\mathrm{op}}

\newcommand{\st}{\mathrm{st}}
\newcommand{\ad}{\mathrm{ad}}

\newcommand{\ca}{\mathrm{ca}}
\newcommand{\cg}{\mathrm{cg}}
\newcommand{\cont}{\mathrm{cont}}
\newcommand{\rig}{\mathrm{rig}}
\newcommand{\tors}{\mathrm{tors}}
\newcommand{\cplt}{\mathrm{cplt}}
\newcommand{\loc}{\mathrm{loc}}
\newcommand{\at}{\mathrm{at}}
\newcommand{\dbl}{\mathrm{dbl}}
\newcommand{\Loc}{\operatorname{Loc}}

\newcommand{\iL}{\mathrm{iL}}

\usepackage{relsize}

\title{Chromatic Purity of Dualizable Categories}
\author{Yifan Jin\thanks{School of Mathematical Sciences, Beijing Normal University, Beijing 100875, China; email: \texttt{\href{mailto:yfjin@mail.bnu.edu.cn}{yfjin@mail.bnu.edu.cn}}.}}
\date{\today}

\begin{document}

\maketitle

\begin{abstract}
    We develop the chromatic theory of dualizable stable categories, with chromatic purity for continuous $K$-theory at its center. We generalize the chromatic purity theorem for algebraic $K$-theory to continuous $K$-theory of dualizable stable categories and reformulate it as the purity of $H$-unital rings. Using categorical completion theory, we construct chromatic fracture squares and obtain a refinement of chromatic purity: for every dualizable stable category $\C$, the $T(n)\oplus T(n-1)$-completion map induces an equivalence
	\[
		K_{T(n)}^{\cont}(\C)\xrightarrow{\simeq}K_{T(n)}^{\cont}(\Nuc_{T(n)\oplus T(n-1)}(\C)).
	\]
	We also establish chromatic descent for continuous $K$-theory of dualizable homotopy fixed points and prove a categorical nuclear refinement of descent.
	
	This framework allows us to lift redshift bounds, Tate vanishing, and blueshift from spectra to dualizable stable categories. We show that the $T(n)$-completion of a rigid symmetric monoidal stable category is equivalent to the dualizable limit of module categories over a tower of type $n$ generalized Moore spectra. As applications, we prove that both nuclear solid and nuclear gaseous module categories satisfy chromatic redshift.
\end{abstract}

\begin{figure}[htbp]
	\centering
    \includegraphics[scale=0.045]{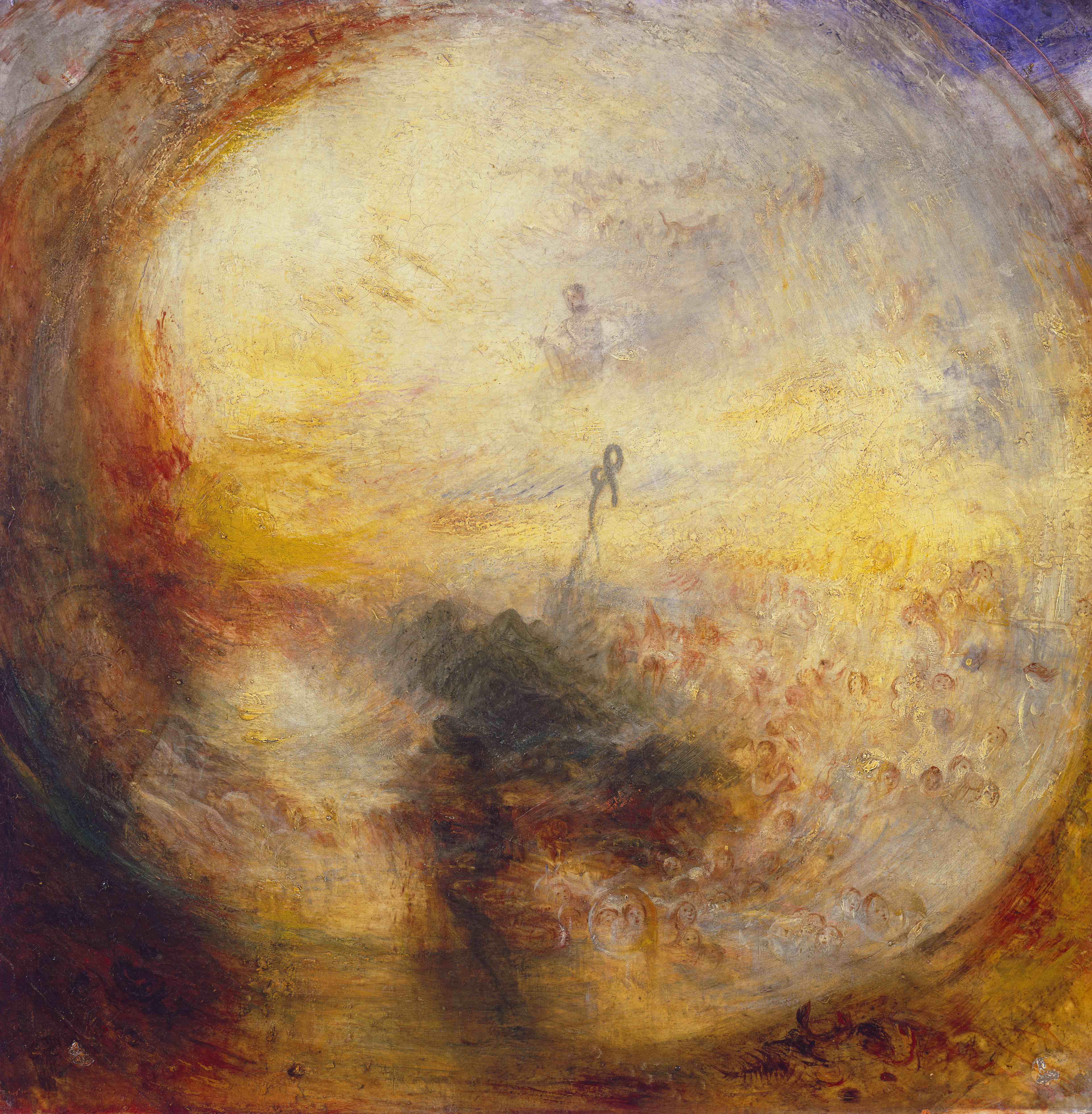}
    \caption{Light and Colour (Goethe's Theory), J.M.W. Turner}
\end{figure}

\newpage
\tableofcontents

\section*{Introduction}
\addcontentsline{toc}{section}{Introduction}

Chromatic purity is a beautiful theorem concerning the chromatic behavior of algebraic $K$-theory of small stable categories, or equivalently, compactly generated stable categories. The study of this behavior is guided by the redshift philosophy proposed by Ausoni--Rognes \cite{AR08}, which predicts that algebraic $K$-theory raises chromatic height by $1$. Chromatic purity asks which telescopic localization of an $\bbE_1$-ring suffices to determine its $T(n)$-local algebraic $K$-theory. The cases of heights $0$ and $1$ were considered by Waldhausen \cite{Wal78} and Bhatt--Clausen--Mathew \cite{BCM20}, respectively. A complete answer at arbitrary height was finally given, in categorical form, by Land--Mathew--Meier--Tamme \cite{LMMT24} and Clausen--Mathew--Naumann--Noel \cite{CMNN24}: the $T(n)$-local $K$-theory of a compactly generated stable category depends only on its $L_n^f$-localization and vanishes on its $L_{n-2}^f$-localization. More recently, Ben-Moshe gave a new proof of the vanishing statement \cite{BM26redshift}.

\begin{theorem*}[Purity, {\cite[Theorem A]{LMMT24}, \cite[Theorem C]{CMNN24}}]
	Let $n\geq1$, and let $\C$ be a compactly generated stable category. Then the canonical maps induce equivalences
	\[
		K_{T(n)}(\C)\xrightarrow{\simeq} K_{T(n)}(L_n^f\C),\quad K_{T(n)}(L_{n-2}^f\C)\simeq0.
	\]
\end{theorem*}

Equivalently, the $T(n)$-local algebraic $K$-theory of $\C$ agrees with that of its $T(n)\oplus T(n-1)$-localization $L_{T(n)\oplus T(n-1)}\C\simeq L_{T(n)\oplus T(n-1)}\Sp\otimes\C$. For an $\bbE_1$-ring, the chromatic fracture square therefore gives the following purity theorem for $\bbE_1$-rings.

\begin{theorem*}[Purity of $\bbE_1$-rings, {\cite[Corollary 4.11]{CMNN24}, \cite[Purity Theorem]{LMMT24}}]
	Let $n\geq1$, and let $R$ be an $\bbE_1$-ring. Then the canonical map $R\to L_{T(n-1)\oplus T(n)}R$ induces an equivalence
	\[
		K_{T(n)}(R)\xrightarrow{\simeq} K_{T(n)}(L_{T(n)\oplus T(n-1)}R).
	\]
\end{theorem*}

For a category, however, the corresponding equivalence is not induced directly by $\C\to L_{T(n)\oplus T(n-1)}\C$, since this localization map need not be an internal left adjoint. This naturally raises the question of whether one can construct a categorical chromatic fracture square and obtain a refined version of purity.

Moreover, Ausoni--Rognes \cite{AR08} proposed $T(n)$-local Galois descent as part of the redshift philosophy. The height $0$ case was proved by Thomason--Trobaugh \cite{Tho85,TT90}, and Clausen--Mathew--Naumann--Noel \cite{CMNN24} established the theorem at arbitrary height for finite $p$-groups.

\begin{theorem*}[Descent, {\cite[Theorem C, Proposition 4.1]{CMNN24}}]
	Let $n\geq0$, and let $\C$ be an $L_n^f$-local compactly generated stable category acted on by a finite $p$-group $G$. Then the canonical maps induce equivalences
	\[
		K_{T(n+1)}(\C^{hG,\cg})\xrightarrow{\simeq}K_{T(n+1)}(\C)^{hG},\quad K_{T(n+1)}(\C)_{hG}\xrightarrow{\simeq}K_{T(n+1)}(\C_{hG}).
	\]
\end{theorem*}

Combining this result with purity, they proved that $T(n+1)$-local $K$-theory satisfies descent for $T(n)$-local Galois extensions. Ben-Moshe--Carmeli--Schlank--Yanovski \cite{BMCSY25redshift} generalized this result to categorical descent and Galois descent for higher $p$-groups. Burklund--Clausen \cite{Cla24perfection} announced a proof for arbitrary finite groups, thereby resolving the Ausoni--Rognes conjecture on Galois descent.

As a significant application of chromatic algebraic $K$-theory, Burklund--Hahn--Levy--Schlank \cite{BHLS23telescope} disproved the Telescope Conjecture by constructing an algebraic $K$-theoretic counterexample. Their result also provides natural examples of stable categories that are not compactly generated: for $n\geq2$, the fiber category $\fib(L_n:\Sp\to L_n\Sp)$ cannot be compactly generated, since otherwise it would agree with $\fib(L_n^f:\Sp\to L_n^f\Sp)$, forcing $L_n\simeq L_n^f$. Similar categories arise naturally in analytic geometry. For example, the nuclear module categories of Clausen--Scholze are not necessarily compactly generated, but can be realized as fibers of internal localizations between compactly generated categories. Consequently, ordinary algebraic $K$-theory does not directly apply to these categories.

Efimov \cite{Efi25localizing} introduced the theory of \emph{continuous $K$-theory} for a certain class of well-behaved large categories. It extends algebraic $K$-theory from compactly generated stable categories to \emph{dualizable stable categories}, i.e. dualizable objects in $\PrL_{\st}$. We denote by $\Pr^{\dbl}_{\st}$ the category of dualizable stable categories with internal left adjoints. Every dualizable category admits a presentation $\C\simeq\fib(\D\to\E)$ for an internal localization between compactly generated stable categories. Its continuous $K$-theory is then given by
\[
	K^{\cont}(\C):=\fib(K(\D)\to K(\E)).
\]
Every $\C$ admits a canonical presentation of this form through the continuous Calkin category construction.

Our goal in this paper is to develop the chromatic theory of dualizable stable categories. The guiding principle is the passage
\begin{center}
	\emph{from $T(n)$-localization of spectra to $T(n)$-completion of dualizable stable categories.}
\end{center}
For this purpose, we use Zhou's categorical completion theory \cite{Zhob} to define chromatic completions of dualizable stable categories. Together with localizations obtained by base change, these completions give rise to torsion--complete equivalences. We extend chromatic purity and descent to dualizable stable categories and refine purity through categorical completion and categorical chromatic fracture squares. We also lift redshift bounds, Tate vanishing, and blueshift from spectra to categories, prove continuity with respect to suitable towers of $\bbE_1$-algebras, and apply these results to nuclear solid and nuclear gaseous module categories.

An important role is played by the notion of a \emph{rigid symmetric monoidal stable category}, introduced by Gaitsgory--Rozenblyum \cite{GR17}. A presentably symmetric monoidal stable category $\C$ is called rigid if $\mathbbm{1}\in\C^{\omega}$ and the multiplication $\C\otimes\C\to\C$ admits a colimit preserving right adjoint satisfying the projection formula. In the compactly generated case, this is equivalent to the compact and dualizable objects coinciding. Thus, the $\Ind$-completion of a small symmetric monoidal stable category is rigid precisely when every object of the small category is dualizable, recovering the classical notion of rigidity. Since categorical completions preserve rigidity, the nuclear refinement of purity is particularly well suited to studying the continuous $K$-theory and chromatic behavior of rigid symmetric monoidal stable categories.

\subsection*{Main results}
\addcontentsline{toc}{subsection}{Main results}

To develop chromatic theory at a general height, we first define the relevant chromatic layers.

\begin{definition*}[{\cref{def:chromatic_layer}}]
	Let $n\geq0$, and let $1\leq m\leq n+1$. We define the \emph{$n$-th length-$m$ chromatic layer spectrum} as
	\[
		M_n^{m,f}:=\fib(L_n^f\bbS\to L_{n-m}^f\bbS),
	\]
	where $L_{-1}^f\bbS:=0$. In particular, $M_n^{1,f}$ is the $n$-th monochromatic layer spectrum $M_n^f:=\fib(L_n^f\bbS\to L_{n-1}^f\bbS)$.
\end{definition*}

Since the continuous $K$-theory of a dualizable stable category is determined by the algebraic $K$-theory of compactly generated stable categories, the purity theorem extends naturally as follows.

\begin{maintheorem}[Purity, {\cref{thm:purity_dualizable}}]
	Let $n\geq1$, and let $\C$ be a dualizable stable category. Then the canonical maps induce equivalences
	\[
		K_{T(n)}^{\cont}(\C)\xrightarrow{\simeq}K_{T(n)}^{\cont}(L_n^f\C)\xleftarrow{\simeq}K_{T(n)}^{\cont}(M_n^{2,f}\C).
	\]
	In particular, we have an equivalence
	\[
		K_{T(n)}^{\cont}(\C)\simeq K_{T(n)}^{\cont}(L_{T(n)\oplus T(n-1)}\C).
	\]
\end{maintheorem}

This theorem is the foundation for the $K$-theory results of the paper. Its last equivalence identifies the relevant chromatic localization at the level of $K$-theory spectra, but does not provide a map of dualizable stable categories inducing this equivalence. Zhou's categorical completion theory \cite{Zhob} supplies a canonical refinement.

Let $n\geq1$ and let $\C$ be a dualizable stable category. We define the \emph{$T(n)\oplus T(n-1)$-completion} of $\C$ as
\[
	\Nuc_{T(n)\oplus T(n-1)}(\C):=\iHom^{\dbl}_{L_n^f\Sp}(L_{T(n)\oplus T(n-1)}\Sp,L_n^f\C),
\]
where $\iHom^{\dbl}_{L_n^f\Sp}$ denotes the internal Hom in $\Pr^{\dbl}_{L_n^f\Sp}$. The notation $\Nuc$ comes from the fact that $\Nuc_{T(n)\oplus T(n-1)}(\C)$ is the category of nuclear objects in $L_{T(n)\oplus T(n-1)}\C$ whenever $\C$ is rigid symmetric monoidal. Zhou's completion theory gives an equivalence between $T(n)\oplus T(n-1)$-local and $T(n)\oplus T(n-1)$-complete categories, given by nuclear completion $\Nuc_{T(n)\oplus T(n-1)}$ with inverse $L_{T(n)\oplus T(n-1)}$. Here a dualizable stable category $\C$ is called \emph{$T(n)\oplus T(n-1)$-complete} if the canonical completion map $\C\to L_n^f\C\to\Nuc_{T(n)\oplus T(n-1)}(\C)$ is an equivalence.

\begin{maintheorem}[Nuclear purity, {\cref{thm:Nuc_T(n),thm:nuc_purity}}]
	Let $n\geq1$, and let $\C$ be a dualizable stable category. Then we have a pullback diagram of dualizable stable categories
	\[\begin{tikzcd}
		{L_n^f\C} & {L_{n-2}^f\C} \\
		{\Nuc_{T(n)\oplus T(n-1)}(\C)} & {L_{n-2}^f\Nuc_{T(n)\oplus T(n-1)}(\C)}
		\arrow[from=1-1, to=1-2]
		\arrow[from=1-1, to=2-1]
		\arrow[from=1-2, to=2-2]
		\arrow[""{name=0, anchor=center, inner sep=0}, from=2-1, to=2-2]
		\arrow["\lrcorner"{anchor=center, pos=0.125}, draw=none, from=1-1, to=0]
	\end{tikzcd}\]
	and the canonical map $\C\to\Nuc_{T(n)\oplus T(n-1)}(\C)$ induces an equivalence
	\[
		K_{T(n)}^{\cont}(\C)\xrightarrow{\simeq}K_{T(n)}^{\cont}(\Nuc_{T(n)\oplus T(n-1)}(\C)).
	\]
\end{maintheorem}
The $K$-theoretic statements of purity and nuclear purity are equivalent, but the latter realizes purity through a canonical map in $\Pr^{\dbl}_{\st}$ that fits into a categorical fracture square. More precisely, it lifts the chromatic fracture square of localization functors to the level of dualizable stable categories.

More generally, we establish a construction of fracture squares using rigid envelopes, of which the square above is an instance. A \emph{rigid envelope} of a symmetric monoidal stable category $\I$ is a rigid symmetric monoidal stable category $\C$ which contains $\I$ as a \emph{smashing ideal}, i.e. $\I\subseteq\C$ is a $\C$-submodule whose inclusion is a $\C$-internal left adjoint. We prove that every dualizable $\C$-module gives rise to a fracture square relating localization and completion. Applied to the smashing ideal $L_{T(n)\oplus T(n-1)}\Sp\simeq M_n^{2,f}\Sp\subseteq L_n^f\Sp$, the construction gives the square above and yields a second proof of nuclear purity. Thus, rigid envelopes provide the connection between categorical completion, fracture squares and purity used throughout the paper.

There is also a ring-theoretic description of the framework. An important class of dualizable stable categories consists of fiber categories associated with homological epimorphisms. Let $R\to S$ be a \emph{homological epimorphism}, i.e. $S$ is idempotent over $R$, and write $I:=\fib(R\to S)$. Then the fiber category $\Mod(R,I):=\fib(S\otimes_R-:\Mod(R)\to\Mod(S))$ is dualizable. The special case $S=\bbS$ leads to the following notion. A nonunital $\bbE_1$-ring $I$ is called \emph{$H$-unital} if the map $I^+\to\bbS$ is a homological epimorphism. Its \emph{category of $H$-modules} is defined as $\Mod_H(I):=\Mod(I^+,I)$. Moreover, every dualizable stable category is equivalent to $\Mod_H(I)$ for some $H$-unital ring $I$. This viewpoint follows Tamme's reformulation \cite{Tam18} of the excision theory of Suslin--Wodzicki \cite{SW92,Sus95}; see also \cite{KNP24}.

Since $L_{T(n)\oplus T(n-1)}$ need not preserve $H$-unitality, we instead use $M_n^{2,f}$ to make chromatic purity compatible with these presentations. We prove that coidempotent $\bbS$-coalgebras are $H$-unital and that their tensor products with $H$-unital rings remain $H$-unital. In particular, this applies to chromatic layers.

\begin{proposition*}[{\cref{prop:multichromatic_Hunital}}]
	Let $n\geq0$ and $1\leq m\leq n+1$. Then $M_n^{m,f}$ is an $H$-unital $\bbE_\infty$-ring.
\end{proposition*}

Consequently, $M^{2,f}_n\otimes-$ preserves $H$-unitality, and for every $H$-unital ring $I$ we have $M_n^{2,f}\Mod_H(I)\simeq\Mod_H(M_n^{2,f}\otimes I)$. Writing $K^{\cont}(I):=K^{\cont}(\Mod_H(I))$, we obtain the following form of purity.

\begin{corollary*}[Purity of $H$-unital rings, {\cref{cor:purity_Hunital}}]
	Let $n\geq1$, and let $I$ be an $H$-unital ring. Then we have an equivalence
	\[
		K_{T(n)}^{\cont}(I)\simeq K_{T(n)}^{\cont}(M_n^{2,f}\otimes I).
	\]
\end{corollary*}

We next turn to descent. Let $G\in\Grp(\An)$ be a group anima acting on a dualizable stable category $\C$. The homotopy orbits $\C_{hG}:=\colim^{\dbl}_{BG}\C$ agree with the colimit in $\PrL_{\st}$, whereas the homotopy fixed points $\C^{hG,\dbl}:=\lim^{\dbl}_{BG}\C$ generally differ from the limit in $\PrL_{\st}$. For a finite group $G$, both constructions preserve short exact sequences \cite{Efi25rigidity}, thereby allowing us to extend chromatic descent.

\begin{maintheorem}[Descent, {\cref{thm:descent_dualizable}}]
	Let $n\geq0$, and let $\C$ be an $L_n^f$-local dualizable stable category acted on by a finite $p$-group $G$. Then the canonical maps induce equivalences
	\[
		K_{T(n+1)}^{\cont}(\C^{hG,\dbl})\xrightarrow{\simeq} K_{T(n+1)}^{\cont}(\C)^{hG},\quad K_{T(n+1)}^{\cont}(\C)_{hG}\xrightarrow{\simeq}K_{T(n+1)}^{\cont}(\C_{hG})
	\]
\end{maintheorem}

Consider the categorical norm map defined by Zhou \cite{Zhoa}, following Efimov's construction \cite{Efi25limit}. Tate vanishing in $L_{T(n+1)}\Sp$, together with the descent theorem above, implies that the norm map induces an equivalence on $T(n+1)$-local continuous $K$-theory. In fact, we prove that the norm map itself already becomes an equivalence after $T(n)$-localization, or equivalently, via the torsion--complete equivalence, after $T(n)$-completion. To prove this, we identify the norm map with the completion map with respect to the smashing ideal $\Sp_{hG}\subseteq\Sp^{hG,\dbl}$. This yields the following refinement of chromatic descent, which holds for arbitrary finite groups.

\begin{maintheorem}[Nuclear descent, {\cref{thm:nuc_descent}}]
	Let $n\geq1$, and let $\C$ be an $L_n^f$-local dualizable stable category acted on by a finite group $G$. Then the norm map induces equivalences
	\[
		\Nuc_{T(n)}(\C_{hG})\xrightarrow{\simeq}\Nuc_{T(n)}(\C)^{hG,\dbl},
	\]
	and
	\[
		K^{\cont}_{T(n+1)}(\C_{hG})\xrightarrow{\simeq}K^{\cont}_{T(n+1)}(\C^{hG,\dbl}).
	\]
\end{maintheorem}

The categorical assertion is proved using the Tate vanishing theorem stated below, while the $K$-theoretic assertion follows from the first and nuclear purity.

For algebraic $K$-theory of $\bbE_\infty$-rings, the redshift philosophy has a precise form. By a theorem of Hahn \cite{Hah22Hinfty}, for a nonzero $\bbE_\infty$-ring $R$, the set $\{n:L_{T(n)}R\not\simeq0\}$ is either $\varnothing$, all of $\bbN$, or an interval $\{0,1,\cdots,n\}$. We define the \emph{chromatic height} $\operatorname{ht}(R)$ to be the maximal such $n$, with value $-1$ if $L_{T(0)}R\simeq0$ and value $\infty$ if $L_{T(n)}R\not\simeq0$ for all $n\geq0$. The redshift theorem for $\bbE_\infty$-rings states that algebraic $K$-theory raises chromatic height exactly by $1$: for a nonzero $\bbE_\infty$-ring $R$ with $\operatorname{ht}(R_{(p)})\geq0$,
\[
	\operatorname{ht}(K(R))=\operatorname{ht}(R)+1.
\]
The upper bound follows from the categorical purity results of \cite{CMNN24,LMMT24}, whereas the lower bound uses ring-theoretic results of \cite{Yua21redshift,BSY22nullstellensatz}. This distinction motivates the corresponding question for symmetric monoidal stable categories. We define the chromatic height of such a category through the endomorphism $\bbE_\infty$-ring of its unit.

\begin{definition*}[{\cref{def:ht(C)}}]
	Let $\C$ be a presentably symmetric monoidal stable category. We define the \emph{chromatic height} of $\C$ as
	\[
		\operatorname{ht}(\C):=\operatorname{ht}(\End(\mathbbm{1}_{\C})).
	\]
\end{definition*}

Our purity theorem immediately gives the redshift upper bound. For a rigid symmetric monoidal stable category, Ramzi's dimension map and noshift theorem \cite{Ram26noshift} give the lower bound. Thus, continuous $K$-theory can only preserve the chromatic height of a rigid symmetric monoidal stable category or raise it by $1$.

\begin{maintheorem}[Redshift bound; Redshift or noshift, {\cref{thm:redshift}}]
	Let $\C$ be a nonzero dualizably symmetric monoidal stable category such that $K^{\cont}(\C)$ is nonzero. Then
	\[
		\operatorname{ht}(K^{\cont}(\C))\leq\operatorname{ht}(\C)+1.
	\]

	Moreover, if $\C$ is rigid, then we have
	\[
		\operatorname{ht}(\C)\leq\operatorname{ht}(K^{\cont}(\C))\leq\operatorname{ht}(\C)+1.
	\]
\end{maintheorem}

In contrast to redshift, Tate fixed points exhibit blueshift. Kuhn's Tate vanishing theorem \cite{Kuhn04Tate} implies that the Tate fixed points of an $L_n^f$-local spectrum with a finite group action are $L_{n-1}^f$-local; in particular, the $T(n)$-local Tate construction vanishes on $T(n)$-local spectra. We use Zhou's Tate construction \cite{Zhoa} for dualizable stable categories, which is defined as the cofiber of the categorical norm map. Our identification of the norm map with a completion map identifies the Tate construction with a generic fiber, giving it a lax symmetric monoidal structure. The calculation $\End(\mathbbm{1}_{(L_n^f\Sp)^{tG,\dbl}})\simeq(L_n^f\bbS)^{tG}$ of the endomorphism $\bbE_\infty$-ring of the unit then lifts Tate vanishing to the categorical level.

For the lower bound, Hahn \cite{Hah22Hinfty} proves that if an $\bbE_\infty$-ring $R$ with trivial $C_p$-action has nonzero Tate fixed points, then $\operatorname{ht}(R^{tC_p})\geq\operatorname{ht}(R)-1$. Together, the two bounds yield categorical blueshift.

\begin{maintheorem}[Tate vanishing; Blueshift, {\cref{thm:tate_vanishing,cor:blueshift}}]
	Let $n\geq1$, and let $\C$ be an $L_n^f$-local dualizable stable category acted on by a finite group $G$. Then $\C^{tG,\dbl}$ is $L_{n-1}^f$-local and we have
	\[
		L_{T(n)}\C^{tG,\dbl}\simeq0.
	\]
	In particular, we have $L_{T(n)}\C^{tG,\dbl}\simeq0$ whenever $\C$ is $T(n)$-local.

	Moreover, if $\C$ is a dualizably symmetric monoidal stable category with $\C^{tC_p,\dbl}$ nonzero, then we have
	\[
		\operatorname{ht}(\C^{tC_p,\dbl})\geq\operatorname{ht}(\C)-1.
	\]
	Therefore, if $\C$ is $T(n)$-complete for $n\geq1$, we have
	\[
		\operatorname{ht}(\C^{tC_p,\dbl})=\operatorname{ht}(\C)-1.
	\]
\end{maintheorem}
The first part proves that the categorical norm map is a $T(n)$-local equivalence, and hence completes the proof of nuclear descent.

One reason for the term continuous $K$-theory comes from formal geometry. For an affine Noetherian formal scheme $\Spf(R^{\wedge}_I)$, its continuous $K$-theory is defined as $\lim_nK(R/I^n)$, and one seeks a category whose (continuous) $K$-theory realizes this limit. The nuclear solid $R^{\wedge}_I$-module category $\Nuc(R^{\wedge,\solid}_I)$ of Clausen--Scholze \cite{Sch26analytic} is well suited to this purpose. Efimov \cite{Efi25limit} proved that it has the desired continuous $K$-theory, and that the same is true for the larger nuclear module category $\Nuc(R^{\wedge}_I)$ obtained by rigidifying $\Mod(R)_{I\text{-}\cplt}\simeq\lim_n\Mod(R/I^n)$. Thus
\[
	K^{\cont}(\Nuc(R^{\wedge,\solid}_I))\simeq K^{\cont}(\Nuc(R^{\wedge}_I))\simeq\lim_nK(R/I^n).
\]

At higher chromatic heights, we replace the $I$-adic tower by a suitable tower $(V_r)_{r\geq0}$ of type $n$ generalized Moore spectra with compatible $\bbE_1$-algebra structures, constructed using Burklund's results \cite{Bur22}. We prove a limit topology theorem generalizing that of \cite{LZ25}, together with a continuity theorem for continuous $K$-theory.

\begin{maintheorem}[Nuclear limit topology, {\cref{thm:nuc_lim_topology}}]
	Let $n\geq1$, and let $\C$ be a presentably symmetric monoidal stable category. Then the base change functors induce symmetric monoidal equivalences
	\[
		L_{T(n)}\C\xrightarrow{\simeq}\lim_rL_{T(n)}\Mod_{V_r}(\C)\xleftarrow{\simeq}\lim_r\Mod_{V_r}(L_{T(n)}\C).
	\]
	If $\C$ is rigid, then the base change functors induce symmetric monoidal equivalences
	\[
		\Nuc_{T(n)}\C\xrightarrow{\simeq}\lim_r^{\dbl}L_{T(n)}\Mod_{V_r}(\C)\xleftarrow{\simeq}\lim_r^{\dbl}\Mod_{V_r}(L_{T(n)}\C).
	\]
	
	Moreover, if $\C$ is rigid and $L_n^f$-local, then the base change functors induce an equivalence
	\[
		K_{T(n+1)}^{\cont}(\C)\xrightarrow{\simeq}L_{T(n+1)}\left(\lim_rK^{\cont}(\Mod_{L_{T(n)}V_r}(\C))\right).
	\]
\end{maintheorem}

More generally, motivated by \cite{MW25refinedTC}, we prove analogous limit topology and continuity theorems for suitable towers of $\bbE_1$-algebras whose terms and transition maps admit increasingly commutative structures. This framework also recovers the $I$-adic case of $\Nuc(R^{\wedge}_I)$. We also compare the $\bbE_1$- and $\bbE_2$-variants of our results with Ben-Moshe's recent work \cite{BM26quotient}, which treats limit topology under more general Mittag-Leffler conditions and proves continuity for quotients of even $\bbE_2$-ring spectra.

As applications, we prove chromatic redshift for two classes of nuclear module categories of height $0$. The first consists of nuclear solid module categories under suitable finiteness assumptions. Recall that a discrete commutative $\bbF_p$-algebra is \emph{$F$-finite} if its absolute Frobenius map $\Frob:R\to R$ is finite. In the $p$-adic case, it is enough to assume that the input is a connective $\bbE_\infty$-$\bbZ$-algebra.

\begin{corollary*}[Redshift for nuclear solid modules, {\cref{cor:redshift_nuc_solid}}]
	\begin{enumerate}
		\item Let $R$ be a Noetherian discrete commutative ring and $I\subseteq R$ an ideal. Suppose that $\pi_0(R^{\wedge}_I/p)$ is $F$-finite and $R^{\wedge}_I[1/p]$ is nonzero. Then both $\Nuc(R^{\wedge}_I)$ and $\Nuc(R^{\wedge,\solid}_I)$ satisfy chromatic redshift.
		\item Let $R$ be a connective $\bbE_\infty$-$\bbZ$-algebra. Suppose that $R^{\wedge}_p[1/p]$ is nonzero. Then both $\Nuc(R^{\wedge}_p)$ and $\Nuc(R^{\wedge,\solid}_p)$ satisfy chromatic redshift.
	\end{enumerate}
\end{corollary*}

The second class consists of nuclear gaseous module categories from Clausen--Scholze \cite{CS26complex,Cla24deligne}, which arise in complex analytic geometry.

\begin{corollary*}[Redshift for nuclear gaseous modules, {\cref{cor:redshift_nuc_gas}}]
	Let $M$ be a nonempty compact complex manifold. Then $\Nuc(M^{\gas})$ satisfies chromatic redshift.
\end{corollary*}

\subsection*{Relation to other work}
\addcontentsline{toc}{subsection}{Relation to other work}

This work builds on the theory of dualizable stable categories and continuous $K$-theory developed by Efimov \cite{Efi25localizing,Efi25limit,Efi25rigidity}, together with Zhou's categorical completion theory \cite{Zhob}. We apply the latter to chromatic layers to construct chromatic completions of dualizable stable categories. The theories of rigid symmetric monoidal stable categories, rigidifications, and nuclear categories in \cite{KNP24,Ram26locallyrigid,Ram26freerigid,Aok25} provide the corresponding symmetric monoidal setting. Within this setting, we introduce the notion of maps of rigid envelopes and prove that they give rise to symmetric monoidal fracture squares. This generalizes the symmetric monoidal fracture square of Naumann--Pol--Ramzi \cite{NPR24} and is used both for chromatic purity and for comparisons between nuclear module categories. In the adic setting, the resulting fracture squares recover the $K$-theoretic comparison in \cite{And23thesis} and the corresponding categorical comparison in \cite{Cor23}.

The purity and descent results for continuous $K$-theory build on the compactly generated versions of the theorems of Land--Mathew--Meier--Tamme \cite{LMMT24} and Clausen--Mathew--Naumann--Noel \cite{CMNN24}, and place them in the setting of dualizable categories with dualizable limits. The $H$-unital reformulation builds on Tamme's reformulation of excision for $\bbE_1$-rings \cite{Tam18}; see also Krause--Nikolaus--P\"utzst\"uck \cite{KNP24} for an account in the setting of dualizable stable categories. For descent, Zhou's norm map and Tate construction \cite{Zhoa} are essential inputs. We identify the norm map with a completion map and the Tate construction with a generic fiber. This gives the Tate construction its lax symmetric monoidal structure and yields the categorical Tate vanishing and nuclear descent.

We study the chromatic height of symmetric monoidal stable categories, defined through the endomorphism $\bbE_\infty$-ring of the unit using Hahn's description of chromatic support \cite{Hah22Hinfty}. Our redshift bounds follow from the purity result and Ramzi's work on chromatic noshift \cite{Ram26noshift}. For dualizably symmetric monoidal stable categories, purity immediately gives the upper bound on the chromatic height of continuous $K$-theory. For rigid symmetric monoidal stable categories, Ramzi's dimension map construction \cite{Ram26noshift} gives the lower bound showing that $K$-theory does not decrease chromatic height, while his noshift theorem provides examples for which the height is preserved. Combining these bounds, we obtain that redshift and noshift are the only possible behaviors for rigid symmetric monoidal stable categories. We prove redshift, under suitable hypotheses, for two classes of rigid symmetric monoidal stable categories: nuclear solid module categories, as introduced in \cite{Sch26analytic} and reformulated in \cite{Efi25limit,LLS26}, and nuclear gaseous module categories from \cite{Cla24deligne}. The proofs use the continuity results of \cite{CMM21} and the computations in \cite{Cla24deligne}, respectively. For blueshift, the spectral inputs are Kuhn's Tate vanishing theorem \cite{Kuhn04Tate} and Hahn's lower bound \cite{Hah22Hinfty}; the categorical input is the lax symmetric monoidal structure on the Tate construction established here.

The limit topology results combine four inputs. Li--Zhang's  approach \cite{LZ25} motivates our methods and the limit description. Burklund supplies $\bbE_1$-quotient constructions \cite{Bur22}. Meyer--Wagner provide the factorization condition for towers and the construction of towers of generalized Moore spectra \cite{MW25refinedTC}. Finally, Efimov's continuity theorem identifies the continuous $K$-theory of suitable dualizable limits with the corresponding limits of $K$-theory spectra \cite{Efi25limit}. We formulate these inputs in a common categorical setting and identify the resulting dualizable limits with nuclear completions. This gives both adic and chromatic examples, and applying nuclear purity in the latter case yields a continuity theorem for chromatically localized $K$-theory.

After an earlier version of this paper had been written, Ben-Moshe's work on algebraic $K$-theory of quotient ring spectra \cite{BM26quotient} appeared. His limit topology theorem is formulated under the strongly Mittag-Leffler condition, and his work establishes continuity for quotients of even $\bbE_2$-ring spectra using the constructions of \cite{HW18}. Our factorization condition implies his strongly Mittag-Leffler condition, giving another proof of our limit topology statement. The $\bbE_2$-quotient towers considered there satisfy our condition after passing to a suitable cofinal subtower, and the $\bbE_1$- and $\bbE_2$-variants of our results recover the corresponding continuity statements. We give the precise comparisons in \cref{sec:3.4}. In the symmetric monoidal setting, our formalism of maps of rigid envelopes also provides an independent proof of the relevant comparisons between completed module categories and nuclear module categories.

\subsection*{Outline}
\addcontentsline{toc}{subsection}{Outline}

\cref{sec:1} contains preliminaries on $H$-unital rings and rigid categories.
	In \cref{sec:1.1}, we recall $H$-unital rings and dualizable categories.
	In \cref{sec:1.2}, we establish a new method for constructing $H$-unital rings via coidempotent $\bbS$-coalgebras. 
	In \cref{sec:1.3}, we introduce the notion of maps of rigid envelopes, building on the theory of rigid categories from \cite{Ram26locallyrigid,Ram26freerigid} and the categorical completion theory of \cite{Zhob}.

\cref{sec:2} develops the categorical framework relating chromatic layers, categorical completion, and fracture squares.
	In \cref{sec:2.1}, we develop a relative theory of Bousfield localizations and formulate fracture squares of localization functors.
	In \cref{sec:2.2}, we apply categorical completion theory to construct fracture squares for dualizable categories, including symmetric monoidal fracture squares arising from maps of rigid envelopes, which generalize the construction of \cite{NPR24}.
	In \cref{sec:2.3}, we apply these formalisms to chromatic homotopy theory, study nuclear module categories, and construct fracture squares associated to chromatic layers.

\cref{sec:3} develops the chromatic theory of dualizable categories and proves the main results of the paper.
	In \cref{sec:3.1}, we generalize the chromatic purity theorem for algebraic $K$-theory to continuous $K$-theory of dualizable stable categories, reformulate it in terms of $H$-unital rings, and obtain a nuclear refinement via categorical completion theory.
	In \cref{sec:3.2}, we establish chromatic descent for dualizable homotopy fixed points and nuclear module categories, study norm maps for dualizable categories, and give applications to $T(n)$-local Galois descent.
	In \cref{sec:3.3}, we study the chromatic height of symmetric monoidal stable categories and prove redshift bounds, Tate vanishing, and blueshift theorems. In particular, continuous $K$-theory raises chromatic height by at most $1$, while $C_p$-Tate fixed points for the trivial action lower it by at most $1$.

	In \cref{sec:3.4}, we establish a limit topology theorem for towers of $\bbE_1$-algebras under suitable conditions, generalizing \cite{LZ25}, and apply it to nuclear module categories to prove the continuity of continuous $K$-theory. Following the construction of \cite{MW25refinedTC}, we apply these results in the $T(n)$-local and $T(n)$-nuclear settings, and then compare them with the continuity theorem of \cite{BM26quotient}.
	In \cref{sec:3.5}, we apply these results to nuclear solid and nuclear gaseous module categories and prove that both satisfy chromatic redshift.

Readers primarily interested in purity may begin with \cref{sec:3.1} and refer to \cref{sec:1.3,sec:2.2,sec:2.3} for the categorical completion theory and fracture square constructions. The limit topology results in \cref{sec:3.4} can be read independently of the descent and Tate vanishing results; purity is used only in the $K$-theoretic application.

\subsection*{Notation and conventions}
\addcontentsline{toc}{subsection}{Notation and conventions}

We will work with the theory of $\infty$-categories as developed by Lurie in \cite{HTT,HA,SAG}.
\begin{enumerate}
	\item We simply write categories and $2$-categories for $(\infty,1)$-categories and $(\infty,2)$-categories, respectively. 
	\item We denote by $\An$ the category of anima (i.e. spaces or groupoids), $\widehat{\Cat}$ the (very large) category of categories, and $\widehat{\Cat}_{\st}$ the (very large) category of stable categories with exact functors.
	\item We denote by $\PrL$ the category of presentable categories with left adjoint functors, and $\PrL_{\st}$ the corresponding category of presentable stable categories. For a regular cardinal $\kappa$, we denote by $\Pr^{\kappa}$ the category of $\kappa$-compactly generated categories with left adjoint functors whose right adjoint preserves $\kappa$-filtered colimits, and $\Pr^{\kappa}_{\st}$ the corresponding category of $\kappa$-compactly generated stable categories. For $\kappa=\omega$, we simply write $\Pr^{\cg}:=\Pr^{\omega}$ and $\Pr^{\cg}_{\st}:=\Pr^{\omega}_{\st}$. For a presentably symmetric monoidal category $\C\in\CAlg(\PrL)$, we denote by $\PrL_{\C}:=\Mod_{\C}(\PrL)$ the category of $\C$-modules, and $\Pr^{\kappa}_{\C}:=\Mod_{\C}(\Pr^{\kappa})$ the category of $\kappa$-compactly generated $\C$-modules.
	\item We denote by $\Pr^{\iL}_{\st}$ the category of presentable stable categories with $\Sp$-internal left adjoint functors (i.e. left adjoint functors whose right adjoint is an $\Sp$-linear left adjoint), and $\Pr^{\dbl}_{\st}\subseteq\Pr^{\iL}_{\st}$ the full subcategory spanned by dualizable stable categories (i.e. dualizable objects in $\PrL_{\st}$). For a presentably symmetric monoidal category $\C\in\CAlg(\PrL)$, we denote by $\Pr^{\iL}_{\C}$ the category of $\C$-modules with $\C$-internal left adjoints, and $\Pr^{\dbl}_{\C}\subseteq\Pr_{\C}^{\iL}$ the full subcategory spanned by dualizable $\C$-modules. We denote by $\Pr^{\ca}$ the category of compactly assembled categories (i.e. retracts of compactly generated categories) with compactly assembled functors (i.e. left adjoint functors whose right adjoint preserves filtered colimits). The categories appearing above are introduced and studied in \cite{Efi25localizing,KNP24,Ram24dualizable,Ram26locallyrigid,Ram26noshift,Ram26freerigid,Sch26gestalten}.
	\item We fix a prime $p$. For $n\geq0$, we denote by $L_n^f$ the finite localization $L_{T(0)\oplus\cdots\oplus T(n)}$, where $T(0):=\bbS[1/p]$ and, for $n\geq1$, $T(n)$ denotes the telescope of a $v_n$-self map of a type $n$ finite $p$-local spectrum $F(n)$. We denote by $L_n$ the localization $L_{T(0)\oplus K(1)\oplus\cdots\oplus K(n)}$, where, for $n\geq1$, $K(n)$ denotes a Morava $K$-theory spectrum of height $n$. In the $p$-local setting, these agree with the usual localizations $(L_n^f)_{(p)}\simeq L_{\bbQ\oplus T(1)\oplus\cdots\oplus T(n)}$ and $(L_n)_{(p)}\simeq L_{K(0)\oplus\cdots\oplus K(n)}\simeq L_{E_n}$, where $E_0=K(0):=\bbQ$ and, for $n\geq1$, $E_n$ denotes an even periodic Morava $E$-theory $\bbE_\infty$-ring of height $n$.
	\item We denote by $K:\Pr^{\cg}_{\st}\to\Sp$ the nonconnective algebraic $K$-theory, and $K^{\cont}:\Pr^{\dbl}_{\st}\to\Sp$ the continuous $K$-theory introduced and developed in \cite{Efi25localizing,Efi25limit,Efi25rigidity}. We denote by $K_{T(n)}$ and $K_{T(n)}^{\cont}$ the $T(n)$-localized algebraic $K$-theory and continuous $K$-theory, respectively.
	\item For an $\bbE_1$-ring $R\in\Alg(\Sp)$, we denote by $\Mod(R)$ the category of left $R$-modules. For a presentably symmetric monoidal category $\C\in\CAlg(\PrL)$, we simply write $\C$-algebras and $\C$-algebra maps for objects and maps in $\CAlg(\PrL_{\C})$. For an $\bbE_1$-algebra $R\in\Alg(\C)$, we denote by $\Mod_R(\C)$ the category of left $R$-modules in $\C$. For an $\bbE_2$-ring $R\in\Alg_{\bbE_2}(\Sp)$, we simply write
	\[
		\PrL_R:=\PrL_{\Mod(R)},\quad\Pr^{\dbl}_R:=\Pr^{\dbl}_{\Mod(R)},
	\]
	where the tensor product is denoted by $\otimes_R$. For $n\geq0$, we simply write 
	\[
		\PrL_{L_n^f}:=\PrL_{L_n^f\Sp},\quad\Pr^{\dbl}_{L_n^f}:=\Pr^{\dbl}_{L_n^f\Sp}
	\]
	and
	\[
		\PrL_{T(n)}:=\PrL_{L_{T(n)}\Sp},\quad\Pr^{\dbl}_{T(n)}:=\Pr^{\dbl}_{L_{T(n)}\Sp}.
	\]
	\item For a presentably symmetric monoidal category $\C\in\CAlg(\PrL)$ and a $\C$-module $\M\in\PrL_{\C}$, we denote by $\iHom_{\C}$ the internal Homs in $\M$ relative to $\C$.
\end{enumerate}

\subsection*{Acknowledgements}
\addcontentsline{toc}{subsection}{Acknowledgements}

This paper grew out of my Bachelor's thesis at Beijing Normal University and subsequent work on the project. I would like to express my deep gratitude to my advisors, Guchuan Li and Yin Tian, for their guidance and support. I am grateful to Daming Zhou for explaining his work to me; much of this paper is built upon the theory he developed. I also thank Zhenpeng Li, Maxime Ramzi, Georg Tamme, Longke Tang, Xiangdong Wu, Yuchen Wu, and Shuai Zong for helpful conversations related to this work. I am especially grateful to Jiacheng Liang and Vladimir Sosnilo for discussing this work with me in detail. I thank the organizers of IWoAT, especially Hana Jia Kong, for creating a stimulating environment. Some helpful discussions related to this work took place during IWoAT 2025.

\section{\texorpdfstring{$H$}{H}-unital rings}\label{sec:1}

In this section, we review the theory of $H$-unital rings and rigid categories, establish a new method for constructing $H$-unital rings, and develop the theory of maps of rigid envelopes.

\subsection{\texorpdfstring{$H$}{H}-unital rings and dualizable categories}\label{sec:1.1}

We begin by reviewing the definition and basic properties of $H$-unital rings and collecting background on dualizable categories.
\begin{definition}[{\cite[Lemma 22]{Tam18}, \cite[Lemma 2.9.9]{KNP24}}]
	Let $R\to S$ be an $\bbE_1$-ring map with fiber $I:=\fib(R\to S)$. We say $R\to S$ is \emph{idempotent} or $S$ is an \emph{idempotent} $R$-algebra if it satisfies one of the following equivalent conditions:
	\begin{enumerate}
		\item The map $R\to S$ is a \emph{homological epimorphism}, i.e.  the base change functor $S\otimes_R-:\Mod(R)\to\Mod(S)$ is a localization.
		\item The map $S\otimes_RS\to S$ induced by multiplication is an equivalence.
		\item The map $S\to S\otimes_RS$ induced from $R\to S$ by $S\otimes_R-$ is an equivalence.
		\item We have $S\otimes_RI\simeq0$.
		\item The map $I\otimes_RI\to I$ induced by the multiplication is an equivalence.
	\end{enumerate}
	We denote $\Mod(R,I):=\fib(S\otimes_R-:\Mod(R)\to\Mod(S))$.
\end{definition}

A \emph{short exact sequence} $\C\to\D\to\E$ in $\PrL_{\st}$ is a fiber-cofiber sequence, i.e. $\C\to\D$ is fully faithful and $\D/\C\simeq\E$, or equivalently, $\D\to\E$ is a Bousfield localization and $\C\simeq\fib(\D\to\E)$. A \emph{short exact sequence} $\C\to\D\to\E$ in $\Pr^{\iL}_{\st}$ is a fiber-cofiber sequence, or equivalently, a short exact sequence in $\PrL_{\st}$ such that $\C\to\D$ and $\D\to\E$ are $\Sp$-internal left adjoints.

Let $R\to S$ be a homological epimorphism. Then we have a short exact sequence in $\Pr^{\iL}_{\st}$:
\begin{equation}\label{eq:Mod(R,I)}
\begin{tikzcd}[column sep=4em]
	{\Mod(R,I)} & {\Mod(R)} & {\Mod(S)}
	\arrow[shift left=3, hook, from=1-1, to=1-2]
	\arrow["{\Hom_R(I,-)}"', shift right=3, from=1-1, to=1-2]
	\arrow["{I\otimes_R-}"{description}, from=1-2, to=1-1]
	\arrow["{S\otimes_R-}", shift left=3, from=1-2, to=1-3]
	\arrow["{\Hom_R(S,-)}"', shift right=3, from=1-2, to=1-3]
	\arrow[hook', from=1-3, to=1-2]
\end{tikzcd}
\end{equation}
with equivalences $\Mod(R,I)\simeq\coMod_{I}(\Mod(R))$ and $\Mod(S)\simeq\Mod_S(\Mod(R))$.
\begin{example}
	Let $R\in\Alg(\Sp)$ be an $\bbE_1$-ring. In \cite[\Section 7.2.3]{HA} (or see \cite[\Section A]{Nik17}), Lurie establishes the theory of localizations of $\bbE_1$-rings with respect to a set of elements. Let $S\subseteq\pi_*(R)$ be a set of homogeneous elements, and let $M\in\Mod(R)$.
	\begin{enumerate}
		\item We say $M$ is \emph{$S$-torsion} if it lies in the localizing stable subcategory $\Mod(R)^{S\text{-}\tors}:=\Loc(\{R/Rs\}_{s\in S})$ $\subseteq\Mod(R)$ generated by $\{R/Rs\}_{s\in S}$.
		\item We say $M$ is \emph{$S$-local} if it lies in the right orthogonal complement $\Mod(R)^{S^{-1}}:=(\Mod(R)^{S\text{-}\tors})^{\perp}\subseteq\Mod(R)$, or equivalently, every $s:M[d]\to M$ is an equivalence. 
		\item We say $M$ is \emph{$S$-complete} if it lies in the right orthogonal complement $\Mod(R)^{S\text{-}\cplt}:=(\Mod(R)^{S^{-1}})^{\perp}$ $\subseteq\Mod(R)$.
	\end{enumerate}

	Then we have a short exact sequence in $\Pr^{\iL}_{\st}$:
	\[\begin{tikzcd}[column sep=4em]
		{\Mod(R)^{S\text{-}\tors}} & {\Mod(R)} & {\Mod(R)^{S^{-1}}}
		\arrow[shift left=3, hook, from=1-1, to=1-2]
		\arrow["{(-)^{\wedge}_S}"', shift right=3, hook, from=1-1, to=1-2]
		\arrow["{C_{S^{-1}}}"{description}, from=1-2, to=1-1]
		\arrow["{S^{-1}}", shift left=3, from=1-2, to=1-3]
		\arrow[shift right=3, from=1-2, to=1-3]
		\arrow[hook', from=1-3, to=1-2]
	\end{tikzcd}\]
	yielding inverse equivalences
	\[
		(-)^{\wedge}_S:\Mod(R)^{S\text{-}\tors}\simeq\Mod(R)^{S\text{-}\cplt}:C_{S^{-1}}.
	\]

	Moreover, we have equivalences $\Mod(R)^{S^{-1}}\simeq\Mod(S^{-1}R)$ and $S^{-1}\simeq S^{-1}R\otimes_R-$, so this short exact sequence is equivalent to \eqref{eq:Mod(R,I)} for the homological epimorphism $R\to S^{-1}R$.
		
	We will provide another perspective on this example in \cref{exm:Mod(R)^S-tors/cplt}.
\end{example}

\begin{definition}[{\cite[Definition 2.10.2]{SW92,Tam18,KNP24}}]
	Let $I$ be a nonunital $\bbE_1$-ring. We say $I$ is \emph{$H$-unital} if the augmentation $\bbE_1$-ring map
	\[
		I^+\to\bbS
	\]
	is a homological epimorphism, where $I^+$ denotes the unitalization of $I$. We define the \emph{category of $H$-unital $I$-modules} as 
	\[
		\Mod_H(I):=\Mod(I^+,I).
	\]
\end{definition}
\begin{example}
	Let $A$ be an $\bbE_1$-ring. Then it is $H$-unital as a nonunital $\bbE_1$-ring. Indeed, $A^+\simeq A\times\bbS$ as an $\bbE_1$-ring. Then $\Mod(A^+)\simeq\Mod(A)\times\Sp$ and $A^+\to\bbS$ induces the projection to $\Sp$, which is a localization. So the category of $H$-unital $A$-modules is the same as the category of $A$-modules $\Mod_H(A)\simeq\Mod(A)$.
\end{example}
\begin{proposition}[{\cite[Lemma 23]{Tam18}, \cite[Proposition 2.10.5]{KNP24}}]\label{prop:Tam_pullback_homoepi}
	Consider a pullback diagram of $\bbE_1$-rings
	\[\begin{tikzcd}
		R & S \\
		{R'} & {S'}
		\arrow[from=1-1, to=1-2]
		\arrow[from=1-1, to=2-1]
		\arrow["\lrcorner"{anchor=center, pos=0.125}, draw=none, from=1-1, to=2-2]
		\arrow[from=1-2, to=2-2]
		\arrow[from=2-1, to=2-2]
	\end{tikzcd}\]
	and suppose that $R\to S$ is a homological epimorphism. Then $S'\simeq S\otimes_RR'$, $R'\to S'$ is a homological epimorphism, and we have a pullback square of module categories
	\[\begin{tikzcd}
		{\Mod(R)} & {\Mod(S)} \\
		{\Mod(R')} & {\Mod(S')}
		\arrow[from=1-1, to=1-2]
		\arrow[from=1-1, to=2-1]
		\arrow["\lrcorner"{anchor=center, pos=0.125}, draw=none, from=1-1, to=2-2]
		\arrow[from=1-2, to=2-2]
		\arrow[from=2-1, to=2-2]
	\end{tikzcd}\]
	and the induced functor on fibers
	\[
		\Mod(R,I)\to\Mod(R',I)
	\]
	is an equivalence, where $I:=\fib(R\to S)\simeq\fib(R'\to S')$.
\end{proposition}
\begin{corollary}[{\cite[Corollary 2.10.6]{KNP24}}]\label{cor:fib_Hunital}
	Let $R\to S$ be an $\bbE_1$-ring map whose fiber $I$ is $H$-unital. Then $R\to S$ is a homological epimorphism and $\Mod(R,I)\simeq\Mod_H(I)$.
\end{corollary}
In particular, given a pullback diagram of $\bbE_1$-rings
	\[\begin{tikzcd}
		R & S \\
		{R'} & {S'}
		\arrow[from=1-1, to=1-2]
		\arrow[from=1-1, to=2-1]
		\arrow["\lrcorner"{anchor=center, pos=0.125}, draw=none, from=1-1, to=2-2]
		\arrow[from=1-2, to=2-2]
		\arrow[from=2-1, to=2-2]
	\end{tikzcd}\]
such that the fiber $I:=\fib(R\to S)\simeq\fib(R'\to S')$ is $H$-unital, $R\to S$ and $R'\to S'$ are homological epimorphisms and the pullback diagram induces a pullback square of module categories.

There is a partial converse of this corollary.
\begin{proposition}[{\cite[Remark 27]{Tam18}, \cite[Proposition 2.10.7]{KNP24}}]\label{prop:fib_conn_Hunital}
	Consider a pullback diagram of connective $\bbE_1$-rings
	\[\begin{tikzcd}
		R & S \\
		{R'} & {S'}
		\arrow[from=1-1, to=1-2]
		\arrow[from=1-1, to=2-1]
		\arrow["\lrcorner"{anchor=center, pos=0.125}, draw=none, from=1-1, to=2-2]
		\arrow[from=1-2, to=2-2]
		\arrow[from=2-1, to=2-2]
	\end{tikzcd}\]
	and suppose that the horizontal fibers are connective. If $R' \to S'$ is a homological epimorphism, then so is $R\to S$. In particular, if $R\to S$ is a homological epimorphism of connective $\bbE_1$-rings with $\pi_0(R)\to\pi_0(S)$ surjective, then the fiber $I:=\fib(R\to S)$ is $H$-unital.
\end{proposition}
This proposition fails without connectivity assumptions. For example, by \cite[Example 2.10.8]{KNP24}, $\bbQ[x]\to\bbQ[x^{\pm1}]$ is a homological epimorphism, whereas the induced pullback $\bbQ\to\bbQ[x^{-1}]$ is not.

One reason $H$-unital rings are important is that every dualizable stable category is equivalent to the $H$-module category of some $H$-unital ring.
\begin{definition}[{\cite[Proposition D.7.3.1]{SAG}, \cite[Theorem 2.10.16]{KNP24}}]
	Let $\C\in\PrL_{\st}$. We say $\C$ is \emph{dualizable stable} if one of the following equivalent conditions holds:
	\begin{enumerate}
		\item $\C$ is dualizable as an object in $\PrL_{\st}$.
		\item $\C$ is a retract of a compactly generated stable category.
		\item $\C$ can be internally left adjoint fully faithful embedded into some $\D\in\Pr^{\cg}_{\st}$, or equivalently, $\C$ is the fiber of a Bousfield localization
		\[
			\D\to\E
		\]
		where $\D\to\E$ lies in $\Pr^{\cg}_{\st}$.
		\item $\C$ is the fiber of a base change functor
		\[
			S\otimes_R-:\Mod(R)\to\Mod(S)
		\]
		where $R\to S$ is a homological epimorphism of $\bbE_1$-rings, i.e. $\C\simeq\Mod(R,I)$ with $I:=\fib(R\to S)$.
		\item $\C$ is the fiber of a base change functor 
		\[
			\bbS\otimes_{I^+}-:\Mod(I^+)\to\Sp
		\]
		where $I$ is an $H$-unital ring, i.e. $\C\simeq\Mod_H(I)$.
	\end{enumerate}
\end{definition}
The following relative version of dualizability in condition (1) will be useful later. We refer to \cite{Ram24dualizable} for the general theory. 
\begin{definition}
	Let $\C\in\CAlg(\PrL)$ and $\M\in\PrL_{\C}$.
	\begin{enumerate}
		\item We say $\M$ is a \emph{dualizable $\C$-module} if it is dualizable as an object in $\PrL_{\C}$, and let $\Pr^{\dbl}_{\C}\subseteq\Pr^{\iL}_{\C}$ denote the full subcategory spanned by dualizable $\C$-modules.
		\item We say an object $X\in\M$ is \emph{$\C$-atomic} if the $\C$-linear functor $-\otimes X:\C\to\M$ is a $\C$-internal left adjoint, and let $\M^{\C\text{-}\at}$ denote the full subcategory spanned by $\C$-atomic objects.
		\item We say $\M$ is an \emph{atomically generated $\C$-module} if $\M$ is generated by $\M^{\C\text{-}\at}$ as a localizing $\C$-submodule. Equivalently, $\M\simeq\P_{\C}(\M_0)$ for some small $\C$-category $\M_0$. Let $\Pr^{\at}_{\C}\subseteq\Pr^{\iL}_{\C}$ denote the full subcategory spanned by atomically generated $\C$-modules.
	\end{enumerate}
\end{definition}
\begin{remark}
	Let $\C\in\CAlg(\PrL)$ and $\M\in\PrL_{\C}$. By \cite[Theorem 1.49]{Ram24dualizable}, $\M$ is a dualizable $\C$-module if and only if it is a retract of an atomically generated $\C$-module, if and only if it can be $\C$-internally left adjoint fully faithful embedded into a $\C$-atomically generated module. If $\C\in\CAlg(\Pr^{\kappa})$ with $\kappa>\omega$, and $\M\in\Pr^{\lambda}_{\C}$ with $\lambda\geq\kappa$, then $\M$ is dualizable if and only if the canonical functor $k:\P_{\C}(\M^{\lambda})\to\M$ admits a left adjoint, namely, $\widehat{j}:\M\hookrightarrow\P_{\C}(\M^{\lambda})$. In particular, by \cite[Theorem 3.1]{Ram24dualizable}, dualizable $\C$-modules are $\kappa$-compactly generated $\C$-modules.
\end{remark}
By \cite[Corollary 1.32, Proposition 1.62]{Ram24dualizable}, the forgetful functor $\Pr^{\star}_{\C}\to\PrL_{\C}$ creates colimits for $\star\in\{\iL,\at,\dbl\}$. So we only need to compute colimits in $\PrL_{\C}$. On the other hand, we denote by $\lim^{\dbl}$ the limits computed in $\Pr^{\dbl}_{\C}$, which generally do not coincide with those in $\PrL_{\C}$.
\begin{remark}\label{rmk:compass}
	In the unstable case, condition (2) leads to the concept of compactly assembled categories introduced and developed in \cite{SAG,KNP24,Ram24dualizable,Efi25localizing}. Let $\C\in\PrL$. We say $\C$ is \emph{compactly assembled} if it is a retract of a compactly generated category in $\PrL$, and let $\Pr^{\ca}$ denote the category of compactly assembled categories with left adjoint functors whose right adjoints preserve filtered colimits. Then $\C$ is compactly assembled if and only if it is generated by \emph{compactly exhaustible} objects (i.e. objects of the form $\colim_{\bbN}X_n$ where all maps $X_n\to X_{n+1}$ are compact) under colimits by \cite[Theorem 2.2.15]{KNP24} or \cite[Theorem 2.39]{Ram24dualizable}.
	
	If $\C\in\Pr^{\kappa}$ with $\kappa>\omega$, then $\C$ is compactly assembled if and only if the colimit functor $k:\Ind(\C^{\kappa})\to\C$ admits a left adjoint, namely, $\widehat{j}:\C\hookrightarrow\Ind(\C^{\kappa})$ by \cite[Theorem 21.1.2.10]{SAG} or \cite[Theorem 2.2.15]{KNP24}. Since compactly assembled categories are $\omega_1$-compactly generated, we may take $\kappa=\omega_1$. In particular, the notion of \emph{$\kappa$-compactly assembled} is redundant: a retract of a $\kappa$-compactly generated category remains $\kappa$-compactly generated by \cite[Observation 1.69, Proposition 2.27]{Ram24dualizable}.
\end{remark}

Our focus here is on condition (5); we refer the reader to \cite[Theorem 2.10.16]{KNP24} for its proof. $H$-unital rings also play a role in the general setting (4). Let $\C$ be the fiber of a base change functor $S\otimes_R-:\Mod(R)\to\Mod(S)$ induced by an $\bbE_1$-ring map $R\to S$ with fiber $I:=\fib(R\to S)$. If $I$ is an $H$-unital ring, then \cref{cor:fib_Hunital} implies that $\C$ is equivalent to the $H$-module category $\Mod_H(I)$. In particular, by \cref{prop:fib_conn_Hunital}, $I$ is $H$-unital whenever $R\to S$ is a \emph{$\pi_0$-surjective homological epimorphism} of connective $\bbE_1$-rings, i.e. a homological epimorphism with $\pi_0(R)\to\pi_0(S)$ surjective.

We briefly record the additive analogue of the preceding discussion. A \emph{dualizable additive} category is a dualizable object in the category of presentable additive categories $\PrL_{\ad}\simeq\PrL_{\Sp_{\geq0}}$, or equivalently, a dualizable $\Sp_{\geq0}$-module. In \cite{LLS26}, Levy--Liang--Sosnilo provide several equivalent characterizations of dualizable additive categories. The one relevant here is the following:
\begin{theorem}[{\cite[Theorem A]{LLS26}}]\label{thm:LLS}
	Let $\C\in\PrL_{\ad}$. Then $\C$ is dualizable additive if and only if it is the fiber of a base change functor 
	\[
		S\otimes_R-:\Mod(R)_{\geq0}\to\Mod(S)_{\geq0}
	\]
	where $R\to S$ is a $\pi_0$-surjective homological epimorphism of connective $\bbE_1$-rings.
\end{theorem}
For completeness, we reformulate this characterization in terms of connective $H$-unital rings.
\begin{definition}
	Let $I$ be an $H$-unital ring. We say $I$ is a \emph{connective $H$-unital ring} if $I$ is connective, and we define the \emph{connective $H$-module category} as $\Mod_H(I)_{\geq0}:=\fib(\Mod(I^+)_{\geq0}\to\Sp_{\geq0})$.
\end{definition}
By \cref{prop:fib_conn_Hunital}, a nonunital $\bbE_1$-ring is connective $H$-unital if and only if it is equivalent, as a nonunital $\bbE_1$-ring, to the fiber $I\simeq\fib(R\to S)$ of a $\pi_0$-surjective homological epimorphism of connective $\bbE_1$-rings $R\to S$.
\begin{corollary}
	Every dualizable additive category is equivalent to the connective $H$-module category $\Mod_H(I)_{\geq0}$ for some connective $H$-unital ring $I$.
\end{corollary}
\begin{proof}
	By \cref{thm:LLS}, every dualizable additive category $\C$ is equivalent to the fiber of a base change functor $\Mod(R)_{\geq0}\to\Mod(S)_{\geq0}$ where $R\to S$ is a $\pi_0$-surjective homological epimorphism of connective $\bbE_1$-rings. Then $I:=\fib(R\to S)$ is connective $H$-unital and $\C\simeq\fib(\Mod(R)_{\geq0}\to\Mod(S)_{\geq0})\simeq\fib(\Mod(I^+)_{\geq0}\to\Sp_{\geq0})\simeq\Mod_H(I)_{\geq0}$.
\end{proof}

\subsection{Constructions of \texorpdfstring{$H$}{H}-unital rings}\label{sec:1.2}
Using an argument similar to the proof of \cref{prop:Tam_pullback_homoepi},
we show that coidempotent $\bbS$-coalgebras are $H$-unital and their tensor products with $H$-unital rings remain $H$-unital, thereby providing a new method for constructing $H$-unital rings.
\begin{remark}
	We will use the following argument. Given a pullback diagram of $\bbE_1$-rings
	\[\begin{tikzcd}
		R & S \\
		{R'} & {S'}
		\arrow[from=1-1, to=1-2]
		\arrow[from=1-1, to=2-1]
		\arrow["\lrcorner"{anchor=center, pos=0.125}, draw=none, from=1-1, to=2-2]
		\arrow[from=1-2, to=2-2]
		\arrow[from=2-1, to=2-2]
	\end{tikzcd}\]
	we have a pair of adjoint functors
	\[
		F:\Mod(R)\rightleftarrows\Mod(R')\times_{\Mod(S')}\Mod(S):G,
	\]
	where $F$ sends an $R$-module $M$ to $(R'\otimes_RM,S\otimes_RM,S'\otimes_{R'}(R'\otimes_RM)\simeq S'\otimes_S(S\otimes_RM))$, and $G$ sends $(M,N,S'\otimes_{R'}M\simeq S'\otimes_SN)$ to the pullback $R$-module $M\times_{S'\otimes_SN}N$. The functor $F$ is always fully faithful since the unit $\id_{\Mod(R)}\simeq GF$ is an equivalence.
	
	To show that $F$ and $G$ form a pair of inverse equivalences, it suffices to show that $G$ is conservative. Then we only need to check that $G$ detects $0$, i.e. for any object $(M,N,S'\otimes_{R'}M\simeq S'\otimes_SN)$ of the pullback $\Mod(R')\times_{\Mod(S')}\Mod(S)$ such that 
	\begin{equation}\label{eq:pullback}\begin{tikzcd}
		0 & N \\
		M & {S'\otimes_{R'}M\simeq S'\otimes_SN}
		\arrow[from=1-1, to=1-2]
		\arrow[from=1-1, to=2-1]
		\arrow["\lrcorner"{anchor=center, pos=0.125}, draw=none, from=1-1, to=2-2]
		\arrow[from=1-2, to=2-2]
		\arrow[from=2-1, to=2-2]
	\end{tikzcd}\end{equation}
	we have $M\simeq0$ and $N\simeq0$.

	If $F$ and $G$ form a pair of inverse equivalences, then we obtain a pullback diagram of module categories
	\[\begin{tikzcd}
		{\Mod(R)} & {\Mod(S)} \\
		{\Mod(R')} & {\Mod(S')}
		\arrow[from=1-1, to=1-2]
		\arrow[from=1-1, to=2-1]
		\arrow["\lrcorner"{anchor=center, pos=0.125}, draw=none, from=1-1, to=2-2]
		\arrow[from=1-2, to=2-2]
		\arrow[from=2-1, to=2-2]
	\end{tikzcd}\]
	In this case, if $R'\to S'$ is a homological epimorphism, then the induced functor $\Mod(R')\to\Mod(S')$ is a localization. Consequently, $\Mod(R)\to\Mod(S)$ is also a localization, which implies that $R\to S$ is a homological epimorphism.
\end{remark}
\begin{proposition}\label{prop:pullback_homoepi_idem}
	Consider a pullback diagram of $\bbE_1$-rings
	\[\begin{tikzcd}
		R & S \\
		{R'} & {S'}
		\arrow[from=1-1, to=1-2]
		\arrow[from=1-1, to=2-1]
		\arrow["\lrcorner"{anchor=center, pos=0.125}, draw=none, from=1-1, to=2-2]
		\arrow[from=1-2, to=2-2]
		\arrow[from=2-1, to=2-2]
	\end{tikzcd}\]
	and suppose that $R'\to S'$ is a homological epimorphism with fiber $I:=\fib(R'\to S')$. If $S'\otimes I\simeq 0$, then $R\to S$ is a homological epimorphism and we have a pullback diagram of module categories
	\[\begin{tikzcd}
		{\Mod(R)} & {\Mod(S)} \\
		{\Mod(R')} & {\Mod(S')}
		\arrow[from=1-1, to=1-2]
		\arrow[from=1-1, to=2-1]
		\arrow["\lrcorner"{anchor=center, pos=0.125}, draw=none, from=1-1, to=2-2]
		\arrow[from=1-2, to=2-2]
		\arrow[from=2-1, to=2-2]
	\end{tikzcd}\]
	In particular, the fiber of a ring map between idempotent $\bbS$-algebras is $H$-unital.
\end{proposition}
\begin{proof}
	We show that the commutative diagram of module categories is a pullback diagram. Applying $S'\otimes-$ to the pullback diagram \eqref{eq:pullback} in $\Sp$, we obtain a pullback diagram
	\[\begin{tikzcd}
		0 & {S'\otimes N} \\
		{S'\otimes M} & {S'\otimes S'\otimes_{R'}M}
		\arrow[from=1-1, to=1-2]
		\arrow[from=1-1, to=2-1]
		\arrow["\lrcorner"{anchor=center, pos=0.125}, draw=none, from=1-1, to=2-2]
		\arrow[from=1-2, to=2-2]
		\arrow[from=2-1, to=2-2]
	\end{tikzcd}\]
	Since $S'\otimes I\simeq0$, we have $S'\otimes R'\simeq S'\otimes S'$ and then the bottom map $S'\otimes M\simeq S'\otimes R'\otimes_{R'}M\to S'\otimes S'\otimes_{R'}M$ is an equivalence. Therefore, $S'\otimes N\simeq0$ and also $S'\otimes_SN\simeq0$. It follows from the pullback diagram \eqref{eq:pullback} that both $M\simeq0$ and $N\simeq0$.
	
	Let $R\to S$ be a ring map between idempotent $\bbS$-algebras with fiber $I=\fib(R\to S)$. Then $S\otimes I\simeq\fib(R\otimes S\to S\otimes S)\simeq0$ and $R\to S$ is in particular a homological epimorphism. Consider the following pullback diagram of $\bbE_1$-rings 
	\[\begin{tikzcd}
		{I^+} & \bbS \\
		R & S
		\arrow[from=1-1, to=1-2]
		\arrow[from=1-1, to=2-1]
		\arrow["\lrcorner"{anchor=center, pos=0.125}, draw=none, from=1-1, to=2-2]
		\arrow[from=1-2, to=2-2]
		\arrow[from=2-1, to=2-2]
	\end{tikzcd}\]
	It follows that $I^+\to\bbS$ is a homological epimorphism, and hence $I$ is an $H$-unital ring.
\end{proof}
\begin{corollary}
	Let $A$ be an $\bbE_1$-ring. Then the fiber $I:=\fib(\bbS\to A)$ is $H$-unital if and only if it is coidempotent.
\end{corollary}

\begin{remark}
We now turn to the tensor product of $H$-unital rings. We refer to \cite[Section 2]{ACB22} for the preliminaries on total fibers. Let $R\to S$ and $R'\to S'$ be homological epimorphisms with fibers $I,I'$, respectively. Consider the commutative diagram
\[\begin{tikzcd}
	{I\otimes I'} & {I\otimes R'} & {I\otimes S'} \\
	{R\otimes I'} & {R\otimes R'} & {R\otimes S'} \\
	{S\otimes I'} & {S\otimes R'} & {S\otimes S'}
	\arrow[from=1-1, to=1-2]
	\arrow[from=1-1, to=2-1]
	\arrow[from=1-2, to=1-3]
	\arrow[from=1-2, to=2-2]
	\arrow[from=1-3, to=2-3]
	\arrow[from=2-1, to=2-2]
	\arrow[from=2-1, to=3-1]
	\arrow[from=2-2, to=2-3]
	\arrow[from=2-2, to=3-2]
	\arrow[from=2-3, to=3-3]
	\arrow[from=3-1, to=3-2]
	\arrow[from=3-2, to=3-3]
\end{tikzcd}\]
where each row and column is a fiber sequence. Then $I\otimes I'$ is the fiber of the $\bbE_1$-ring map $R\otimes R'\to(S\otimes R')\times_{S\otimes S'}R\otimes S'$. Moreover, $R\otimes R'\to(S\otimes R')\times_{S\otimes S'}R\otimes S'$ is a homological epimorphism since
\begin{align*}
	((S\otimes R')\times_{S\otimes S'}&(R\otimes S'))\otimes_{R\otimes R'}(I\otimes I')\\
	&\simeq(S\otimes R')\otimes_{R\otimes R'}(I\otimes I')\times_{(S\otimes S')\otimes_{R\otimes R'}(I\otimes I')}(R\otimes S')\otimes_{R\otimes R'}(I\otimes I')\\
	&\simeq(S\otimes_RI\otimes I')\times_{S\otimes_RI\otimes S'\otimes_{R'}I'}I\otimes S'\otimes_{R'}I'\simeq0.
\end{align*}
This is also the dual version of \cite[Theorem 3.8(2)]{Aok26sheavesspectrum}.

Then consider the pullback diagram of $\bbE_1$-rings
\[\begin{tikzcd}
	{(I\otimes I')^+} & \bbS \\
	{R\otimes R'} & {(S\otimes R')\times_{S\otimes S'}(R\otimes S')}
	\arrow[from=1-1, to=1-2]
	\arrow[from=1-1, to=2-1]
	\arrow["\lrcorner"{anchor=center, pos=0.125}, draw=none, from=1-1, to=2-2]
	\arrow[from=1-2, to=2-2]
	\arrow[from=2-1, to=2-2]
\end{tikzcd}\]
To show that $I\otimes I'$ is $H$-unital, it suffices to show that the commutative diagram of module categories is a pullback diagram. If so, then $(I\otimes I')^+\to\bbS$ is a homological epimorphism and we have 
\[
	\Mod_H(I\otimes I')\simeq\fib(\Mod(R\otimes R')\to\Mod((S\otimes R')\times_{S\otimes S'}(R\otimes S'))).
\]

Suppose that $I,I'$ are $H$-unital rings such that $I\otimes I'$ is also $H$-unital. Consider the commutative diagram of module categories
\[\begin{tikzcd}
	{\Mod_H(I)\otimes\Mod_H(I')} & {\Mod_H(I)\otimes\Mod(R')} & {\Mod_H(I)\otimes\Mod(S')} \\
	{\Mod(R)\otimes\Mod_H(I')} & {\Mod(R)\otimes\Mod(R')} & {\Mod(R)\otimes\Mod(S')} \\
	{\Mod(S)\otimes\Mod_H(I')} & {\Mod(S)\otimes\Mod(R')} & {\Mod(S)\otimes\Mod(S')}
	\arrow[from=1-1, to=1-2]
	\arrow[from=1-1, to=2-1]
	\arrow[from=1-2, to=1-3]
	\arrow[from=1-2, to=2-2]
	\arrow[from=1-3, to=2-3]
	\arrow[from=2-1, to=2-2]
	\arrow[from=2-1, to=3-1]
	\arrow[from=2-2, to=2-3]
	\arrow[from=2-2, to=3-2]
	\arrow[from=2-3, to=3-3]
	\arrow[from=3-1, to=3-2]
	\arrow[from=3-2, to=3-3]
\end{tikzcd}\]
where each row and column is a short exact sequence. Then 
\[
	\Mod_H(I)\otimes\Mod_H(I')\simeq\fib(\Mod(R\otimes R')\to\Mod(S\otimes R')\times_{\Mod(S\otimes S')}\Mod(R\otimes S')).
\]
Since $\Mod((S\otimes R')\times_{S\otimes S'}(R\otimes S'))\to\Mod(S\otimes R')\times_{\Mod(S\otimes S')}\Mod(R\otimes S')$ is fully faithful, we have
\begin{align*}
	\Mod_H(I)\otimes\Mod_H(I')&\simeq\fib(\Mod(R\otimes R')\to\Mod(S\otimes R')\times_{\Mod(S\otimes S')}\Mod(R\otimes S'))\\
	&\simeq\fib(\Mod(R\otimes R')\to\Mod((S\otimes R')\times_{S\otimes S'}(R\otimes S')))\\
	&\simeq\Mod_H(I\otimes I').
\end{align*}
\end{remark}
\begin{corollary}
	Let $I,I'$ be connective $H$-unital rings. Then $I\otimes I'$ is also a connective $H$-unital ring. Moreover, we have an equivalence $\Mod_H(I\otimes I')\simeq\Mod_H(I)\otimes\Mod_H(I')$.
\end{corollary}
\begin{proof}
	Consider homological epimorphisms $I^+\to\bbS$ and $(I')^+\to\bbS$ with fibers $I$ and $I'$, respectively. Then $I^+\otimes(I')^+\to(\bbS\otimes(I')^+)\times_{\bbS\otimes\bbS}(I^+\otimes\bbS)\simeq(I')^+\times_{\bbS}I^+$ is a $\pi_0$-surjective homological epimorphism. So by \cref{prop:fib_conn_Hunital}, $I\otimes I'$ is connective $H$-unital. Hence $\Mod_H(I\otimes I')\simeq\Mod_H(I)\otimes\Mod_H(I')$.
\end{proof}

\begin{proposition}\label{prop:tensor-H-unital}
	Let $I$ be an $H$-unital ring, and let $R'\to S'$ be a ring map between idempotent $\bbS$-algebras with fiber $I'=\fib(R'\to S')$. Then $I\otimes I'$ is also an $H$-unital ring. Moreover, we have $\Mod_H(I\otimes I')\simeq\Mod_H(I)\otimes\Mod_H(I')$.
\end{proposition}
\begin{proof}
	Consider the homological epimorphisms $I^+\to\bbS$ and $R'\to S'$ with fibers $I$ and $I'$, respectively. Then it suffices to show that the following pullback diagram of $\bbE_1$-rings induces a pullback diagram of module categories:
	\[\begin{tikzcd}
		{(I\otimes I')^+} & \bbS \\
		{I^+\otimes R'} & {R'\times_{S'}(I^+\otimes S')}
		\arrow[from=1-1, to=1-2]
		\arrow[from=1-1, to=2-1]
		\arrow["\lrcorner"{anchor=center, pos=0.125}, draw=none, from=1-1, to=2-2]
		\arrow[from=1-2, to=2-2]
		\arrow[from=2-1, to=2-2]
	\end{tikzcd}\]
	Applying $I'\otimes-$ to the corresponding pullback diagram of \eqref{eq:pullback} in $\Sp$, we obtain a pullback diagram
	\[\begin{tikzcd}
		0 & {I'\otimes N} \\
		{I'\otimes M} & {I'\otimes(R'\times_{S'}(I^+\otimes S'))\otimes N}
		\arrow[from=1-1, to=1-2]
		\arrow[from=1-1, to=2-1]
		\arrow["\lrcorner"{anchor=center, pos=0.125}, draw=none, from=1-1, to=2-2]
		\arrow[from=1-2, to=2-2]
		\arrow[from=2-1, to=2-2]
	\end{tikzcd}\]
	Since $I'\otimes N\to I'\otimes(R'\times_{S'}(S'\otimes I^+))\otimes N\simeq I'\otimes R'\otimes N$ is an equivalence, we have $I'\otimes M\simeq0$ and hence $S'\otimes M\simeq M$. Observe that in the commutative diagram
	\[\begin{tikzcd}
		M & {(I^+\otimes S')\otimes_{I^+\otimes R'}M} \\
		{R'\otimes_{I^+\otimes R'}M} & {S'\otimes_{I^+\otimes R'}M}
		\arrow[from=1-1, to=1-2]
		\arrow[from=1-1, to=2-1]
		\arrow[from=1-2, to=2-2]
		\arrow[from=2-1, to=2-2]
	\end{tikzcd}\]
	we have $M\simeq(I^+\otimes S')\otimes_{I^+\otimes R'}M$ and $R'\otimes_{I^+\otimes R'}M \simeq S'\otimes_{I^+\otimes R'}M$. Then this is a pullback diagram and $M\simeq(R'\times_{S'}(I^+\otimes S'))\otimes_{I^+\otimes R'}M$, and hence $N\simeq0$ from the pullback diagram \eqref{eq:pullback}.
\end{proof}
One can instead first prove directly that coidempotent $\bbS$-coalgebras are $H$-unital and that their tensor products with $H$-unital rings remain $H$-unital, and then recover \cref{prop:pullback_homoepi_idem,prop:tensor-H-unital}. The Bar construction gives another proof, which we record below. This proof does not establish that the tensor product of two arbitrary $H$-unital rings remains $H$-unital, since $\Delta^{\op}_{\mathrm{inj}}$ is not sifted.
\begin{proof}[An alternative proof of {\cref{prop:tensor-H-unital}}]
	By definition, a nonunital $\bbE_1$-ring $I'$ is $H$-unital if and only if $I'\otimes_{(I')^+}I'\simeq I'$. It suffices to show that the augmented semi-simplicial diagram
	\[\begin{tikzcd}
		{I'} & {I'\otimes I'} & {I'\otimes I'\otimes I'} & \cdots
		\arrow[from=1-2, to=1-1]
		\arrow[shift left, from=1-3, to=1-2]
		\arrow[shift right, from=1-3, to=1-2]
		\arrow[from=1-4, to=1-3]
		\arrow[shift left=2, from=1-4, to=1-3]
		\arrow[shift right=2, from=1-4, to=1-3]
	\end{tikzcd}\]
	is a colimit diagram, as the left Kan extension of this diagram to $\Delta^{\op}$ is the Bar construction computing $I'\otimes_{(I')^+}I'$. Since all its face maps and its augmentation map are induced by the multiplication $I'\otimes I'\to I'$, which is an equivalence, this diagram is equivalent to the constant diagram and hence is a colimit diagram. Similarly, the nonunital $\bbE_1$-ring $I\otimes I'$ is $H$-unital if and only if the augmented semi-simplicial diagram
	\[\begin{tikzcd}
		{I\otimes I'} & {(I\otimes I')\otimes(I\otimes I')} & {(I\otimes I')\otimes(I\otimes I')\otimes(I\otimes I')} & \cdots
		\arrow[from=1-2, to=1-1]
		\arrow[shift left, from=1-3, to=1-2]
		\arrow[shift right, from=1-3, to=1-2]
		\arrow[from=1-4, to=1-3]
		\arrow[shift left=2, from=1-4, to=1-3]
		\arrow[shift right=2, from=1-4, to=1-3]
	\end{tikzcd}\]
	is a colimit diagram. Note that this diagram is equivalent to the augmented semi-simplicial diagram
	\[\begin{tikzcd}
		{I\otimes I'} & {(I\otimes I)\otimes I'} & {(I\otimes I\otimes I)\otimes I'} & \cdots
		\arrow[from=1-2, to=1-1]
		\arrow[shift left, from=1-3, to=1-2]
		\arrow[shift right, from=1-3, to=1-2]
		\arrow[from=1-4, to=1-3]
		\arrow[shift left=2, from=1-4, to=1-3]
		\arrow[shift right=2, from=1-4, to=1-3]
	\end{tikzcd}\]
	which is obtained by tensoring the augmented semi-simplicial diagram for $I$ with $I'$. Since $I$ is $H$-unital and $-\otimes I'$ preserves colimits, it is a colimit diagram. Therefore, $I\otimes I'$ is also $H$-unital.
\end{proof}
\begin{question}
	Does the tensor product of two $H$-unital rings remain $H$-unital?
\end{question}
In \cref{sec:2}, we will show that many common nonunital $\bbE_1$-rings arising in chromatic homotopy theory are $H$-unital.

\subsection{Rigid envelopes and completion theory}\label{sec:1.3}

We review the theory of rigid categories from \cite{Ram26locallyrigid,Ram26freerigid} together with the completion theory of \cite{Zhob}. We then introduce the notion of a map of rigid envelopes and establish its basic properties.
\begin{definition}[{\cite[Definition 9.1.2]{GR17}, \cite[Definition C.1.1]{AGK+22}}]
	Let $\C\to\D\in\CAlg(\PrL)$. 
	\begin{enumerate}
		\item We say $\C\to\D$ is \emph{rigid} if $\C\to\D$ is a $\C$-internal left adjoint and the multiplication $\D\otimes_{\C}\D\to\D$ is a $\D\otimes_{\C}\D$-internal left adjoint.
		\item We say $\C\to\D$ is \emph{locally rigid} if $\D$ is dualizable over $\C$ and the multiplication $\D\otimes_{\C}\D\to\D$ is a $\D\otimes_{\C}\D$-internal left adjoint.
	\end{enumerate}
	Let $\CAlg^{\rig}(\PrL_{\C})\subseteq\CAlg(\PrL_{\C})$ denote the full subcategory spanned by rigid $\C$-algebras.
\end{definition}
See also \cite[Definition 4.2.5]{KNP24} or \cite[Definition 1.3]{Neu26} for the general notion of rigid commutative algebras in a symmetric monoidal $2$-category and its monoidal version.

The rigidity condition precisely ensures that the adjunction $\D\otimes_{\C}-:\PrL_{\C}\rightleftarrows\PrL_{\D}$ restricts to $\D\otimes_{\C}-:\Pr^{\iL}_{\C}\rightleftarrows\Pr^{\iL}_{\D}$. Moreover, if $\D\in\CAlg^{\rig}(\PrL_{\C})$, then $\D\in\Pr^{\dbl}_{\C}$ by \cite[Lemma 4.57]{Ram26locallyrigid}, and hence
\[
	\Pr^{\dbl}_{\D}\simeq\Mod_{\D}(\Pr^{\dbl}_{\C})
\]
by \cite[Corollary 4.49]{Ram26locallyrigid}. In particular, a symmetric monoidal functor $\C\to\D$ is rigid if and only if it is locally rigid and the unit $\mathbbm{1}_{\D}$ is $\C$-atomic.
\begin{example}\label{exm:compact_dbl}
	Let $\C\to\D\in\CAlg(\PrL)$. Then $\mathbbm{1}_{\D}$ is $\C$-atomic if and only if $\D^{\dbl}\subseteq\D^{\C\text{-}\at}$. If $\D\in\Pr^{\at}_{\C}$, then $\D$ is locally rigid if and only if $\D^{\C\text{-}\at}\subseteq\D^{\dbl}$ by \cite[Example 4.6]{Ram26locallyrigid}.

	For $\C=\Sp$ and $\M\in\PrL_{\st}$, we have $\M^{\Sp\text{-}\at}=\M^\omega$ and $\Pr^{\at}_{\Sp}=\Pr^{\cg}_{\st}$. Let $\D\in\CAlg(\PrL_{\st})$. Then $\mathbbm{1}\in\D^{\omega}$ if and only if $\D^{\dbl}\subseteq\D^{\omega}$. If $\D\in\Pr^{\cg}_{\st}$, then $\D$ is locally rigid over $\Sp$ if and only if $\D^\omega\subseteq\D^{\dbl}$. In particular, $\Ind(\D^{\dbl})$ is always rigid over $\Sp$.
\end{example}
\begin{theorem}[{\cite[Corollary 4.73]{Ram26locallyrigid}}]
	Let $\C\in\CAlg(\PrL)$. Then the inclusion of rigid $\C$-algebras admits a right adjoint, namely, the \emph{$\C$-rigidification}
	\[\begin{tikzcd}
		{\CAlg^{\rig}(\PrL_{\C})	} & {\CAlg(\PrL_{\C}) }
		\arrow[shift left, hook, from=1-1, to=1-2]
		\arrow["{(-)^{\rig}_{\C}}", shift left, from=1-2, to=1-1]
	\end{tikzcd}\]
\end{theorem}
\begin{remark}
	Indeed, let $\kappa>\omega$ be such that $\C\to\D\in\CAlg(\Pr^{\kappa})$. The colimit functor $k:\P_{\C}(\D^{\kappa})\to\D$ induces an equivalence $\P_{\C}(\D^{\kappa})^{\rig}_{\C}\simeq\D^{\rig}_{\C}$ by \cite[Lemma 4.71]{Ram26locallyrigid}. The rigidification $\D^{\rig}_{\C}\simeq\P_{\C}(\D^{\kappa})^{\rig}_{\C}\hookrightarrow\P_{\C}(\D^{\kappa})$ is given by the colimit of all rigid full sub-$\C$-algebras of $\P_{\C}(\D^{\kappa})$ by \cite[Lemma 4.72]{Ram26locallyrigid}. In particular, for $\kappa\geq\omega$, $\C=\Sp$ and $\D\in\CAlg(\Pr^{\kappa}_{\st})$, we have $\Ind(\D^{\kappa})^{\rig}_{\st}\simeq\D^{\rig}_{\st}$.
	
	More generally, if $\mathbbm{1}_{\D}$ is $\C$-atomic, then $\D^{\rig}\hookrightarrow\D$ is a fully faithful $\C$-internal left adjoint, which is given by the colimit of rigid full sub-$\C$-algebras of $\D$ by \cite[Corollary 4.75]{Ram26locallyrigid}.
\end{remark}

Let $\C\to\overline{\C}\in\CAlg(\PrL)$ be rigid. Then a $\overline{\C}$-algebra is $\overline{\C}$-rigid if and only if it is $\C$-rigid, and hence $\C$-rigidifications and $\overline{\C}$-rigidifications of $\overline{\C}$-algebras coincide by \cite[Observation 4.66]{Ram26locallyrigid}. 
\begin{definition}
	Let $\C\in\CAlg(\PrL)$ and $\I\subseteq\C$ be a localizing subcategory. We say $\I\subseteq\C$ is a \emph{smashing ideal} if it is an ideal and is a $\C$-internal left adjoint with respect to the $\C$-module structure induced by itself. Equivalently, the right adjoint $\C\to\I$ is a symmetric monoidal localization and $\I\subseteq\C$ is $\C$-linear with respect to the $\C$-module structure induced by the right adjoint.
\end{definition}
\begin{remark}
	Let $\C\in\CAlg(\PrL)$ and $\I\subseteq\C$ be a localizing ideal. If $\C$ is locally rigid over $\overline{\C}$, then $\I\subseteq\C$ is smashing if and only if it is a $\overline{\C}$-internal left adjoint by \cite[Proposition 4.18]{Ram26locallyrigid}. 
	
	Let $\I\subseteq\C$ be a smashing ideal. Then $\I$ canonically admits a coidempotent $\C$-coalgebra structure. So the algebra map $\C\to\I$ is idempotent and hence is locally rigid; see also \cite[Proposition 3.24]{Lia26telescope}.

	More generally, let $\I$ be a locally rigid $\C$-algebra. Then the $\C$-algebra map $\I^{\rig}_{\C}\to\I$ admits a fully faithful left adjoint such that $\I$ is a smashing ideal of $\I^{\rig}_{\C}$ by \cite[Theorem 4.67]{Ram26locallyrigid} or \cite[Corollary 5.2]{Ram26freerigid}.
\end{remark}
Let $\I\subseteq\C$ be a smashing ideal. Then we have $\I\simeq\coMod_I(\C)$ where $I$ is a coidempotent coalgebra in $\C$. Denote $j_!:\I\hookrightarrow\C$, $j^*:=j_!^R$ and $j_*:=(j^*)^R$. We have $\Gamma_I:=j_!j^*\simeq I\otimes-$ and $(-)^{\wedge}_I:=j_*j^*\simeq\iHom_{\C}(I,-)$. This induces an equivalence 
\[
	(-)^{\wedge}_I:\C^{I\text{-}\tors}:=\{X\in\C:I\otimes X\xrightarrow{\simeq}X\}\simeq\C^{I\text{-}\cplt}:=\{X\in\C:X\xrightarrow{\simeq}\iHom_{\C}(I,X)\}:\Gamma_I.
\]
In \cite{Zhob}, Zhou establishes a completion theory for dualizable $\C$-modules, which is also developed in \cite{Efi25rigidity,Ram26freerigid}. Let $\iHom^{\dbl}_{\C}:=\iHom_{\Pr^{\dbl}_{\C}}$ denote the internal Hom.
\begin{definition}[{\cite{Zhob}}]
	Let $\C\in\CAlg(\PrL)$, and let $\I\subseteq\C$ be a smashing ideal. Let $\M\in\Pr^{\dbl}_{\C}$.
	\begin{enumerate}
		\item We say $\M$ is \emph{$\I$-torsion} if $\I\otimes_{\C}\M\to\M$ is an equivalence, and denote $\Pr^{\dbl}_{\C,\I\text{-}\tors}\subseteq\Pr^{\dbl}_{\C}$.
		\item We say $\M$ is \emph{$\I$-complete} if $\M\to\iHom^{\dbl}_{\C}(\I,\M)$ is an equivalence, and denote $\Pr^{\dbl}_{\C,\I\text{-}\cplt}\subseteq\Pr^{\dbl}_{\C}$.
	\end{enumerate}
	We define 
	\[
		\M^{\wedge\I}:=\iHom^{\dbl}_{\C}(\I,\M)
	\]
	as the \emph{$\I$-completion of $\M$}, or the \emph{completion of $\M$ with respect to $\I$}.
\end{definition}
The following theorem identifies torsion modules, complete modules, and dualizable $\I$-modules.
\begin{theorem}[{\cite[Proposition 5.7]{Efi25rigidity}, \cite[Corollary 5.4, Remark 5.6]{Ram26freerigid} or \cite{Zhob}}]\label{thm:unst_completion}
	Let $\C\in\CAlg(\PrL)$, and let $\I\subseteq\C$ be a smashing ideal. Then the forgetful functor $\Pr^{\dbl}_{\I}\to\Pr^{\dbl}_{\C}$ is fully faithful and admits a right adjoint
	\[\begin{tikzcd}[column sep=4em]
		{\Pr^{\dbl}_{\I}} & {\Pr^{\dbl}_{\C}}
		\arrow[shift left=3, hook, from=1-1, to=1-2]
		\arrow["{\iHom^{\dbl}_{\C}(\I,-)}"', shift right=3, hook, from=1-1, to=1-2]
		\arrow["{\I\otimes_{\C}-}"{description}, from=1-2, to=1-1]
	\end{tikzcd}\]
	identifying dualizable $\I$-modules, $\I$-comodules and $\I$-torsion $\C$-modules
	\[
		\Pr^{\dbl}_{\I}\simeq\coMod_{\I}(\Pr^{\dbl}_{\C})\simeq\Pr^{\dbl}_{\C,\I\text{-}\tors},
	\]
	and providing a pair of inverse equivalences
	\[
		\iHom^{\dbl}_{\C}(\I,-):\Pr^{\dbl}_{\C,\I\text{-}\tors}\simeq\Pr^{\dbl}_{\C,\I\text{-}\cplt}:\I\otimes_{\C}-.
	\]

	Moreover, if $\D\in\CAlg^{\rig}(\PrL_{\C})$, then $\I\otimes_{\C}\D$ is locally rigid over $\C$ with $\C$-rigidification
	\[
		(\I\otimes_{\C}\D)^{\rig}_{\C}\simeq\iHom^{\dbl}_{\C}(\I,\D)\simeq\iHom^{\dbl}_{\C}(\I,\I\otimes_{\C}\D).
	\]
	In particular, we have $\I^{\rig}_{\C}\simeq\iHom^{\dbl}_{\C}(\I,\C)\simeq\iHom^{\dbl}_{\C}(\I,\I)$.
\end{theorem}
\begin{corollary}\label{cor:completion}
	Let $\C\in\CAlg(\PrL)$ and $\I\subseteq\C$ be a smashing ideal. Let $\M\in\Pr^{\dbl}_{\C}$. Then
	\begin{enumerate}
		\item The $\I$-completion functor $\iHom^{\dbl}_{\C}(\I,-):\Pr^{\dbl}_{\C}\to\Pr^{\dbl}_{\C}$ is lax symmetric monoidal, and preserves rigid $\C$-algebras.
		\item The inclusion $\I\subseteq\C$ induces a commutative diagram that is both horizontally and vertically right adjointable
		\[\begin{tikzcd}
			{\I\otimes_{\C}\M} & \M \\
			{\I\otimes_{\C}\M^{\wedge\I}} & {\M^{\wedge\I}}
			\arrow[hook, from=1-1, to=1-2]
			\arrow["\simeq"', from=1-1, to=2-1]
			\arrow[from=1-2, to=2-2]
			\arrow[hook, from=2-1, to=2-2]
		\end{tikzcd}\]
	\end{enumerate}
\end{corollary}
We further study the functoriality of completion theory. Adapting the terminology of \cite[Definition 4.3.14]{KNP24}, we call the following objects rigid envelopes and define maps between them.
\begin{definition}
	Let $\C\in\CAlg(\PrL)$ and $\I\in\CAlg(\PrL_{\C})$. We say a $\C$-algebra map $\overline{\C}\to\I$ is a \emph{$\C$-rigid envelope} of $\I$ if $\overline{\C}$ is rigid over $\C$ and $\overline{\C}\to\I$ admits a fully faithful $\overline{\C}$-linear left adjoint. Equivalently, a $\C$-rigid envelope of $\I$ is a rigid $\C$-algebra $\overline{\C}$ with a fully faithful functor $\I\subseteq\overline{\C}$ that is a smashing ideal.

	Let $\overline{\C},\widetilde{\C}$ be $\C$-rigid envelopes of $\I$. We say a $\C$-algebra map $\overline{\C}\to\widetilde{\C}$ over $\I$ is a \emph{map of $\C$-rigid envelopes} of $\I$ if $\I\to\I\otimes_{\overline{\C}}\widetilde{\C}$ is an equivalence.
\end{definition}
The notion of a $\C$-rigid envelope also appears in \cite[Definition 3.26]{Lia26telescope} under the name \emph{locally rigid localization}.
\begin{remark}
	Since smashing ideals are locally rigid over the ambient category, a $\C$-algebra admits a rigid envelope if and only if it is locally rigid over $\C$. Given a $\C$-rigid envelope $\overline{\C}\to\I$, by \cref{thm:unst_completion} the $\C$-rigidification of $\I$ can be computed by 
	\[
		\I^{\rig}_{\C}\simeq\I^{\rig}_{\overline{\C}}\simeq\iHom^{\dbl}_{\overline{\C}}(\I,\overline{\C}).
	\]
	Then the $\I$-completion functor $\overline{\C}\to\I^{\rig}_{\C}$ is a map of $\C$-rigid envelopes of $\I$, and $\I\hookrightarrow\overline{\C}\to\I^{\rig}_{\C}$ is identified with the fully faithful left adjoint $\I\hookrightarrow\I^{\rig}_{\C}$ by \cref{cor:completion}.
\end{remark}
\begin{proposition}\label{prop:map_of_rigid_envelopes}
	Let $\C\in\CAlg(\PrL)$ and $\I\in\CAlg(\PrL_{\C})$ be locally rigid.
	\begin{enumerate}
		\item Let $j^*:\overline{\C}\to\I$ be a $\C$-rigid envelope of $\I$ and $f:\overline{\C}\to\widetilde{\C}$ be a $\C$-algebra map over $\I$ with $\widetilde{\C}$ rigid. If $\I\to\I\otimes_{\overline{\C}}\widetilde{\C}$ is an equivalence, then $\overline{\C}\to\widetilde{\C}$ is a map of $\C$-rigid envelopes of $\I$. In particular, this holds whenever $\overline{\C}\to\widetilde{\C}$ is a localization or $\widetilde{\C}\to\I^{\rig}_{\C}$ is fully faithful.
		\item Let $\overline{\C}\to\widetilde{\C}$ be a map of $\C$-rigid envelopes of $\I$ and let $\M\in\Pr^{\dbl}_{\widetilde{\C}}$. Then
		\[
			\iHom^{\dbl}_{\overline{\C}}(\I,\M)\simeq\iHom^{\dbl}_{\widetilde{\C}}(\I,\M).
		\]
	\end{enumerate}
\end{proposition}
\begin{proof}
	For the first part, consider the horizontally right adjointable diagrams
	\[\begin{tikzcd}
		\I & {\overline{\C}} \\
		{\I\otimes_{\overline{\C}}\widetilde{\C}} & {\widetilde{\C}} \\
		{\I\otimes_{\overline{\C}}\I} & \I
		\arrow[shift left, hook, from=1-1, to=1-2]
		\arrow[from=1-1, to=2-1]
		\arrow[shift left, from=1-2, to=1-1]
		\arrow[from=1-2, to=2-2]
		\arrow[shift left, hook, from=2-1, to=2-2]
		\arrow[from=2-1, to=3-1]
		\arrow[shift left, from=2-2, to=2-1]
		\arrow[from=2-2, to=3-2]
		\arrow[shift left, hook, from=3-1, to=3-2]
		\arrow["\simeq", shift left, from=3-2, to=3-1]
	\end{tikzcd}\]
	Since the composition $\I\to\I\otimes_{\overline{\C}}\widetilde{\C}\to\I\otimes_{\overline{\C}}\I\to\I$ is the identity, the right adjoint of $\I\to\I\otimes_{\overline{\C}}\widetilde{\C}$ coincides with the map $\I\otimes_{\overline{\C}}\widetilde{\C}\to\I\otimes_{\overline{\C}}\I\to\I$. Therefore, $\I\to\I\otimes_{\overline{\C}}\widetilde{\C}\hookrightarrow\widetilde{\C}$ is the left adjoint of $\widetilde{\C}\to\I$. In particular, since the composition $\I\to\I\otimes_{\overline{\C}}\widetilde{\C}\to\I\otimes_{\overline{\C}}\I^{\rig}_{\C}$ is an equivalence, $\I\to\I\otimes_{\overline{\C}}\widetilde{\C}$ is an equivalence whenever $\overline{\C}\to\widetilde{\C}$ is a localization or $\widetilde{\C}\to\I^{\rig}_{\C}$ is fully faithful.
	
	The second part follows from $\iHom_{\overline{\C}}^{\dbl}(\I,\M)\simeq\iHom_{\widetilde{\C}}^{\dbl}(\I\otimes_{\overline{\C}}\widetilde{\C},\M)\simeq\iHom_{\widetilde{\C}}^{\dbl}(\I,\M)$.
\end{proof}
Let $\C\in\CAlg(\PrL)$ and $\I\in\CAlg(\PrL_{\C})$ be locally rigid. For $\M\in\Pr^{\dbl}_{\I}$, we simply denote
\[
	\M^{\wedge\I}:=\iHom_{\I^{\rig}_{\C}}^{\dbl}(\I,\M).
\]
This notation is independent of the choice of a $\C$-rigid envelope. Indeed, for any $\C$-rigid envelope $\overline{\C}\to\I$, \cref{prop:map_of_rigid_envelopes} implies that $\iHom_{\I^{\rig}_{\C}}^{\dbl}(\I,\M)\simeq\iHom_{\overline{\C}}^{\dbl}(\I,\M)$.
\begin{proposition}\label{prop:corig}
	Let $\C\in\CAlg(\PrL)$ and $\I\in\CAlg(\PrL_{\C})$.
	\begin{enumerate}
		\item Suppose that $\I$ is corigid\footnote{Rigid commutative algebras in $\Pr^{\mathrm{L},\mathrm{co}}_{\C}$, which is the $2$-opposite $2$-category of $\PrL_{\C}$. The term ``corigid" was suggested by Jiacheng Liang.}, i.e. the unit $\C\to\I$ admits a $\C$-linear left adjoint, and the multiplication $\I\otimes_{\C}\I\to\I$ admits an $\I\otimes_{\C}\I$-linear left adjoint. Then the left adjoint of $\I\otimes_{\C}-:\Pr^{\dbl}_{\C}\to\Pr^{\dbl}_{\I}$ is the forgetful functor.
		\item Suppose that $\I$ is locally rigid. Then the right adjoint of $\I\otimes_{\C}-:\Pr^{\dbl}_{\C}\to\Pr^{\dbl}_{\I}$ is the $\I$-completion functor $(-)^{\wedge\I}$.
	\end{enumerate}
	In particular, if $\I$ is corigid and locally rigid, then we have adjunctions
	\[\begin{tikzcd}[column sep=4em]
		{\Pr^{\dbl}_{\I}} & {\Pr^{\dbl}_{\C}}
		\arrow[shift left=3, from=1-1, to=1-2]
		\arrow["{(-)^{\wedge\I}}"', shift right=3, from=1-1, to=1-2]
		\arrow["{\I\otimes_{\C}-}"{description}, from=1-2, to=1-1]
	\end{tikzcd}\]
\end{proposition}
\begin{proof}
	\begin{enumerate}
		\item Passing to left adjoints, the adjunction between $\I\otimes_{\C}-:\PrL_{\C}\to\PrL_{\I}$ and the forgetful functor induces an adjunction between the forgetful functor and $\I\otimes_{\C}-:\Pr^{\iL}_{\C}\to\Pr^{\iL}_{\I}$. By the dual version of \cite[Lemma 4.57]{Ram26locallyrigid}, $\I$ is dualizable over $\C$. Thus, this adjunction restricts to the adjunction between the forgetful functor and $\I\otimes_{\C}-:\Pr^{\dbl}_{\C}\to\Pr^{\dbl}_{\I}$.
		\item Since the right adjoint of $\I^{\rig}_{\C}\otimes_{\C}-:\Pr^{\dbl}_{\C}\to\Pr^{\dbl}_{\I^{\rig}_{\C}}$ is the forgetful functor, by \cref{thm:unst_completion} the right adjoint of $\I\otimes_{\C}-\simeq\I\otimes_{\I^{\rig}_{\C}}\I^{\rig}_{\C}\otimes_{\C}-$ is the $\I$-completion functor $(-)^{\wedge\I}$. \qedhere
	\end{enumerate}
\end{proof}
We generalize \cite[Lemma 5.10]{Efi25rigidity} to obtain a criterion for determining when a rigid envelope embeds into the rigidification.
\begin{lemma}\label{lem:1_rig}
	Let $\C\in\CAlg(\PrL)$ and $\I\in\CAlg(\PrL_{\C})$. The right adjoint of $\I^{\rig}_{\C}\to\I$ sends $\mathbbm{1}_{\I}$ to $\mathbbm{1}_{\I^{\rig}_{\C}}$, and induces an equivalence
	\[
		\iEnd_{\C}(\mathbbm{1}_{\I^{\rig}_{\C}})\simeq\iEnd_{\C}(\mathbbm{1}_{\I}).
	\]

	Suppose that $\I$ is locally rigid. Let $j^*:\overline{\C}\to\I$ be a $\C$-rigid envelope with $j_!:=(j^*)^L$, $j_*:=(j^*)^R$. Then
	\begin{enumerate}
		\item The canonical map $f:\overline{\C}\to\I^{\rig}_{\C}$ is fully faithful if and only if $\mathbbm{1}_{\overline{\C}}\to j_*\mathbbm{1}_{\I}$ is an equivalence. Consequently, we always have maps of rigid envelopes $\overline{\C}\to\Mod_{j_*\mathbbm{1}_{\I}}(\overline{\C})\hookrightarrow\I^{\rig}_{\C}$.
		\item If $\overline{\C}\to\I^{\rig}_{\C}$ is fully faithful, let $\widetilde{\C}$ be a $\C$-rigid envelope of $\I$ and $\widetilde{\C}\to\overline{\C}$ be a $\C$-algebra map over $\I^{\rig}_{\C}$. Then $\widetilde{\C}\to\overline{\C}$ is a map of $\C$-rigid envelopes of $\I$.
	\end{enumerate}
\end{lemma}
\begin{proof}
	Let $\kappa>\omega$ be such that $\C\in\CAlg(\Pr^{\kappa})$ and $\I\in\CAlg(\Pr^{\kappa}_{\C})$. Then $\I^{\rig}_{\C}\to\I$ is given by the composite $\I^{\rig}_{\C}\hookrightarrow\P_{\C}(\I^{\kappa})\xrightarrow{k}\I$. It follows that the right adjoint sends $\mathbbm{1}_{\I}$ to its Yoneda image $j(\mathbbm{1}_{\I})$, and then to $\mathbbm{1}_{\I^{\rig}_{\C}}$.

	Suppose that $\I$ is locally rigid. Since $f^Rf\simeq f^R\mathbbm{1}_{\I^{\rig}}\otimes-$, $f$ is fully faithful if and only if $\mathbbm{1}_{\overline{\C}}\to f^R\mathbbm{1}_{\I^{\rig}}$ is an equivalence. Then the first assertion follows from $f^R\mathbbm{1}_{\I^{\rig}}\simeq j_*\mathbbm{1}_{\I}$ since $j_*$ is the composition of the right adjoint of $\I^{\rig}\to\I$ and $f^R$. Since $\overline{\C}\to\I$ factors through $\overline{\C}\to\Mod_{j_*\mathbbm{1}_{\I}}(\overline{\C})\to\I$ and $\Mod_{j_*\mathbbm{1}_{\I}}(\overline{\C})\to\I$ admits a fully faithful left adjoint, $\Mod_{j_*\mathbbm{1}_{\I}}(\overline{\C})$ is a rigid envelope of $\I$. Therefore, the first criterion applied to $\Mod_{j_*\mathbbm{1}_{\I}}(\overline{\C})$ shows that $\Mod_{j_*\mathbbm{1}_{\I}}(\overline{\C})\hookrightarrow\I^{\rig}_{\C}$ is fully faithful and hence $\overline{\C}\to\Mod_{j_*\mathbbm{1}_{\I}}(\overline{\C})$ is a map of rigid envelopes by \cref{prop:map_of_rigid_envelopes}.
	
	The last assertion follows from \cref{prop:map_of_rigid_envelopes}.
\end{proof}

In the stable case, we can not only define a categorical analogue of completion, but also that of the generic fiber. Let $\A:=\C/\I\simeq\Mod_A(\C)$ where $A$ is an idempotent algebra in $\C$. Denote $i^*:\C\to\A$, $i_*:=(i^*)^R$ and $i^!:=i_*^R$. We have $L_I:=i_*i^*\simeq A\otimes-$ and $\Delta_I:=\iHom_{\C}(A,-)$. For $\C\in\CAlg(\PrL_{\st})$, we simply say that $\C$ is locally rigid if it is locally rigid over $\Sp$, and $\C$ is rigid if it is rigid over $\Sp$.
\begin{definition}[{\cite{Zhob}}]
	Let $\C\in\CAlg(\PrL_{\st})$ and $\I\subseteq\C$ be a smashing ideal with $\A:=\C/\I$. Let $\M\in\Pr^{\dbl}_{\C}$. We define 
	\[\M_{\eta}:=\A\otimes_{\C}\M^{\wedge\I}\]
	as the \emph{$\I$-generic fiber of $\M$} or the \emph{generic fiber of $\M$ with respect to $\I$}.
\end{definition}
\begin{corollary}\label{cor:generic_fiber}
	Let $\C\in\CAlg(\PrL_{\st})$ and $\I\subseteq\C$ be a smashing ideal. Let $\M\in\Pr^{\dbl}_{\C}$. Then
	\begin{enumerate}
		\item The $\I$-generic fiber functor $(-)_{\eta}:\Pr^{\dbl}_{\C}\to\Pr^{\dbl}_{\C}$ is lax symmetric monoidal, and preserves rigid $\C$-algebras.
		\item The short exact sequence $\I\hookrightarrow\C\to\A$ induces a commutative diagram of short exact sequences in $\Pr^{\dbl}_{\C}$ 
		\[\begin{tikzcd}
			{\I\otimes_{\C}\M} & \M & {\A\otimes_{\C}\M} \\
			{\I\otimes_{\C}\M^{\wedge\I}} & {\M^{\wedge\I}} & {\M_{\eta}}
			\arrow[hook, from=1-1, to=1-2]
			\arrow["\simeq"', from=1-1, to=2-1]
			\arrow[from=1-2, to=1-3]
			\arrow[from=1-2, to=2-2]
			\arrow[from=1-3, to=2-3]
			\arrow[hook, from=2-1, to=2-2]
			\arrow[from=2-2, to=2-3]
		\end{tikzcd}\]
	\end{enumerate}
\end{corollary}
In fact, by \cref{lem:ses_pullback} the right square in the commutative diagram in \cref{cor:generic_fiber}(2) is a pullback diagram. This observation will play an important role in \cref{sec:2}.

Let $\C\in\CAlg(\PrL_{\st})$ and $\I\in\CAlg(\PrL_{\C})$ be locally rigid. For $\M\in\Pr^{\dbl}_{\I}$, we simply denote
\[
	\M_{\eta}:=(\I^{\rig}_{\C}/\I)\otimes_{\I^{\rig}_{\C}}\iHom_{\I^{\rig}_{\C}}^{\dbl}(\I,\M).
\]
This notation is also independent of the choice of a $\C$-rigid envelope. Indeed, for any $\C$-rigid envelope $\overline{\C}\to\I$, we have $\iHom_{\I^{\rig}_{\C}}^{\dbl}(\I,\M)\simeq\iHom_{\overline{\C}}^{\dbl}(\I,\M)$ and $\I\otimes_{\I^{\rig}_{\C}}\M\simeq\I\otimes_{\overline{\C}}\I^{\rig}_{\C}\otimes_{\I^{\rig}_{\C}}\M\simeq\I\otimes_{\overline{\C}}\M$. Therefore, $(\I^{\rig}_{\C}/\I)\otimes_{\I^{\rig}_{\C}}\iHom_{\I^{\rig}_{\C}}^{\dbl}(\I,\M)\simeq(\overline{\C}/\I)\otimes_{\overline{\C}}\iHom^{\dbl}_{\overline{\C}}(\I,\M)$.

We conclude this section with the $\bbE_\infty$-version of $H$-unital rings, which provide a basic class of locally rigid categories equipped with a canonical rigid envelope.
\begin{definition}
	Let $I$ be a nonunital $\bbE_\infty$-ring. We say $I$ is an \emph{$H$-unital $\bbE_\infty$-ring} if it is $H$-unital as a nonunital $\bbE_1$-ring.
\end{definition}
\begin{example}
	Let $I$ be an $H$-unital $\bbE_\infty$-ring. Since $\Mod_H(I)\subseteq\Mod(I^+)$ is a smashing ideal and $\Mod(I^+)$ is rigid, $\Mod_H(I)$ is locally rigid with rigid envelope $\Mod(I^+)$, and its rigidification is given by $\Mod_H(I)^{\rig}\simeq\iHom^{\dbl}_{I^+}(\Mod_H(I),\Mod(I^+))$. By \cref{lem:1_rig}, we have a map of rigid envelopes $\Mod(\iHom_{I^+}(I,I^+))\hookrightarrow\Mod_H(I)^{\rig}$. The completion theory with respect to $\Mod_H(I)$ is given as follows: for $\M\in\Pr^{\dbl}_{I^+}$, its $I$-completion and $I$-generic fiber are
	\[
		\M^{\wedge I}:=\iHom_{I^+}^{\dbl}(\Mod_H(I),\M),\quad\M_{\eta}:=\Sp\otimes_{I^+}\iHom_{I^+}^{\dbl}(\Mod_H(I),\M).
	\]
\end{example}
\begin{example}
	More generally, let $R\to S$ be a homological epimorphism of $\bbE_\infty$-rings with fiber $I:=\fib(R\to S)$. The fiber category $\Mod(R,I)$ is locally rigid, with rigid envelope $\Mod(R)$, and its rigidification is given by $\Mod(R,I)^{\rig}\simeq\iHom^{\dbl}_R(\Mod(R,I),\Mod(R))$. By \cref{lem:1_rig}, we also have a map of rigid envelopes $\Mod(\iHom_R(I,R))\hookrightarrow\Mod(R,I)^{\rig}$. The completion theory with respect to $\Mod(R,I)$ is given as follows: for $\M\in\Pr^{\dbl}_R$, the $I$-completion and $I$-generic fiber are
	\[
		\M^{\wedge I}:=\iHom^{\dbl}_R(\Mod(R,I),\M),\quad\M_{\eta}:=\Mod(S)\otimes_R\iHom^{\dbl}_R(\Mod(R,I),\M).
	\]

	Suppose now that $I$ is $H$-unital. We show that the completion theory with respect to $\Mod(R,I)$ is recovered by that with respect to $\Mod_H(I)$, providing a universal property of the completion theory with respect to an $H$-unital ring. By \cref{prop:Tam_pullback_homoepi}, we have a symmetric monoidal functor $\Mod(I^+)\to\Mod(R)$, and an equivalence $\Mod(S)\simeq\Sp\otimes_{I^+}\Mod(R)$. Then $\Mod(I^+)\to\Mod(R)$ is a map of rigid envelopes of $\Mod_H(I)$, since $\Mod_H(I)\simeq\Mod(R,I)\simeq\fib(\Mod(R)\to\Mod(S))\simeq\Mod_H(I)\otimes_{I^+}\Mod(R)$. Hence, for $\M\in\Pr^{\dbl}_R$, by \cref{prop:map_of_rigid_envelopes} we have $\iHom^{\dbl}_{I^+}(\Mod_H(I),\M)\simeq\iHom^{\dbl}_R(\Mod(R,I),\M)$, and $\Sp\otimes_{I^+}\iHom^{\dbl}_{I^+}(\Mod_H(I),\M)\simeq\Mod(S)\otimes_R\iHom^{\dbl}_R(\Mod(R,I),\M)$. Thus, the two notions of $I$-completion and $I$-generic fiber coincide.
\end{example}
\begin{remark}
	Let $A$ be an idempotent $\bbS$-algebra. Then $I:=\fib(\bbS\to A)$ is a coidempotent $\bbS$-coalgebra and hence an $H$-unital $\bbE_\infty$-ring. Therefore, $\Mod_H(I)\simeq\coMod(I)$ is locally rigid with rigid envelope $\Sp$. 
	
	More generally, for $\C\in\CAlg(\PrL_{\st})$, one can also define the notion of \emph{$H$-unital rings relative to $\C$} for nonunital $\bbE_1$-rings in $\C$. Let $I$ be an $H$-unital $\bbE_\infty$-ring relative to $\C$. The \emph{category of $H$-unital $I$-modules} $\Mod_{H,I}(\C)$ is locally rigid over $\C$. Let $\overline{\C}\in\CAlg(\PrL_{\C})$. Then every smashing ideal $\I\subseteq\overline{\C}$ is equivalent to the comodule category of a coidempotent coalgebra $I$, where $I$ is therefore an $H$-unital $\bbE_\infty$-ring relative to $\overline{\C}$. Thus, every locally rigid $\C$-algebra $\I$ is equivalent to the category of $H$-unital modules over an $H$-unital $\bbE_\infty$-ring $I$ relative to a rigid envelope $\overline{\C}$ (e.g. $\I^{\rig}_{\C}$), i.e.
	\[
		\I\simeq\coMod_I(\overline{\C})\simeq\Mod_{H,I}(\overline{\C}).
	\]

	In particular, let $R\to S$ be a homological epimorphism of $\bbE_\infty$-rings with fiber $I:=\fib(R\to S)$. Then $\Mod(R,I)$ is equivalent to the category of $H$-unital modules over an $H$-unital $\bbE_\infty$-ring $I$ relative to a rigid envelope $\Mod(R)$, i.e. $\Mod(R,I)\simeq\Mod_{H,I}(\Mod(R))$.
\end{remark}

\section{Chromatic layers}\label{sec:2}

In this section, we review the theory of chromatic layers and apply the completion theory developed in \cref{sec:1.3} to construct fracture squares for dualizable categories and their chromatic counterparts.

\subsection{Localizations and fracture squares}\label{sec:2.1}

We introduce the theory of $E$-localizations relative to a base symmetric monoidal category, extending the classical theory of Bousfield localizations (see, e.g. \cite{Man24}). This framework unifies several existing constructions, including complete modules in \cite{SAG} and stable height in \cite{CSY21height}. After reviewing local duality from \cite{BHV18,NPR24}, we establish their basic properties and formulate the theory of fracture squares of localization functors, relating them to short exact sequences of categories.

\begin{definition}
	Let $\C\in\CAlg(\PrL_{\st})$ and $\M\in\PrL_{\C}$, and let $E\in\C$.
	\begin{enumerate}
		\item We define $C_E\M:=\fib(E\otimes-)\subseteq\M$ as the full subcategory spanned by the objects $X$ such that $E\otimes X\simeq0$, and call its objects \emph{$E$-acyclic}.
		\item We define $L_E\M:=(C_E\M)^{\perp}\subseteq\M$ (also denoted by $\M_E$) as the full subcategory spanned by the objects right orthogonal to $C_E\M$, and call its objects \emph{$E$-local}.
	\end{enumerate}
\end{definition}

Since $E\otimes-:\M\to\M$ is a $\C$-linear left adjoint with right adjoint $\iHom_{\M}(E,-):\M\to\M$, $C_E\M\subseteq\M$ is also a $\C$-linear left adjoint with right adjoint $C_E:\M\to C_E\M$, and $L_E\M\subseteq\M$ admits a $\C$-linear left adjoint $L_E:\M\to L_E\M$ which is equivalent to $\M\to\M[\{X\to Y\in\M:E\otimes X\xrightarrow{\simeq} E\otimes Y\}^{-1}]$ (e.g. by the theory of Bousfield localization in \cite[\Section 5.5.4]{HTT} or \cite[\Section 2.2]{MNN17}). So we have a short exact sequence in $\PrL_{\C}$:
\[\begin{tikzcd}
	{C_E\M} & \M & {L_E\M}
	\arrow[shift left, hook, from=1-1, to=1-2]
	\arrow["{C_E}", shift left, from=1-2, to=1-1]
	\arrow["{L_E}", shift left, from=1-2, to=1-3]
	\arrow[shift left, hook', from=1-3, to=1-2]
\end{tikzcd}\]
\begin{remark}
	Let $f:\C\to\D\in\CAlg(\PrL)$ and $E\in\C$. Then we have $L_{f(E)}\D\simeq L_E\D$. Indeed, $L_{f(E)}\D=\{X\in\D:\Hom_{\D}(Y,X)\simeq0\text{ for all }Y\in C_{f(E)}\D\}=\{X\in\D:\Hom_{\D}(Y,X)\simeq0\text{ for all }Y\in C_E\D\}=L_E\D$. More generally, let $\M\in\PrL_{\D}$. Then we have $L_E\M\simeq L_{f(E)}\M$ when viewing $\M$ as a $\C$-module.
\end{remark}
For a set of objects $\E\subseteq\C$, we denote $C_{\E}\M:=\bigcap_{E\in\E}\fib(E\otimes-)$ and $L_{\E}\M:=(C_{\E}\M)^{\perp}\simeq\M[\{X\to Y\in\M:E\otimes X\xrightarrow{\simeq}E\otimes Y\text{ for all }E\in\E\}^{-1}]$.
\begin{definition}
	Let $\C\in\CAlg(\PrL)$, and let $L:\C\to L\C$ be a symmetric monoidal localization. We say $L:\C\to L\C$ is a \emph{smashing localization} if it is a $\C$-internal left adjoint. Equivalently, a smashing localization $L:\C\to L\C$ is equivalent to $L\mathbbm{1}\otimes-:\C\to\Mod_{L\mathbbm{1}}(\C)$ with $L\mathbbm{1}$ an idempotent algebra in $\C$.
\end{definition}
\begin{example}
	Let $\C\in\CAlg(\PrL_{\st})$. Then a localizing ideal $\I\subseteq\C$ is a smashing ideal if and only if $\C\to\C/\I$ is a smashing localization.
	
	Let $\C\in\CAlg(\PrL_{\st})$ and $E\in\C$. Then $C_E\C\subseteq\C$ is a localizing ideal and $L_E:\C\to L_E\C$ is a symmetric monoidal localization. We say $E\in\C$ is \emph{smashing} if $\C\to L_E\C$ is smashing, or equivalently, $C_E\C\subseteq\C$ is smashing. If $E$ is smashing, then $C_E\C\hookrightarrow\C\to L_E\C$ is equivalent to the short exact sequence in $\Pr^{\dbl}_{\C}$:
	\[\begin{tikzcd}[column sep=4em]
		{\coMod_{C_E\mathbbm{1}}(\C)} & \C & {\Mod_{L_E\mathbbm{1}}(\C)}
		\arrow[shift left=3, hook, from=1-1, to=1-2]
		\arrow["{\iHom_{\C}(C_E\mathbbm{1},-)}"', shift right=3, hook, from=1-1, to=1-2]
		\arrow["{C_E\mathbbm{1}\otimes-}"{description}, from=1-2, to=1-1]
		\arrow["{L_E\mathbbm{1}\otimes-}", shift left=3, from=1-2, to=1-3]
		\arrow["{\iHom_{\C}(L_E\mathbbm{1},-)}"', shift right=3, from=1-2, to=1-3]
		\arrow[hook', from=1-3, to=1-2]
	\end{tikzcd}\]
	which is also identified with $C_{L_E\mathbbm{1}}\C\hookrightarrow\C\to L_{L_E{\mathbbm{1}}}\C$.
\end{example}

Similarly, $\fib(\iHom_{\M}(E,-))\subseteq\M$ (also denoted by $\M^E$) admits a $\C$-linear left adjoint and then we can establish the theory of $\iHom_{\M}(E,-)$-colocalization $^{\perp}(\M^E)$. More generally, let $\K\subseteq\M$ be a set of objects. We recall the formalism of local duality.

\begin{definition}
	Let $\C\in\CAlg(\PrL_{\st})$ and $\M\in\PrL_{\C}$, and let $\K\subseteq\M$.
	\begin{enumerate}
		\item We define $\M^{\K\text{-}\tors}:=\Loc^\otimes(\K)\subseteq\M$ as the localizing $\C$-submodule generated by $\K$, and call its objects \emph{$\K$-torsion}.
		\item We define $\M^{\K\text{-}\loc}:=(\M^{\K\text{-}\tors})^{\perp}\subseteq\M$ as the full subcategory spanned by the objects right orthogonal to $\M^{\K\text{-}\tors}$, and call its objects \emph{$\K$-local}.
		\item We define $\M^{\K\text{-}\cplt}:=(\M^{\K\text{-}\loc})^{\perp}\subseteq\M$ as the full subcategory spanned by the objects right orthogonal to $\M^{\K\text{-}\loc}$, and call its objects \emph{$\K$-complete}.
	\end{enumerate}
\end{definition}

For any object $K\in\K$, $K\otimes-:\C\to\M$ is a $\C$-linear left adjoint with right adjoint $\iHom_{\M}(K,-):\M\to\C$. Then $i_{\tors}:\M^{\K\text{-}\tors}\hookrightarrow\M$ is also a $\C$-linear left adjoint with right adjoint $\Gamma^{\K}:\M\to\M^{\K\text{-}\tors}$, and $i_{\loc}:\M^{\K\text{-}\loc}\hookrightarrow\M$ admits a $\C$-linear left adjoint $L^{\K}:\M\to\M^{\K\text{-}\loc}$ which is equivalent to $\M\to\M[\{X\otimes K\to 0:K\in\K,X\in\C\}^{-1}]$. So we have a short exact sequence in $\PrL_{\C}$:
\[\begin{tikzcd}
	{\M^{\K\text{-}\tors}} & \M & {\M^{\K\text{-}\loc}}
	\arrow["{i_{\tors}}", shift left, hook, from=1-1, to=1-2]
	\arrow["{\Gamma^{\K}}", shift left, from=1-2, to=1-1]
	\arrow["{L^{\K}}", shift left, from=1-2, to=1-3]
	\arrow["{i_{\loc}}", shift left, hook', from=1-3, to=1-2]
\end{tikzcd}\]
By \cite[Lemma 4.2]{NPR24}, $\M^{\K\text{-}\loc}\simeq\bigcap_{K\in\K}\fib(\iHom_{\M}(K,-))$. Let $\C^0\subseteq\C$ be a set of generators. Then $\M^{\K\text{-}\tors}\simeq\Loc(\K\otimes\C^0)$ is the localizing stable subcategory generated by $\K\otimes\C^0$. Let $E\in\C$ be an object and let $\M^0\subseteq\M$ be a set of generators. Then $\M^{E\otimes\M^0\text{-}\loc}\simeq\fib(\iHom_{\M}(E,-))$ and this recovers the theory of $\iHom_{\M}(E,-)$-colocalization.

Moreover, if every object $K\in\K$ is $\C$-atomic, i.e. $K\otimes-:\C\to\M$ is a $\C$-internal left adjoint, then $i_{\loc}:\M^{\K\text{-}\loc}\hookrightarrow\M$ admits a right adjoint $\Delta^{\K}:\M\to\M^{\K\text{-}\loc}$ and $i_{\cplt}:\M^{\K\text{-}\cplt}\hookrightarrow\M$ admits a left adjoint $\Lambda^{\K}:\M\to\M^{\K\text{-}\cplt}$. This leads to the local duality.

\begin{proposition}[Local duality, {\cite[Proposition 4.3, Lemma 4.6]{NPR24}}]\label{prop:local_duality}
	Let $\C\in\CAlg(\PrL_{\st})$ and $\M\in\PrL_{\C}$, and let $\K\subseteq\M^{\C\text{-}\at}$. Then we have a short exact sequence in $\Pr^{\iL}_{\C}$:
	\[\begin{tikzcd}[column sep=4em]
		{\M^{\K\text{-}\tors}} & \M & {\M^{\K\text{-}\loc}}
		\arrow["{i_{\tors}}", shift left=3, hook, from=1-1, to=1-2]
		\arrow["{\Lambda^{\K}}"', shift right=3, hook, from=1-1, to=1-2]
		\arrow["{\Gamma^{\K}}"{description}, from=1-2, to=1-1]
		\arrow["{\Delta^{\K}}"', shift right=3, from=1-2, to=1-3]
		\arrow["{L^{\K}}", shift left=3, from=1-2, to=1-3]
		\arrow["{i_{\loc}}"{description}, hook', from=1-3, to=1-2]
	\end{tikzcd}\]
	yielding inverse equivalences: 
	\[
		\Lambda^{\K}i_{\tors}:\M^{\K\text{-}\tors}\simeq\M^{\K\text{-}\cplt}:\Gamma^{\K}i_{\cplt}.
	\]
\end{proposition}
\begin{example}
	There is another form of local duality arising from coidempotent coalgebras, which does not require atomicity and underlies the completion theory developed in \cref{sec:1.3}.

	Let $\C\in\CAlg(\PrL)$ and $I$ be a coidempotent coalgebra in $\C$. We denote $\C^{I\text{-}\tors}:=\{X\in\C:I\otimes X\simeq X\}$ and $\C^{I\text{-}\cplt}:=\{X\in\C:X\simeq\iHom_{\C}(I,X)\}$. Then $\C^{I\text{-}\tors}\simeq\coMod_I(\C)\hookrightarrow\C$ is a smashing ideal, and it induces inverse equivalences $\iHom_{\C}(I,-):\C^{I\text{-}\tors}\simeq\C^{I\text{-}\cplt}:I\otimes-$.

	Suppose moreover that $\C\in\CAlg(\PrL_{\st})$ and let $A:=\mathbbm{1}/I$ be the associated idempotent algebra. We denote $\C^{I\text{-}\loc}:=\C/\C^{I\text{-}\tors}\simeq\Mod_A(\C)$. Then $A\otimes-:\C\to\C^{I\text{-}\loc}$ is a smashing localization. Moreover, when $I$ is viewed simply as an object of $\C$, these notions agree with the preceding notions, since $\C^{I\text{-}\tors}\simeq\Loc^{\otimes}(I)$. Thus we have identifications $C_I\C=\C^{I\text{-}\loc}\subseteq\C$ and $L_I\C=\C^{I\text{-}\cplt}\subseteq\C$.
\end{example}
\begin{example}\label{exm:local_duality}
	Let $\C\in\CAlg(\PrL_{\st})$ and $\K\subseteq\C$. Then $i_{\tors}:\C^{\K\text{-}\tors}\hookrightarrow\C$ is a localizing ideal and $L^{\K}:\C\to\C^{\K\text{-}\loc}$ is a symmetric monoidal localization. If $\K\subseteq\C^{\dbl}=\C^{\C\text{-}\at}$ is a set of dualizable objects, then $\C^{\K\text{-}\tors}\hookrightarrow\C\to\C^{\K\text{-}\loc}$ is a short exact sequence in $\Pr^{\iL}_{\C}$. So $i_{\tors}:\C^{\K\text{-}\tors}\hookrightarrow\C$ is a smashing ideal and $L^{\K}:\C\to\C^{\K\text{-}\loc}$ is a smashing localization. Passing to right adjoints, $i_{\loc}:\C^{\K\text{-}\loc}\hookrightarrow\C$ is a localizing ideal and $\Lambda^{\K}:\C\to\C^{\K\text{-}\cplt}$ is a symmetric monoidal localization. Then $L^{\K}\simeq L^{\K}\mathbbm{1}\otimes-\simeq L_{L^{\K}\mathbbm{1}}$ and $\Gamma^{\K}\simeq\Gamma^{\K}\mathbbm{1}\otimes-\simeq C_{L^{\K}\mathbbm{1}}$. Note that $\C^{\K\text{-}\loc}\simeq C_{\Gamma^{\K}\mathbbm{1}}\C$. Since $\Gamma^{\K}\mathbbm{1}\otimes\K=\K$ and $\Gamma^{\K}\mathbbm{1}\in\C^{\K\text{-}\tors}$, $C_{\Gamma^{\K}\mathbbm{1}}\C\simeq C_{\K}\C$. Thus we have identifications $C_{\K}\C=\C^{\K\text{-}\loc}\subseteq\C$ and $L_{\K}\C=\C^{\K\text{-}\cplt}\subseteq\C$ with corresponding functors $C_{\K}\simeq L^{\K}$ and $L_{\K}\simeq\Lambda^{\K}$.
\end{example}
\begin{remark}
	Let $\C\in\CAlg(\PrL_{\st})$ and $\K\subseteq\C^{\dbl}$. Then $\C^{\K\text{-}\tors}\hookrightarrow\C$ is an \emph{atomic smashing ideal} in the sense of \cite[Definition 1.5]{Lia26telescope}, i.e. it lies in $\Pr^{\at}_{\C}$. This generalizes the notion of \emph{finite smashing localization} in the case of $\C=\Sp$, i.e. smashing localizations with compactly generated ideals. In particular, if $\C$ is locally rigid, $\C^{\K\text{-}\tors}$ is atomic for any set $\K\subseteq\C^{\omega}$. Suppose that $\C$ is locally rigid compactly generated (i.e. $\C$ is compactly generated by a set of dualizable objects as in \cite{BHV18,NPR24}). Then, for $\K\subseteq\C^{\dbl}$, $\C^{\K\text{-}\tors}=\Loc(\K\otimes\C^\omega)$ is compactly generated since $\C^{\dbl}\otimes\C^{\omega}\subseteq\C^{\omega}$, and hence we may choose $\K'\subseteq\C^{\omega}$ such that $\C^{\K'\text{-}\tors}=\C^{\K\text{-}\tors}$. So a smashing ideal is atomic if and only if it is compactly generated \cite[Example 1.7(1)]{Lia26telescope}. Suppose moreover that $\C$ is rigid compactly generated. Then atomic smashing ideals correspond to \emph{finite smashing ideals}, i.e. thick ideals of $\C^\omega$ (cf. \cite[Definition 2.47]{ABC+25}).
\end{remark}
\begin{corollary}\label{cor:local_duality_fracture_square_functors}
	Let $\C\in\CAlg(\PrL_{\st})$ and $\K\subseteq\C^{\dbl}$. Then we have a pullback square
	\[\begin{tikzcd}
		\id & {L_{\K}} \\
		{L^{\K}} & {L^{\K}L_{\K}}
		\arrow[from=1-1, to=1-2]
		\arrow[from=1-1, to=2-1]
		\arrow["\lrcorner"{anchor=center, pos=0.125}, draw=none, from=1-1, to=2-2]
		\arrow[from=1-2, to=2-2]
		\arrow[from=2-1, to=2-2]
	\end{tikzcd}\]
\end{corollary}
\begin{proof}
	By \cite[Lemma 4.11]{NPR24}, we have $\Lambda^{\K}\Gamma^{\K}\simeq\Lambda^{\K}$ and $\Gamma^{\K}\Lambda^{\K}\simeq\Gamma^{\K}$. Since we have a fiber sequence $\Gamma^{\K}\to\id\to L^{\K}$, $\Gamma^{\K}$ and $L^{\K}$ are jointly conservative. Replacing $L_{\K}$ by $\Lambda^{\K}$, we may check this square is a pullback square by applying $\Gamma^{\K}$ and $L^{\K}$.
\end{proof}
\begin{example}\label{exm:Mod(R)^S-tors/cplt}
	Let $\C=\Sp$ and $\M\in\PrL_{\st}$. Then $\M^{\C\text{-}\at}=\M^{\omega}$. Let $R\in\Alg(\Sp)$ and $\M=\Mod(R)$. For a set of homogeneous elements $S\subseteq\pi_*(R)$, we have $\{R/Rs\}_{s\in S}\subseteq\Mod(R)^\omega$. Then the definitions of $S$-torsion, $S$-local and $S$-complete coincide with those of $\{R/Rs\}_{s\in S}$-torsion, $\{R/Rs\}_{s\in S}$-local and $\{R/Rs\}_{s\in S}$-complete, respectively.
\end{example}

Now we recall some basic properties of localizations and generalize \cref{prop:local_duality} and \cref{cor:local_duality_fracture_square_functors}.
\begin{definition}[{\cite[Definition 5.2.9]{CSY21height}}]
	A \emph{mode} is an idempotent algebra in $\PrL$. Let $\C\in\CAlg(\PrL)$. A \emph{$\C$-mode} is an idempotent algebra in $\PrL_{\C}$.
\end{definition}
\begin{lemma}[cf. {\cite[Proposition 5.2.10]{CSY21height}}]
	Let $\C\in\CAlg(\PrL)$, and let $L:\C\to L\C$ be a symmetric monoidal localization. Then $L\C$ is a $\C$-mode\footnote{We thank Yuchen Wu for reminding us of this proposition.}. Let $\M\in\PrL_{\C}$. Then $L\C\otimes_{\C}\M\subseteq\M$ is the full subcategory spanned by the objects $X\in\M$ such that $\iHom_{\C}(Y,X)\in L\C$ for all $Y\in\M$.
\end{lemma}
\begin{remark}
	Let $\C\in\CAlg(\PrL_{\st})$ and $\M\in\PrL_{\C}$, and let $E\in\C$. We have two notions of $E$-localizations of $\M$: $L_E\M$ and $L_E\C\otimes_{\C}\M$. They coincide in the following cases:
	\begin{enumerate}
		\item If $\M\in\Pr^{\dbl}_{\C}$, then $L_E\C\otimes_{\C}\M\simeq L_E\M$ by \cref{prop:L_FL_E}. 
		\item If $I$ is a coidempotent coalgebra or $A$ is an idempotent algebra, then $L_I\C\otimes_{\C}\M\simeq L_I\M$ or $L_A\C\otimes_{\C}\M\simeq L_A\M$, respectively, by \cref{prop:L_AM}.
		\item Let $\K\subseteq\C^{\dbl}$. Then we have $L_{\K}\C\otimes_{\C}\M\simeq L_{\K}\M$ by \cref{cor:L_KM}.
	\end{enumerate}
\end{remark}
\begin{proposition}\label{prop:L_FL_E}
	Let $\C\in\CAlg(\PrL_{\st})$ and let $E\in\C$. Then $L_E\C$ is a $\C$-mode. Let $F\in\C$. Then $L_F\C\otimes_{\C}L_E\C\simeq L_F\C\cap L_E\C\subseteq\C$. Let $\M\in\PrL_{\C}$. Then $L_FL_E\M\simeq L_{E\otimes F}\M\simeq L_{L_EF}L_E\M$.

	If $\M\in\Pr^{\dbl}_{\C}$, then $L_F\C\otimes_{\C}\M\simeq L_F\M$. Then if $L_E\C$ is dualizable, we have
	\[
		L_F\C\otimes_{\C}L_E\C\simeq L_FL_E\C\simeq L_{E\otimes F}\C.
	\]
	In particular, if $L_E$ is smashing, then $L_F\C\otimes_{\C}L_E\C\simeq L_{E\otimes F}\C\simeq L_{L_EF}\C$.
\end{proposition}
\begin{proof}
	Let $E\in\C$. Then $L_E\C$ is a $\C$-mode since $L_E:\C\to L_E\C$ is a symmetric monoidal localization. Let $F\in\C$. Then $L_F\C\otimes_{\C}L_E\C=\{X\in L_E\C:\iHom_{\C}(Y,X)\in L_F\C\text{ for all }Y\in L_E\C\}=\{X\in L_E\C:\Hom_{\C}(Z,\iHom_{\C}(Y,X))\simeq0\text{ for all }Y\in L_E\C,Z\in C_F\C\}=\{X\in L_E\C:\Hom_{\C}(Z\otimes L_EY,X)\simeq0\text{ for all }Y\in\C,Z\in C_F\C\}=\{X\in L_E\C:\Hom_{\C}(Z\otimes Y,X)\simeq0\text{ for all }Y\in\C,Z\in C_F\C\}=L_F\C\cap L_E\C$.
	
	Let $\M\in\PrL_{\C}$. Then $L_FL_E\M=\{X\in L_E\M:\Hom_{L_E\M}(Y,X)\simeq0\text{ for all }Y\in L_E\M\text{ such that }L_E(F\otimes Y)\simeq0\}=\{X\in L_E\M:\Hom_{\M}(Y,X)\simeq0\text{ for all }Y\in \M\text{ such that }E\otimes F\otimes Y\simeq0\}=L_{E\otimes F}\M$. Since $E\otimes L_EF\simeq E\otimes F$, $L_FL_E\M\simeq L_{E\otimes F}\M=L_{L_EF}L_E\M$.
	
	If $\M\in\Pr^{\dbl}_{\C}$, then $-\otimes_{\C}\M$ preserves both limits and colimits. So $C_F\C\otimes_{\C}\M\simeq \fib(F\otimes:\C\to\C)\otimes_{\C}\M\simeq\fib(F\otimes-:\M\to\M)\simeq C_F\M$ and $L_F\C\otimes_{\C}\M\simeq(\C/C_F\C)\otimes_{\C}\M\simeq\M/C_F\C\otimes_{\C}\M\simeq\M/C_F\M\simeq L_F\M$. Therefore, if $L_E\C$ is dualizable, then $L_E\C\otimes_{\C}L_F\C\simeq L_FL_E\C\simeq L_{E\otimes F}\C$.
\end{proof}
\begin{proposition}\label{prop:L_AM}
	Let $\C\in\CAlg(\PrL_{\st})$ and $I$ be a coidempotent coalgebra in $\C$, and let $A:=\mathbbm{1}/I$ be the associated idempotent algebra. Then for $\M\in\PrL_{\C}$, the short exact sequence $C_A\M\hookrightarrow\M\to L_A\M$ is identified with $C_A\C\otimes_{\C}\M\hookrightarrow\M\to L_A\C\otimes_{\C}\M$, and the short exact sequence $C_I\M\hookrightarrow\M\to L_I\M$ is identified with $C_I\C\otimes_{\C}\M\hookrightarrow\M\to L_I\C\otimes_{\C}\M$. In particular, we have $L_A\C\otimes_{\C}\M\simeq L_A\M$ and $L_I\C\otimes_{\C}\M\simeq L_I\M$.
\end{proposition}
\begin{proof}
	Since $C_A\C\hookrightarrow\C\to L_A\C$ is a short exact sequence in $\Pr^{\iL}_{\C}$, $C_A\C\otimes_{\C}\M\hookrightarrow\M\to L_A\C\otimes_{\C}\M$ is also a short exact sequence in $\Pr^{\iL}_{\C}$. In particular, $L_A\C\hookrightarrow\C$ induces a fully faithful functor $L_A\C\otimes_{\C}\M\hookrightarrow\M$ such that precomposing with $L_A\otimes_{\C}\M:\M\to L_A\C\otimes_{\C}\M$ gives $A\otimes-:\M\to\M$. Then $C_A\C\otimes_{\C}\M\simeq\fib(L_A\otimes_{\C}\M:\M\to L_A\C\otimes_{\C}\M)\simeq\fib(A\otimes-:\M\to\M)\simeq C_A\M$, and the desired identification follows.

	Similarly, passing to right adjoints gives a short exact sequence $C_I\C\otimes_{\C}\M\hookrightarrow\M\to L_I\C\otimes_{\C}\M$ in $\PrL_{\C}$. In particular, $L_I\C\simeq C_A\C\hookrightarrow\C$ induces a fully faithful functor $L_I\C\otimes_{\C}\M\hookrightarrow\M$ such that precomposing with $L_I\otimes_{\C}\M:\M\to L_I\C\otimes_{\C}\M$ gives $I\otimes-:\M\to\M$. Then $C_I\C\otimes_{\C}\M\simeq\fib(L_I\otimes_{\C}\M:\M\to L_I\C\otimes_{\C}\M)\simeq\fib(I\otimes-:\M\to\M)\simeq C_I\M$, and the desired identification follows.
\end{proof}
\begin{lemma}\label{lem:LC=LC(Mod(L1))}
	Let $f:\C\to\D\in\CAlg(\PrL_{\st})$ be a $\C$-internal left adjoint with $f^R$ conservative, or equivalently, $\D\simeq\Mod_R(\C)$ for some $R\in\CAlg(\C)$. Let $L:\C\to L\C$ be a symmetric monoidal localization. Then $L\D\simeq L\C\otimes_{\C}\D\subseteq\D$ is the full subcategory spanned by the objects $X\in\D$ such that $f^R(X)\in L\C$.
\end{lemma}
\begin{proof}
	Consider the commutative diagram
	\[\begin{tikzcd}
		\C & \D \\
		{L\C} & {L\C\otimes_{\C}\D}
		\arrow["f", shift left, from=1-1, to=1-2]
		\arrow["L"', shift right, from=1-1, to=2-1]
		\arrow[shift left, from=1-2, to=1-1]
		\arrow["{L\otimes_{\C}\D}"', shift right, from=1-2, to=2-2]
		\arrow[shift right, hook, from=2-1, to=1-1]
		\arrow["{L\C\otimes_{\C}f}", shift left, from=2-1, to=2-2]
		\arrow[shift right, hook, from=2-2, to=1-2]
		\arrow[shift left, from=2-2, to=2-1]
	\end{tikzcd}\]
	Since $f:\C\to\D$ is a $\C$-internal left adjoint, the commutative diagram is horizontally right adjointable. Let $X\in\D$. Then $X\in L\C\otimes_{\C}\D$ if and only if $X\to(L\otimes_{\C}\D)(X)$ is an equivalence, if and only if $f^R(X)\to f^R(L\otimes_{\C}\D)(X)\simeq(L\otimes_{\C}f^R)(X)\simeq L(f^R(X))$ is an equivalence, which is equivalent to $f^R(X)\in L\C$. Moreover, since $\D$ is a dualizable $\C$-module, $L\D\simeq L\C\otimes_{\C}\D$.
\end{proof}
Consequently, let $\C\in\CAlg(\PrL_{\st})$ and $R\in\CAlg(\C)$, and $\C\to L\C$ be a symmetric monoidal localization. Then $L\Mod_R(\C)\subseteq\Mod_R(\C)$ is the full subcategory spanned by $L$-local $R$-modules, which is identified with the full subcategory $L\Mod_{LR}(\C)\subseteq\Mod_{LR}(\C)$ spanned by $L$-local $LR$-modules.

Let $\C\in\PrL_{\st}$ and $L_0,L_1$ be two localizations. We have a commutative diagram 
\begin{equation}\label{eq:fracture}
	\begin{tikzcd}
	\id & {L_1} \\
	{L_0} & {L_0L_1}
	\arrow[from=1-1, to=1-2]
	\arrow[from=1-1, to=2-1]
	\arrow[from=1-2, to=2-2]
	\arrow[from=2-1, to=2-2]
\end{tikzcd}
\end{equation}
which induces a functor 
\[\C\to L_0\C\overrightarrow{\times}_{L_0}L_1\C:=(L_0\C)^{[1]}\times_{\ev_1,L_0\C,L_0}L_1\C\]
in $\PrL_{\st}$ that sends $X$ to $(L_0X,L_1X,L_0L_1X)$, and whose right adjoint sends $(X,Y,X\to L_0Y)$ to $X\times_{L_0Y}Y$.
\begin{lemma}[{\cite[Proposition A.8.11]{HA}, or \cite[Theorem 3.20]{MNN17}}]\label{lem:recollement}
	The functor $\C\to L_0\C\overrightarrow{\times}_{L_0}L_1\C$ is fully faithful if and only if \eqref{eq:fracture} is a pullback diagram, or equivalently, $L_0$ and $L_1$ are jointly conservative and $L_0$ preserves $L_1$-acyclic objects. It is an equivalence if and only if $\C$ is a recollement \textup{(}also called a semi-orthogonal decomposition in $\PrL_{\st}$\textup{)} of $L_0\C,L_1\C$, i.e. $L_0$ and $L_1$ are jointly conservative and $L_1L_0\simeq0$. 
\end{lemma}
If $L_1$ and $L_0$ define a recollement of $L_0\C,L_1\C$, then we have a short exact sequence in $\Pr^{\iL}_{\st}$:
\[\begin{tikzcd}
	{L_1\C} & \C & {L_0\C}
	\arrow["{L_1^L}", shift left=3, hook, from=1-1, to=1-2]
	\arrow[shift right=3, hook, from=1-1, to=1-2]
	\arrow["{L_1}"{description}, from=1-2, to=1-1]
	\arrow["{L_0}", shift left=3, from=1-2, to=1-3]
	\arrow["{L_0^R}"', shift right=3, from=1-2, to=1-3]
	\arrow[hook', from=1-3, to=1-2]
\end{tikzcd}\]
yielding inverse equivalences
\[
	L_1^L:L_1\C\simeq C_0\C:L_1.
\]
By the lax additivity of $\PrL_{\st}$, we also have $\C\simeq L_1\C {}_{L_0}\!\overrightarrow{\times}L_0\C:=L_1\C\times_{L_0,L_0\C,\ev_0}(L_0\C)^{[1]}$.

We specialize the preceding discussion to localizations with respect to objects $E,F\in\C$, recovering in a more general setting the familiar fracture square and recollement of chromatic homotopy theory (cf. \cite[Lemma 2.8]{LMMT24}, \cite[Lecture 34]{Lur10}, \cite[\Section 6]{BHV18}).
\begin{proposition}\label{prop:fracture_square_functors}
	Let $\C\in\CAlg(\PrL_{\st})$, and let $E,F\in\C$.
	\begin{enumerate}
		\item Suppose that $L_E$ preserves $L_F$-acyclic objects. Then we have a pullback diagram
		\begin{equation}
			\begin{tikzcd}
				{L_{E\oplus F}} & {L_F} \\
				{L_E} & {L_EL_F}
				\arrow[from=1-1, to=1-2]
				\arrow[from=1-1, to=2-1]
				\arrow["\lrcorner"{anchor=center, pos=0.125}, draw=none, from=1-1, to=2-2]
				\arrow[from=1-2, to=2-2]
				\arrow[from=2-1, to=2-2]
			\end{tikzcd}
		\end{equation}
		In particular, this assumption holds whenever $E$ is smashing.
		\item Suppose that $L_FL_E\simeq0$. Then we have a short exact sequence in $\Pr^{\iL}_{\C}$: 
		\[\begin{tikzcd}
			{L_F\C} & L_{E\oplus F}\C & {L_E\C}
			\arrow["{C_E}", shift left=3, hook, from=1-1, to=1-2]
			\arrow[shift right=3, hook, from=1-1, to=1-2]
			\arrow["{L_F}"{description}, from=1-2, to=1-1]
			\arrow["{L_E}", shift left=3, from=1-2, to=1-3]
			\arrow[shift right=3, from=1-2, to=1-3]
			\arrow[hook', from=1-3, to=1-2]
		\end{tikzcd}\]
		yielding inverse equivalences
		\[
			C_E:L_F\C\simeq C_EL_{E\oplus F}\C:L_F.
		\]
		Suppose moreover that $L_{E\oplus F}\C$ is a dualizable $\C$-module. Then this is a short exact sequence in $\Pr^{\dbl}_{\C}$, and $L_EL_{E\oplus  F}\C\simeq L_E\C$, and $L_{E\oplus F}E\in L_{E\oplus F}\C$ is smashing. In particular, these assumptions hold whenever both $E\oplus F$ and $E$ are smashing and $E\otimes F\simeq0$.
		
		\item The assumptions of (2) hold whenever $L_E$ and $L_F$ are jointly conservative, $F\in\C^{\dbl}$, and $E\otimes F\simeq0$.
		
		More generally, let $\K\subseteq\C^{\dbl}$. Suppose that $L_E$ and $L_{\K}$ are jointly conservative and $E\otimes\K=0$. Then we have a short exact sequence in $\Pr^{\at}_{\C}$
		\[\begin{tikzcd}
			{L_{\K}\C} & \C & {L_E\C}
			\arrow["{C_E}", shift left=3, hook, from=1-1, to=1-2]
			\arrow[shift right=3, hook, from=1-1, to=1-2]
			\arrow["{L_{\K}}"{description}, from=1-2, to=1-1]
			\arrow["{L_E}", shift left=3, from=1-2, to=1-3]
			\arrow[shift right=3, from=1-2, to=1-3]
			\arrow[hook', from=1-3, to=1-2]
		\end{tikzcd}\]
		and the short exact sequence $C_E\C\hookrightarrow\C\to L_E\C$ in $\Pr^{\iL}_{\st}$ identifies with $\C^{\K\text{-}\tors}\hookrightarrow\C\to\C^{\K\text{-}\loc}$.
	\end{enumerate}
\end{proposition}
\begin{proof}
	\begin{enumerate}
		\item Applying \cref{lem:recollement} to the localizations $L_E:L_{E\oplus F}\C\to L_E\C$ and $L_F:L_{E\oplus F}\C\to L_F\C$, we obtain the desired pullback diagram. In particular, if $E$ is smashing, then $L_E\simeq L_E\mathbbm{1}\otimes-$ preserves $L_F$-acyclic objects.
		\item Similarly, applying \cref{lem:recollement} we obtain the desired short exact sequence in $\Pr^{\iL}_{\st}$. In particular, if $E$ is smashing and $E\otimes F=0$, then $F\otimes L_FL_E\simeq F\otimes L_E\simeq L_EF\otimes-\simeq0$, and hence $L_FL_E\simeq0$. 
		
		Suppose moreover that $L_{E\oplus F}\C$ is a dualizable $\C$-module. Then $L_{L_{E\oplus F}E}L_{E\oplus F}\C\simeq L_EL_{E\oplus F}\C\simeq L_E\C\otimes_{\C}L_{E\oplus F}\C\simeq L_E\C$ by \cref{prop:L_FL_E}, and hence $L_{E\oplus F}E\in L_{E\oplus F}\C$ is smashing. In particular, if $E\oplus F$ is smashing, then $L_{E\oplus F}\C$ is a dualizable $\C$-module.
		
		\item Assume that $L_E$ and $L_{\K}$ are jointly conservative, $\K\subseteq\C^{\dbl}$, and $E\otimes\K\simeq0$. Then $\Gamma^{\K}\mathbbm{1}\in\C^{{\K}\text{-}\tors}\subseteq C_E\C$, and it implies that $L_{\K}L_E\simeq\iHom_{\C}(\Gamma^{\K}\mathbbm{1},L_E)\simeq\iHom_{\C}(L_E\Gamma^{\K}\mathbbm{1},L_E)\simeq0$. Combining this with \cref{exm:local_duality} we obtain the desired identification. \qedhere
	\end{enumerate}
\end{proof}
\begin{corollary}\label{cor:L_KM}
	Let $\C\in\CAlg(\PrL_{\st})$ and $\M\in\PrL_{\C}$. Let $\K\subseteq\C^{\dbl}$ and $E\in\C$. Suppose that $L_E$ and $L_{\K}$ are jointly conservative and $E\otimes\K\simeq0$. Then we have a short exact sequence in $\Pr^{\iL}_{\C}$
	\[\begin{tikzcd}
			{L_{\K}\M} & \M & {L_{L_E\mathbbm{1}}\M}
			\arrow["{C_{L_E\mathbbm{1}}}", shift left=3, hook, from=1-1, to=1-2]
			\arrow[shift right=3, hook, from=1-1, to=1-2]
			\arrow["{L_{\K}}"{description}, from=1-2, to=1-1]
			\arrow["{L_{L_E\mathbbm{1}}}", shift left=3, from=1-2, to=1-3]
			\arrow[shift right=3, from=1-2, to=1-3]
			\arrow[hook', from=1-3, to=1-2]
	\end{tikzcd}\]
	with $L_{\K}\C\otimes_{\C}\M\simeq L_{\K}\M$. In particular, this assumption holds for $E=L^{\K}\mathbbm{1}$.
\end{corollary}
\begin{proof}
	By \cref{prop:fracture_square_functors}(3), we have $L_E\C\otimes_{\C}\M=\{X\in\M:\iHom_{\C}(Y,X)\in L_E\C\text{ for all }Y\in\M\}=\{X\in\M:\K\otimes\iHom_{\C}(Y,X)\simeq0\text{ for all }Y\in\M\}=\{X\in\M:\iHom_{\C}(Y,\K\otimes X)\simeq0\text{ for all }Y\in\M\}=C_{\K}\M\subseteq\M$. By \cref{prop:L_AM}, we have $L_E\C\otimes_{\C}\M\simeq L_{L_E\mathbbm{1}}\C\otimes_{\C}\M\simeq L_{L_E\mathbbm{1}}\M$. Then the desired short exact sequence follows, and we have $L_{\K}\C\otimes_{\C}\M\simeq\M/C_{\K}\C\otimes_{\C}\M\simeq \M/L_E\C\otimes_{\C}\M\simeq\M/C_{\K}\M\simeq L_{\K}\M$.
\end{proof}

\begin{example}\label{exm:Mod(R)_I-tors/cplt}
	Let $R\in\CAlg(\Sp)$. In \cite[\Section 7]{SAG}, Lurie establishes the localization theory of $\bbE_2$-rings with respect to a homogeneous ideal. For $M\in\Mod(R)$ and $s\in\pi_d(R)$, we denote $M[s^{-1}]:=s^{-1}M\simeq s^{-1}R\otimes_RM\simeq\colim(M\xrightarrow{s\cdot}M[-d]\xrightarrow{s\cdot}\cdots)$. Let $I\subseteq\pi_*(R)$ be a homogeneous ideal.
	
	\begin{enumerate}
		\item We say $M$ is \emph{$I$-torsion} or \emph{$I$-nilpotent} if $M[x^{-1}]\simeq0$ for all homogeneous elements $x\in I$, denoted by $\Mod(R)_{I\text{-}\tors}\subseteq\Mod(R)$ with right adjoint $\Gamma_I$.
		\item We say $M$ is \emph{$I$-local} if it lies in the right orthogonal complement $\Mod(R)_{I\text{-}\loc}:=(\Mod(R)_{I\text{-}\tors})^{\perp}\subseteq\Mod(R)$ with left adjoint $L_I$ and, when it exists, a right adjoint $\Delta_I$.
		\item We say $M$ is \emph{$I$-complete} if it lies in the right orthogonal complement $\Mod(R)_{I\text{-}\cplt}:=(\Mod(R)_{I\text{-}\loc})^{\perp}$ $\subseteq\Mod(R)$ with, when it exists, a left adjoint $(-)^{\wedge}_I$.
	\end{enumerate}

	Let $S$ be a set of homogeneous generators of $I$. Then we have $\Mod(R)_{I\text{-}\tors}\simeq C_{\{R[s^{-1}]\}_{s\in S}}\Mod(R)$ and $\Mod(R)_{I\text{-}\loc}\simeq L_{\{R[s^{-1}]\}_{s\in S}}\Mod(R)$, with corresponding functors $\Gamma_I\simeq C_{\{R[s^{-1}]\}_{s\in S}}$ and $L_I\simeq L_{\{R[s^{-1}]\}_{s\in S}}$. 
	
	Suppose that $I$ is finitely generated and write $I=(x_0,\cdots,x_n)$ for a finite set of generators $x_i\in\pi_{d_i}(R)$. We denote $R/(x_0^{k_0},\cdots,x_n^{k_n}):=\bigotimes_{i=0}^nR/x_i^{k_i}\in\Mod(R)^{\omega}$. Then the definitions of $I$-torsion, $I$-local, and $I$-complete coincide with those of $R/(x_0,\cdots,x_n)$-torsion, $R/(x_0,\cdots,x_n)$-local, and $R/(x_0,\cdots,x_n)$-complete, respectively. Therefore, we have a short exact sequence in $\Pr^{\cg}_{\st}$:
	\[\begin{tikzcd}
		{\Mod(R)_{I\text{-}\tors}} & {\Mod(R)} & {\Mod(R)_{I\text{-}\loc}}
		\arrow[shift left=3, hook, from=1-1, to=1-2]
		\arrow["(-)^{\wedge}_I"', shift right=3, hook, from=1-1, to=1-2]
		\arrow["{\Gamma_I}"{description}, from=1-2, to=1-1]
		\arrow["{L_I}", shift left=3, from=1-2, to=1-3]
		\arrow["{\Delta_I}"', shift right=3, from=1-2, to=1-3]
		\arrow[hook', from=1-3, to=1-2]
	\end{tikzcd}\]
	yielding inverse equivalences $(-)^{\wedge}_I:\Mod(R)_{I\text{-}\tors}\simeq\Mod(R)_{I\text{-}\cplt}:\Gamma_I$.
	\begin{enumerate}
		\item $\Mod(R)_{I\text{-}\tors}\simeq\Loc(R/(x_0,\cdots,x_n))\simeq\Mod(R,\Gamma_IR)\simeq C_{L_IR}\Mod(R)$ with $\Gamma_I\simeq\Gamma_{(x_n)}\cdots\Gamma_{(x_0)}\simeq\Gamma_IR\otimes_R-\simeq\colim_{k_0,\cdots,k_n}R/(x_0^{k_0},\cdots,x_n^{k_n})[-k_0d_0-\cdots-k_nd_n][-n-1]\otimes_R-$.
		\item $\Mod(R)_{I\text{-}\loc}\simeq\Loc(\{R[x_i^{-1}]\}_{i=0}^n)\simeq\Mod(L_IR)\simeq L_{L_IR}\Mod(R)$ with $L_I\simeq\lim_{(\Delta_{\mathrm{inj}})_{/[n]}}x_{i_0}^{-1}\cdots x_{i_k}^{-1}\simeq L_IR\otimes_R-\simeq\lim_{(\Delta_{\mathrm{inj}})_{/[n]}}R[x_{i_0}^{-1},\cdots,x_{i_k}^{-1}]\otimes_R-$.
		\item $\Mod(R)_{I\text{-}\cplt}\simeq\bigcap_{i=0}^nL_{R/x_i}\Mod(R)\simeq L_{R/x_n}\cdots L_{R/x_0}\Mod(R)\simeq L_{R/(x_0,\cdots,x_n)}\Mod(R)$ with $(-)^{\wedge}_I\simeq L_{R/x_n}\cdots L_{R/x_0}\simeq\lim_{k_0,\cdots,k_n}(-)/(x_0^{k_0},\cdots,x_n^{k_n})$.
	\end{enumerate}

	More generally, let $\M\in\PrL_R$, and let $I$ be a homogeneous ideal with a set $S$ of homogeneous generators. Define $\M_{I\text{-}\tors}:=C_{\{R[s^{-1}]\}_{s\in S}}\M$, $\M_{I\text{-}\loc}:=L_{\{R[s^{-1}]\}_{s\in S}}\M$, and $\M_{I\text{-}\cplt}:=(\M_{I\text{-}\loc})^{\perp}\subseteq\M$.

	If $I=(x_0,\cdots,x_n)$ is finitely generated, taking $\K=\{R/(x_0,\cdots,x_n)\}$ in \cref{cor:L_KM} yields $\M_{I\text{-}\tors}\simeq C_{L_IR}\M\simeq\Mod(R)_{I\text{-}\tors}\otimes_R\M$, $\M_{I\text{-}\loc}\simeq L_{L_IR}\M\simeq\Mod(R)_{I\text{-}\loc}\otimes_R\M$ and $\M_{I\text{-}\cplt}\simeq L_{R/(x_0,\cdots,x_n)}\M\simeq\Mod(R)_{I\text{-}\cplt}\otimes_R\M$.
\end{example}
\begin{remark}
	More generally, let $F(0),\cdots,F(n)\in\C$ be a collection of objects such that $L_{F(j)}L_{F(i)}\simeq0$ for $j>i$, and let $E=F(0)\oplus\cdots\oplus F(n)$. In \cite[Proposition 3.2, Theorem 5.5]{ACB22}, Antol\'in-Camarena--Barthel construct a fracture $n$-cube of localization functors for $L_E$ and a fracture $n$-cube of localization categories for $L_E\C$. More precisely,
	\begin{enumerate}
		\item The $(\Delta_{\mathrm{inj}})_{/[n]}$-diagram of localization functors whose $\{i_0<\cdots<i_k\}$-term is $L_{F(i_0)}\cdots L_{F(i_k)}$, and whose transition maps are induced by natural transformations $\id\to L_{F(i)}$ has limit $L_E$.
		\item The $(\Delta_{\mathrm{inj}})_{/[n]}$-diagram of localization categories whose $\{i\}$-term is $(L_{F(i)}\C)^{[1]^{n-i}}$, and whose transition maps are induced by the natural transformations $\id\to L_{F(i)}$ has limit $L_E\C$. Equivalently, we have
		\[
			L_E\C\simeq L_{F(0)}\C\overrightarrow{\times}_{L_{F(0)}}L_{F(1)}\C\overrightarrow{\times}_{(L_{F(0)}\to L_{F(0)}L_{F(1)}\leftarrow L_{F(1)})}L_{F(2)}\C\overrightarrow{\times}\cdots L_{F(n)}\C.
		\]
	\end{enumerate}
\end{remark}
\begin{example}\label{exm:S[1/p]_S/p}
	Let $\C=\Sp$. Take $E=\bbS[1/p]$ and $F=\bbS/p$. Since $\bbS[1/p]\oplus\bbS/p$ detects $0$, $\bbS/p\in\Sp^{\dbl}$ and $\bbS[1/p]\otimes\bbS/p\simeq0$, \cref{prop:fracture_square_functors}(3) yields a short exact sequence in $\Pr^{\dbl}_{\st}$
	\[\begin{tikzcd}[column sep=4em]
		{\Sp_{p\text{-}\cplt}} & \Sp & {\Sp[1/p]}
		\arrow[shift left=3, hook, from=1-1, to=1-2]
		\arrow[shift right=3, hook, from=1-1, to=1-2]
		\arrow["{(-)^{\wedge}_p}"{description}, from=1-2, to=1-1]
		\arrow["{\bbS[1/p]\otimes-}", shift left=3, from=1-2, to=1-3]
		\arrow[shift right=3, from=1-2, to=1-3]
		\arrow[hook', from=1-3, to=1-2]
	\end{tikzcd}\]
	which is equivalent to the short exact sequence in \cref{exm:Mod(R)_I-tors/cplt} by taking $R=\bbS$ and $I=(p)$. Then we have a fracture square
	\[\begin{tikzcd}
		\id & {(-)^{\wedge}_p} \\
		{(-)[1/p]} & {(-)^{\wedge}_p[1/p]}
		\arrow[from=1-1, to=1-2]
		\arrow[from=1-1, to=2-1]
		\arrow["\lrcorner"{anchor=center, pos=0.125}, draw=none, from=1-1, to=2-2]
		\arrow[from=1-2, to=2-2]
		\arrow[from=2-1, to=2-2]
	\end{tikzcd}\]

	Take $E=\bbQ$ and $\K=\{\bbS/p\}_{p\text{ prime}}\subseteq\Sp^{\dbl}$. Then $L_E\simeq\bbQ\otimes-$ and $L_{\K}\simeq\prod_{p\text{ prime}}(-)^{\wedge}_p$. Since $\bbQ\oplus\{\bbS/p\otimes-\}_{p\text{ prime}}$ jointly detect $0$ and $\bbQ\otimes\bbS/p\simeq0$ for all primes $p$, we have a fracture square called the \emph{arithmetic fracture square}
	\[\begin{tikzcd}
		\id & {\prod_{p\text{ prime}}(-)^{\wedge}_p} \\
		{\bbQ\otimes-} & {\bbQ\otimes\prod_{p\text{ prime}}(-)^{\wedge}_p}
		\arrow[from=1-1, to=1-2]
		\arrow[from=1-1, to=2-1]
		\arrow["\lrcorner"{anchor=center, pos=0.125}, draw=none, from=1-1, to=2-2]
		\arrow[from=1-2, to=2-2]
		\arrow[from=2-1, to=2-2]
	\end{tikzcd}\]
\end{example}

\subsection{Fracture square of dualizable categories}\label{sec:2.2}

In this subsection, we apply completion theory to construct fracture squares for dualizable categories. We show that maps of rigid envelopes induce symmetric monoidal fracture squares, thereby generalizing the symmetric monoidal fracture square in Naumann--Pol--Ramzi \cite{NPR24}.

We begin with some technical lemmas providing pullback diagrams.
\begin{lemma}\label{lem:ses_pullback}
	Given a commutative diagram between short exact sequences in $\PrL_{\st}$
	\begin{equation}\label{eq:ses_pullback}\begin{tikzcd}
		\C & \D & \E \\
		\C & {\D'} & {\E'}
		\arrow["{j_!}", shift left, hook, from=1-1, to=1-2]
		\arrow[equals, from=1-1, to=2-1]
		\arrow["{j^*}", shift left, from=1-2, to=1-1]
		\arrow["{i^*}", shift left, from=1-2, to=1-3]
		\arrow["{f^*}"', shift right, from=1-2, to=2-2]
		\arrow["{i_*}", shift left, hook', from=1-3, to=1-2]
		\arrow["{g^*}"', shift right, from=1-3, to=2-3]
		\arrow["{j_!'}", shift left, hook, from=2-1, to=2-2]
		\arrow["{f_*}"', shift right, from=2-2, to=1-2]
		\arrow["{j^{\prime,*}}", shift left, from=2-2, to=2-1]
		\arrow["{i^{\prime,*}}", shift left, from=2-2, to=2-3]
		\arrow["{g_*}"', shift right, from=2-3, to=1-3]
		\arrow["{i'_*}", shift left, hook', from=2-3, to=2-2]
	\end{tikzcd}\end{equation}
	Suppose that the right square is both horizontally and vertically right adjointable. Then the right square is a pullback diagram in $\PrL_{\st}$.
\end{lemma}
\begin{proof}
	The right square induces a pair of adjoint functors 
	\[
		F:\D\rightleftarrows\D'\times_{\E'}\E:G
	\]
	where $F$ sends $D\in\D$ to $(f^*D,i^*D,i^{\prime,*}f^*D\simeq g^*i^*D)\in\D'\times_{\E'}\E$, and $G$ sends $(D',E,i^{\prime,*}D'\simeq g^*E)\in\D'\times_{\E'}\E$ to $f_*D'\times_{i_*g_*g^*E}i_*E\in\D$.

	It suffices to show that $F$ is fully faithful and $G$ is conservative. $F$ is fully faithful if and only if the commutative diagram
	\[\begin{tikzcd}
		{\id_{\D}} & {i_*i^*} \\
		{f_*f^*} & {f_*i'_*i^{\prime,*}f^*\simeq i_*g_*g^*i^*}
		\arrow[from=1-1, to=1-2]
		\arrow[from=1-1, to=2-1]
		\arrow[from=1-2, to=2-2]
		\arrow[from=2-1, to=2-2]
	\end{tikzcd}\]
	is a pullback diagram. Passing to fibers of rows, we obtain that $F$ is fully faithful if and only if $j_!j^*\to f_*j_!'j^{\prime,*}f^*$ is an equivalence. By vertical right adjointability, $f_*j'_!\simeq j_!$. Since $j_!$ is fully faithful, $j_!j^*\to f_*j_!'j^{\prime,*}f^*$ is an equivalence if and only if $j^*\to j^{\prime,*}f^*$ is an equivalence. Thus, it follows from horizontal right adjointability that $F$ is fully faithful.

	To show that $G$ is conservative, it suffices to show that for any object $(D',E,i^{\prime,*}D'\simeq g^*E)\in\D'\times_{\E'}\E$ such that
	\[\begin{tikzcd}
		0 & {i_*E} \\
		{f_*D'} & {f_*i'_*i^{\prime,*}D'\simeq i_*g_*g^*E}
		\arrow[from=1-1, to=1-2]
		\arrow[from=1-1, to=2-1]
		\arrow["\lrcorner"{anchor=center, pos=0.125}, draw=none, from=1-1, to=2-2]
		\arrow[from=1-2, to=2-2]
		\arrow[from=2-1, to=2-2]
	\end{tikzcd}\]
	we have $D'\simeq0$ and $E\simeq0$. Applying $i^*$ and using the vertical right adjointability $i^*f_*\simeq g_*i^{\prime,*}$, the bottom row becomes an equivalence since $i_*$ and $i'_*$ are fully faithful. Then $E\simeq0$ and hence $f_*D'\simeq0$. It follows that $j^{\prime,*}D'\simeq j^*f_*D'\simeq0$, and $i^{\prime,*}D'\simeq g^*E\simeq0$. Therefore, $D'\simeq0$.
\end{proof}
\begin{remark}
	This is the right adjointable version of \cite[Proposition 2.10]{Aok26}. Dually, given a commutative diagram between short exact sequences in $\PrL_{\st}$
	\[\begin{tikzcd}
		\C & \D & \E \\
		{\C'} & {\D'} & \E
		\arrow["{j_!}", shift left, hook, from=1-1, to=1-2]
		\arrow["{f^*}"', shift right, from=1-1, to=2-1]
		\arrow["{j^*}", shift left, from=1-2, to=1-1]
		\arrow["{i^*}", shift left, from=1-2, to=1-3]
		\arrow["{g^*}"', shift right, from=1-2, to=2-2]
		\arrow["{i_*}", shift left, hook', from=1-3, to=1-2]
		\arrow[equals, from=1-3, to=2-3]
		\arrow["{f_*}"', shift right, from=2-1, to=1-1]
		\arrow["{j_!'}", shift left, hook, from=2-1, to=2-2]
		\arrow["{g_*}"', shift right, from=2-2, to=1-2]
		\arrow["{j^{\prime,*}}", shift left, from=2-2, to=2-1]
		\arrow["{i^{\prime,*}}", shift left, from=2-2, to=2-3]
		\arrow["{i'_*}", shift left, hook', from=2-3, to=2-2]
	\end{tikzcd}\]
	The left square is a pushout diagram in $\PrL_{\st}$ whenever it is both horizontally and vertically right adjointable. Indeed, after passing to right adjoints in the left square, both left and right squares are vertically left adjointable. It follows from the $\widehat{\Cat}_{\st}$ version of \cite[Proposition 2.10]{Aok26} that, after passing to right adjoints of all functors, the left square is a pullback diagram, and hence the original left square is a pushout diagram.
\end{remark}
\begin{remark}
	More generally, given a commutative diagram between short exact sequences in $\widehat{\Cat}_{\st}$
	\[\begin{tikzcd}
		\C & \D & \E \\
		{\C'} & {\D'} & \E
		\arrow["j", hook, from=1-1, to=1-2]
		\arrow["f"', from=1-1, to=2-1]
		\arrow["i", from=1-2, to=1-3]
		\arrow["g"', from=1-2, to=2-2]
		\arrow[equals, from=1-3, to=2-3]
		\arrow["{j'}", hook, from=2-1, to=2-2]
		\arrow["{i'}", from=2-2, to=2-3]
	\end{tikzcd}\]
	Then the left square is a pullback diagram in $\widehat{\Cat}_{\st}$. Indeed, since $\C'\times_{\D'}\D\hookrightarrow\D$ is fully faithful, $\C\hookrightarrow\C'\times_{\D'}\D$ is also fully faithful. It remains to show that it is essentially surjective. This is because every object $(c',d,j'c'\simeq gd)$ satisfies $id\simeq i'gd\simeq i'j'c'\simeq0$ and hence lies in the image of $\C$.

	Conversely, given a commutative diagram between short exact sequences in $\PrL_{\st}$ as in \eqref{eq:ses_pullback}, the right square is a pushout diagram in $\PrL_{\st}$. Indeed, passing to right adjoints, it follows that the right square is a pullback diagram in $\widehat{\Cat}_{\st}$.
\end{remark}

\begin{lemma}[{\cite[Proposition 1.87]{Efi25localizing}, \cite[Lemma 4.1, Corollary 4.5]{Ram24dualizable}, \cite[Remark 3.10]{LLS26}}]\label{lem:pullback_dbl}
	Let $\C\in\CAlg(\PrL_{\st})$. Given a pullback diagram in $\PrL_{\C}$
	\[\begin{tikzcd}
		{\M} & {\N} \\
		{\M'} & {\N'}
		\arrow[from=1-1, to=1-2]
		\arrow[from=1-1, to=2-1]
		\arrow["\lrcorner"{anchor=center, pos=0.125}, draw=none, from=1-1, to=2-2]
		\arrow[from=1-2, to=2-2]
		\arrow[from=2-1, to=2-2]
	\end{tikzcd}\]

	Suppose that $\M'\to\N'$ is a $\C$-internal localization and $\N\to\N'$ is a $\C$-internal left adjoint. Then the diagram is also a pullback diagram in $\Pr^{\iL}_{\C}$ and it is both horizontally and vertically right adjointable.
	
	Suppose moreover that the diagram lies in $\Pr^{\dbl}_{\C}$. Then the diagram is also a pullback diagram in $\Pr^{\dbl}_{\C}$.
\end{lemma}
Consequently, if the diagram \eqref{eq:ses_pullback} is a commutative diagram between short exact sequences in $\Pr^{\iL}_{\st}$ or $\Pr^{\dbl}_{\st}$ that is both horizontally and vertically right adjointable, then the right square is a pullback diagram in $\Pr^{\iL}_{\st}$ or $\Pr^{\dbl}_{\st}$, respectively.
\begin{lemma}\label{lem:pullback_mechanism}
	Let $\C\in\CAlg(\PrL_{\st})$, and let $\D\in\PrL_{\C}$. Given a pullback diagram in $\PrL_{\C}$
	\begin{equation}\label{eq:pullback_of_modules}\begin{tikzcd}
		\M & \N \\
		{\M'} & {\N'}
		\arrow[from=1-1, to=1-2]
		\arrow[from=1-1, to=2-1]
		\arrow["\lrcorner"{anchor=center, pos=0.125}, draw=none, from=1-1, to=2-2]
		\arrow[from=1-2, to=2-2]
		\arrow[from=2-1, to=2-2]
	\end{tikzcd}\end{equation}
	Consider the canonical commutative diagram
	\begin{equation}\label{eq:induced_pullback_of_modules}\begin{tikzcd}
		{\M\otimes_{\C}\D} & {\N\otimes_{\C}\D} \\
		{\M'\otimes_{\C}\D} & {\N'\otimes_{\C}\D}
		\arrow[from=1-1, to=1-2]
		\arrow[from=1-1, to=2-1]
		\arrow[from=1-2, to=2-2]
		\arrow[from=2-1, to=2-2]
	\end{tikzcd}\end{equation}
	\begin{enumerate}
		\item Suppose that $\D\in\Pr^{\dbl}_{\C}$. Then \eqref{eq:induced_pullback_of_modules} is a pullback diagram in $\PrL_{\C}$.
		
		Suppose moreover that \eqref{eq:pullback_of_modules} is a pullback diagram in $\CAlg(\PrL_\C)$ and $\D\in\CAlg(\PrL_{\C})$. Then \eqref{eq:induced_pullback_of_modules} is a pullback diagram in $\CAlg(\PrL_{\C})$.
		\item Suppose that \eqref{eq:pullback_of_modules} lies in $\Pr^{\iL}_{\C}$ and $\M'\to\N'$ is a localization. Then \eqref{eq:induced_pullback_of_modules} is a pullback diagram in $\Pr^{\iL}_{\C}$.
		
		Suppose moreover that \eqref{eq:pullback_of_modules} is a pullback diagram in $\CAlg(\PrL_{\C})$ and $\D\in\CAlg(\PrL_{\C})$. Then \eqref{eq:induced_pullback_of_modules} is a pullback diagram in $\CAlg(\PrL_{\C})$.
		\item Suppose that \eqref{eq:pullback_of_modules} lies in $\Pr^{\dbl}_{\C}$ and $\M'\to\N'$ is a localization, and $\D\in\Pr^{\dbl}_{\C}$. Then \eqref{eq:induced_pullback_of_modules} is a pullback diagram in $\Pr^{\dbl}_{\C}$.
		
		Suppose moreover that \eqref{eq:pullback_of_modules} is a pullback diagram in $\CAlg(\Pr^{\dbl}_{\C})$ and $\D\in\CAlg(\Pr^{\dbl}_{\C})$. Then \eqref{eq:induced_pullback_of_modules} is a pullback diagram in $\CAlg(\Pr^{\dbl}_{\C})$.
	\end{enumerate}
\end{lemma}
\begin{proof}
	\begin{enumerate}
		\item Since $\D\in\Pr^{\dbl}_{\C}$, $-\otimes_{\C}\D:\PrL_{\C}\to\PrL_{\C}$ preserves limits. The symmetric monoidal version follows from the fact that $\CAlg(\PrL_{\C})\to\PrL_{\C}$ reflects limits. 
		\item Since $\fib(\M'\to\N')\to\M'\to\N'$ is a short exact sequence in $\Pr^{\iL}_{\C}$, $\fib(\M'\to\N')\otimes_{\C}\D\to\M'\otimes_{\C}\D\to\N'\otimes_{\C}\D$ is a short exact sequence in $\Pr^{\iL}_{\C}$. Similarly, since $\M\to\N$ is also a $\C$-internal localization, $\fib(\M\to\N)\otimes_{\C}\D\to\M\otimes_{\C}\D\to\N\otimes_{\C}\D$ is also a short exact sequence in $\Pr^{\iL}_{\C}$. Then \eqref{eq:induced_pullback_of_modules} is both horizontally and vertically right adjointable, and induces equivalences on fibers of rows. By \cref{lem:ses_pullback}, \eqref{eq:induced_pullback_of_modules} is a pullback diagram in $\PrL_{\st}$, and hence in $\PrL_{\C}$ since $\PrL_{\C}\to\PrL_{\st}$ reflects limits. Moreover, by \cref{lem:pullback_dbl} we see that \eqref{eq:induced_pullback_of_modules} is a pullback diagram in $\Pr^{\iL}_{\C}$. The symmetric monoidal version again follows from the fact that $\CAlg(\PrL_{\C})\to\PrL_{\C}$ reflects limits. 
		\item Combining (2) and \cref{lem:pullback_dbl}, we see that \eqref{eq:induced_pullback_of_modules} is a pullback diagram in $\Pr^{\dbl}_{\C}$. The symmetric monoidal version follows from the fact that $\CAlg(\Pr^{\dbl}_{\C})\to\Pr^{\dbl}_{\C}$ reflects limits.
		\qedhere
	\end{enumerate}
\end{proof}

This lemma provides a mechanism for constructing pullback diagrams. In particular, let $\M\to\N$ and $\M'\to\N'$ be functors in $\Pr^{\iL}_{\C}$ where $\M\to\N$ is a localization. Then the commutative diagram is both horizontally and vertically right adjointable
\[\begin{tikzcd}
	{\M\otimes_{\C}\M'} & {\N\otimes_{\C}\M'} \\
	{\M\otimes_{\C}\N'} & {\N\otimes_{\C}\N'}
	\arrow[from=1-1, to=1-2]
	\arrow[from=1-1, to=2-1]
	\arrow[from=1-2, to=2-2]
	\arrow[from=2-1, to=2-2]
\end{tikzcd}\]
Thus, the right adjointability assumptions arise naturally in practice.

\begin{example}\label{exm:fracture_square_L_E_L_F}
	Let $\C\in\CAlg(\PrL_{\st})$, and let $E,F\in\C$. Suppose that $L_FL_E\simeq0$. Then by \cref{prop:fracture_square_functors} we have a pullback diagram of $\bbE_\infty$-rings in $L_{E\oplus F}\C$
	\[\begin{tikzcd}
		{L_{E\oplus F}\mathbbm{1}} & {L_F\mathbbm{1}} \\
		{L_E\mathbbm{1}} & {L_EL_F\mathbbm{1}}
		\arrow[from=1-1, to=1-2]
		\arrow[from=1-1, to=2-1]
		\arrow["\lrcorner"{anchor=center, pos=0.125}, draw=none, from=1-1, to=2-2]
		\arrow[from=1-2, to=2-2]
		\arrow[from=2-1, to=2-2]
	\end{tikzcd}\]
	Since $L_E:L_{E\oplus F}\C\to L_E\C$ is smashing, $L_{E\oplus F}\mathbbm{1}\to L_E\mathbbm{1}$ is idempotent. By \cref{prop:Tam_pullback_homoepi}, we have a pullback diagram in $\CAlg(\PrL_{\st})$
	\[\begin{tikzcd}
		{L_{E\oplus F}\C} & {\Mod_{L_F\mathbbm{1}}(L_{E\oplus F}\C)} \\
		{\Mod_{L_E\mathbbm{1}}(L_{E\oplus F}\C)} & {\Mod_{L_EL_F\mathbbm{1}}(L_{E\oplus F}\C)}
		\arrow[from=1-1, to=1-2]
		\arrow[from=1-1, to=2-1]
		\arrow["\lrcorner"{anchor=center, pos=0.125}, draw=none, from=1-1, to=2-2]
		\arrow[from=1-2, to=2-2]
		\arrow[from=2-1, to=2-2]
	\end{tikzcd}\]
	which is also a pushout in $\CAlg(\PrL_{\st})$. Since this pullback diagram satisfies \cref{lem:pullback_mechanism}(2), for an arbitrary $\M\in\PrL_{L_{E\oplus F}\C}$, we have a pullback diagram in $\PrL_{\C}$
	\[\begin{tikzcd}
		{\M} & {\Mod_{L_F\mathbbm{1}}(\M)} \\
		{\Mod_{L_E\mathbbm{1}}(\M)} & {\Mod_{L_EL_F\mathbbm{1}}(\M)}
		\arrow[from=1-1, to=1-2]
		\arrow[from=1-1, to=2-1]
		\arrow["\lrcorner"{anchor=center, pos=0.125}, draw=none, from=1-1, to=2-2]
		\arrow[from=1-2, to=2-2]
		\arrow[from=2-1, to=2-2]
	\end{tikzcd}\]
	If $\M\in\Pr^{\dbl}_{\C}$, then the pullback diagram is a pullback diagram in $\Pr^{\dbl}_{\C}$.

	For example, take $E=\bbS[1/p]$ and $F=\bbS/p$ as in \cref{exm:S[1/p]_S/p}. Then the desired pullback diagram in $\CAlg(\PrL_{\st})$ is given by
	\[\begin{tikzcd}
		{\Sp} & {\Mod(\bbS^{\wedge}_p)} \\
		{\Sp[1/p]} & {\Mod(\bbS^{\wedge}_p)[1/p]}
		\arrow[from=1-1, to=1-2]
		\arrow[from=1-1, to=2-1]
		\arrow["\lrcorner"{anchor=center, pos=0.125}, draw=none, from=1-1, to=2-2]
		\arrow[from=1-2, to=2-2]
		\arrow[from=2-1, to=2-2]
	\end{tikzcd}\]
\end{example}
In \cite{NPR24}, Naumann--Pol--Ramzi construct a symmetric monoidal fracture square, based on the following construction used in \cite{LMMT24} to prove the $T(n)$-local Galois descent of $K$-theory. In this spirit, we reprove this descent result in \cref{sec:3}.
\begin{construction}[{\cite[Construction 4.86]{Ram26locallyrigid}, \cite[Remark 7.21]{Aok25}}]\label{constr:Ind_dbl}
	We have an adjunction
		\[\begin{tikzcd}
			{\CAlg^{\rig}(\PrL_{\st})^{\cg}} & {\CAlg^{\rig}(\PrL_{\st})} & {\CAlg(\PrL_{\st})}
			\arrow[shift left, hook, from=1-1, to=1-2]
			\arrow["{\Ind((-)^{\dbl})}", shift left, from=1-2, to=1-1]
			\arrow[shift left, hook, from=1-2, to=1-3]
			\arrow["{(-)^{\rig}}", shift left, from=1-3, to=1-2]
		\end{tikzcd}\]
	where $\CAlg^{\rig}(\PrL_{\st})^{\cg}\subseteq\CAlg^{\rig}(\PrL_{\st})$ denotes the full subcategory spanned by compactly generated rigid algebras. Moreover, we have $\Ind((\C^{\rig})^{\dbl})\simeq\Ind(\C^{\dbl})\hookrightarrow\C^{\rig}$ for $\C\in\CAlg^{\rig}(\PrL_{\st})$.

	Let $\C\in\CAlg(\PrL_{\st})$. Then any $\Sp$-algebra map $\D\to\C$ from $\D\in\CAlg^{\rig}(\PrL_{\st})^{\cg}$ factors through $\Ind(\C^{\dbl})\hookrightarrow\C^{\rig}\to\C$, and it induces an $\Sp$-algebra map $\D\to\Ind(\C^{\dbl})$.
\end{construction}
\begin{theorem}[{\cite[Theorem 5.11]{NPR24}}]\label{thm:NPR24}
	Let $\C\in\CAlg^{\rig}(\PrL_{\st})^{\cg}$, and let $\K\subseteq\C^{\dbl}$. The completion functor $\Lambda^{\K}:\C\to\C^{\K\text{-}\cplt}$ induces an $\Sp$-algebra map $\C\to\Ind((\C^{\K\text{-}\cplt})^{\dbl})$, which is a $\C$-internal left adjoint. Then the canonical commutative diagram is a pullback diagram in $\CAlg(\PrL_{\st})$
	\[\begin{tikzcd}
		\C & {\C^{\K\text{-}\loc}} \\
		{\Ind((\C^{\K\text{-}\cplt})^{\dbl})} & {\C^{\K\text{-}\loc}\otimes_{\C}\Ind((\C^{\K\text{-}\cplt})^{\dbl})}
		\arrow[from=1-1, to=1-2]
		\arrow[from=1-1, to=2-1]
		\arrow[from=1-2, to=2-2]
		\arrow[""{name=0, anchor=center, inner sep=0}, from=2-1, to=2-2]
		\arrow["\lrcorner"{anchor=center, pos=0.125}, draw=none, from=1-1, to=0]
	\end{tikzcd}\]
\end{theorem}
In fact, this pullback diagram is a special case of the following fracture square.
\begin{proposition}\label{prop:pullback_diagram_of_completion}
	Let $\C\in\CAlg(\PrL_{\st})$, $\I\subseteq\C$ be a smashing ideal, and let $\A:=\C/\I$. Let $\M\in\Pr^{\dbl}_{\C}$. Then the canonical commutative diagram is a pullback diagram in $\Pr^{\dbl}_{\C}$
	\[\begin{tikzcd}
		\M & {\A\otimes_{\C}\M} \\
		{\M^{\wedge\I}} & {\M_{\eta}}
		\arrow[from=1-1, to=1-2]
		\arrow[from=1-1, to=2-1]
		\arrow["\lrcorner"{anchor=center, pos=0.125}, draw=none, from=1-1, to=2-2]
		\arrow[from=1-2, to=2-2]
		\arrow[from=2-1, to=2-2]
	\end{tikzcd}\]
	In particular, if $\C\in\CAlg^{\rig}(\PrL_{\st})$, then it is a pullback diagram in $\Pr^{\dbl}_{\st}$.
\end{proposition}
\begin{proof}
	From \cref{cor:generic_fiber} we have a commutative diagram between short exact sequences in $\Pr^{\iL}_{\C}$
	\[\begin{tikzcd}
		{\I\otimes_{\C}\M} & \M & {\A\otimes_{\C}\M} \\
		{\I\otimes_{\C}\M^{\wedge\I}} & {\M^{\wedge\I}} & {\M_{\eta}}
		\arrow[hook, from=1-1, to=1-2]
		\arrow["\simeq"', from=1-1, to=2-1]
		\arrow[from=1-2, to=1-3]
		\arrow[from=1-2, to=2-2]
		\arrow[from=1-3, to=2-3]
		\arrow[hook, from=2-1, to=2-2]
		\arrow[from=2-2, to=2-3]
	\end{tikzcd}\]
	where the left vertical arrow is an equivalence by \cref{thm:unst_completion}. Since $\C\to\A$ is a $\C$-internal localization and $\M\to\M^{\wedge\I}$ is a $\C$-internal left adjoint, the right square is both horizontally and vertically right adjointable by the discussion following \cref{lem:pullback_mechanism}. By \cref{lem:ses_pullback}, the right square is a pullback diagram in $\PrL_{\C}$, and hence in $\Pr^{\dbl}_{\C}$ by \cref{lem:pullback_dbl}.
\end{proof}
\begin{theorem}\label{thm:pullback_of_map_of_rigid_envelope}
	Let $\C\in\CAlg(\PrL_{\st})$, $\I\subseteq\C$ be a smashing ideal, and let $\A:=\C/\I$. Then the canonical commutative diagram is a pullback diagram in $\CAlg(\Pr^{\dbl}_{\C})$
	\begin{equation}\label{eq:pullback_I_rig}\begin{tikzcd}
		\C & \A \\
		{\I^{\rig}_{\C}} & {\I_{\eta}}
		\arrow[from=1-1, to=1-2]
		\arrow[from=1-1, to=2-1]
		\arrow["\lrcorner"{anchor=center, pos=0.125}, draw=none, from=1-1, to=2-2]
		\arrow[from=1-2, to=2-2]
		\arrow[from=2-1, to=2-2]
	\end{tikzcd}\end{equation}
	In particular, if $\C\in\CAlg^{\rig}(\PrL_{\st})$, then $\I^{\rig}_{\C}\simeq\I^{\rig}$ and the diagram is a pullback diagram in $\CAlg(\Pr^{\dbl}_{\st})$.

	More generally, let $\C\in\CAlg(\PrL_{\st})$ and $\I\in\CAlg(\PrL_{\C})$ be locally rigid. Let $\overline{\C}\to\widetilde{\C}$ be a map of $\C$-rigid envelopes of $\I$, and let $\A:=\overline{\C}/\I$. Then the canonical diagram is a pullback diagram in $\CAlg(\Pr^{\dbl}_{\C})$
	\[\begin{tikzcd}
		{\overline{\C}} & \A \\
		{\widetilde{\C}} & {\A\otimes_{\overline{\C}}\widetilde{\C}}
		\arrow[from=1-1, to=1-2]
		\arrow[from=1-1, to=2-1]
		\arrow["\lrcorner"{anchor=center, pos=0.125}, draw=none, from=1-1, to=2-2]
		\arrow[from=1-2, to=2-2]
		\arrow[from=2-1, to=2-2]
	\end{tikzcd}\]
	In particular, if $\C\in\CAlg^{\rig}(\PrL_{\st})$, then the diagram is a pullback diagram in $\CAlg(\Pr^{\dbl}_{\st})$.
\end{theorem}
\begin{proof}
	For the case that $\I\subseteq\C$ is a smashing ideal, the first part follows by applying the second part to the map of $\C$-rigid envelopes $\C\to\I^{\rig}_{\C}$. 

	For the second part, consider the commutative diagram between short exact sequences in $\Pr^{\iL}_{\C}$
	\[\begin{tikzcd}
		{\I} & {\overline{\C}} & {\A}\\
		{\I\otimes_{\overline{\C}}\widetilde{\C}} & {\widetilde{\C}} & {\A\otimes_{\overline{\C}}\widetilde{\C}}
		\arrow[hook, from=1-1, to=1-2]
		\arrow["\simeq"', from=1-1, to=2-1]
		\arrow[from=1-2, to=1-3]
		\arrow[from=1-2, to=2-2]
		\arrow[from=1-3, to=2-3]
		\arrow[hook, from=2-1, to=2-2]
		\arrow[from=2-2, to=2-3]
	\end{tikzcd}\]
	Then the result follows from \cref{lem:ses_pullback} and \cref{lem:pullback_dbl}.
\end{proof}
\begin{remark}
	By \cite[Corollary 4.89]{Ram26locallyrigid}, the forgetful functor $\CAlg^{\rig}(\PrL_{\st})\to\Pr^{\dbl}_{\st}$ preserves limits. Let $\C\in\CAlg^{\rig}(\PrL_{\st})$ and $\I\subseteq\C$ be a smashing ideal, or more generally, $\I\in\CAlg(\PrL_{\C})$ be locally rigid. Then the diagrams in \cref{thm:pullback_of_map_of_rigid_envelope} are pullback diagrams in $\CAlg^{\rig}(\PrL_{\st})$.
\end{remark}
\begin{corollary}
	Let $\C\in\CAlg^{\rig}(\PrL_{\st})$, $\I\subseteq\C$ be a smashing ideal, and let $\A:=\C/\I$. Then we have a pullback diagram of $\bbE_\infty$-rings
	\[\begin{tikzcd}
		{\End(\mathbbm{1}_{\C})} & {\End(\mathbbm{1}_{\A})} \\
		{\End(\mathbbm{1}_{\I^{\rig}})} & {\End(\mathbbm{1}_{\I_{\eta}})}
		\arrow[from=1-1, to=1-2]
		\arrow[from=1-1, to=2-1]
		\arrow["\lrcorner"{anchor=center, pos=0.125}, draw=none, from=1-1, to=2-2]
		\arrow[from=1-2, to=2-2]
		\arrow[from=2-1, to=2-2]
	\end{tikzcd}\]
	
	More generally, let $\C\in\CAlg(\PrL_{\st})$ and $\I\in\CAlg(\PrL_{\C})$ be locally rigid. Let $\overline{\C}\to\widetilde{\C}$ be a map of $\C$-rigid envelopes of $\I$, and let $\A:=\overline{\C}/\I$. Then we have a pullback diagram in $\CAlg(\C)$
	\[\begin{tikzcd}
		{\iEnd_{\C}(\mathbbm{1}_{\overline{\C}})} & {\iEnd_{\C}(\mathbbm{1}_{\A})} \\
		{\iEnd_{\C}(\mathbbm{1}_{\widetilde{\C}})} & {\iEnd_{\C}(\mathbbm{1}_{\A\otimes_{\overline{\C}}\widetilde{\C}})}
		\arrow[from=1-1, to=1-2]
		\arrow[from=1-1, to=2-1]
		\arrow["\lrcorner"{anchor=center, pos=0.125}, draw=none, from=1-1, to=2-2]
		\arrow[from=1-2, to=2-2]
		\arrow[from=2-1, to=2-2]
	\end{tikzcd}\]
\end{corollary}
\begin{corollary}\label{cor:Ind_dbl}
	Let $\C\in\CAlg^{\rig}(\PrL_{\st})^{\cg}$ and $\I\subseteq\C$ be a smashing ideal, and let $\A:=\C/\I$. Then the canonical commutative diagrams are pullback diagrams in $\CAlg^{\rig}(\PrL_{\st})$
	\[\begin{tikzcd}
		\C & \A \\
		{\Ind(\I^{\dbl})} & {\A\otimes_{\C}\Ind(\I^{\dbl})} \\
		{\I^{\rig}} & {\A\otimes_{\C}\I^{\rig}}
		\arrow[from=1-1, to=1-2]
		\arrow[from=1-1, to=2-1]
		\arrow[from=1-2, to=2-2]
		\arrow[""{name=0, anchor=center, inner sep=0}, from=2-1, to=2-2]
		\arrow[hook, from=2-1, to=3-1]
		\arrow[from=2-2, to=3-2]
		\arrow[""{name=1, anchor=center, inner sep=0}, from=3-1, to=3-2]
		\arrow["\lrcorner"{anchor=center, pos=0.125}, draw=none, from=1-1, to=0]
		\arrow["\lrcorner"{anchor=center, pos=0.125}, draw=none, from=2-1, to=1]
	\end{tikzcd}\]
	In particular, this recovers \cref{thm:NPR24} by taking $\I=\C^{\K\text{-}\tors}$.
\end{corollary}
\begin{proof}
	From \cref{constr:Ind_dbl} we obtain a fully faithful $\Sp$-algebra map $\Ind(\I^{\dbl})\hookrightarrow\I^{\rig}$ over $\I$, and an $\Sp$-algebra map $\C\to\Ind(\I^{\dbl})$ as $\C$ is rigid compactly generated. Since $\Ind(\I^{\dbl})\hookrightarrow\I^{\rig}$ is fully faithful, by \cref{prop:map_of_rigid_envelopes}, they are maps of $\C$-rigid envelopes of $\I$. Thus, the result follows from \cref{thm:pullback_of_map_of_rigid_envelope}.
\end{proof}
\begin{example}\label{exm:Ind(R_I^dbl)}
	Let $R\in\CAlg(\Sp)$, and let $I\subseteq\pi_*(R)$ be a finitely generated ideal. Then $\Mod(R)\in\CAlg^{\rig}(\PrL_{\st})^{\cg}$, $\Mod(R)_{I\text{-}\tors}\subseteq\Mod(R)$ is a smashing ideal, and $\Mod(R)/\Mod(R)_{I\text{-}\tors}\simeq\Mod(R)_{I\text{-}\loc}\simeq\Mod(L_IR)$. Observe that by \cref{lem:LC=LC(Mod(L1))}, we have $\Mod(R)_{I\text{-}\cplt}\simeq\Mod(R^{\wedge}_I)_{I\text{-}\cplt}$. Then the inclusion $\Perf(R^{\wedge}_I)\subseteq\Mod(R)_{I\text{-}\cplt}^{\dbl}$ induces a fully faithful $\Sp$-algebra map $\Mod(R^{\wedge}_I)\hookrightarrow\Ind(\Mod(R)^{\dbl}_{I\text{-}\cplt})$. Since $\Ind(\Mod(R)^{\dbl}_{I\text{-}\cplt})\hookrightarrow(\Mod(R)_{I\text{-}\cplt})^{\rig}$ is fully faithful, $\Mod(R^{\wedge}_I)\hookrightarrow(\Mod(R)_{I\text{-}\cplt})^{\rig}$ is also fully faithful and hence a map of rigid envelopes by \cref{prop:map_of_rigid_envelopes}. In fact, this follows directly from \cref{lem:1_rig}. Therefore, we obtain pullback diagrams
	\[\begin{tikzcd}
		{\Mod(R)} & {\Mod(L_IR)}\\
		{\Mod(R^{\wedge}_I)} & {\Mod(L_IR\otimes_RR^{\wedge}_I)}\\
		{\Ind(\Mod(R)^{\dbl}_{I\text{-}\cplt})} & {\Mod(L_IR)\otimes_R\Ind(\Mod(R)^{\dbl}_{I\text{-}\cplt})}\\
		{(\Mod(R)_{I\text{-}\cplt})^{\rig}} & {\Mod(L_IR)\otimes_R(\Mod(R)_{I\text{-}\cplt})^{\rig}}
		\arrow[from=1-1, to=1-2]
		\arrow[from=1-1, to=2-1]
		\arrow[from=1-2, to=2-2]
		\arrow[""{name=0, anchor=center, inner sep=0}, from=2-1, to=2-2]
		\arrow[hook, from=2-1, to=3-1]
		\arrow[hook, from=2-2, to=3-2]
		\arrow[""{name=1, anchor=center, inner sep=0}, from=3-1, to=3-2]
		\arrow[hook, from=3-1, to=4-1]
		\arrow[hook, from=3-2, to=4-2]
		\arrow[""{name=2, anchor=center, inner sep=0}, from=4-1, to=4-2]
		\arrow["\lrcorner"{anchor=center, pos=0.125}, draw=none, from=1-1, to=0]
		\arrow["\lrcorner"{anchor=center, pos=0.125}, draw=none, from=2-1, to=1]
		\arrow["\lrcorner"{anchor=center, pos=0.125}, draw=none, from=3-1, to=2]
	\end{tikzcd}\]
	The top pullback diagrams recover the pullback square coming from \cref{exm:fracture_square_L_E_L_F,exm:Mod(R)_I-tors/cplt}. 

	The middle vertical maps are equivalences whenever $R$ is an \emph{adic $\bbE_\infty$-ring}, i.e. $R$ is connective and $I\subseteq\pi_0(R)$ by \cite[Theorem 11.3]{NP24}. Conversely, \cite[Remark 11.5]{NP24} shows that the middle vertical maps need not be equivalences when $R$ is nonconnective.
\end{example}
Consequently, we may construct pullback diagrams of modules from these symmetric monoidal fracture squares via \cref{lem:pullback_mechanism}.

\subsection{Chromatic layers and nuclear modules}\label{sec:2.3}

We apply the theories developed in the previous subsections to chromatic homotopy theory, reformulating parts of \cite[\Section 6]{BHV18}. We then introduce nuclear module categories as a relative version of nuclear categories, and construct fracture squares for nuclear module categories associated to chromatic layers.

We fix a prime $p$. For height $n=0$, let $F(0):=\bbS$ with $v_0=p^k$. For height $n\geq1$, let $F(n)$ denote a type $n$ finite $p$-local spectrum, equipped with a $v_n$-self map $v_n^{k_n}$ \cite{HS98}. Moreover, we can assume that $F(n)\simeq F(n-1)/v_{n-1}$, which is called a \emph{type $n$ generalized Moore spectrum} \cite{HS99} and can be constructed inductively as
\[
	F(n)\simeq\bbS/(p^k,v_1^{k_1},\cdots,v_{n-1}^{k_{n-1}}).
\]
Denote $T(n):=F(n)[v_n^{-1}]$, $L_n^f:=L_{T(0)\oplus\cdots\oplus T(n)}$, and $C_n^f:=C_{T(0)\oplus\cdots\oplus T(n)}$. Recall that $T(m)\otimes T(n)\simeq0$ and $F(m)\otimes T(n)\simeq0$ for all $m>n$, and $L_{F(m)\otimes T(n)}\simeq L_{T(n)\otimes T(n)}\simeq L_{T(n)}$ for all $m\leq n$.
\begin{lemma}\label{lem:L_n^f}
	Let $n\geq0$. Then we have a short exact sequence in $\Pr^{\dbl}_{\st}$
	\[\begin{tikzcd}[column sep=4em]
		{L_{F(n+1)}\Sp} & \Sp & {L_n^f\Sp}
		\arrow["{C_n^f}", shift left=3, hook, from=1-1, to=1-2]
		\arrow[shift right=3, hook, from=1-1, to=1-2]
		\arrow["{L_{F(n+1)}}"{description}, from=1-2, to=1-1]
		\arrow["{L_n^f}", shift left=3, from=1-2, to=1-3]
		\arrow[shift right=3, from=1-2, to=1-3]
		\arrow[hook', from=1-3, to=1-2]
	\end{tikzcd}\]
	yielding inverse equivalences
	\[
		C_n^f:L_{F(n+1)}\Sp\simeq C_n^f\Sp:L_{F(n+1)}.
	\]
	Then $L_n^f$ is an atomic (or finite) smashing localization and $L_{F(n+1)}\Sp\otimes L_n^f\Sp\simeq0$. In particular, this recovers \cref{exm:S[1/p]_S/p} by taking $n=0$.

	Moreover, let $\M\in\PrL_{\st}$. Then we have $L_{F(n+1)}\Sp\otimes\M\simeq L_{F(n+1)}\M$ and $L_n^f\Sp\otimes\M\simeq L_n^f\M$.
\end{lemma}
\begin{proof}
	Observe that $(T(0)\oplus\cdots\oplus T(n))\oplus F(n+1)$ detects $0$, and $(T(0)\oplus\cdots\oplus T(n))\otimes F(n+1)\simeq0$. Then applying \cref{prop:fracture_square_functors} gives us the desired short exact sequence. The last part follows from \cref{cor:L_KM}.
\end{proof}

More generally, we introduce the notion of length-$m$ chromatic layer.
\begin{definition}\label{def:chromatic_layer}
	Let $n\geq0$ and $1\leq m\leq n+1$. We define the \emph{$n$-th length-$m$ chromatic layer spectrum} as
	\[
		M_n^{m,f}:=\fib(L_n^f\bbS\to L_{n-m}^f\bbS),
	\]
	where $L_{-1}^f\bbS:=0$, and denote $M_n^{m,f}(-):=\fib(L_n^f\to L_{n-m}^f)\simeq M_n^{m,f}\otimes-$. Then $M_n^{1,f}$ is the \emph{$n$-th monochromatic layer spectrum} $M_n^f:=\fib(L_n^f\bbS\to L_{n-1}^f\bbS)$. For $n\geq0$, we also define
	\[
		C_n^f:=\fib(\bbS\to L_n^f\bbS).
	\]
\end{definition}
\begin{proposition}\label{prop:multichromatic_Hunital}
	Let $n\geq0$ and $1\leq m\leq n+1$. Then $M_n^{m,f}$ is an $H$-unital $\bbE_\infty$-ring, and $\Mod_H(M_n^{m,f})$ is a locally rigid $\Sp$-algebra.
\end{proposition}
\begin{proof}
	Since $L_n^f\bbS\to L_{n-m}^f\bbS$ is an $\bbE_\infty$-ring map between idempotent $\bbS$-algebras, the fiber $M_n^{m,f}$ is an $H$-unital $\bbE_\infty$-ring by \cref{prop:pullback_homoepi_idem}. Then $\Mod_H(M_n^{m,f})$ is a smashing ideal of $L_n^f\Sp$, and hence a locally rigid $\Sp$-algebra.
\end{proof}
\begin{proposition}\label{prop:M_n^mf}
	Let $n\geq0$ and $1\leq m\leq n+1$. Then we have a short exact sequence in $\Pr^{\dbl}_{\st}$
	\[\begin{tikzcd}[column sep=6em]
		{L_{T(n)\oplus\cdots\oplus T(n-m+1)}\Sp} & {L_n^f\Sp} & {L_{n-m}^f\Sp}
		\arrow["{M_n^{m,f}}", shift left=3, hook, from=1-1, to=1-2]
		\arrow[shift right=3, hook, from=1-1, to=1-2]
		\arrow["{L_{F(n-m+1)}}"{description}, from=1-2, to=1-1]
		\arrow["{L_{n-m}^f}", shift left=3, from=1-2, to=1-3]
		\arrow[shift right=3, from=1-2, to=1-3]
		\arrow[hook', from=1-3, to=1-2]
	\end{tikzcd}\]
	yielding inverse equivalences
	\[
		M_n^{m,f}:L_{T(n)\oplus\cdots\oplus T(n-m+1)}\Sp\simeq M_n^{m,f}\Sp:L_{T(n)\oplus\cdots\oplus T(n-m+1)}.
	\]
	Then we have 
	\[
		L_{T(n)\oplus\cdots\oplus T(n-m+1)}\Sp\simeq L_n^f\Sp\otimes L_{F(n-m+1)}\Sp\simeq L_{L_n^fF(n-m+1)}\Sp,
	\]
	and $L_{T(n)\oplus\cdots\oplus T(n-m+1)}\Sp\otimes L_{n-m}^f\Sp\simeq0$. In particular, taking $m=1$ gives the $n$-th monochromatic layer category $L_{T(n)}\Sp\simeq M_n^f\Sp$.

	Moreover, let $\M\in\PrL_{\st}$. Then we have $L_{T(n)\oplus\cdots\oplus T(n-m+1)}\Sp\otimes\M\simeq L_{L_n^fF(n-m+1)}\M\simeq M_n^{m,f}\M$.
\end{proposition}
\begin{proof}
	Since both $(T(0)\oplus\cdots\oplus T(n-m))$ and $(T(0)\oplus\cdots\oplus T(n-m))\oplus(T(n-m+1)\oplus\cdots\oplus T(n))$ are smashing, and $(T(0)\oplus\cdots\oplus T(n-m))\otimes(T(n-m+1)\oplus\cdots\oplus T(n))\simeq0$, applying \cref{prop:fracture_square_functors} gives the desired short exact sequence. 
	
	We apply $L_n^f\Sp\otimes-$ to the short exact sequence $L_{F(n-m+1)}\Sp\hookrightarrow\Sp\to L_{n-m}^f\Sp$. Since $(T(0)\oplus\cdots\oplus T(n))\otimes F(n-m+1)\simeq(T(n-m+1)\oplus\cdots\oplus T(n))\otimes F(n-m+1)$, we have $L_n^f\Sp\otimes L_{F(n-m+1)}\Sp\simeq L_{L_n^fF(n-m+1)}\Sp\simeq L_{(T(0)\oplus\cdots\oplus T(n))\otimes F(n-m+1)}\Sp\simeq L_{T(n-m+1)\oplus\cdots\oplus T(n)}\Sp$ by \cref{prop:L_FL_E}. Then by \cref{lem:LC=LC(Mod(L1))}, the localization functor $L_n^f\Sp\to L_{F(n-m+1)}L_n^f\Sp$ is identified with $L_{F(n-m+1)}$.

	Let $\M\in\PrL_{\st}$. Then $L_{L_n^fF(n-m+1)}\M\simeq L_{L_n^f\bbS\otimes F(n-m+1)}\M\simeq L_{F(n-m+1)}L_n^f\M\simeq L_{F(n-m+1)}\Sp\otimes L_n^f\Sp\otimes\M\simeq L_{T(n)\oplus\cdots\oplus T(n-m+1)}\Sp\otimes\M\simeq M_n^{m,f}\M$ by \cref{prop:L_FL_E,prop:L_AM}.
\end{proof}
\begin{corollary}
	Let $n\geq0$ and $1\leq m\leq n+1$. Then we have an equivalence of locally rigid $\Sp$-algebras
	\[
		L_{T(n)\oplus\cdots\oplus T(n-m+1)}\Sp\simeq\Mod_H(M_n^{m,f}),
	\]
	with rigid envelope $L_n^f\Sp\to L_{T(n)\oplus\cdots\oplus T(n-m+1)}\Sp$.
\end{corollary}

\begin{corollary}
	Let $n\geq0$ and $\M\in\PrL_{\st}$. Then we have $L_{T(n)}\Sp\otimes\M\simeq L_{T(n)}\M$.
\end{corollary}
\begin{proof}
	It follows from $L_n^fF(n)\simeq T(n)$ (e.g. \cite[Lecture 28, Proposition 1]{Lur10}).
\end{proof}
\begin{remark}
	Let $\M\in\PrL_{\st}$. In \cite[Definition 5.5.1]{CSY21height}, Carmeli--Schlank--Yanovski introduce the notion of \emph{stable height} and define $\M_{\leq^{\st}n}:=C_{F(n+1)}\M$, $\M_{>^{\st}n}:=L_{F(n+1)}\M$, and $\M_{n^{\st}}:=C_{F(n+1)}\M\cap L_{F(n)}\M$. Then \cref{lem:L_n^f,prop:M_n^mf} recover \cite[Proposition 5.5.5]{CSY21height}, i.e. $\M_{\leq^{\st}n}\simeq L_n^f\Sp\otimes\M$, $\M_{>^{\st}n}\simeq L_{F(n+1)}\Sp\otimes\M$ and $\M_{n^{\st}}\simeq L_{T(n)}\M\simeq L_{T(n)}\Sp\otimes\M$.
\end{remark}
Next, we turn to Morava $E$-theories and reformulate some results in \cite{HS99} (see also \cite{Hea23}).

For height $n=0$, let $E_0=K(0)=\bbQ$. For height $n\geq1$, let $E_n$ denote the \emph{even periodic Morava $E$-theory $\bbE_\infty$-ring of height $n$} with coefficients
\[
	\pi_*E_n\simeq\bbZ_p[\![v_1,\cdots,v_{n-1}]\!][u^{\pm1}],\quad|v_i|=0,~|u|=2.
\]
Let $K(n)$ denote the \emph{Morava $K$-theory of height $n$} with coefficients
\[
	\pi_*K(n)\simeq\bbF_p[v_n^{\pm1}],\quad|v_n|=2(p^n-1).
\]
Denote $L_n:=L_{T(0)\oplus K(1)\oplus\cdots\oplus K(n)}$. Recall that $L_{K(n)}\Sp\subseteq L_{T(n)}\Sp$ and $L_n\Sp\subseteq L_n^f\Sp$, $K(m)\otimes K(n)\simeq0$ and $K(m)\otimes F(n)\simeq 0$ for $m<n$, $L_{K(n)\otimes K(n)}\simeq L_{K(n)}$, and $L_{K(m)\otimes F(n)}\simeq L_{K(m)}$ for $m\geq n$.
\begin{lemma}
	$L_n$ is a smashing localization. Moreover, let $\M\in\PrL_{\st}$. Then we have $L_n\Sp\otimes\M\simeq L_n\M$.
\end{lemma}
\begin{proof}
	Observe that $L_{\bbQ\oplus\{\bbS_{(q)}\}_{q\neq p\text{ prime}}}\simeq L_{T(0)}$. Then $L_n\simeq L_{E_n\oplus\{\bbS_{(q)}\}_{q\neq p\text{ prime}}}$ and $\bbQ\otimes L_n\simeq\bbQ\otimes-$, so that $(L_n)_{(p)}\simeq L_{K(0)\oplus\cdots\oplus K(n)}\simeq L_{E_n}$ and $(L_n)_{(q)}\simeq(-)_{(q)}$ are smashing localizations of $\Sp_{(p)}$ and $\Sp_{(q)}$ for all primes $q\neq p$, respectively. By \cite[Theorem 6.6]{Lia26telescope}, we obtain that $L_n$ is a smashing localization.
\end{proof}
Moreover, since $\pi_*K(n)$ is a graded field, for every spectrum $X\in\Sp$ such that $K(n)\otimes X\neq0$, $K(n)\otimes X$ is a direct sum of shifts of $K(n)$. It follows that every $L_{K(n)}\Sp$ is a \emph{smashing field} \cite[Definition 1.11]{Lia26telescope}.
\begin{proposition}\label{prop:L_n}
	Let $n\geq1$ and $1\leq m\leq n$. Then we have a short exact sequence in $\Pr^{\dbl}_{\st}$
	\[\begin{tikzcd}[column sep=6em]
		{L_{K(n)\oplus\cdots\oplus K(n-m+1)}\Sp} & L_n\Sp & {L_{n-m}\Sp}
		\arrow[shift left=3, hook, from=1-1, to=1-2]
		\arrow[shift right=3, hook, from=1-1, to=1-2]
		\arrow["{L_{F(n-m+1)}}"{description}, from=1-2, to=1-1]
		\arrow["{L_{n-m}}", shift left=3, from=1-2, to=1-3]
		\arrow[shift right=3, from=1-2, to=1-3]
		\arrow[hook', from=1-3, to=1-2]
	\end{tikzcd}\]
	Then we have
	\begin{align*}
		L_{K(n)\oplus\cdots\oplus K(n-m+1)}\Sp&\simeq L_n\Sp\otimes L_{T(n)\oplus\cdots\oplus T(n-m+1)}\Sp\\
		&\simeq L_n\Sp\otimes L_{F(n-m+1)}\Sp\\
		&\simeq L_{L_nF(n-m+1)}\Sp,
	\end{align*}
	and $L_{K(n)\oplus\cdots\oplus K(n-m+1)}\Sp\otimes L_{n-m}\Sp\simeq0$. In particular, taking $m=1$ gives $L_{K(n)}\Sp\simeq L_{T(n)}\Sp\otimes L_n\Sp$.

	Moreover, let $\M\in\PrL_{\st}$. Then we have $L_{K(n)\oplus\cdots\oplus K(n-m+1)}\Sp\otimes\M\simeq L_{L_nF(n-m+1)}\M$.
\end{proposition}
\begin{proof}
	Since both $T(0)\oplus K(1)\oplus\cdots\oplus K(n-m)$ and $(T(0)\oplus K(1)\oplus\cdots\oplus K(n-m))\oplus(K(n-m+1)\oplus\cdots\oplus K(n))$ are smashing, and $(T(0)\oplus K(1)\oplus\cdots\oplus K(n-m))\otimes(K(n-m+1)\oplus\cdots\oplus K(n))\simeq0$, applying \cref{prop:fracture_square_functors} gives the desired short exact sequence. 
	
	We apply $L_n\Sp\otimes-$ to the short exact sequence $L_{F(n-m+1)}\Sp\hookrightarrow\Sp\to L_{n-m}^f\Sp$. Since we have $L_{(T(0)\oplus K(1)\oplus\cdots\oplus K(n))\otimes F(n-m+1)}\Sp\simeq L_{K(n-m+1)\oplus\cdots\oplus K(n)}\Sp$, $L_n\Sp\otimes L_{F(n-m+1)}\Sp\simeq L_{L_nF(n-m+1)}\Sp\simeq L_{(T(0)\oplus K(1)\oplus\cdots\oplus K(n))\otimes F(n-m+1)}\Sp\simeq L_{K(n-m+1)\oplus\cdots\oplus K(n)}\Sp$ by \cref{prop:L_FL_E}. It follows that $L_n\Sp\otimes L_{n-m}^f\Sp\simeq L_{n-m}\Sp$. Therefore, \cref{lem:LC=LC(Mod(L1))} implies that the localization $L_n\Sp\to L_{F(n-m+1)}L_n\Sp$ is identified with $L_{F(n-m+1)}$.
	
	On the other hand, applying $L_n\Sp\otimes-$ to the short exact sequence $L_{T(n)\oplus\cdots\oplus T(n-m+1)}\Sp\hookrightarrow L_n^f\Sp\to L_{n-m}^f\Sp$ gives equivalences $L_n\Sp\otimes L_{T(n)\oplus\cdots\oplus T(n-m+1)}\Sp\simeq L_n\Sp\otimes L_{K(n)\oplus\cdots\oplus K(n-m+1)}\Sp$.

	Let $\M\in\PrL_{\st}$. Then $L_{L_nF(n-m+1)}\M\simeq L_{L_n\bbS\otimes F(n-m+1)}\M\simeq L_{F(n-m+1)}L_n\M\simeq L_{F(n-m+1)}\Sp\otimes L_n\Sp\otimes\M\simeq L_{K(n)\oplus\cdots\oplus K(n-m+1)}\Sp\otimes\M$.
\end{proof}

We now specialize the theory of rigid envelopes to locally rigid categories with $\omega_1$-compact unit, a condition satisfied by both $\Mod(R)_{I\text{-}\cplt}$ and $L_{T(n)\oplus\cdots\oplus T(n-m+1)}\Sp$. In this setting, rigidification admits a concrete description in terms of nuclear objects. We introduce the notion of nuclear module categories and construct symmetric monoidal fracture squares for them.
\begin{definition}[{\cite[Definition 8.1]{CS26complex}}]
	Let $\C\in\CAlg(\PrL)$. We say a map $X\to Y\in\C$ is \emph{trace-class} if the map $\mathbbm{1}\to\iHom_{\C}(X,Y)$ corresponding to $X\to Y$ admits a lift 
	\[\begin{tikzcd}
		& {X^{\vee}\otimes Y} \\
		{\mathbbm{1}} & {\iHom_{\C}(X,Y)}
		\arrow[from=1-2, to=2-2]
		\arrow[dashed, from=2-1, to=1-2]
		\arrow[from=2-1, to=2-2]
	\end{tikzcd}\]
	Equivalently (cf. \cite[Definition 4.4.1]{KNP24} or \cite[Definition 3.1]{Ram26locallyrigid}), $X\to Y$ is trace-class if and only if there exists an object $Z\in\C$ with maps $\mathbbm{1}\to Y\otimes Z$ and $Z\otimes X\to\mathbbm{1}$, and a commutative diagram
	\[\begin{tikzcd}
		& {Y\otimes Z\otimes X} & \\
		X && Y
		\arrow[from=1-2, to=2-3]
		\arrow[from=2-1, to=1-2]
		\arrow[from=2-1, to=2-3]
	\end{tikzcd}\]
\end{definition}
The second characterization immediately implies that symmetric monoidal functors preserve trace-class maps. Moreover, it extends to the $\bbE_1$-monoidal case, where we refer to it as the notion of a \emph{left trace-class} map. Equivalently, it is characterized by the first definition with $X^{\vee}\otimes Y$ replaced by $Y\otimes X^{l\vee}$.
\begin{definition}[{\cite[Definition 2.17, Lemma 2.18, Lemma 2.21]{Aok25}}]\label{def:nuclear}
	Let $\C\in\CAlg(\PrL)$. We say an $\Ind$-object $X\in\Ind(\C)$ is \emph{nuclear} if it satisfies one of the following equivalent conditions:
	\begin{enumerate}
		\item Any map $j(C)\to X$ from $C\in\C$ factors as $j(C)\to j(C')\to X$ for some trace-class map $C\to C'$.
		\item $X$ is equivalent to $\colim_Ij(C_i)$ with $I$ directed and $C_i\in\C$, so that for every $i\in I$ there exists some $i\leq j$ with $C_i\to C_j$ trace-class.
		\item $X$ is the filtered colimit of \emph{basic nuclear objects} (i.e. objects of the form $\colim_{\bbN}j(C_n)$ where all maps $C_n\to C_{n+1}$ are trace-class). 
	\end{enumerate}
	Let $\Nuc(\C)$ denote the full subcategory of $\Ind(\C)$ spanned by nuclear $\Ind$-objects, and call it the \emph{nuclear module category} of $\C$.
\end{definition}
\begin{proposition}
	Let $\C\in\CAlg(\PrL)$. Then $\Nuc(\C)$ is $\omega_1$-accessible with $\omega_1$-compact objects consisting of basic nuclear objects, and is closed under tensor products. Let $\C\in\CAlg(\PrL_{\st})$. Then $\Nuc(\C)$ is $\omega_1$-compactly generated, and $\Nuc(\C)\in\CAlg(\PrL_{\st})$.
\end{proposition}
\begin{proof}
	Basic nuclear objects are $\omega_1$-compact in $\Ind(\C)$, and hence $\omega_1$-compact in $\Nuc(\C)$. Suppose that $\C\in\Pr^{\kappa}$ and $\mathbbm{1}\in\C^{\kappa}$. Then every trace-class map factors through a $\kappa$-compact object and $\Nuc(\C)\subseteq\Ind(\C^{\kappa})$. So basic nuclear objects are of the form $\colim_{\bbN}j(C_n)$ where all maps $C_n\to C_{n+1}$ are trace-class between $\kappa$-compacts. As in the proof of \cite[Proposition 1.24]{Efi25localizing}, we can show that every nuclear object is the $\omega_1$-filtered colimit of basic nuclear objects. It follows that $\Nuc(\C)$ is $\omega_1$-accessible. Conversely, let $X\in\Nuc(\C)^{\omega_1}$. Then $X$ is a retract of a basic nuclear object $\colim_{\bbN}j(C_n)$ associated to $f:\colim_{\bbN}j(C_n)\to\colim_{\bbN}j(C_n)$. We may assume that $f$ is induced by a sequence of maps $(C_n\to C_{i(n)})_n$ with a cofinal functor $i:\bbN\to\bbN$. Hence, $X\simeq\colim_{m\in\bbN}\colim_{n\in\bbN}j(C_{i^m(n)})\simeq\colim_{\bbN}j(C_{i^n(n)})$ is basic nuclear. Suppose moreover that $\C\in\CAlg(\Pr^{\kappa})$. Then basic nuclear objects are closed under tensor products. Therefore, $\Nuc(\C)$ is closed under tensor products.

	Let $\C\in\CAlg(\PrL_{\st})$. Then by \cite[Lemma 7.7]{Aok25}, basic nuclear objects are closed under cofibers. It follows that $\Nuc(\C)\in\CAlg(\PrL_{\st})$.
\end{proof}
\begin{remark}
	The notion of nuclear module category in \cref{def:nuclear} recovers the notions of nuclear module categories of Clausen--Scholze, Meyer--Wagner, and Efimov in the corresponding settings.
	\begin{enumerate}
		\item Let $\C\in\Pr^{\cg}$ with $\mathbbm{1}\in\C^\omega$. Then $\Nuc(\C)\subseteq\Ind(\C^{\omega})\simeq\C$ coincides with the nuclear module category in the sense of Clausen--Scholze \cite[Definition 8.5]{Sch26analytic,CS26complex}, namely, the full subcategory spanned by objects $X\in\C$ such that 
		\[
			\Hom_{\C}(\mathbbm{1},C^{\vee}\otimes X)\simeq\Hom_{\C}(C,X)
		\]
		for every $C\in\C^{\omega}$, or equivalently, such that every map $C\to X$ from $C\in\C^{\omega}$ is trace-class. Indeed, objects in $\Nuc(\C)$ are filtered colimits of basic nuclear objects, and the claim follows from \cite[Lemma 4.7.8]{KNP24}.
		
		Moreover, if $\C\in\CAlg(\Pr^{\cg})$, then by \cite[Proposition 7.11]{Aok25}, $X\in\C$ is nuclear if and only if
		\[
			C^{\vee}\otimes X\simeq\iHom_{\C}(C,X)
		\]
		for every $C\in\C^{\omega}$.
		\item Let $\C\in\Pr^{\kappa}$ with $\mathbbm{1}\in\C^{\kappa}$. Then $\Nuc(\C)\subseteq\Ind(\C^{\kappa})$ coincides with the nuclear module category in the sense of Meyer--Wagner \cite[Remark 2.5]{MW25refinedTC}, namely, the $\Ind_{\omega_1}$-completion of the full subcategory spanned by the basic nuclear objects.
		
		Moreover, if $\kappa\geq\omega$ and $\C\in\CAlg(\Pr^{\kappa})$, then the colimit functor $k:\Ind(\C^\kappa)\to\C$ induces an equivalence 
		\[
			\Nuc(\Ind(\C^{\kappa}))\simeq\Nuc(\C).
		\]
		Thus, it suffices to study nuclear module categories for $\C\in\CAlg(\Pr^{\cg})$. Indeed, we only need to verify the observations in \cite[Remark 1.24]{Efi25limit}. Since $k:\Ind(\C^{\kappa})\to\C$ is symmetric monoidal and $j:\C\hookrightarrow\Ind(\C^{\kappa})$ is fully faithful, $j$ preserves internal Homs, and in particular, duals. Since $j$ is also symmetric monoidal, a map $X\to Y$ is trace-class in $\C$ if and only if $j(X)\to j(Y)$ is trace-class in $\Ind(\C^{\kappa})$. Therefore, under the equivalence $k:\Ind(\Ind(\C^{\kappa})^{\omega})\simeq\Ind(\C^{\kappa})$, the nuclear module category $\Nuc(\C)\subseteq\Ind(\C^{\kappa})$ coincides with $\Nuc(\Ind(\C^{\kappa}))\subseteq\Ind(\Ind(\C^{\kappa})^{\omega})$.
		\item Let $\C\in\Pr^{\ca}$ with $\mathbbm{1}\in\C^{\omega}$. Then every trace-class map is compact, $\Nuc(\C)\subseteq\widehat{j}(\C)\subseteq\Ind(\C^{\omega_1})$, and $k(\Nuc(\C))\subseteq\C$ coincides with the nuclear module category in the sense of Efimov \cite[Definition 3.2]{Efi25rigidity}, namely, the full subcategory spanned by objects $X\in\C$ such that every compact map $C\to X$ from $C\in\C^{\omega_1}$ is trace-class.
		
		Indeed, $\Nuc(\C)\subseteq\widehat{j}(\C)$ consists of objects of the form $\widehat{j}(X)\simeq\colim_Ij(C_i)\simeq\colim_I\widehat{j}(C_i)$ with $I$ directed and $C_i\in\C$, so that for every $i\in I$ there exists some $i\leq j$ with $C_i\to C_j$ trace-class. It follows that every object of $k(\Nuc(\C))$ is the $\omega_1$-filtered colimit of basic nuclear objects. The converse follows from the same argument as in \cite[Lemma 4.7.8]{KNP24}. This recovers \cite[Proposition 3.5 and 3.6]{Efi25rigidity}.
	\end{enumerate}
\end{remark}
\begin{remark}
	More generally, Aoki introduces the notion of $\I$-restricted $\Ind$-objects in \cite{Aok25}, unifying basic nuclear objects and compactly exhaustible objects. 
	
	Let $\C\in\PrL$ and $\I$ be an \emph{idealoid} in $\C$, i.e., a class of morphisms such that for any $f\in\I$ we have $gfh\in\I$ for any composable maps $g,h\in\C$. We say an $\Ind$-object $X\in\Ind(\C)$ is \emph{$\I$-restricted} if any map $j(C)\to X$ factors as $j(C)\to j(C')\to X$ for some map $C\to C'\in\I$, and let $\Ind_{\I}(\C)$ denote the full subcategory of $\Ind(\C)$ spanned by $\I$-restricted $\Ind$-objects.
	
	We say $\I$ is \emph{accessible} if there exists some $\kappa$ such that every map in $\I$ factors through a $\kappa$-compact object. In this case, $\Ind_{\I}(\C)$ is $\omega_1$-accessible with $\omega_1$-compact objects precisely of the form $\colim_{\bbN}j(C_n)$ where all maps $C_n\to C_{n+1}\in\I$. 
	
	Let $\I_{\mathrm{sd}}$ denote the subset of $\I$ consisting of maps $f\in\C$ such that there exists a functor $F:\bbQ\cap[0,1]\to\C$ with $F(a)\to F(b)\in\I$ for all $a<b$ and $F(0)\to F(1)$ is identified with $f$. Then $\Ind_{\I_{\mathrm{sd}}}(\C)$ recovers $(\C,\I)^{\ca}$ in \cite[Definition 2.7.1]{KNP24} and $(\C,\I)$ in \cite[Definition 4.22]{Ram24dualizable} in the corresponding settings.
	
	The basic examples are as follows:
	\begin{enumerate}
		\item Let $\C\in\PrL$ and $\I=\{\text{compact maps}\}$. We call an object of $\Ind_{\I}(\C)$ a \emph{Schwartz} $\Ind$-object and an object of $\Ind_{\I_{\mathrm{sd}}}(\C)$ a \emph{very Schwartz} $\Ind$-object. By \cref{rmk:compass} or \cite[Proposition 3.1]{Aok25}, $\C$ is compactly assembled if and only if every object is the colimit of some Schwartz $\Ind$-object. In particular, every Schwartz $\Ind$-object of a compactly assembled category is very Schwartz.
		\item Let $\C\in\CAlg(\PrL_{\st})$ and $\I=\{\text{trace-class maps}\}$. Then $\Nuc(\C):=\Ind_{\I}(\C)$ by definition, and we call an object in $\VNuc(\C):=\Ind_{\I_{\mathrm{sd}}}(\C)$ a \emph{very nuclear} $\Ind$-object. By \cite[Proposition 7.15]{Aok25}, $\C$ is rigid if and only if $\mathbbm{1}\in\C^{\omega}$ and every object is the colimit of some nuclear $\Ind$-object. In particular, every nuclear $\Ind$-object of a rigid $\Sp$-algebra is very nuclear.
	\end{enumerate}
\end{remark}
\begin{theorem}[{\cite[Theorem 4.4.16]{KNP24}, \cite[Theorem 4.81]{Ram26locallyrigid}, \cite[Corollary 7.20]{Aok25}}]\label{thm:Nuc=rig}
	Let $\C\in\CAlg(\PrL_{\st})$. We have an equivalence
	\[\C^{\rig}\simeq\VNuc(\C),\]
	i.e., $\C^{\rig}\subseteq\Ind(\C)$ is the localizing stable subcategory generated by objects of the form $\colim_{\bbQ_{\geq0}}j(C_p)$ where $C_p\to C_q$ are trace-class for all $p<q$.
\end{theorem}
Analogous descriptions of $\C$-rigid algebras and $\C$-rigidifications over a general base $\C\in\CAlg(\PrL)$ are given in \cite[Theorem 4.98, Construction 4.100]{Ram26locallyrigid}.
\begin{theorem}[{\cite{Efi25limit,Aok25}}]
	Let $\C\in\CAlg(\PrL_{\st})$. Suppose that every trace-class map is the composition of two trace-class maps. Then we have an equivalence
	\[
		\C^{\rig}\simeq\Nuc(\C).
	\]
	In particular, the assumption holds in the following two cases.
	\begin{enumerate}
		\item Let $\C\in\Pr^{\kappa}_{\st}$ and $\mathbbm{1}\in\C^{\kappa}$. Suppose that every trace-class map between $\kappa$-compacts is the composition of two trace-class maps between $\kappa$-compacts.
		\item Let $\C\in\CAlg(\PrL_{\st})$ be locally rigid with $\mathbbm{1}\in\C^{\omega_1}$. In this case, the right adjoint $\C\to\C^{\rig}$ generates $\C^{\rig}$ under colimits and is symmetric monoidal.
	\end{enumerate}
	
	If the assumption holds, let $\overline{\C}\to\C$ be a rigid envelope. Then $\Nuc(\C)\simeq\C^{\rig}\simeq\iHom^{\dbl}_{\overline{\C}}(\C,\overline{\C})$.
\end{theorem}
\begin{proof}
	It follows that every nuclear $\Ind$-object is very nuclear, so we have $\C^{\rig}\simeq\Nuc(\C)$.

	The first case follows from the fact that every sequence of trace-class maps factors as a sequence of trace-class maps between $\kappa$-compacts, while the second case follows from \cite[Theorem 4.2]{Efi25limit}.
\end{proof}
As a corollary, we recover the following special case of a theorem in \cite{Zhob}. Zhou's proof applies under the weaker assumptions that $j_*$ generates $\C$ under colimits and preserves tensor products.
\begin{corollary}
	Let $\I\in\CAlg(\PrL_{\st})$ be locally rigid with $\mathbbm{1}\in\I^{\omega_1}$, and $j^*:\C\to\I$ be a rigid envelope. Then $\C\simeq\I^{\rig}$ if and only if the right adjoint $j_*:=(j^*)^R:\I\hookrightarrow\C$ is symmetric monoidal.
\end{corollary}
\begin{proof}
	Let $\I\to\C$ be symmetric monoidal. Then $F:\C\to\I^{\rig}$ is fully faithful by \cref{lem:1_rig}. It remains to show that $F^R$ is conservative. Since the composition of $F^R$ and the right adjoint $\I\hookrightarrow\I^{\rig}$ is given by $j_*$, which is symmetric monoidal, and $\I\hookrightarrow\I^{\rig}$ generates $\I^{\rig}$ under colimits, $F^R$ is also symmetric monoidal. Let $X\in\I^{\rig}$ such that $F^R(X)\simeq0$. Since $\I^{\rig}$ is self-dual over $\C$, we have $\I^{\rig}\simeq\FunL_{\C}(\I^{\rig},\C)$, which sends $X$ to $F^R(X\otimes-)\simeq F^R(X)\otimes F^R(-)\simeq0$. Hence $X\simeq0$.
\end{proof}
Let $R\in\CAlg(\Sp)$ and $I\subseteq\pi_*(R)$ be a finitely generated homogeneous ideal. Then the $I$-completion $R^{\wedge}_I$ is $\omega_1$-compact in $\Mod(R)_{I\text{-}\cplt}$ since $(-)^{\wedge}_I:\Mod(R)\to\Mod(R)$ preserves $\omega_1$-filtered colimits. More generally, we have the following lemma.
\begin{lemma}
	Let $\C\in\CAlg(\Pr^{\kappa}_{\st})$ and $E\in\C^{\kappa}$ where $\kappa>\omega$. Then $C_E\C\hookrightarrow\C\to L_E\C$ is a short exact sequence in $\Pr^{\kappa}_{\st}$. In particular, $L_E\mathbbm{1}\in(L_E\C)^{\kappa}$.
\end{lemma}
\begin{proof}
	Since $\C^{\kappa}\cap C_E\C\subseteq(C_E\C)^{\kappa}$, it suffices to write every object $X\in C_E\C$ as a colimit of $E$-acyclic $\kappa$-compact objects. Write $X\simeq\colim_PX_p$ where $P$ is $\kappa$-filtered directed and $X_p\in\C^{\kappa}$. Since $E\otimes X_p\in\C^{\kappa}$, $E\otimes X_p\to\colim_PE\otimes X_p\simeq E\otimes X\simeq0$ factors through a null map $E\otimes X_p\to E\otimes X_{p'}$ induced by $p\leq p'$. Repeating this process, we may write $X\simeq\colim_IX_i$ with $I$ directed and $X_i\simeq\colim_{\bbN}X_{i,n}$, where $X_{i,n}\in\C^{\kappa}$ and $E\otimes X_{i,n}\to E\otimes X_{i,n+1}$ is null. It follows from $\kappa>\omega$ that $X_i\in\C^{\kappa}\cap C_E\C$ and hence $C_E\C$ is $\kappa$-compactly generated. Therefore, we have $\C^{\kappa}\cap C_E\C=(C_E\C)^{\kappa}$ and $C_E\C\hookrightarrow\C\to L_E\C$ is a short exact sequence in $\Pr^{\kappa}_{\st}$.

	This can also be deduced from \cite[Proposition 2.31]{Ram24dualizable}\footnote{We thank Jiacheng Liang for pointing this out.}.
\end{proof}
Let $n\geq1$ and $1\leq m\leq n$. It follows that $L_{T(n)\oplus\cdots\oplus T(n-m+1)}\bbS$ is $\omega_1$-compact. Moreover, let $\C\in\CAlg^{\rig}(\PrL_{\st})$. Then $L_{T(n)\oplus\cdots\oplus T(n-m+1)}\C$ is locally rigid with $\omega_1$-compact unit.
\begin{definition}
	\begin{enumerate}
		\item Let $R\in\CAlg(\Sp)$ and $I\subseteq\pi_*(R)$ be a finitely generated homogeneous ideal. We define the \emph{category of nuclear $R^{\wedge}_I$-modules} as
		\[
			\Nuc(R^{\wedge}_I):=\Nuc(\Mod(R)_{I\text{-}\cplt}).
		\]
		\item Let $n\geq0$ and $1\leq m\leq n+1$. We define the \emph{category of nuclear $L_{T(n)\oplus\cdots\oplus T(n-m+1)}\bbS$-modules} as
		\[
			\Nuc(L_{T(n)\oplus\cdots\oplus T(n-m+1)}\bbS):=\Nuc(L_{T(n)\oplus\cdots\oplus T(n-m+1)}\Sp).
		\]
	\end{enumerate}
\end{definition}
For example, for $R\in\CAlg(\Sp)$, $\Nuc(L_{K(n)\oplus\cdots\oplus K(n-m+1)}R)\simeq\Nuc(L_{K(n)\oplus\cdots\oplus K(n-m+1)}\Mod(R))$.
\begin{example}\label{exm:R^BS^1}
		Let $R\in\CAlg(\Sp)$ be complex orientable (e.g. $R$ is even) and $t\in\pi_{-2}(R^{hS^1})$ be a complex orientation. By \cite[Lemma 3.2]{MW25refinedTC}, the homotopy $S^1$-fixed points functor induces a symmetric monoidal equivalence
		\[
			(-)^{hS^1}:\Mod(R)^{BS^1}\xrightarrow{\simeq}\Mod(R^{hS^1})_{t\text{-}\cplt},
		\]
		where $\Mod(R)^{BS^1}$ is equipped with the pointwise symmetric monoidal structure.
		
		In fact, this is a Koszul duality phenomenon. Indeed, $R^{hS^1}$ is an augmented $\bbE_\infty$-algebra in $\Mod(R)$, and the augmentation exhibits $R\simeq R^{hS^1}/t$ as a dualizable $R^{hS^1}$-module. Then $\mathrm{Bar}(R^{hS^1})\simeq R^{S^1}$ is a dualizable commutative $R$-bialgebra with dual $\mathrm{Bar}(R^{hS^1})^{\vee}\simeq R[S^1]$. By \cite[Corollary 2.3.3]{Rak26derham}, we have
		\[
			\coMod_{R^{S^1}}(\Mod(R))\simeq\Mod(R[S^1])\simeq\Mod(R)^{BS^1}.
		\]
		Moreover, by \cite[Proposition 3.2.5]{Rak26derham}, $R\otimes_{R^{hS^1}}-:\Mod(R^{hS^1})\to\Mod(R)$ lifts to a symmetric monoidal functor $R\otimes_{R^{hS^1}}-:\Mod(R^{hS^1})\to\Mod(R)^{BS^1}$ with fully faithful right adjoint $(-)^{hS^1}$, whose essential image consists of $R\simeq R^{hS^1}/t$-local objects. This gives the desired equivalence.
		
		In particular, $\Mod(R)^{BS^1}$ is locally rigid with $\omega_1$-compact unit. We have an equivalence
		\[
			\Nuc(R^{hS^1})\simeq(\Mod(R)^{BS^1})^{\rig}\simeq\Nuc(\Mod(R)^{BS^1}).
		\]
		This category provides a natural target of \emph{refined topological Hochschild homology} \cite{Sch24refined,MW25refinedTC}.
\end{example}
\begin{example}\label{exm:R_I^solid}
		Let $R\in\CAlg(\Sp)$ and $I\subseteq\pi_*(R)$ be a finitely generated homogeneous ideal. Following \cite[Subsection 7.2]{Efi25limit}, we reconstruct the category of nuclear solid modules defined in \cite{Sch26analytic} without using condensed mathematics. We denote by $\Perf(R^{\wedge,\solid}_I)\subseteq\Mod(R^{\wedge}_I)_{I\text{-}\cplt}$ the thick subcategory generated by $(\bigoplus_{\bbN}R)^{\wedge}_I$, and denote $\Mod(R^{\wedge,\solid}_I):=\Ind(\Perf(R^{\wedge,\solid}_I))\simeq\Mod(\End_R(\bigoplus_{\bbN}R)^{\wedge}_I)$. By \cite[Proposition 7.4]{Efi25limit} (see \cite[Appendix D]{LLS26} for more details), $\Mod(R^{\wedge,\solid}_I)$ satisfies the property that every trace-class map between compacts is the composition of two trace-class maps between compacts.
		
		We define the \emph{category of nuclear solid $R^{\wedge}_I$-modules} as
		\[
			\Nuc(R^{\wedge,\solid}_I):=\Nuc(\Mod(R^{\wedge,\solid}_I)).
		\]
		By \cite[Corollary 7.6]{Efi25limit}, we have a natural fully faithful $\Sp$-algebra map
		\[
			\Nuc(R^{\wedge,\solid}_I)\simeq\Mod(R^{\wedge,\solid}_I)^{\rig}
			\hookrightarrow\Ind(\Mod(R)_{I\text{-}\cplt}^{\omega_1})^{\rig}\simeq(\Mod(R)_{I\text{-}\cplt})^{\rig}\simeq\Nuc(R^{\wedge}_I).
		\]
		Since the embedding $\Perf(R^{\wedge}_I)\hookrightarrow\Perf(R^{\wedge,\solid}_I)$ induces a fully faithful $\Sp$-algebra map $\Mod(R^{\wedge}_I)\hookrightarrow\Mod(R^{\wedge,\solid}_I)$, by \cref{prop:map_of_rigid_envelopes} we obtain maps of rigid envelopes of $\Mod(R)_{I\text{-}\cplt}$
		\[
			\Mod(R)\to\Mod(R^{\wedge}_I)\hookrightarrow\Nuc(R^{\wedge,\solid}_I)\hookrightarrow\Nuc(R^{\wedge}_I).
		\]
		
		We define $\Calk_R^{\mathrm{top}}(\bigoplus_{\bbN}R)^{\wedge}_I:=\End_{\Calk(R)}(\bigoplus_{\bbN}R)^{\wedge}_I\simeq\cof[((\bigoplus_{\bbN}R)^{\vee}\otimes(\bigoplus_{\bbN}R))^{\wedge}_I\to\End_R((\bigoplus_{\bbN}R)^{\wedge}_I)]$, where $\Calk(R):=\Calk^{\cont}(\Mod(R))\subseteq\Ind(\Mod(R))$ denotes the image of $\widehat{j}/j:\Mod(R)\to\Ind(\Mod(R))$ by \cite[Proposition 1.64]{Efi25localizing}. By \cite[Corollary 7.7]{Efi25limit}, we have a short exact sequence in $\Pr^{\dbl}_{\st}$
		\[
			\Nuc(R^{\wedge,\solid}_I)\hookrightarrow\Mod(\End_R(\bigoplus_{\bbN}R)^{\wedge}_I)\to\Mod(\Calk^{\mathrm{top}}_R(\bigoplus_{\bbN}R)^{\wedge}_I),
		\]
		where $\End_R(\bigoplus_{\bbN}R)^{\wedge}_I\to\Calk^{\mathrm{top}}_R(\bigoplus_{\bbN}R)^{\wedge}_I$ is an idempotent $\bbE_1$-ring map. Let $\End^f_R(\bigoplus_{\bbN}R)^{\wedge}_I$ denote the fiber, which is an $H$-unital ring over $R$. Therefore, we have
		\[
			\Nuc(R^{\wedge,\solid}_I)\simeq\Mod_{H,\End^f_R(\bigoplus_{\bbN}R)^{\wedge}_I}\Mod(R).
		\]

		In particular, if $I=(0)$, we obtain $\Mod(R^{\solid})\simeq\Mod(\End_R(\bigoplus_{\bbN}R))$ and $\Nuc(R^{\solid})\simeq\Mod(R)\simeq\Mod_{H,\End_R^f(\bigoplus_{\bbN}R)}\Mod(R)$.
\end{example}
\begin{remark}
	Let $R$ be an adic $\bbE_\infty$-ring. Then $\End^f_R(\bigoplus_{\bbN}R)^{\wedge}_I$ is an $H$-unital ring as it is connective. Moreover, \cite[Theorem C]{LLS26} shows that $\Nuc(R^{\wedge,\solid}_I)$ is the additive rigidification
	\[
		\Nuc(R^{\wedge,\solid}_I)\simeq\Mod_H(\End^f_R(\bigoplus_{\bbN}R)^{\wedge}_I)\simeq\Sp(\Mod(R)_{I\text{-}\cplt,\geq0})^{\rig}_{\ad}.
	\]
\end{remark}

\cref{thm:Nuc=rig} motivates us to define the nuclear module category with respect to a rigid envelope $\C\to\I$ as the $\I$-completion category.
\begin{definition}
	Let $\I\in\CAlg(\PrL_{\st})$ be locally rigid with rigid envelope $\C\to\I$. Suppose that $\mathbbm{1}_{\I}\in\I^{\omega_1}$. Let $\M\in\Pr^{\dbl}_{\C}$. We define the \emph{nuclear $\I$-module category} of $\M$ as
	\[
		\Nuc_{\I}(\M):=\M^{\wedge\I}\simeq\iHom_{\C}^{\dbl}(\I,\M).
	\]
	Let $\M\in\Pr^{\dbl}_{\I}$. We simply denote $\Nuc_{\I}(\M):=\M^{\wedge\I}\simeq\iHom_{\I^{\rig}}^{\dbl}(\I,\M)\simeq\iHom^{\dbl}_{\C}(\I,\M)$.

	Suppose moreover that $\C$ is idempotent over $\Sp$ and let $\M\in\Pr^{\dbl}_{\st}$. We denote $\Nuc_{\I}(\M):=\Nuc_{\I}(\C\otimes\M)$, which agrees with the preceding definition whenever $\M\in\Pr^{\dbl}_{\C}$.
\end{definition}
In particular, $\Nuc_{\I}(\D)\simeq\Nuc(\I\otimes_{\C}\D)$ for any $\D\in\CAlg^{\rig}(\PrL_{\C})$, which justifies the terminology.
\begin{definition}
	\begin{enumerate}
		\item Let $R\in\CAlg(\Sp)$ and $I\subseteq\pi_*(R)$ be a finitely generated homogeneous ideal. We denote
		\[
			\Nuc_I(-):=\Nuc_{\Mod(R)_{I\text{-}\cplt}}(-).
		\]
		Let $S\in\CAlg(\Mod(R))$. We define $\Nuc(S^{\wedge}_I):=\Nuc_I(\Mod(S))$ and $\Nuc(L_IS^{\wedge}_I):=\Nuc_I(\Mod(S))_{\eta}\simeq L_I\Nuc_I(\Mod(S))$.
		\item Let $n\geq0$ and $1\leq m\leq n+1$. We denote
		\[
			\Nuc_{T(n)\oplus\cdots\oplus T(n-m+1)}(-):=\Nuc_{L_{T(n)\oplus\cdots\oplus T(n-m+1)}\Sp}(-).
		\]
		Let $R\in\CAlg(\Sp)$. We define $\Nuc(L_{T(n)\oplus\cdots\oplus T(n-m+1)}R):=\Nuc_{T(n)\oplus\cdots\oplus T(n-m+1)}(\Mod(R))$ and
		\begin{align*}
			\Nuc(L_{n-m}^fL_{T(n)\oplus\cdots\oplus T(n-m+1)}R):=&\;\Nuc_{T(n)\oplus\cdots\oplus T(n-m+1)}(\Mod(R))_{\eta}\\
			\simeq&\;L_{n-m}^f\Nuc_{T(n)\oplus\cdots\oplus T(n-m+1)}(\Mod(R)).
		\end{align*}
	\end{enumerate}
\end{definition}
\begin{example}
	Let $R\in\CAlg(\Sp)$ and $I\subseteq\pi_*(R)$ be a finitely generated homogeneous ideal. We define
	\[
		\Nuc(L_IR^{\wedge,\solid}_I):=\Nuc(R^{\wedge,\solid}_I)_{\eta}\simeq\Mod(L_IR)\otimes_R\Nuc(R^{\wedge,\solid}_I)
	\]
	and
	\[
		\Nuc(L_IR^{\wedge}_I):=\Nuc(R^{\wedge}_I)_{\eta}\simeq L_I\Nuc(R^{\wedge}_I).
	\]
\end{example}

The $\omega_1$-compactness ensures that $\Nuc_{\I}$ is an exact functor by the following results of Efimov.
\begin{definition}[{\cite[Definition 1.48]{Efi25limit}, \cite[Theorem C.6]{Efi25localizing}}]
	Let $\C\in\CAlg(\PrL)$ and $\D\in\Pr^{\dbl}_{\C}$.
	\begin{enumerate}
		\item We say $\D$ is \emph{proper} over $\C$ if the evaluation map $\ev_{\D}:\D^{\vee}\otimes_{\C}\D\to\C$ is a $\C$-internal left adjoint. 
		\item We say $\D$ is \emph{smooth} over $\C$ if the coevaluation map $\coev_{\D}:\C\to\D\otimes_{\C}\D^{\vee}$ is a $\C$-internal left adjoint. 
		\item Let $\C\in\CAlg(\Pr^{\lambda})$ with $\kappa\geq\lambda>\omega$. We say $\D$ is \emph{$\kappa$-compact} if $\D\in(\Pr^{\dbl}_{\C})^{\kappa}$, or equivalently, $\D\in(\Pr^{\kappa}_{\C})^{\dbl}$ by \cite[Theorem 4.24]{Sch26gestalten} (i.e. both $\ev_{\D}:\D^{\vee}\otimes_{\C}\D\to\C$ and $\coev_{\D}:\C\to\D^{\vee}\otimes_{\C}\D$ admit right adjoints that preserve $\kappa$-filtered colimits).
	\end{enumerate}
\end{definition}
\begin{theorem}[{\cite[Theorem 3.6]{Efi25limit}}]\label{thm:Homdbl_proper}
	Let $\C\in\CAlg^{\rig}(\PrL_{\st})$ and $\D\in\PrL_{\C}$. Suppose that $\D$ is proper and $\omega_1$-compact over $\C$, i.e. $\ev_{\D}$ is an internal left adjoint and $\coev_{\D}(\mathbbm{1})\in(\D\otimes_{\C}\D^{\vee})^{\omega_1}$. Then $\iHom^{\dbl}_{\C}(\D,-):\Pr^{\dbl}_{\C}\to\Pr^{\dbl}_{\C}$ preserves short exact sequences.
\end{theorem}
\begin{corollary}
	Let $\I\in\CAlg(\PrL_{\st})$ be locally rigid with $\mathbbm{1}\in\I^{\omega_1}$, and $\C\to\I$ be a rigid envelope. Then $\Nuc_{\I}:\Pr^{\dbl}_{\C}\to\Pr^{\dbl}_{\C}$ preserves short exact sequences.
\end{corollary}

Let $n\geq0$ and $1\leq m\leq n+1$. We focus on the rigid envelope $L_n^f\Sp\to L_{T(n)\oplus\cdots\oplus T(n-m+1)}\Sp$ where $L_n^f\Sp$ is idempotent over $\Sp$, and denote by $\Nuc_{T(n)\oplus\cdots\oplus T(n-m+1)}$ the nuclear $L_{T(n)\oplus\cdots\oplus T(n-m+1)}\Sp$-module category. Combining \cref{prop:pullback_diagram_of_completion,thm:pullback_of_map_of_rigid_envelope}, we obtain:
\begin{theorem}\label{thm:Nuc_T(n)}
	Let $n\geq0$ and $1\leq m\leq n+1$.
	\begin{enumerate}
		\item The $T(n)\oplus\cdots\oplus T(n-m+1)$-completion functor $\Nuc_{T(n)\oplus\cdots\oplus T(n-m+1)}:\Pr^{\dbl}_{\st}\to\Pr^{\dbl}_{\st}$ is lax symmetric monoidal, and preserves rigid $\Sp$-algebras and short exact sequences.
		\item Let $\C\in\Pr^{\dbl}_{\st}$. Then we have an equivalence
		\[
			L_{T(n)\oplus\cdots\oplus T(n-m+1)}\C\simeq L_{T(n)\oplus\cdots\oplus T(n-m+1)}\Nuc_{T(n)\oplus\cdots\oplus T(n-m+1)}(\C),
		\]
		which is the horizontal fiber of the pullback diagram in $\Pr^{\dbl}_{\st}$
		\[\begin{tikzcd}
			{L_n^f\C} & {L_{n-m}^f\C} \\
			{\Nuc_{T(n)\oplus\cdots\oplus T(n-m+1)}(\C)} & {L_{n-m}^f\Nuc_{T(n)\oplus\cdots\oplus T(n-m+1)}(\C)}
			\arrow[from=1-1, to=1-2]
			\arrow[from=1-1, to=2-1]
			\arrow[from=1-2, to=2-2]
			\arrow[""{name=0, anchor=center, inner sep=0}, from=2-1, to=2-2]
			\arrow["\lrcorner"{anchor=center, pos=0.125}, draw=none, from=1-1, to=0]
		\end{tikzcd}\]
		\item Let $\C\in\CAlg^{\rig}(\PrL_{\st})$. Then we have a pullback diagram in $\CAlg^{\rig}(\PrL_{\st})$
		\[\begin{tikzcd}
			{L_n^f\C} & {L_{n-m}^f\C} \\
			{\Nuc_{T(n)\oplus\cdots\oplus T(n-m+1)}(\C)} & {L_{n-m}^f\Nuc_{T(n)\oplus\cdots\oplus T(n-m+1)}(\C)}
			\arrow[from=1-1, to=1-2]
			\arrow[from=1-1, to=2-1]
			\arrow[from=1-2, to=2-2]
			\arrow[""{name=0, anchor=center, inner sep=0}, from=2-1, to=2-2]
			\arrow["\lrcorner"{anchor=center, pos=0.125}, draw=none, from=1-1, to=0]
		\end{tikzcd}\]
		whose endomorphism rings form a pullback diagram of $\bbE_\infty$-rings
		\[\begin{tikzcd}
			{L_n^f\End(\mathbbm{1}_{\C})} & {L_{n-m}^f\End(\mathbbm{1}_{\C})} \\
			{L_{T(n)\oplus\cdots\oplus T(n-m+1)}\End(\mathbbm{1}_{\C})} & {L_{n-m}^fL_{T(n)\oplus\cdots\oplus T(n-m+1)}\End(\mathbbm{1}_{\C})}
			\arrow[from=1-1, to=1-2]
			\arrow[from=1-1, to=2-1]
			\arrow[from=1-2, to=2-2]
			\arrow[""{name=0, anchor=center, inner sep=0}, from=2-1, to=2-2]
			\arrow["\lrcorner"{anchor=center, pos=0.125}, draw=none, from=1-1, to=0]
		\end{tikzcd}\]
	\end{enumerate}
\end{theorem}
In particular, the last pullback diagram of $\bbE_\infty$-rings comes from \cref{lem:End(1_L)}.

\section{Chromatic purity}\label{sec:3}
Throughout this section, we use the theory of localizing invariants and their continuous extensions developed by Efimov. We begin by recalling the background.

Let $\E\in\PrL_{\st}$. A functor $F:\Pr^{\dbl}_{\st}\to\E$ or $F:\Pr^{\cg}_{\st}\to\E$ is called a (\emph{finitary}) \emph{localizing invariant} if it preserves short exact sequences and filtered colimits. Denote by $\Fun^{\loc}(\Pr^{\dbl}_{\st},\E)$ and $\Fun^{\loc}(\Pr^{\cg}_{\st},\E)$ the categories of localizing invariants, respectively. By \cite[Theorem 4.10]{Efi25localizing} (or more generally, \cite[Theorem 1.19]{Efi25limit}), the precomposition functor induces an equivalence
\[
	\Fun^{\loc}(\Pr^{\dbl}_{\st},\E)\xrightarrow{\simeq}\Fun^{\loc}(\Pr^{\cg}_{\st},\E),
\]
with inverse functor $(-)^{\cont}$ defined by
\[
	F^{\cont}(\C):=F(\Ind(\Calk^{\cont}_{\omega_1}(\C)))[-1].
\]
Here, for $\kappa>\omega$, we define the \emph{continuous Calkin $\kappa$-category} of $\C\in\Pr^{\dbl}_{\st}$ as $\Calk^{\cont}_{\kappa}(\C):=(\Ind(\C^{\kappa})/\widehat{j}(\C))^{\omega}$, where $\Ind(\Calk^{\cont}_{\kappa}(\C))\simeq\Ind(\C^{\kappa})/\widehat{j}(\C)\simeq\fib(k:\Ind(\C^{\kappa})\to\C)$. We also define the \emph{continuous Calkin category} $\Calk^{\cont}(\C):=\colim_{\kappa}\Calk^{\cont}_{\kappa}(\C)$, which coincides with the definition in \cref{exm:R_I^solid}.

We write $K_{T(n)}:=L_{T(n)}K$ and $K^{\cont}_{T(n)}:=L_{T(n)}K^{\cont}\simeq(K_{T(n)})^{\cont}$ for the $T(n)$-local $K$-theory and $T(n)$-local continuous $K$-theory, respectively, and regard them as taking values in $L_{T(n)}\Sp$.

Using this formalism, we develop the chromatic theory of dualizable stable categories, with chromatic purity for continuous $K$-theory at its center. Working within the framework of nuclear completion, we establish purity, descent, redshift, blueshift, and continuity theorems, together with their nuclear refinements and applications.

\subsection{Purity and nuclear purity}\label{sec:3.1}

We generalize the chromatic purity theorem for algebraic $K$-theory to continuous $K$-theory of dualizable stable categories, and reformulate it in terms of $H$-unital rings. We then use categorical completion theory to obtain a canonical nuclear refinement of the purity theorem. The resulting refinement extends further to arbitrary rigid envelopes.
\begin{theorem}[Purity]\label{thm:purity_dualizable}
	Let $n\geq1$ and $\C\in\Pr^{\dbl}_{\st}$. Then the canonical maps induce equivalences
	\[
		K^{\cont}_{T(n)}(\C)\xrightarrow{\simeq}K^{\cont}_{T(n)}(L_n^f\C)\xleftarrow{\simeq}K^{\cont}_{T(n)}(M_n^{2,f}\C).
	\]
	In particular, we have an equivalence
	\[
		K_{T(n)}^{\cont}(\C)\simeq K_{T(n)}^{\cont}(L_{T(n)\oplus T(n-1)}\C).
	\]
\end{theorem}
\begin{proof}
	Since $L_n^f$ and $M_n^{2,f}$ preserve short exact sequences in $\Pr^{\dbl}_{\st}$, $K_{T(n)}^{\cont}(L_n^f(-))$ and $K^{\cont}_{T(n)}(M_n^{2,f}(-))$ are localizing invariants. Then $K_{T(n)}^{\cont}\to K_{T(n)}^{\cont}(L_n^f(-))\leftarrow K_{T(n)}^{\cont}(M_n^{2,f}(-))$ are equivalences since they restrict to equivalences on $\Pr^{\cg}_{\st}$ by the compactly generated version of the purity theorem recalled in the Introduction. Therefore, the canonical maps $\C\to L_n^f\C\leftarrow M_n^{2,f}\C$ induce equivalences on $K_{T(n)}^{\cont}$. The last assertion follows from the equivalence $L_{T(n)\oplus T(n-1)}\C\simeq M_n^{2,f}\C$ by \cref{prop:M_n^mf}.
\end{proof}

Since every dualizable stable category is equivalent to the $H$-module category over an $H$-unital ring, \cref{thm:purity_dualizable} can be reformulated as the following purity of $H$-unital rings.
\begin{corollary}[Purity of $H$-unital rings]\label{cor:purity_Hunital}
	Let $n\geq1$ and $I$ be an $H$-unital ring. Then we have an equivalence
	\[
		K_{T(n)}^{\cont}(I)\simeq K_{T(n)}^{\cont}(M_n^{2,f}\otimes I).
	\]
\end{corollary}
\begin{proof}
	By \cref{prop:tensor-H-unital,prop:multichromatic_Hunital}, $\Mod_H(M_n^{2,f}\otimes I)\simeq\Mod_H(M_n^{2,f})\otimes\Mod_H(I)\simeq M_n^{2,f}\Mod_H(I)$. It follows that $K_{T(n)}^{\cont}(I)\simeq K_{T(n)}^{\cont}(L_n^fI)\simeq K_{T(n)}^{\cont}(M_n^{2,f}\otimes I)$.
\end{proof}
\begin{remark}
	We compare \cref{cor:purity_Hunital} with the Purity of $\bbE_1$-rings recalled in the Introduction. Let $n\geq1$, and let $R\to S$ be an $\bbE_1$-ring map. Then the Purity of $\bbE_1$-rings says that $K_{T(n)}(R)\to K_{T(n)}(S)$ is an equivalence whenever $R\to S$ is a $T(n)\oplus T(n-1)$-local equivalence, while \cref{cor:purity_Hunital} says that it is an equivalence whenever $R\to S$ is an $M_n^{2,f}$-local equivalence. Note that $R\to S$ is a $T(n)\oplus T(n-1)$-local equivalence if and only if $L_n^fR\to L_n^fS$ is, which is equivalent to $R\to S$ being an $M_n^{2,f}$-local equivalence by \cref{prop:M_n^mf}. In this sense, when restricted to maps of $\bbE_1$-rings, \cref{cor:purity_Hunital} is equivalent to the Purity of $\bbE_1$-rings.
\end{remark}

Unlike the case of $\bbE_1$-rings, the equivalence $K^{\cont}_{T(n)}(\C)\simeq K^{\cont}_{T(n)}(L_{T(n)\oplus T(n-1)}\C)$ is not induced directly by the localization functor $\C\to L_{T(n)\oplus T(n-1)}\C$, since the latter need not be an $\Sp$-internal left adjoint. To overcome this issue, we invoke categorical completion theory, which provides a replacement for the localization map and leads to the following nuclear refinement of the purity theorem.
\begin{theorem}[Nuclear purity]\label{thm:nuc_purity}
	Let $n\geq1$ and $\C\in\Pr^{\dbl}_{\st}$. Then the canonical map $\C\to\Nuc_{T(n)\oplus T(n-1)}(\C)$ induces an equivalence
	\[
		K_{T(n)}^{\cont}(\C)\xrightarrow{\simeq}K_{T(n)}^{\cont}(\Nuc_{T(n)\oplus T(n-1)}(\C)).
	\]
\end{theorem}
\begin{proof}
	Since the canonical map $\C\to\Nuc_{T(n)\oplus T(n-1)}(\C)$ induces an equivalence 
	\[
		L_{T(n)\oplus T(n-1)}\C\simeq L_{T(n)\oplus T(n-1)}\Nuc_{T(n)\oplus T(n-1)}(\C),
	\]
	it induces an equivalence on $K^{\cont}_{T(n)}$ by \cref{thm:purity_dualizable}.
\end{proof}
\cref{thm:nuc_purity} and \cref{thm:purity_dualizable} are equivalent. Indeed, observe that 
\[
	\Nuc_{T(n)\oplus T(n-1)}(\C)\simeq\Nuc_{T(n)\oplus T(n-1)}(L_{T(n)\oplus T(n-1)}\C).
\]
However, \cref{thm:nuc_purity} is better suited to studying continuous $K$-theory and chromatic behavior of rigid $\Sp$-algebras, since $\Nuc_{T(n)}(-)$ preserves rigidity, whereas $L_{T(n)}(-)$ does not. In fact, we will show in \cref{sec:3.3} that $L_{T(n)}\C$ cannot be rigid when $n\geq1$ and $L_{T(n)}\C$ is nonzero, so $\Nuc_{T(n)}(\C)$ is the natural object for studying rigid $\Sp$-algebras.

We now give a second proof of nuclear purity using the fracture square.
\begin{lemma}
	Given a commutative diagram in $\Pr^{\dbl}_{\st}$
	\[\begin{tikzcd}
		\D & \E \\
		{\D'} & {\E'}
		\arrow[from=1-1, to=1-2]
		\arrow[from=1-1, to=2-1]
		\arrow[from=1-2, to=2-2]
		\arrow[from=2-1, to=2-2]
	\end{tikzcd}\]
	Suppose that both $\D\to\E$ and $\D'\to\E'$ are localizations and the induced functor $\C:=\fib(\D\to\E)\to\C':=\fib(\D'\to\E')$ is an equivalence. Then it induces a pullback diagram
	\[\begin{tikzcd}
		K^{\cont}(\D) & K^{\cont}(\E) \\
		{K^{\cont}(\D')} & {K^{\cont}(\E')}
		\arrow[from=1-1, to=1-2]
		\arrow[from=1-1, to=2-1]
		\arrow[from=1-2, to=2-2]
		\arrow[from=2-1, to=2-2]
	\end{tikzcd}\]
	In particular, this holds for a pullback diagram in $\Pr^{\dbl}_{\st}$ where the horizontal arrows are localizations.
\end{lemma}
\begin{proof}
	The given square induces a commutative diagram between short exact sequences in $\Pr^{\dbl}_{\st}$
	\[\begin{tikzcd}
		\C & \D & \E \\
		{\C'} & {\D'} & {\E'}
		\arrow[hook, from=1-1, to=1-2]
		\arrow["\simeq"', from=1-1, to=2-1]
		\arrow[from=1-2, to=1-3]
		\arrow[from=1-2, to=2-2]
		\arrow[from=1-3, to=2-3]
		\arrow[hook, from=2-1, to=2-2]
		\arrow[from=2-2, to=2-3]
	\end{tikzcd}\]
	Then the continuous $K$-theory sends this to a commutative diagram of fiber sequences in $\Sp$
	\[\begin{tikzcd}
		K^{\cont}(\C) & K^{\cont}(\D) & K^{\cont}(\E) \\
		{K^{\cont}(\C')} & {K^{\cont}(\D')} & {K^{\cont}(\E')}
		\arrow[from=1-1, to=1-2]
		\arrow["\simeq"', from=1-1, to=2-1]
		\arrow[from=1-2, to=1-3]
		\arrow[from=1-2, to=2-2]
		\arrow[from=1-3, to=2-3]
		\arrow[from=2-1, to=2-2]
		\arrow[from=2-2, to=2-3]
	\end{tikzcd}\]
	As the left vertical map is an equivalence, the right square is a pullback diagram.
\end{proof}
\begin{proof}[Alternative proof of {\cref{thm:nuc_purity}}.]
	By \cref{thm:Nuc_T(n)}, we have a pullback diagram
	\[\begin{tikzcd}
			{K^{\cont}_{T(n)}(L_n^f\C)} & {K^{\cont}_{T(n)}(L_{n-2}^f\C)} \\
			{K^{\cont}_{T(n)}(\Nuc_{T(n)\oplus T(n-1)}(\C))} & {K^{\cont}_{T(n)}(L_{n-2}^f\Nuc_{T(n)\oplus T(n-1)}(\C))}
			\arrow[from=1-1, to=1-2]
			\arrow[from=1-1, to=2-1]
			\arrow[from=1-2, to=2-2]
			\arrow[""{name=0, anchor=center, inner sep=0}, from=2-1, to=2-2]
			\arrow["\lrcorner"{anchor=center, pos=0.125}, draw=none, from=1-1, to=0]
	\end{tikzcd}\]
	Then the result follows from \cref{thm:purity_dualizable}.
\end{proof}
More generally, every rigid envelope functorially gives rise to a fracture square of continuous $K$-theory.
\begin{proposition}
	Let $\I\in\CAlg(\PrL_{\st})$ be locally rigid and $\C\to\I$ be a rigid envelope with $\A:=\C/\I$. Let $\M\in\Pr^{\dbl}_{\C}$. Then we have a fracture square
	\[\begin{tikzcd}
			{K^{\cont}(\M)} & {K^{\cont}(\A\otimes_{\C}\M)} \\
			{K^{\cont}(\M^{\wedge\I})} & {K^{\cont}(\M_{\eta})}
			\arrow[from=1-1, to=1-2]
			\arrow[from=1-1, to=2-1]
			\arrow[from=1-2, to=2-2]
			\arrow[""{name=0, anchor=center, inner sep=0}, from=2-1, to=2-2]
			\arrow["\lrcorner"{anchor=center, pos=0.125}, draw=none, from=1-1, to=0]
	\end{tikzcd}\]
\end{proposition}
\begin{proof}
	Apply $K^{\cont}$ to the pullback diagram in \cref{prop:pullback_diagram_of_completion}.
\end{proof}
Applying the proposition to the maps of rigid envelopes in \cref{exm:R_I^solid} yields the following corollary.
\begin{corollary}\label{cor:K(Nuc(R_I))}
	Let $R\in\CAlg(\Sp)$ and let $I\subseteq\pi_*(R)$ be a finitely generated homogeneous ideal. Then we have pullback diagrams
	\[\begin{tikzcd}
		{K(R)} & {K(R^{\wedge}_I)} & {K^{\cont}(\Nuc(R^{\wedge,\solid}_I))} & {K^{\cont}(\Nuc(R^{\wedge}_I))} \\
		{K(L_IR)} & {K(L_IR^{\wedge}_I)} & {K^{\cont}(\Nuc(L_IR^{\wedge,\solid}_I))} & {K^{\cont}(\Nuc(L_IR^{\wedge}_I))}
		\arrow[from=1-1, to=1-2]
		\arrow[from=1-1, to=2-1]
		\arrow["\lrcorner"{anchor=center, pos=0.125}, draw=none, from=1-1, to=2-2]
		\arrow[from=1-2, to=1-3]
		\arrow[from=1-2, to=2-2]
		\arrow["\lrcorner"{anchor=center, pos=0.125}, draw=none, from=1-2, to=2-3]
		\arrow[from=1-3, to=1-4]
		\arrow[from=1-3, to=2-3]
		\arrow[from=1-4, to=2-4]
		\arrow[from=2-1, to=2-2]
		\arrow[from=2-2, to=2-3]
		\arrow[""{name=0, anchor=center, inner sep=0}, from=2-3, to=2-4]
		\arrow["\lrcorner"{anchor=center, pos=0.125}, draw=none, from=1-3, to=0]
	\end{tikzcd}\]
	where the common fiber of the vertical maps is $K(\Mod(R)_{I\text{-}\cplt})$.
\end{corollary}
\begin{proof}
	By \cref{exm:R_I^solid}, we obtain a sequence of maps of rigid envelopes
	\[
		\Mod(R)\to\Mod(R^{\wedge}_I)\hookrightarrow\Nuc(R^{\wedge,\solid}_I)\hookrightarrow\Nuc(R^{\wedge}_I).
	\]
	By \cref{thm:pullback_of_map_of_rigid_envelope}, we have pullback diagrams
	\[\begin{tikzcd}
		{\Mod(R)} & {\Mod(R^{\wedge}_I)} & {\Nuc(R^{\wedge,\solid}_I)} & {\Nuc(R^{\wedge}_I)} \\
		{\Mod(L_IR)} & {\Mod(L_IR^{\wedge}_I)} & {\Nuc(L_IR^{\wedge,\solid}_I)} & {\Nuc(L_IR^{\wedge}_I)}
		\arrow[from=1-1, to=1-2]
		\arrow[from=1-1, to=2-1]
		\arrow["\lrcorner"{anchor=center, pos=0.125}, draw=none, from=1-1, to=2-2]
		\arrow[from=1-2, to=1-3]
		\arrow[from=1-2, to=2-2]
		\arrow["\lrcorner"{anchor=center, pos=0.125}, draw=none, from=1-2, to=2-3]
		\arrow[from=1-3, to=1-4]
		\arrow[from=1-3, to=2-3]
		\arrow[from=1-4, to=2-4]
		\arrow[from=2-1, to=2-2]
		\arrow[from=2-2, to=2-3]
		\arrow[""{name=0, anchor=center, inner sep=0}, from=2-3, to=2-4]
		\arrow["\lrcorner"{anchor=center, pos=0.125}, draw=none, from=1-3, to=0]
	\end{tikzcd}\]
	where the common fiber of the vertical maps is $\Mod(R)_{I\text{-}\cplt}$. Applying $K$-theory completes the proof.
\end{proof}
A special case of the middle pullback square appears in \cite[(5.9)]{And23thesis}, where it is established by showing that the fibers of its two vertical maps are equivalent. For the corresponding categorical statement for condensed adic $\bbE_\infty$-rings, see \cite[Lemma 1.33]{Cor23}. One can also prove directly that the two vertical maps in the outer rectangle have equivalent fibers. 

For rigid symmetric monoidal categories, the nuclear completion $\Nuc_{T(n)\oplus T(n-1)}\C$ can be replaced by any rigid envelope of $L_{T(n)\oplus T(n-1)}\C$ under $L_n^f\C$.
\begin{proposition}
	Let $n\geq1$ and $\C\in\CAlg^{\rig}(\PrL_{\st})$, and let $L_n^f\C\to\D$ be a map of rigid envelopes of $L_{T(n)\oplus T(n-1)}\C$. Then $\C\to\D$ induces an equivalence
	\[
		K^{\cont}_{T(n)}(\C)\xrightarrow{\simeq}K^{\cont}_{T(n)}(\D).
	\]
	In particular, this recovers \cref{thm:nuc_purity} for rigid $\Sp$-algebras by taking $L_n^f\C\to\Nuc_{T(n)\oplus T(n-1)}(\C)$.
\end{proposition}
\begin{proof}
	Consider the locally rigid $\Sp$-algebra $L_{T(n)\oplus T(n-1)}\C$ with rigid envelope $L_n^f\C$. By \cref{thm:pullback_of_map_of_rigid_envelope}, we have a pullback diagram in $\Sp$
	\[\begin{tikzcd}
			{K^{\cont}_{T(n)}(L_n^f\C)} & {K^{\cont}_{T(n)}(L_{n-2}^f\C)} \\
			{K^{\cont}_{T(n)}(\D)} & {K^{\cont}_{T(n)}(L_{n-2}^f\D)}
			\arrow[from=1-1, to=1-2]
			\arrow[from=1-1, to=2-1]
			\arrow[from=1-2, to=2-2]
			\arrow[""{name=0, anchor=center, inner sep=0}, from=2-1, to=2-2]
			\arrow["\lrcorner"{anchor=center, pos=0.125}, draw=none, from=1-1, to=0]
	\end{tikzcd}\]
	Then the result follows from \cref{thm:purity_dualizable}.
\end{proof}
\begin{corollary}
	Let $n\geq1$ and $\C\in\CAlg^{\rig}(\PrL_{\st})^{\cg}$. Then $\C\to\Ind(L_{T(n)\oplus T(n-1)}\C)^{\dbl}$ induces an equivalence
	\[
		K^{\cont}_{T(n)}(\C)\xrightarrow{\simeq}K^{\cont}_{T(n)}(\Ind(L_{T(n)\oplus T(n-1)}\C)^{\dbl}).
	\]
\end{corollary}
\begin{proof}
	Since $L_n^f\C\in\CAlg^{\rig}(\PrL_{\st})^{\cg}$ is a rigid envelope of $L_{T(n)\oplus T(n-1)}\C$, by \cref{cor:Ind_dbl}, $L_n^f\C\to\Ind(L_{T(n)\oplus T(n-1)}\C)^{\dbl}$ is a map of rigid envelopes. Then the result follows.
\end{proof}
In particular, for $R\in\CAlg(\Sp)$, we obtain equivalences
\begin{align*}
	K^{\cont}_{T(n)}(R)&\simeq K^{\cont}_{T(n)}(L_{T(n)\oplus T(n-1)}R)\simeq K^{\cont}_{T(n)}(L_{T(n)\oplus T(n-1)}\Mod(R))\\
	&\simeq K^{\cont}_{T(n)}(\Ind(L_{T(n)\oplus T(n-1)}\Mod(R)^{\dbl}))\simeq K^{\cont}_{T(n)}(\Nuc(L_{T(n)\oplus T(n-1)}R)).
\end{align*}

As a corollary, we restrict to $T(n+1)$-localized $K$-theory of $L_n^f$-local categories.
\begin{corollary}\label{cor:purity_T(n)}
	Let $n\geq0$.
	\begin{enumerate}
		\item Let $\C\in\Pr^{\dbl}_{L_n^f}$. Then we have an equivalence
		\[
			K^{\cont}_{T(n+1)}(\C)\simeq K^{\cont}_{T(n+1)}(L_{T(n)}\C),
		\]
		which also appears in \cite[Proposition 2.24]{BMCSY25redshift} in the form $K^{\cont}_{T(n+1)}(\C)\simeq K^{\cont}_{T(n+1)}(M_n^f\C)$.
		
		Moreover, the $T(n)$-completion map induces an equivalence
		\[
			K^{\cont}_{T(n+1)}(\C)\xrightarrow{\simeq}K^{\cont}_{T(n+1)}(\Nuc_{T(n)}(\C)).
		\]
		\item Let $\C\in\CAlg^{\rig}(\PrL_{L_n^f})^{\cg}$. Then $\C\to\Ind(L_{T(n)}\C)^{\dbl}$ induces an equivalence
		\[
			K^{\cont}_{T(n+1)}(\C)\xrightarrow{\simeq}K^{\cont}_{T(n+1)}(\Ind(L_{T(n)}\C)^{\dbl}).
		\]
	\end{enumerate}
\end{corollary}
\begin{proof}
	Let $\C\in\Pr^{\dbl}_{L_n^f}$. Since $\C$ is $L_n^f$-local, $L_{T(n+1)\oplus T(n)}\C\simeq\fib(L_{n+1}^f\C\to L_{n-1}^f\C)\simeq\fib(\C\to L_{n-1}^f\C)\simeq L_{T(n)}\C$. Moreover, we have $\Nuc_{T(n+1)\oplus T(n)}(\C)\simeq\iHom^{\dbl}_{L_{n+1}^f\Sp}(L_{T(n+1)\oplus T(n)}\Sp,\C)\simeq\iHom^{\dbl}_{L_n^f\Sp}(L_n^f\Sp\otimes L_{T(n+1)\oplus T(n)}\Sp,\C)\simeq\iHom^{\dbl}_{L_n^f\Sp}(L_{T(n)}\Sp,\C)\simeq\Nuc_{T(n)}(\C)$. Therefore, the results follow.
\end{proof}
In particular, for $R\in\CAlg(L_n^f\Sp)$ (e.g. $R\in\CAlg(L_{T(n)}\Sp)$), we obtain equivalences
\begin{align*}
	K^{\cont}_{T(n+1)}(R)&\simeq K^{\cont}_{T(n+1)}(L_{T(n)}R)\simeq K^{\cont}_{T(n+1)}(L_{T(n)}\Mod(R))\\
	&\simeq K^{\cont}_{T(n+1)}(\Ind(L_{T(n)}\Mod(R)^{\dbl}))\simeq K^{\cont}_{T(n+1)}(\Nuc(L_{T(n)}R)).
\end{align*}
The equivalence $K^{\cont}_{T(n+1)}(R)\simeq K^{\cont}_{T(n+1)}(\Ind(L_{T(n)}\Mod(R)^{\dbl}))$ recovers \cite[Proposition 4.15]{CMNN24}.

\subsection{Descent and nuclear descent}\label{sec:3.2}

We generalize the chromatic descent theorem to continuous $K$-theory of dualizable homotopy fixed points. We then study norm maps of dualizable stable categories and establish a descent theorem for nuclear module categories. As an application of the purity and descent theorems, we reprove the $T(n)$-local Galois descent theorem for finite $p$-groups and then prove its nuclear analogue for arbitrary finite groups.

\begin{theorem}[Descent]\label{thm:descent_dualizable}
	Let $n\geq0$ and $\C\in(\Pr^{\dbl}_{L_n^f})^{BG}$ with $G$ a finite $p$-group. Then the canonical maps induce equivalences
	\[
		K_{T(n+1)}^{\cont}(\C^{hG,\dbl})\xrightarrow{\simeq}K_{T(n+1)}^{\cont}(\C)^{hG},\quad K_{T(n+1)}^{\cont}(\C)_{hG}\xrightarrow{\simeq}K_{T(n+1)}^{\cont}(\C_{hG}).
	\]
\end{theorem}
\begin{lemma}[{cf. \cite[Theorem 1.91]{Efi25localizing}}]\label{lem:limdbl}
	Let $\C_i:I\to\Pr^{\dbl}_{\st}$ be a functor. Then we have
	\begin{equation}\label{eq:limit_ses}
		\lim^{\dbl}_I\C_i\simeq\fib^{\dbl}\left(\lim^{\cg}_I\Ind(\C_i^{\omega_1})\to\lim^{\cg}_I\Ind(\Calk^{\cont}_{\omega_1}(\C_i))\right).
	\end{equation}
	In particular, if $\lim^{\dbl}_I$ preserves short exact sequences, then \eqref{eq:limit_ses} is a short exact sequence in $\Pr^{\dbl}_{\st}$.
\end{lemma}
\begin{proof}
	By \cite[Proposition 1.89]{Efi25localizing} (or more generally, \cite[Theorem 1.63]{Ram24dualizable}), we have adjunctions:
	\[\begin{tikzcd}
		{\Pr^{\cg}_{\st}} & {\Pr^{\dbl}_{\st}} & {\PrL_{\st,\omega_1}}
		\arrow[shift left, hook, from=1-1, to=1-2]
		\arrow["{\Ind((-)^\omega)}", shift left, from=1-2, to=1-1]
		\arrow[shift left, from=1-2, to=1-3]
		\arrow["{\Ind((-)^{\omega_1})}", shift left, from=1-3, to=1-2]
	\end{tikzcd}\]
	For an arbitrary functor $\D_i:I\to\PrL_{\st,\omega_1}$, we have $\lim^{\dbl}_I\Ind(\D_i^{\omega_1})\simeq\lim^{\cg}_I\Ind(\D_i^{\omega_1})$. For an arbitrary functor $\D_i:I\to\Pr^{\dbl}_{\st}$, we have an internal left adjoint fully faithful functor $\lim^{\cg}_I\Ind(\D_i^{\omega})\hookrightarrow\lim^{\dbl}_I\D_i$.

	Now we write the limit as
	\[
		\lim^{\dbl}_I\C_i\simeq\fib^{\dbl}\left(\lim^{\dbl}_I\Ind(\C_i^{\omega_1})\to\lim^{\dbl}_I\Ind(\Calk^{\cont}_{\omega_1}(\C_i))\right).
	\]
	By adjunction, the functor factors as
	\[\begin{tikzcd}
		{\lim^{\dbl}_I\C_i} & {\lim^{\dbl}_I\Ind(\C_i^{\omega_1})} & {\lim^{\dbl}_I\Ind(\Calk^{\cont}_{\omega_1}(\C_i))} \\
		{\lim^{\dbl}_I\C_i} & {\lim^{\cg}_I\Ind(\C_i^{\omega_1})} & {\lim^{\cg}_I\Ind(\Calk^{\cont}_{\omega_1}(\C_i))}
		\arrow[from=1-1, to=1-2]
		\arrow[from=1-2, to=1-3]
		\arrow[equals, from=2-1, to=1-1]
		\arrow[from=2-1, to=2-2]
		\arrow["\simeq"', from=2-2, to=1-2]
		\arrow[from=2-2, to=2-3]
		\arrow[hook, from=2-3, to=1-3]
	\end{tikzcd}\]
	By \cite[Proposition 1.84]{Efi25localizing}, the dualizable fiber $\fib^{\dbl}(F)$ of a functor $F$ in $\Pr^{\dbl}_{\st}$ is the largest dualizable full subcategory of the fiber $\fib(F)$ in $\PrL_{\st}$ such that $\fib^{\dbl}(F)\to\C$ is an internal left adjoint. Then the bottom row is also a fiber sequence in $\Pr^{\dbl}_{\st}$ and \eqref{eq:limit_ses} follows.
	
	In particular, if $\lim^{\dbl}_I$ preserves short exact sequences, then the top row is a short exact sequence, where the functor $\lim^{\dbl}_I\Ind(\C_i^{\omega_1})\to\lim^{\dbl}_I\Ind(\Calk^{\cont}_{\omega_1}(\C_i))$ is a localization and hence surjective. Therefore, the fully faithful functor $\lim^{\cg}_I\Ind(\Calk^{\cont}_{\omega_1}(\C_i))\hookrightarrow\lim^{\dbl}_I\Ind(\Calk^{\cont}_{\omega_1}(\C_i))$ is surjective and thus is an equivalence.
\end{proof}

\begin{definition}
	Let $G\in\Grp(\An)$. We define the \emph{homotopy orbits} and \emph{homotopy fixed points} as the dualizable colimit and dualizable limit
	\[
		(-)_{hG}:=\colim_{BG}:(\Pr^{\dbl}_{\st})^{BG}\to\Pr^{\dbl}_{\st},\quad(-)^{hG,\dbl}:=\lim^{\dbl}_{BG}:(\Pr^{\dbl}_{\st})^{BG}\to\Pr^{\dbl}_{\st}.
	\]
\end{definition}
By \cite[Proposition 4.4]{CSY24cyclotomic}, we have an equivalence $(\Pr^{\dbl}_{\st})^{BG}\simeq\Mod_{\Sp[G]}(\Pr^{\dbl}_{\st})$ where $\Sp[G]\simeq\Sp^G$ is equipped with the Day convolution $\bbE_1$-algebra structure. This can be upgraded to a symmetric monoidal equivalence by \cite[Example 2.2.9]{Rak26derham}, where the symmetric monoidal structure of $(\Pr^{\dbl}_\st)^{BG}$ is pointwise and that of $\Mod_{\Sp[G]}(\Pr^{\dbl}_{\st})$ comes from the cocommutative $\Sp$-bialgebra structure on $\Sp[G]$.

By \cite[Theorem 6.1]{Efi25rigidity}, $e_!:\Sp\to\Sp[G]$ is rigid in the $\bbE_1$-sense. We rewrite $(-)_{hG}\simeq\Sp\otimes_{\Sp[G]}-$ and $(-)^{hG,\dbl}\simeq\iHom^{\dbl}_{\Sp[G]}(\Sp,-)$, where $\iHom^{\dbl}_{\Sp[G]}(\Sp,-)$ denotes the internal Hom relative to $\Pr^{\dbl}_{\st}$. Then $(-)_{hG}$ preserves short exact sequences. If $G$ has compact underlying anima, then by \cite[Theorem 6.2]{Efi25rigidity}, $\Sp$ is proper and $\omega_1$-compact over $\Sp[G]$ in the $\bbE_1$-sense with evaluation $\ev_{\Sp/\Sp[G]}=\operatorname{const}:\Sp\otimes\Sp\simeq\Sp\to\Sp[G]$ and coevaluation $\coev_{\Sp/\Sp[G]}:\Sp\to\Sp\otimes_{\Sp[G]}\Sp\simeq\Sp_{hG}\simeq\Sp^{BG}$. It follows from the $\bbE_1$-analog of \cref{thm:Homdbl_proper} that $(-)^{hG,\dbl}$ also preserves short exact sequences.
\begin{lemma}[{\cite[Theorem 6.2]{Efi25rigidity}}]\label{lem:hGdbl}
	Let $G\in\Grp(\An)$. Then the homotopy orbits functor $(-)_{hG}:(\Pr^{\dbl}_{\st})^{BG}\to\Pr^{\dbl}_{\st}$ preserves short exact sequences. Suppose moreover that $G$ is a compact group anima, i.e. $G\in\An^{\omega}$. Then the homotopy fixed points functor $(-)^{hG,\dbl}:(\Pr^{\dbl}_{\st})^{BG}\to\Pr^{\dbl}_{\st}$ preserves short exact sequences.
\end{lemma}
\begin{proof}[Proof of {\cref{thm:descent_dualizable}}]
	Since $L_n^f\Sp$ is smashing, we have a short exact sequence 
	\[
		\C\hookrightarrow\Ind(\C^{\omega_1})\to\Ind(\Calk^{\cont}_{\omega_1}(\C))
	\]
	in $(\Pr^{\dbl}_{L_n^f})^{BG}$. Combining \cref{lem:limdbl,lem:hGdbl}, we obtain a short exact sequence in $\Pr^{\dbl}_{L_n^f}$
	\[
		\C^{hG,\dbl}\hookrightarrow\Ind(\C^{\omega_1})^{hG,\cg}\to\Ind(\Calk^{\cont}_{\omega_1}(\C))^{hG,\cg}.
	\]
	Applying $K_{T(n+1)}^{\cont}$ and using the compactly generated version of the descent theorem recalled in the Introduction, we obtain
	\begin{align*}
		K_{T(n+1)}^{\cont}(\C^{hG,\dbl})&\simeq\fib\left(K_{T(n+1)}^{\cont}(\Ind(\C^{\omega_1})^{hG,\cg})\to K_{T(n+1)}^{\cont}(\Ind(\Calk^{\cont}_{\omega_1}(\C))^{hG,\cg})\right)\\
		&\simeq\fib\left(K_{T(n+1)}^{\cont}(\Ind(\C^{\omega_1}))^{hG}\to K_{T(n+1)}^{\cont}(\Ind(\Calk^{\cont}_{\omega_1}(\C)))^{hG}\right)\\
		&\simeq K_{T(n+1)}^{\cont}(\C)^{hG}.
	\end{align*}

	Similarly, the second equivalence follows from \cref{lem:hGdbl} and the same compactly generated descent theorem.
\end{proof}
By \cref{thm:tate_vanishing_kuhn}, the norm map of spectra induces an equivalence $K^{\cont}_{T(n+1)}(\C)_{hG}\simeq K^{\cont}_{T(n+1)}(\C)^{hG}$. We show that this norm map is induced by the norm map in $\Pr^{\dbl}_{\st}$.
\begin{construction}[{\cite{Zhoa}}]
	Let $G\in\Grp(\An)$ be a compact group anima and $\C\in(\Pr^{\dbl}_{\st})^{BG}$. We define the \emph{norm map} $\Nm:\C_{hG}\to\C^{hG,\dbl}$ as the fully faithful internal left adjoint
	\[
		\C_{hG}\simeq\Sp\otimes_{\Sp[G]}\C\to\iHom^{\dbl}_{\Sp[G]}(\Sp,\Sp\otimes\Sp\otimes_{\Sp[G]}\C)\xrightarrow{\ev_{\Sp/\Sp[G]}}\iHom^{\dbl}_{\Sp[G]}(\Sp,\C)\simeq\C^{hG,\dbl},
	\]
	which is the left adjoint of the canonical map $\iHom^{\dbl}_{\Sp[G]}(\Sp,\C)\to\Sp\otimes_{\Sp[G]}\C$ by \cite[Proposition 3.4]{Efi25limit}, since $\Sp$ is proper over $\Sp[G]$.
\end{construction}
\begin{lemma}
	Let $G\in\Grp(\An)$ be a compact group anima. The norm map $\Nm:\C_{hG}\hookrightarrow\C^{hG,\dbl}$ coincides with the composition $\C_{hG}\simeq(\Sp[G]_{hG}\otimes\C)_{hG}\hookrightarrow(\Sp[G]^{hG,\dbl}\otimes\C)_{hG}\to\C^{hG}$, where the second map is the usual norm map. If $G$ is a finite group, then $\Sp\simeq\Sp[G]^{hG,\dbl}$ and $\Nm:\C_{hG}\hookrightarrow\C^{hG}$ coincides with the usual norm map.
\end{lemma}
\begin{proof}
	We have the following commutative diagram
	\[\begin{tikzcd}
		{\Sp\otimes_{\Sp[G]}\C} & {\iHom^{\dbl}_{\Sp[G]}(\Sp,\Sp\otimes_{\Sp[G]}\C)} & {\iHom^{\dbl}_{\Sp[G]}(\Sp,\Sp[G]\otimes_{\Sp[G]}\C)} \\
		{\Sp\otimes\Sp\otimes_{\Sp[G]}\C} & {\iHom^{\dbl}_{\Sp[G]}(\Sp,\Sp)\otimes\Sp\otimes_{\Sp[G]}\C} & {\iHom^{\dbl}_{\Sp[G]}(\Sp,\Sp[G])\otimes\Sp\otimes_{\Sp[G]}\C}
		\arrow[from=1-1, to=1-2]
		\arrow["{\ev_{\Sp/\Sp[G]}}", from=1-2, to=1-3]
		\arrow["\simeq", from=2-1, to=1-1]
		\arrow[from=2-1, to=2-2]
		\arrow[from=2-2, to=1-2]
		\arrow["{\ev_{\Sp/\Sp[G]}}"', from=2-2, to=2-3]
		\arrow[from=2-3, to=1-3]
	\end{tikzcd}\]
	The composition of the bottom functors is identified with $\C_{hG}\to(\Sp[G]^{hG,\dbl}\otimes\C)_{hG}$. The right vertical functor is identified with the composition
	\begin{align*}
		\Sp\otimes_{\Sp[G]}(\iHom^{\dbl}_{\Sp[G]}(\Sp,\Sp[G])\otimes\C)&\to\Sp\otimes_{\Sp[G]}\iHom^{\dbl}_{\Sp[G]}(\Sp,\Sp[G]\otimes\C)\\
		&\to\iHom^{\dbl}_{\Sp[G]}(\Sp,\Sp\otimes_{\Sp[G]}(\Sp[G]\otimes\C))\\
		&\simeq\iHom^{\dbl}_{\Sp[G]}(\Sp,\C),
	\end{align*}
	which is the usual norm map.

	Now let $G$ be a finite group. To show that $\Sp[G]^{hG,\dbl}\to\lim^{\PrL_{\st}}_{BG}\Sp[G]\simeq\Sp$ is an equivalence, by \cite[Corollary 2.7.11, Lemma 2.7.8]{KNP24}, it suffices to show that $\const:\Sp\to\Sp[G]$ preserves and detects compact maps. Since $\const$ and $\const^L\simeq\colim_G$ are $\Sp$-internal left adjoints, they preserve compact maps. It remains to show that a map $f:X\to Y\in\Sp$ is compact if $\const^L\circ\const f\simeq f[G]:X[G]\to Y[G]$ is compact. Since $\Sp$ is pointed and $G$ is finite, $f$ is a retract of $f[G]$. Therefore, $\const:\Sp\to\Sp[G]$ preserves and detects compact maps.
\end{proof}
\begin{remark}\label{rmk:lim^dbl}
	More generally, let $\C_i:I\to\Pr^{\dbl}_{\st}$ be a functor. By \cite[Corollary 2.7.11]{KNP24}, the dualizable limit can be computed by 
	\[
		\lim^{\dbl}_I\C_i\simeq\Ind_{\I_{\mathrm{sd}}}(\lim_I^{\PrL}\C_i),
	\]
	where $\I$ denotes the idealoid of maps whose image is compact in each $\C_i$.
\end{remark}
\begin{corollary}
	Let $n\geq0$ and $\C\in(\Pr^{\dbl}_{L_n^f})^{BG}$ with $G$ a finite $p$-group. Then the norm map induces an equivalence
	\[
		K_{T(n+1)}^{\cont}(\C_{hG})\xrightarrow{\simeq}K_{T(n+1)}^{\cont}(\C^{hG,\dbl}).
	\]
\end{corollary}
In fact, the $p$-group assumption can be removed using the Tate vanishing result that will be proved in \cref{sec:3.3}. To prepare for this, we study $\Sp_{hG}$-nuclear categories. This study also provides a natural categorical analogue of chromatic descent.
\begin{lemma}\label{lem:Sp^BG}
	Let $G\in\Grp(\An)$ be a compact group anima and $\Sp^{BG}$ be equipped with the pointwise symmetric monoidal structure. Then $\Sp^{BG}$ is locally rigid over $\Sp$ with $\mathbbm{1}_{\Sp^{BG}}\in(\Sp^{BG})^{\omega_1}$. Consequently, let $e:*\to BG$ be a base point. Then $e^*:\Sp^{BG}\to\Sp$ is rigid.
\end{lemma}
\begin{proof}
	For the first part we follow \cite[Example 4.4.10]{KNP24}. Let $e:*\to BG$ be a base point. Since $e^*:\Sp^{BG}\to\Sp$ creates colimits and $e^*e_*\simeq\iHom_{\bbS}(\bbS[G],-)$ preserves filtered colimits, $e^*$ is an $\Sp$-internal left adjoint and preserves compact objects. Then $(\Sp^{BG})^{\omega}\subseteq(\Sp^{\omega})^{BG}=(\Sp^{\dbl})^{BG}=(\Sp^{BG})^{\dbl}$. Since $\Sp^{BG}$ is compactly generated, by \cref{exm:compact_dbl} we obtain that $\Sp^{BG}$ is locally rigid over $\Sp$. Since $\const^R\simeq\lim_{BG}:\Sp^{BG}\to\Sp$ preserves $\omega_1$-filtered colimits, $\mathbbm{1}_{\Sp^{BG}}$ is $\omega_1$-compact.

	For the second part, since both $\const$ and $e^*\circ\const\simeq\id$ are locally rigid, $e^*$ is locally rigid by \cite[Corollary 4.19]{Ram26locallyrigid}. Since $e^*$ is an $\Sp$-internal left adjoint and $\Sp^{BG}$ is locally rigid over $\Sp$, $e^*$ is an $\Sp^{BG}$-internal left adjoint by \cite[Proposition 4.18]{Ram26locallyrigid}. Therefore, $e^*:\Sp^{BG}\to\Sp$ is rigid.
\end{proof}
Consequently, $\const:\C\to\C^{BG}$ is locally rigid and $e^*:\C^{BG}\to\C$ is rigid for any $\C\in\CAlg(\PrL_{\st})$. If $\C$ is locally rigid, then $\C^{BG}$ is also locally rigid. Moreover, we can compute the rigidification of $\C^{BG}$ via dualizable homotopy fixed points whenever $\C$ is rigid.
\begin{lemma}[{\cite{Zhoa}}]\label{cor:C_hG_locrig}
	Let $G\in\Grp(\An)$ be a compact group anima and $\C\in\CAlg(\PrL_{\st})^{BG}$. Then $(\C_{hG})^{\rig}\simeq(\C^{\rig})^{hG,\dbl}$. If $\C$ is rigid, then $\C_{hG}$ is locally rigid over $\Sp$ with $(\C_{hG})^{\rig}\simeq\C^{hG,\dbl}$, and $\mathbbm{1}_{\C_{hG}}\in(\C_{hG})^{\omega_1}$. In particular, the fully faithful left adjoint of $\C^{hG,\dbl}\to\C_{hG}$ is given by the norm map.
\end{lemma}
\begin{proof}
	Since both $(-)^{\rig}:\CAlg(\PrL_{\st})\to\CAlg^{\rig}(\PrL_{\st})$ and $\CAlg^{\rig}(\PrL_{\st})\to\Pr^{\dbl}_{\st}$ preserve limits, we have
	\[
		(\C^{\rig})^{hG,\dbl}\simeq\lim^{\CAlg(\Pr^{\dbl}_{\st})}_{BG}\C^{\rig}\simeq\lim^{\CAlg^{\rig}(\PrL_{\st})}_{BG}\C^{\rig}\simeq(\lim^{\CAlg(\PrL_{\st})}_{BG}\C)^{\rig}\simeq(\C_{hG})^{\rig}.
	\]
	If $\C$ is rigid, then $(\C_{hG})^{\rig}\simeq\C^{hG,\dbl}$. Therefore, $\C_{hG}\hookrightarrow\C^{hG,\dbl}$ is a smashing ideal and hence $\C_{hG}$ is locally rigid. Let $p:\C_{hG}\to\C$ denote the projection map. Since $\Hom_{\C_{hG}}(\mathbbm{1}_{\C_{hG}},-)\simeq\Hom_{\C}(\mathbbm{1}_{\C},p(-))^{hG}$ and $(-)^{hG}$ preserves $\omega_1$-filtered colimits, $\mathbbm{1}_{\C_{hG}}$ is $\omega_1$-compact.
\end{proof}

In particular, let $\C\in\CAlg^{\rig}(\PrL_{\st})^{BG}$ and $G\in\Grp(\An)$ be a compact group anima. Then we have $\End(\mathbbm{1}_{\C^{hG,\dbl}})\simeq\End(\mathbbm{1}_{\C_{hG}})\simeq\End(\mathbbm{1}_{\C})^{hG}$ by \cref{lem:1_rig}.

\begin{definition}
	Let $G\in\Grp(\An)$ be a compact group anima and $R\in\CAlg(\Sp)^{BG}$. We define the \emph{category of nuclear $R^{hG}$-modules} as
	\[
		\Nuc(R^{hG}):=\Nuc(\Mod(R)_{hG}).
	\]
\end{definition}
If $R$ is complex orientable with trivial $S^1$-action, this recovers \cref{exm:R^BS^1}.
\begin{lemma}\label{lem:CX_corig}
	Let $\C\in\CAlg(\PrL)$ and $X\in\An$, and let $\C^X$ be equipped with the pointwise symmetric monoidal structure. Then $\C^X$ is corigid over $\C$ in the sense of \cref{prop:corig}(1).
\end{lemma}
\begin{proof}
	It follows from \cite[Proposition 2.1.2, Lemma 2.1.6]{Rak26Tate} that $\C\to\C^X$ admits a $\C$-linear left adjoint and $\C^X\otimes_{\C}\C^X\simeq\C^{X\times X}\to\C^X$ admits a $\C^X\otimes_{\C}\C^X$-linear left adjoint.
\end{proof}
\begin{theorem}\label{thm:hGdbl=Nuc}
	Let $G\in\Grp(\An)$ be a compact group anima. Then we have adjunctions
	\[\begin{tikzcd}[column sep=6em]
		{\Pr^{\dbl}_{\Sp_{hG}}} & {\Pr^{\dbl}_{\st}}
		\arrow[shift left=3, from=1-1, to=1-2]
		\arrow["{\Nuc_{\Sp_{hG}}}"', shift right=3, from=1-1, to=1-2]
		\arrow["{\Sp_{hG}\otimes-}"{description}, from=1-2, to=1-1]
	\end{tikzcd}\]
	Moreover, we have adjunctions
	\[\begin{tikzcd}[column sep=6em]
		{(\Pr^{\dbl}_{\st})^{BG}} & {\Pr^{\dbl}_{\st}}
		\arrow["{(-)_{hG}}", shift left=3, from=1-1, to=1-2]
		\arrow["{\Nuc_{\Sp_{hG}}(-)_{hG}}"', shift right=3, from=1-1, to=1-2]
		\arrow["{\mathrm{const}}"{description}, from=1-2, to=1-1]
	\end{tikzcd}\]
	In other words, $(-)^{hG,\dbl}\simeq\Nuc_{\Sp_{hG}}((-)_{hG})$ and the norm map coincides with the $\Sp_{hG}$-completion map.
\end{theorem}
\begin{proof}
	The first part follows by combining \cref{lem:CX_corig,cor:C_hG_locrig,prop:corig}. For the second part, since $(-)_{hG}:(\PrL_{\st})^{BG}\to\PrL_{\st}$ factors through $(-)_{hG}:(\PrL_{\st})^{BG}\to\PrL_{\Sp_{hG}}$, we obtain an adjunction
	\[\begin{tikzcd}[column sep=6em]
		{(\PrL_{\st})^{BG}} & {\PrL_{\Sp_{hG}}}
		\arrow["{(-)_{hG}}"', shift right, from=1-1, to=1-2]
		\arrow["{(\Sp\otimes_{\Sp_{hG}}-)}"', shift right, hook', from=1-2, to=1-1]
	\end{tikzcd}\]
	Here, the left adjoint is fully faithful since the unit of $\C\in\PrL_{\Sp_{hG}}$ is given by
	\[
	(\Sp\otimes_{\Sp_{hG}}\C)_{hG}\simeq\colim_{BG}^{\PrL_{\Sp_{hG}}}\Sp\otimes_{\Sp_{hG}}\C\simeq\Sp_{hG}\otimes_{\Sp_{hG}}\C\simeq\C.
	\]
	To describe the counit of $\C\in(\PrL_{\st})^{BG}$, we apply \cite[Theorem 4.7.5.2]{HA} to the totalization of right adjoints
	\[
	\C_{hG}\simeq\Tot^{\PrL_{\st}}(\C^{G^\bullet}).
	\]
	The projection map $p:\C_{hG}\to\C$ admits a right adjoint $p^R:\C\to\C_{hG}$, with $pp^R\simeq\lim_Gg(-)$. Then the counit $\Sp\otimes_{\Sp_{hG}}\C_{hG}\to\C$ is given pointwise by the base change of $p:\C_{hG}\to\C$ along $\Sp_{hG}\to\Sp$.
	
	Since $p$ creates colimits and $\lim_G$ preserves filtered colimits, $p$ is an $\Sp$-internal left adjoint and hence an $\Sp_{hG}$-internal left adjoint as $\Sp_{hG}$ is locally rigid. Then the counit is a pointwise $\Sp$-internal left adjoint since $\Sp_{hG}\to\Sp$ is rigid by \cref{lem:Sp^BG}. Moreover, since $\Sp_{hG}$ is locally rigid, the forgetful functor $\PrL_{\Sp_{hG}}\to\PrL_{\st}$ reflects dualizability by \cite[Proposition 4.17]{Ram26locallyrigid}. Thus, this adjunction restricts to an adjunction
	\[\begin{tikzcd}[column sep=6em]
		{(\Pr^{\dbl}_{\st})^{BG}} & {\Pr^{\dbl}_{\Sp_{hG}}}
		\arrow["{(-)_{hG}}"', shift right, from=1-1, to=1-2]
		\arrow["{(\Sp\otimes_{\Sp_{hG}}-)}"', shift right, hook', from=1-2, to=1-1]
	\end{tikzcd}\]
	Combining this adjunction with that obtained in the first part yields the desired adjunction.
\end{proof}
\begin{remark}
	The homotopy orbits functor $(-)_{hG}:(\PrL_{\st})^{BG}\to\PrL_{\Sp_{hG}}$ is an equivalence if and only if it is conservative. If $G$ is a finite group, or more generally, a discrete group, then we obtain a symmetric monoidal equivalence $(-)_{hG}:(\PrL_{\st})^{BG}\simeq\Mod_{\Sp_{hG}}(\PrL_{\st})$, and hence $(-)_{hG}:(\Pr^{\dbl}_{\st})^{BG}\simeq\Pr^{\dbl}_{\Sp_{hG}}$. Therefore, for a finite group $G$, \cref{thm:hGdbl=Nuc} follows immediately.
\end{remark}
The above theorem provides the key input for constructing the lax symmetric monoidal structure of the Tate construction in \cref{sec:3.3}, thereby answering a question of Zhou. We state nuclear descent here alongside the other descent results, although its proof uses \cref{thm:tate_vanishing}, proved in \cref{sec:3.3} independently of this theorem and the intervening results.
\begin{theorem}[Nuclear descent]\label{thm:nuc_descent}
	Let $n\geq1$ and $\C\in(\Pr^{\dbl}_{L_n^f})^{BG}$ with $G$ a finite group. Then the norm map induces equivalences
	\[
		L_{T(n)}\C_{hG}\xrightarrow{\simeq}L_{T(n)}\C^{hG,\dbl},
	\]
	and
	\[
		\Nuc_{T(n)}(\C_{hG})\xrightarrow{\simeq}\Nuc_{T(n)}(\C^{hG,\dbl})\xrightarrow{\simeq}\Nuc_{T(n)}(\C)^{hG,\dbl}.
	\]

	Consequently, the norm map induces an equivalence
	\[
		K_{T(n+1)}^{\cont}(\C_{hG})\xrightarrow{\simeq}K_{T(n+1)}^{\cont}(\C^{hG,\dbl}).
	\]
\end{theorem}
\begin{proof}
	Assuming \cref{thm:tate_vanishing}, the first equivalence follows from the vanishing $L_{T(n)}\C^{tG,\dbl}\simeq0$. It follows from \cref{thm:unst_completion} that $\Nuc_{T(n)}(-)$ preserves dualizable limits, and hence $\Nuc_{T(n)}(\C_{hG})\simeq\Nuc_{T(n)}(L_{T(n)}\C_{hG})\simeq\Nuc_{T(n)}(L_{T(n)}\C^{hG,\dbl})\simeq\Nuc_{T(n)}(\C^{hG,\dbl})\simeq\Nuc_{T(n)}(\C)^{hG,\dbl}$. Therefore, by \cref{cor:purity_T(n)}, the norm map induces an equivalence on $K^{\cont}_{T(n+1)}$.
\end{proof}
\begin{corollary}
	Let $n\geq1$ and $\C\in\CAlg^{\rig}(\Pr^{\dbl}_{L_n^f})^{BG}$ with $G$ a finite group. Then the canonical maps induce symmetric monoidal equivalences
	\[
		\Nuc_{T(n)}(\C_{hG})\xleftarrow{\simeq}\Nuc_{T(n)}(\C^{hG,\dbl})\xrightarrow{\simeq}\Nuc_{T(n)}(\C)^{hG,\dbl}\xrightarrow{\simeq}(L_{T(n)}\C_{hG})^{\rig}.
	\]
\end{corollary}
\begin{proof}
	By \cref{thm:nuc_descent}, we only need to show the last equivalence. It follows from \cref{thm:unst_completion,cor:C_hG_locrig} that $\Nuc_{T(n)}(\C)^{hG,\dbl}\simeq((L_{T(n)}\C)^{\rig})^{hG,\dbl}\simeq(L_{T(n)}\C_{hG})^{\rig}$.
\end{proof}
\begin{remark}
	If $n=0$, \cref{thm:nuc_descent} holds for $\C\in(\Pr^{\dbl}_{T(0)})^{BG}$ with $G$ a finite $p$-group or $\C\in(\Pr^{\dbl}_{\bbQ})^{BG}$ with $G$ a finite group.
\end{remark}
In particular, for $R\in\CAlg(L_n^f\Sp)^{BG}$, we obtain symmetric monoidal equivalences
\[
	\Nuc_{T(n)}(\Mod(R)_{hG})\simeq\Nuc_{T(n)}(\Nuc(R^{hG}))\simeq\Nuc(L_{T(n)}R)^{hG,\dbl}.
\]

We now apply \cref{thm:purity_dualizable,thm:descent_dualizable} to reprove the $T(n)$-local Galois descent theorem of \cite{CMNN24}. Let $R\to S\in\CAlg(L_{T(n)}\Sp)$ with $G$ a finite group. Recall that we say $R\to S$ is a \emph{$T(n)$-local $G$-Galois extension} if there is an $R$-linear $G$-action on $S$ such that both $R\to S^{hG}$ and $S\otimes_RS\to\prod_GS$, $(a_1,a_2)\mapsto(a_1g(a_2))_{g\in G}$ are $T(n)$-local equivalences. A $T(n)$-local $G$-Galois extension $R\to S$ is always \emph{$T(n)$-locally faithful}, i.e. $L_{T(n)}(S\otimes_R-)$ is conservative. Indeed, $R\simeq S^{hG}\simeq L_{T(n)}S_{hG}$ and hence $\id_{L_{T(n)}\Mod(R)}\simeq L_{T(n)}(S\otimes_R-)_{hG}$.
\begin{corollary}[Galois descent, {\cite[Corollary 4.16]{CMNN24}}]
	Let $n\geq0$ and $R\to S$ be a $T(n)$-local $G$-Galois extension with $G$ a finite $p$-group. Then the canonical map induces an equivalence
	\[
		K_{T(n+1)}(R)\simeq K_{T(n+1)}(S)^{hG}.
	\]
\end{corollary}
\begin{proof}
	By \cite[Theorem 9.4]{Mat16}, we have a natural equivalence $L_{T(n)}\Mod(R)\simeq(L_{T(n)}\Mod(S))_{hG}$. Then it follows from \cref{cor:purity_T(n),thm:descent_dualizable} that
	\begin{align*}
		K_{T(n+1)}(R)\simeq&\;K_{T(n+1)}(L_{T(n)}\Mod(R))\simeq K_{T(n+1)}((L_{T(n)}\Mod(S))_{hG})\\
		\simeq&\;K_{T(n+1)}(L_{T(n)}\Mod(S))^{hG}\simeq K_{T(n+1)}(S)^{hG}.
	\end{align*}
\end{proof}
In fact, our proof only uses the compactly generated version of purity and descent, whereas the original proof in \cite{CMNN24} relies on the first equivalence of \cref{prop:nuc_Galois} together with \cref{cor:purity_T(n)}(2).
\begin{proposition}[Nuclear Galois descent]\label{prop:nuc_Galois}
	Let $n\geq1$ and $R\to S$ be a $T(n)$-local $G$-Galois extension with $G$ a finite group. Then the canonical maps induce symmetric monoidal equivalences
	\[
		\Ind(L_{T(n)}\Mod(R)^{\dbl})\xrightarrow{\simeq}\Ind(L_{T(n)}\Mod(S)^{\dbl})^{hG,\cg},
	\]
	and
	\[
		\Nuc(L_{T(n)}R)\xrightarrow{\simeq}\Nuc_{T(n)}(\Nuc(S^{hG}))\xrightarrow{\simeq}\Nuc(L_{T(n)}S)^{hG,\dbl}.
	\]

	Consequently, the canonical map induces an equivalence
	\[
		K_{T(n+1)}(R)\xrightarrow{\simeq}K^{\cont}_{T(n+1)}(\Nuc(S^{hG})).
	\]
\end{proposition}
\begin{proof}
	Since $L_{T(n)}\Mod(R)\simeq(L_{T(n)}\Mod(S))_{hG}$, we have $L_{T(n)}\Mod(R)^{\dbl}\simeq(L_{T(n)}\Mod(S)^{\dbl})^{hG}$ by \cite[Proposition 4.6.1.11]{HA}, and the first equivalence follows. For the second equivalence, by \cref{thm:nuc_descent} we have $\Nuc(L_{T(n)}R)\simeq\Nuc_{T(n)}(L_{T(n)}\Mod(R))\simeq\Nuc_{T(n)}(L_{T(n)}\Mod(S)_{hG})\simeq\Nuc_{T(n)}(\Mod(S)_{hG})\simeq\Nuc_{T(n)}(\Nuc(S^{hG}))\simeq\Nuc(L_{T(n)}S)^{hG,\dbl}$.
\end{proof}
\begin{remark}
	In their forthcoming work, Burklund--Clausen prove the corresponding Galois descent theorem for arbitrary finite groups \cite[00:15:35]{Cla24perfection}: if $R\to S$ is a $T(n)$-local $G$-Galois extension with $G$ a finite group, then the canonical map induces an equivalence
	\[
		K_{T(n+1)}(R)\xrightarrow{\simeq}K_{T(n+1)}(S)^{hG}.
	\]
	Combining this result with \cref{prop:nuc_Galois} yields equivalences
	\begin{align*}
		K^{\cont}_{T(n+1)}(\Nuc(S^{hG}))\simeq K_{T(n+1)}(R)\simeq K_{T(n+1)}(S)^{hG}.
	\end{align*}
\end{remark}

\subsection{Redshift or noshift, and blueshift}\label{sec:3.3}

We study the chromatic height of symmetric monoidal stable categories. We prove that continuous $K$-theory raises chromatic height by at most $1$, and, combining this with Ramzi's noshift theorem and dimension map construction in \cite{Ram26noshift}, show that redshift and noshift are the only possible behaviors for rigid symmetric monoidal stable categories. We then prove a categorical Tate vanishing theorem for dualizable stable categories and show that the $C_p$-Tate fixed points of the trivial action lower chromatic height by at most $1$, thereby establishing a blueshift theorem for rigid symmetric monoidal stable categories.

We begin by recalling the notion of chromatic height for $\bbE_\infty$-rings. The following theorem ensures that it is well defined.
\begin{theorem}[{\cite[Theorem 1.1]{Hah22Hinfty}, \cite[Theorem 1.5]{BSY22nullstellensatz}}]\label{thm:Hahn}
	Let $R\in\CAlg(\Sp)$ and $n\geq0$. If $L_{T(n)}R\simeq0$, then $L_{T(n+1)}R\simeq0$.
\end{theorem}
\begin{definition}
	Let $R\in\CAlg(\Sp)$ be nonzero. We define the \emph{chromatic height} of $R$ as
	\[
		\operatorname{ht}(R):=\max\{n:L_{T(n)}R\not\simeq0\},
	\]
	where $\operatorname{ht}(R):=-1$ if $L_{T(0)}R\simeq R[1/p]\simeq0$, and $\operatorname{ht}(R):=\infty$ if $L_{T(n)}R\not\simeq0$ for all $n$.
\end{definition}
Equivalently, $\operatorname{ht}(R)=\max\{n:L_{T(n)}R\not\simeq0\}=\min\{n:L_{T(n)}R\simeq0\}-1$. With this notion of height, Ausoni--Rognes' redshift conjecture for $\bbE_\infty$-rings has been proved.
\begin{theorem}[Redshift, {\cite[Redshift Theorem]{LMMT24}, \cite[Theorem E]{BSY22nullstellensatz}}]
	Let $R\in\CAlg(\Sp)$ be nonzero and of height $\operatorname{ht}(R_{(p)})\geq0$. Then
	\[
		\operatorname{ht}(K(R))=\operatorname{ht}(R)+1.
	\]
\end{theorem}
The upper bound $\operatorname{ht}(K(R))\leq\operatorname{ht}(R)+1$ is proved in \cite{LMMT24,CMNN24} using purely categorical arguments. Therefore, it extends naturally to symmetric monoidal stable categories. By contrast, the proof of the lower bound proceeds by reducing to a class of basic examples, called Nullstellensatzian objects. In \cite{BSY22nullstellensatz}, Burklund--Schlank--Yuan show that every $\bbE_\infty$-ring $R$ such that $\operatorname{ht}(R_{(p)})\geq0$ admits a map to a Lubin--Tate theory $\bbE_\infty$-ring $E_n(L)$ of the same height, and such Lubin--Tate theories satisfy chromatic redshift by \cite{Yua21redshift,BSY22nullstellensatz}. Since chromatic redshift is ``downward closed", this proves chromatic redshift under the hypotheses of the preceding theorem.

The same reduction strategy applies to the categorical setting. In particular, Nullstellensatzian objects remain the testing objects for chromatic redshift. In \cite{Ram26noshift}, Ramzi proves that every rigid $\Sp$-algebra admits a map to another rigid $\Sp$-algebra satisfying chromatic noshift. Thus, the testing objects themselves already produce counterexamples to chromatic redshift. Moreover, since chromatic noshift is ``upward closed", we obtain a class of failures of chromatic redshift. Combined with the dimension map constructed in loc. cit., this shows that chromatic redshift and chromatic noshift are the only possible behaviors for rigid $\Sp$-algebras.

Motivated by the case of $\bbE_\infty$-rings, we define the chromatic height of a symmetric monoidal stable category via its endomorphism $\bbE_\infty$-ring of the unit. Recall the adjunction
\[
	\Mod(-):\CAlg(\Sp)\rightleftarrows\CAlg(\PrL_{\st}):\End(\mathbbm{1}_{(-)}).
\]
The unit is an equivalence $R\simeq\End(\mathbbm{1}_{\Mod(R)})$ for $R\in\CAlg(\Sp)$, so that $\Mod(-)$ is fully faithful. The counit is given by the symmetric monoidal functor $\Mod(\End(\mathbbm{1}_{\C}))\to\C$ for $\C\in\CAlg(\PrL_{\st})$, which is fully faithful whenever $\mathbbm{1}\in\C^{\omega}$. 
\begin{lemma}\label{lem:End(1_L)}
	Let $\C\in\CAlg(\PrL_{\st})$ and $L:\Sp\to L\Sp$ be a symmetric monoidal localization. If $\mathbbm{1}\in\C^{\omega}$, then $\End(\mathbbm{1}_{L\Sp\otimes\C})\simeq L\End(\mathbbm{1}_{\C})$ and $L\Mod(\End(\mathbbm{1}_{\C}))\hookrightarrow L\Sp\otimes\C$ is fully faithful.
\end{lemma}
\begin{proof}
	Write $f:\Mod(\End(\mathbbm{1}_{\C}))\to\C$. Then we have a commutative diagram
	\[\begin{tikzcd}
		{\Mod(\End(\mathbbm{1}_{\C}))} & \C \\
		{L\Mod(\End(\mathbbm{1}_{\C}))} & {L\Sp\otimes\C}
		\arrow["f", hook, shift left, from=1-1, to=1-2]
		\arrow["L"', shift right, from=1-1, to=2-1]
		\arrow[shift left, from=1-2, to=1-1]
		\arrow["{L\otimes\C}"', shift right, from=1-2, to=2-2]
		\arrow[shift right, hook, from=2-1, to=1-1]
		\arrow["{L\Sp\otimes f}", shift left, from=2-1, to=2-2]
		\arrow[shift right, hook, from=2-2, to=1-2]
		\arrow[shift left, from=2-2, to=2-1]
	\end{tikzcd}\]
	Since $\mathbbm{1}\in\C^{\omega}$, the commutative diagram is horizontally right adjointable. Then $(L\Sp\otimes f)^R(\mathbbm{1}_{L\Sp\otimes\C})\simeq(L\Sp\otimes f)^R(L\otimes\C)(\mathbbm{1}_{\C})\simeq Lf^R(\mathbbm{1}_{\C})\simeq L\End(\mathbbm{1}_{\C})$. By \cref{lem:LC=LC(Mod(L1))}, $L\Mod(\End(\mathbbm{1}_{\C}))\simeq L\Mod(L\End(\mathbbm{1}_{\C}))$, and therefore $\End(\mathbbm{1}_{L\Sp\otimes\C})\simeq L\End(\mathbbm{1}_{\C})$.
	
	In particular, since $\Mod(\End(\mathbbm{1}))\hookrightarrow\C$ is an $\Sp$-internal left adjoint, $L\Sp\otimes\Mod(\End(\mathbbm{1}_{\C}))\hookrightarrow L\Sp\otimes\C$ is also fully faithful.
\end{proof}
\begin{definition}\label{def:ht(C)}
	Let $\C\in\CAlg(\PrL_{\st})$ be nonzero. We define the \emph{chromatic height} of $\C$ as
	\[
		\operatorname{ht}(\C):=\operatorname{ht}(\End(\mathbbm{1}_{\C})).
	\]
\end{definition}
In particular, if $\C$ is nonzero and $T(n)$-local, then $\End(\mathbbm{1}_{\C})$ is $T(n)$-local and we have $\operatorname{ht}(\C)=n$. 
\begin{theorem}[Redshift bound; Redshift or noshift]\label{thm:redshift}
	Let $\C\in\CAlg(\Pr^{\dbl}_{\st})$ such that $K^{\cont}(\C)$ is nonzero. Then
	\[
		\operatorname{ht}(K^{\cont}(\C))\leq\operatorname{ht}(\C)+1.
	\]
	Moreover, if $\C\in\CAlg^{\rig}(\PrL_{\st})$, then we have 
	\[
		\operatorname{ht}(\C)\leq\operatorname{ht}(K^{\cont}(\C))\leq\operatorname{ht}(\C)+1.
	\]
\end{theorem}
\begin{remark}
	Let $\C\in\CAlg^{\rig}(\PrL_{\st})$. We denote by
	\[
		\U_{\loc}:\Pr^{\dbl}_{\C}\to\mathrm{Mot}^{\loc}_{\C}
	\]
	the universal (finitary) localizing invariant, and call $\mathrm{Mot}^{\loc}_{\C}$ the \emph{category of noncommutative $\C$-motives}. The corepresentability of $K$-theory yields an equivalence of $\bbE_\infty$-rings
	\[
		K^{\cont}(\C)\simeq\End(\mathbbm{1}_{\mathrm{Mot}^{\loc}_{\C}}),
	\]
	whence $\operatorname{ht}(\mathrm{Mot}^{\loc}_{\C})=\operatorname{ht}(K^{\cont}(\C))$. Moreover, by \cite[Theorem 3.1]{Efi25rigidity}, $\mathrm{Mot}^{\loc}_{\C}$ is again a rigid $\Sp$-algebra. Therefore, \cref{thm:redshift} is equivalent to the statement that taking the category of noncommutative motives, $\mathrm{Mot}^{\loc}_{(-)}$, preserves or raises chromatic height by at most $1$:
	\[
		\operatorname{ht}(\C)\leq\operatorname{ht}(\mathrm{Mot}^{\loc}_{\C})\leq\operatorname{ht}(\C)+1.
	\]
\end{remark}

The following proposition, together with the purity theorem, proves the redshift bound of \cref{thm:redshift}. 
\begin{proposition}\label{prop:ht(C)}
	Let $\C\in\CAlg(\PrL_{\st})$.
	\begin{enumerate}
		\item We have 
		\[
			\operatorname{ht}(\C)\geq\max\{n:L_{T(n)}\C\not\simeq0\}\geq\min\{n:L_{T(n)}\C\simeq0\}-1.
		\]
		\item If $\mathbbm{1}\in\C^{\omega}$, then $\End(\mathbbm{1}_{L_{T(n)}\C})=L_{T(n)}\End(\mathbbm{1}_{\C})$ and we have
		\[
			\operatorname{ht}(\C)=\max\{n:L_{T(n)}\C\not\simeq0\}=\min\{n:L_{T(n)}\C\simeq0\}-1.
		\]
	\end{enumerate}
\end{proposition}
\begin{proof}
	\begin{enumerate}
		\item If $n>\operatorname{ht}(\C)$, then $L_{T(n)}\C$ is an algebra over $L_{T(n)}\Mod(\End(\mathbbm{1}_{\C}))\simeq0$, and hence $L_{T(n)}\C\simeq0$.
		\item By \cref{lem:End(1_L)}, we have $\End(\mathbbm{1}_{L_{T(n)}\C})\simeq L_{T(n)}\End(\mathbbm{1})$. If $n\leq\operatorname{ht}(\C)$, then $L_{T(n)}\End(\mathbbm{1})\not\simeq0$ and hence $L_{T(n)}\C\not\simeq0$. \qedhere
	\end{enumerate}
\end{proof}
\begin{example}
	Let $n\geq1$ and $\C=L_{T(n)}\Sp$. Then we have $\operatorname{ht}(L_{T(n)}\Sp)=n$ and $\{m:L_{T(m)}\C\not\simeq0\}=\{n\}$ by \cref{prop:M_n^mf}. So $\operatorname{ht}(L_{T(n)}\Sp)=\max\{m:L_{T(m)}\C\not\simeq0\}>\min\{m:L_{T(m)}\C\simeq0\}-1=-1$. This also shows that the compact unit assumption in \cref{lem:End(1_L)} cannot be omitted, since $L_{T(n-1)}L_{T(n)}\Sp\simeq0$ while $L_{T(n-1)}L_{T(n)}\bbS\not\simeq0$.
\end{example}

To establish the noshift bound, we review the construction of \cite{Ram26noshift}. For rigid $\Sp$-algebras, Ramzi constructs a dimension map 
\[
	\dim:K^{\cont}(\C)\to\End(\mathbbm{1}_{\C}),
\]
which implies that $\operatorname{ht}(K^{\cont}(\C))\geq\operatorname{ht}(\C)$.

As observed in \cite[Remark 1.3]{Ram26noshift}, the construction extends naturally to the general rigid setting.
\begin{definition}
	Let $\C\in\CAlg(\PrL)$. We define the \emph{Hochschild homology} functor as the symmetric monoidal functor
	\[
		\HH(-/\C):\Pr^{\dbl}_{\C}\to\C^{BS^1},
	\]
	obtained by applying \cite[Definition 2.11, Theorem 2.14]{HSS17} to the symmetric monoidal $2$-category $\PrL_{\C}$. Indeed, for $\M\in\Pr^{\dbl}_{\C}$, we define $\HH(\M/\C):=\ev_{\M/\C}\circ\coev_{\M/\C}(\mathbbm{1}_{\C})$. We write
	\[
		\THH(-):=\HH(-/\Sp)
	\]
	for the \emph{topological Hochschild homology}.
\end{definition}
The following base change property follows immediately from the definition. For $\D\in\CAlg(\Pr^{\dbl}_{\C})$ and $\M\in\Pr^{\dbl}_{\C}$, we have $\HH(\D\otimes_{\C}\M/\D)\simeq\mathbbm{1}_{\D}\otimes\HH(\M/\C)$.
\begin{construction}
	By \cite[Theorem 1.54]{Ram26noshift}, there exists a unique natural transformation
	\[
		\HH(-/-)\to\iEnd_{(-)}(\mathbbm{1}_{(-)})
	\]
	of functors with respect to $\C\in\CAlg(\PrL)$ and $\D\in\CAlg^{\rig}(\PrL_{\C})$. Moreover, there exists a unique $S^1$-equivariant refinement of this natural transformation. Indeed, this natural transformation is given by the following explicit construction. Let $\C\in\CAlg(\PrL)$ and $\D\in\CAlg^{\rig}(\PrL_{\C})$. Since the multiplication $\D\otimes_{\C}\D\to\D$ lies in $\Pr^{\dbl}_{\C}$, it induces a map
	\[
		\mathbbm{1}_{\D}\otimes\HH(\D/\C)\simeq\HH(\D\otimes_{\C}\D/\D)\to\HH(\D/\D)\simeq\mathbbm{1}_{\D}.
	\]
	By adjunction, this yields an $S^1$-equivariant $\bbE_\infty$-ring map
	\[
		\HH(\D/\C)\to\End_{\C}(\mathbbm{1}_{\D}),
	\]
	which we call the \emph{trace map}.
\end{construction}
The rigidity condition is essential in the above construction, since it ensures that the multiplication functor is internally left adjoint and hence induces the trace map.

Now assume that $\C=\Sp$. By \cite[Theorem 3.4]{HSS17}, $\THH$ is a localizing invariant. By the corepresentability of $K$-theory, $\THH$ induces a natural transformation
\[
	\mathrm{tr}:K^{\cont}(-)\to\Hom_{\Sp}(\THH(\Sp),\THH(-))\simeq\THH(-),
\]
classified by the element $1\in\pi_0\THH(\Sp)\simeq\pi_0(\bbS)\simeq\bbZ$. This is called the \emph{Dennis trace map}.
\begin{definition}\label{def:dim}
	We define the \emph{dimension map} as the composition
	\[
		\dim:K^{\cont}\to\THH\to\End(\mathbbm{1}_{(-)}).
	\]
	For every $\C\in\CAlg^{\rig}(\PrL_{\st})$, it induces an $\bbE_\infty$-ring map
	\[
		\dim:K^{\cont}(\C)\to\End(\mathbbm{1}_{\C}).
	\]
\end{definition}

Ramzi's main result shows that the chromatic redshift phenomenon of $K$-theory of rigid $\Sp$-algebras fails in general, and provides examples of rigid $\Sp$-algebras satisfying noshift.
\begin{theorem}[Noshift, {\cite[Noshift Theorem]{Ram26noshift}}]\label{thm:noshift}
	Let $n\geq0$ and $\C\in\CAlg^{\rig}(\PrL_{T(n)})$ be nonzero. Then there exists a fully faithful $\C\to\D\in\CAlg^{\rig}(\PrL_{T(n)})$ such that
	\[
		K^{\cont}_{T(n+1)}(\Ind(\D^{\dbl}))\simeq0.
	\]
\end{theorem}
In fact, let $\D\in\CAlg^{\rig}(\PrL_{T(n)})$ be \emph{$\omega_1$-Nullstellensatzian}, namely, $\D$ is nonzero and every nonzero $\E\in\CAlg^{\rig}(\PrL_{\D})^{\omega_1}$ admits a map back to $\D$. Then \cite[$K$-theoretic Nullstellensatz Theorem]{Ram26noshift} gives
\[
	K^{\cont}_{T(n+1)}(\Ind(\D^{\dbl}))\simeq0.
\]
This provides a class of examples for which chromatic redshift fails, since for a nonzero $\C\in\CAlg^{\rig}(\PrL_{T(n)})$, there exists a map $\C\to\D$ where $\D$ is $\omega_1$-Nullstellensatzian in $\CAlg^{\rig}(\PrL_{T(n)})$ by \cite[Example 3.14, Corollary 3.15]{Ram26noshift}.
\begin{proof}[Proof of {\cref{thm:redshift}}]
	If $\operatorname{ht}(\C)=\infty$, the upper bound is tautological. Assume $\operatorname{ht}(\C)=n\geq-1$. By \cref{prop:ht(C)}, $L_{T(n+1)}\C\simeq0$ and $L_{T(n+2)}\C\simeq0$. Then by \cref{thm:purity_dualizable}, we have $K^{\cont}_{T(n+2)}(\C)\simeq K^{\cont}_{T(n+2)}(L_{T(n+2)\oplus T(n+1)}\C)\simeq0$. Therefore, $K$-theory increases the chromatic height by at most $1$
	\[
		\operatorname{ht}(K^{\cont}(\C))\leq\operatorname{ht}(\C)+1.
	\]

	If $\C$ is rigid, then we have a dimension map $\dim:K^{\cont}(\C)\to\End(\mathbbm{1}_{\C})$. It follows that $\operatorname{ht}(K^{\cont}(\C))\geq\operatorname{ht}(\C)$, and hence
	\[
		\operatorname{ht}(\C)\leq\operatorname{ht}(K^{\cont}(\C))\leq\operatorname{ht}(\C)+1.
	\]
	For $n\geq0$, take $\C=\Mod(R)$, where $R$ is a $p$-local $\bbE_\infty$-ring of height $n$. Then we have $\operatorname{ht}(K^{\cont}(\C))=\operatorname{ht}(\C)+1$. Let $\D$ be as in \cref{thm:noshift}. Then we have $\operatorname{ht}(K^{\cont}(\Ind(\D^{\dbl})))=\operatorname{ht}(\Ind(\D^{\dbl}))$ where $\Ind(\D^{\dbl})$ is rigid.
\end{proof}
\begin{corollary}
	Let $n\geq0$. Then there exists a nonzero $\D\in\CAlg^{\rig}(\Pr^{\dbl}_{T(n)})$ such that
	\[
		\operatorname{ht}(K^{\cont}(\Ind(\D^{\dbl})))=\operatorname{ht}(\Ind(\D^{\dbl})),\quad\operatorname{ht}(K^{\cont}(\D^{\rig}))=\operatorname{ht}(\D^{\rig}).
	\]
	In other words, the rigid $\Sp$-algebras $\Ind(\D^{\dbl})$ and $\D^{\rig}$ both satisfy chromatic noshift.
\end{corollary}
\begin{proof}
	By \cref{constr:Ind_dbl} we have a map $\Ind(\D^{\dbl})\to\D^{\rig}$. Then it follows from $K^{\cont}_{T(n+1)}(\Ind(\D^{\dbl}))\simeq0$ that $K^{\cont}_{T(n+1)}(\D^{\rig})\simeq0$, and hence $\operatorname{ht}(K^{\cont}(\D^{\rig}))=\operatorname{ht}(\D^{\rig})$.
\end{proof}
Moreover, we obtain examples of $T(n)$-complete rigid $\Sp$-algebras that satisfy chromatic noshift.
\begin{proposition}\label{prop:End(1_Nuc_T(n))}
	Let $n\geq0$ and $\C\in\CAlg(\PrL_{\st})$.
	\begin{enumerate}
		\item If $n\geq1$ and $\C$ is $T(n)$-local and $\mathbbm{1}\in\C^{\omega}$, then $\C\simeq0$.
		\item If $\C\in\CAlg(\Pr^{\dbl}_{\st})$, then $\Nuc_{T(n)}(\C)\in\CAlg(\Pr^{\dbl}_{\st})$ is $T(n)$-complete and 
		\[
			\End(\mathbbm{1}_{\Nuc_{T(n)}\C})\simeq L_{T(n)}\End(\mathbbm{1}_{\C}).
		\]
	\end{enumerate}
	In particular, if $\C\in\CAlg(\Pr^{\dbl}_{\st})$ is nonzero and $T(n)$-complete, we have $\operatorname{ht}(\C)=n$.
\end{proposition}
\begin{proof}
	\begin{enumerate}
		\item If $\C$ is $T(n)$-local and $\mathbbm{1}\in\C^{\omega}$, then by \cref{lem:End(1_L),prop:M_n^mf}, $L_{T(n-1)}\End(\mathbbm{1}_{\C})\simeq\End(\mathbbm{1}_{L_{T(n-1)}\C})\simeq0$. Since $\End(\mathbbm{1}_{\C})$ is a $T(n)$-local $\bbE_\infty$-ring, $\End(\mathbbm{1}_{\C})\simeq0$ and hence $\C\simeq0$.
		\item If $\C$ is rigid, then by \cref{thm:unst_completion} we have $\Nuc_{T(n)}(\C)\simeq(L_{T(n)}\C)^{\rig}$, and hence $\End(\mathbbm{1}_{\Nuc_{T(n)}(\C)})\simeq\End(\mathbbm{1}_{L_{T(n)}\C})\simeq L_{T(n)}\End(\mathbbm{1}_{\C})$ by \cref{lem:1_rig,lem:End(1_L)}. In general, let $\C\in\CAlg(\Pr^{\dbl}_{\st})$. The fully faithful $\Sp$-internal algebra map $\Mod(\End(\mathbbm{1}_{\C}))\hookrightarrow\C$ induces a fully faithful $\Sp$-internal algebra map $\Nuc(L_{T(n)}\End(\mathbbm{1}_{\C}))\hookrightarrow\Nuc_{T(n)}(\C)$. Since $\Mod(\End(\mathbbm{1}_{\C}))$ is rigid, we have
		\[
			\End(\mathbbm{1}_{\Nuc_{T(n)}(\C)})\simeq\End(\mathbbm{1}_{\Nuc(L_{T(n)}\End(\mathbbm{1}_{\C}))})\simeq L_{T(n)}\End(\mathbbm{1}_{\C}).
		\]
	\end{enumerate}
\end{proof}
For $n\geq1$, it follows that there are no nonzero rigid $T(n)$-local $\Sp$-algebras.
\begin{corollary}
	Let $\D\in\CAlg^{\rig}(\PrL_{\st})$ and $\operatorname{ht}(\D)=n\geq0$. If $\D$ satisfies chromatic redshift or noshift, then the $T(n)$-complete rigid $\Sp$-algebra $\Nuc_{T(n)}(\D)$ also satisfies chromatic redshift or noshift, respectively.
\end{corollary}
\begin{proof}
	Note that $\operatorname{ht}(\Nuc_{T(n)}\D)=n$, and $K^{\cont}_{T(n+1)}(\Nuc_{T(n)}\D)\simeq K^{\cont}_{T(n+1)}(L_{T(n)}\D)\simeq K^{\cont}_{T(n+1)}(\D)$ since $L_{T(n+1)}\D\simeq0$ by \cref{prop:ht(C)}(2).
\end{proof}
\begin{question}
	Under what conditions does a rigid $\Sp$-algebra satisfy chromatic redshift?
\end{question}

In \cite{Zhoa}, Zhou introduces the Tate construction for dualizable categories. Combining this with the results of \cref{sec:3.2}, we show that the Tate construction is lax symmetric monoidal. This allows us to extend the blueshift phenomenon for the Tate construction of spectra to dualizable categories.
\begin{definition}[{\cite{Zhoa}}]
	Let $G\in\Grp(\An)$ be a compact group anima. We define the \emph{dualizable Tate fixed points} as the cofiber of the norm map, or equivalently, the $\Sp_{hG}$-generic fiber
	\[
		(-)^{tG,\dbl}:=\cof(\Nm:(-)_{hG}\hookrightarrow(-)^{hG,\dbl})\simeq((-)_{hG})_{\eta}.
	\]
\end{definition}
\begin{proposition}\label{prop:tG}
	Let $G\in\Grp(\An)$ be a compact group anima. Then both the homotopy fixed points functor $(-)^{hG,\dbl}:(\Pr^{\dbl}_{\st})^{BG}\to\Pr^{\dbl}_{\st}$ and the dualizable Tate fixed points functor $(-)^{tG,\dbl}:(\Pr^{\dbl}_{\st})^{BG}\to\Pr^{\dbl}_{\st}$ are lax symmetric monoidal and preserve rigid $\Sp$-algebras.
\end{proposition}
\begin{proof}
	Since $(-)^{hG,\dbl}$ is the right adjoint of the symmetric monoidal functor $\const:\Pr^{\dbl}_{\st}\to(\Pr^{\dbl}_{\st})^{BG}$, it admits a canonical lax symmetric monoidal structure. By \cref{cor:C_hG_locrig}, $(-)^{hG,\dbl}$ preserves rigid $\Sp$-algebras. For the dualizable Tate fixed points, by \cref{thm:hGdbl=Nuc}, we see that $(-)^{tG,\dbl}\simeq\cof(\Nm:(-)_{hG}\hookrightarrow\Nuc_{\Sp_{hG}}((-)_{hG}))\simeq\Nuc_{\Sp_{hG}}((-)_{hG})_{\eta}\simeq((-)^{hG,\dbl})_{\eta}$ is the $\Sp_{hG}$-generic fiber. Therefore, by \cref{cor:generic_fiber}, $(-)^{tG,\dbl}$ is lax symmetric monoidal and preserves rigid $\Sp$-algebras.
\end{proof}
\begin{proposition}[{\cite{Zhoa}}]\label{prop:End(1_tG)}
	Let $\C\in\CAlg(\Pr^{\dbl}_{\st})^{BG}$ with $G\in\Grp(\An)$ a compact group anima. Then we have
	\[
		\End(\mathbbm{1}_{\C^{hG,\dbl}})\simeq\End(\mathbbm{1}_{\C})^{hG},\quad\End(\mathbbm{1}_{\C^{tG,\dbl}})\simeq\End(\mathbbm{1}_{\C})^{tG}.
	\]
\end{proposition}
\begin{proof}
	The fully faithful $\Sp$-internal algebra map $\Mod(\End(\mathbbm{1}_{\C}))\hookrightarrow\C$ induces a map in $\CAlg(\Pr^{\dbl}_{\st})^{BG}$. Then it induces fully faithful $\Sp$-algebra maps $\Mod(\End(\mathbbm{1}_{\C}))^{hG,\dbl}\hookrightarrow\C^{hG,\dbl}$ and $\Mod(\End(\mathbbm{1}_{\C}))^{tG,\dbl}\hookrightarrow\C^{tG,\dbl}$. It suffices to show that $\End(\mathbbm{1}_{\Mod(R)^{hG,\dbl}})\simeq R^{hG}$ and $\End(\mathbbm{1}_{\Mod(R)^{tG,\dbl}})\simeq R^{tG}$ for $R=\End(\mathbbm{1}_{\C})$.
	
	The first equivalence has been proved after the proof of \cref{cor:C_hG_locrig}. For the second equivalence, we consider the short exact sequence induced by the norm map 
	\[\begin{tikzcd}
		{\Mod(R)_{hG}} & {\Mod(R)^{hG,\dbl}} & {\Mod(R)^{tG,\dbl}}
		\arrow["{j_!}", shift left=3, hook, from=1-1, to=1-2]
		\arrow["{j_*}"', shift right=3, hook, from=1-1, to=1-2]
		\arrow["{j^*}"{description}, from=1-2, to=1-1]
		\arrow["{i^*}", shift left=3, from=1-2, to=1-3]
		\arrow["{i^!}"', shift right=3, from=1-2, to=1-3]
		\arrow["{i_*}"{description}, hook', from=1-3, to=1-2]
	\end{tikzcd}\]
	Let $t:=\cof(j_!\to j_*)\simeq i_*i^*j_*$. Then we obtain a fiber sequence
	\begin{align*}
		\Hom_{\Mod(R)^{hG,\dbl}}(\mathbbm{1}_{\Mod(R)^{hG,\dbl}},j_!(-))\to&\;\Hom_{\Mod(R)_{hG}}(\mathbbm{1}_{\Mod(R)_{hG}},-)\\
		\to&\;\Hom_{\Mod(R)^{tG,\dbl}}(\mathbbm{1}_{\Mod(R)^{tG,\dbl}},t(-)).
	\end{align*}
	Denote by $p:\Mod(R)_{hG}\to\Mod(R)$ the projection functor and by $p^L$ its left adjoint, so that $pp^L\simeq\colim_Gg(-)$. We have a colimit preserving functor $\Hom_R(p(\mathbbm{1}_{\Mod(R)_{hG}}),p(-))\simeq\Hom_R(R,p(-))\simeq p(-):\Mod(R)_{hG}\to\Sp^{BG}$. For every $M\in\Mod(R)_{hG}$, we have
	\[
		\Hom_{\Mod(R)_{hG}}(\mathbbm{1}_{\Mod(R)_{hG}},M)\simeq p(M)^{hG}.
	\]
	Since $\Mod(R)_{hG}$ is dualizable, the canonical map $j_!\to j_*$ is an equivalence on $(\Mod(R)_{hG})^{\omega}$. More precisely, by \cref{constr:Ind_dbl}, $(\Mod(R)^{hG,\dbl})^{\omega}\simeq(\Mod(R)_{hG})^{\rig,\dbl}\simeq(\Mod(R)_{hG})^{\dbl}$. Then $j_!$ and $j_*$ agree with the canonical inclusion $(\Mod(R)_{hG})^{\omega}\subseteq(\Mod(R)_{hG})^{\dbl}$. So, for every $M\in(\Mod(R)_{hG})^{\omega}$, the canonical map induces an equivalence
	\[
		\Hom_{\Mod(R)^{hG,\dbl}}(\mathbbm{1}_{\Mod(R)^{hG,\dbl}},j_!M)\xrightarrow{\simeq}\Hom_{\Mod(R)_{hG}}(\mathbbm{1}_{\Mod(R)_{hG}},M)\simeq p(M)^{hG}.
	\]

	On the other hand, \cite[Theorem I.4.1]{NS18} provides a natural norm map $(\bbS[G]^{hG}\otimes p(-))_{hG}\to p(-)^{hG}$. For every $M\in\Mod(R)$, the $G$-equivariant spectrum $pp^L(M)$ is induced from the underlying spectrum of $M$, and hence the norm map $(\bbS[G]^{hG}\otimes pp^L(M))_{hG}\simeq(\bbS[G]^{hG}\otimes\colim_Gg(M))_{hG}\simeq(\colim_Gg(M))^{hG}\simeq(pp^L(M))^{hG}$ is an equivalence. Since $p^L$ preserves compact objects and generates $\Mod(R)_{hG}$, the norm map $(\bbS[G]^{hG}\otimes p(M))_{hG}\simeq p(M)^{hG}$ is an equivalence for every $M\in(\Mod(R)_{hG})^{\omega}$. So, the functors $(\bbS[G]^{hG}\otimes p(-))_{hG}$ and $\Hom_{\Mod(R)^{hG,\dbl}}(\mathbbm{1}_{\Mod(R)^{hG,\dbl}},j_!(-))$ both preserve colimits, and their natural maps to $p(-)^{hG}$ are equivalences on compact objects. Hence, for every $M\in\Mod(R)_{hG}$, we have
	\[
		\Hom_{\Mod(R)^{hG,\dbl}}(\mathbbm{1}_{\Mod(R)^{hG,\dbl}},j_!M)\simeq(\bbS[G]^{hG}\otimes p(M))_{hG}\to p(M)^{hG}.
	\]
	Therefore, taking cofibers we obtain 
	\[
		\Hom_{\Mod(R)^{tG,\dbl}}(\mathbbm{1}_{\Mod(R)^{tG,\dbl}},t(M))\simeq p(M)^{tG}.
	\]
	The desired result follows by taking $M=\mathbbm{1}_{\Mod(R)_{hG}}$.
\end{proof}
Since $t:\C_{hG}\to\C^{hG,\dbl}$ is lax symmetric monoidal, this proposition gives a categorical interpretation of the lax symmetric monoidal structure on Tate fixed points $(-)^{tG}:\C^{BG}\to\C$ by taking the trivial action.

We can now establish a categorical analogue of the Tate vanishing phenomenon for $L_{T(n)}\Sp$, lifting the following classical result in $\Pr^{\dbl}_{T(n)}$.
\begin{theorem}[Tate vanishing, {\cite[Theorem 1.5]{Kuhn04Tate}, \cite[Corollary 2.7]{CM17Tate}}]\label{thm:tate_vanishing_kuhn}
	Let $n\geq1$ and $X\in(L_{T(n)}\Sp)^{BG}$ with $G$ a finite group. Then $L_{T(n)}(X^{tG})\simeq0$.
\end{theorem}
Equivalently, let $X\in(L_n^f\Sp)^{BG}$. Then $X^{tG}$ is $L_{n-1}^f$-local. Indeed, applying $L_{T(n)}(-)^{tG}$ to the chromatic fracture square of $X$ and using $L_{T(n)}L_{n-1}^f\simeq0$, we obtain $L_{T(n)}X^{tG}\simeq L_{T(n)}(L_{T(n)}X)^{tG}\simeq0$. Then it follows from the chromatic fracture square of $X^{tG}$ that $X^{tG}\simeq L_n^fX^{tG}\simeq L_{n-1}^fX^{tG}$.
\begin{theorem}[Tate vanishing]\label{thm:tate_vanishing}
	Let $n\geq1$ and $\C\in(\Pr^{\dbl}_{L_n^f})^{BG}$ with $G$ a finite group. Then $\C^{tG,\dbl}$ is $L_{n-1}^f$-local and we have an equivalence
	\[
		L_{T(n)}\C^{tG,\dbl}\simeq0.
	\]
	In particular, for $\C\in(\Pr^{\dbl}_{T(n)})^{BG}$ we have $L_{T(n)}\C^{tG,\dbl}\simeq0$.
\end{theorem}
\begin{proof}
	Since $L_n^f\Sp$ is rigid, we have $(\Pr^{\dbl}_{L_n^f})^{BG}\simeq\Mod_{L_n^f\Sp}((\Pr^{\dbl}_{\st})^{BG})$. By \cref{prop:tG}, $(-)^{tG,\dbl}$ is lax symmetric monoidal. Then $(-)^{tG,\dbl}:(\Pr^{\dbl}_{L_n^f})^{BG}\to\Pr^{\dbl}_{\st}$ factors through $\Mod_{(L_n^f\Sp)^{tG,\dbl}}(\Pr^{\dbl}_{\st})$. So it suffices to show that $(L_n^f\Sp)^{tG,\dbl}$ is $L_{n-1}^f$-local. 
	
	By \cref{prop:End(1_tG)}, we have
	\[
		\End(\mathbbm{1}_{(L_n^f\Sp)^{tG,\dbl}})\simeq(L_n^f\bbS)^{tG},
	\]
	which is $L_{n-1}^f$-local by \cref{thm:tate_vanishing_kuhn}. Therefore, since $L_{n-1}^f$ is smashing, $(L_n^f\Sp)^{tG,\dbl}$ is also $L_{n-1}^f$-local, and hence $\C^{tG,\dbl}$ is $L_{n-1}^f$-local. We have $L_{T(n)}\C^{tG,\dbl}\simeq0$ by \cref{prop:M_n^mf}.

	In particular, let $\C\in(\Pr^{\dbl}_{T(n)})^{BG}$. Since $\Pr^{\dbl}_{T(n)}\simeq\Pr^{\dbl}_{L_n^f\Sp,L_{T(n)}\Sp\text{-}\tors}$ by \cref{thm:unst_completion}, $\C^{tG,\dbl}$ is $L_{n-1}^f\Sp$-local and hence $L_{T(n)}\C^{tG,\dbl}\simeq0$.
\end{proof}
This also completes the proof of \cref{thm:nuc_descent}.
\begin{corollary}
	Let $n\geq1$ and $\C\in\CAlg(\Pr^{\dbl}_{L_n^f})^{BG}$ with $G$ a finite group and $\C^{tG,\dbl}$ nonzero. Then $\operatorname{ht}(\C^{tG,\dbl})\leq n-1$.
\end{corollary}

For $\bbE_\infty$-rings, the Tate vanishing theorem admits a partial converse.
\begin{theorem}[{\cite[Proposition 4.7]{Hah22Hinfty}, \cite[Theorem 9.8]{BSY22nullstellensatz}}]
	Let $R\in\CAlg(\Sp)$ and $n\geq0$. If $L_{T(n)}R^{tC_p}\simeq0$, then $L_{T(n+1)}R\simeq0$.
\end{theorem}
\begin{corollary}
	Let $R\in\CAlg(\Sp)$ with $R^{tC_p}$ nonzero. Then we have $\operatorname{ht}(R^{tC_p})\geq\operatorname{ht}(R)-1$. Moreover, if $R\in\CAlg(L_n^f\Sp)$ with $\operatorname{ht}(R)=n\geq1$ \textup{(}e.g. $R$ is $T(n)$-local\textup{)}, then we have $\operatorname{ht}(R^{tC_p})=\operatorname{ht}(R)-1$.
\end{corollary}
Combining this with \cref{prop:End(1_tG)} yields the following blueshift theorem.
\begin{corollary}[Blueshift]\label{cor:blueshift}
	Let $\C\in\CAlg(\Pr^{\dbl}_{\st})$ with $\C^{tC_p,\dbl}$ nonzero. Then we have 
	\[
		\operatorname{ht}(\C^{tC_p,\dbl})\geq\operatorname{ht}(\C)-1.
	\]
	Moreover, if $\C\in\CAlg(\Pr^{\dbl}_{L_n^f})$ with $\operatorname{ht}(\C)=n\geq1$ \textup{(}e.g. $\C$ is $T(n)$-complete\textup{)}, then we have
	\[
		\operatorname{ht}(\C^{tC_p,\dbl})=\operatorname{ht}(\C)-1.
	\]
	In other words, $\C$ satisfies chromatic blueshift.
\end{corollary}
In particular, let $\C\in\CAlg(\Pr^{\dbl}_{\st})^{BG}$ with $G$ a finite group and $\C^{tG,\dbl}$ nonzero. If $\C$ is $L_n^f$-local for some $n$ and $\operatorname{ht}(\C)\geq1$, then $\operatorname{ht}(\C^{tG,\dbl})\leq\operatorname{ht}(\C)-1$ and, if $G=C_p$ and the action is trivial, $\operatorname{ht}(\C^{tC_p,\dbl})=\operatorname{ht}(\C)-1$. Indeed, by \cref{prop:ht(C)}, we obtain that $\C$ is $L_{\operatorname{ht}(\C)}^f$-local.

There is another blueshift phenomenon given by the $T(n)$-generic fiber.
\begin{proposition}
	Let $n\geq1$ and $\C\in\Pr^{\dbl}_{\st}$. Then $\Nuc_{T(n)}(\C)_{\eta}$ is $L_{n-1}^f$-local. 
	
	Let $\C\in\CAlg(\Pr^{\dbl}_{\st})$. If $\operatorname{ht}(\C)\geq n$, then $\operatorname{ht}(\Nuc_{T(n)}(\C)_{\eta})=n-1$. If $\operatorname{ht}(\C)<n$, then $\Nuc_{T(n)}(\C)_{\eta}\simeq0$.
\end{proposition}
\begin{proof}
	Let $\C\in\Pr^{\dbl}_{\st}$. Then we have $\Nuc_{T(n)}(\C)_{\eta}\simeq L_{n-1}^f\Nuc_{T(n)}(\C)$, and hence it is $L_{n-1}^f$-local. Let $\C\in\CAlg(\Pr^{\dbl}_{\st})$. By \cref{lem:End(1_L),prop:End(1_Nuc_T(n))}, we have 
	\[
		\End(\mathbbm{1}_{\Nuc_{T(n)}(\C)_{\eta}})\simeq L_{n-1}^f\End(\mathbbm{1}_{\Nuc_{T(n)}(\C)})\simeq L_{n-1}^fL_{T(n)}\End(\mathbbm{1}_{\C}).
	\]
	If $\operatorname{ht}(\C)\geq n$, then $L_{T(n)}\End(\mathbbm{1}_{\C})\not\simeq0$ and $\operatorname{ht}(\Nuc_{T(n)}(\C)_{\eta})=n-1$. If $\operatorname{ht}(\C)\leq n-1$, then $L_{T(n)}(\C)\simeq0$ and hence $\Nuc_{T(n)}(\C)_{\eta}\simeq\Nuc_{T(n)}(\C)\simeq0$.
\end{proof}

\begin{example}
	Let $n\geq0$. Then we have $\operatorname{ht}(\Nuc(L_{T(n)}\bbS))=\operatorname{ht}(L_{T(n)}\bbS)=n$, $\operatorname{ht}(K^{\cont}(\Nuc(L_{T(n)}\bbS)))=\operatorname{ht}(K(L_{T(n)}\bbS))=n+1$, and for $n\geq1$, $\operatorname{ht}(\Nuc(L_{T(n)}\bbS)^{tC_p,\dbl})=n-1$. Thus $\Nuc(L_{T(n)}\bbS)$ satisfies chromatic redshift, and for $n\geq1$, chromatic blueshift.
\end{example}

\subsection{Limit topology of nuclear categories}\label{sec:3.4}

We identify localization categories as limits of module categories over towers of $\bbE_1$-algebras under certain conditions, thereby generalizing the limit topology theorem of \cite{LZ25}. We then apply this to nuclear module categories and prove the continuity of continuous $K$-theory with respect to such towers. Finally, we construct the towers arising from Burklund's $\bbE_1$-structures following \cite{MW25refinedTC} and apply the resulting theorem to $T(n)$-localizations and $T(n)$-nuclear module categories. Our results also admit $\bbE_1$- and $\bbE_2$-variants, which we compare with the continuity theorem of \cite{BM26quotient} for even $\bbE_2$-quotients as in \cite{HW18}.

We begin with the following assumptions, inspired by \cite[2.20(V)]{MW25refinedTC}.
\begin{assumption}\label{asp:V_r}
	Let $\C\in\CAlg(\PrL_{\st})$ and $V\in\C$. Let
	\[
		\cdots\to V_2\to V_1\to V_0=V
	\]
	be a tower of $\bbE_1$-algebras in $\C$ such that
	\begin{enumerate}
		\item For every $r\geq0$, $V_r\in\C^{\dbl}$, and the induced map $V_r\otimes V_{r+1}\to V_r\otimes V_r$ factors through the multiplication $V_r\otimes V_{r+1}\to V_r$ as a map of $V_r\otimes V_{r+1}^{\op}$-modules.
		\item The canonical map $(-)\to\lim_r(V_r\otimes-)$ exhibits the latter as the $L_V$-localization functor.
	\end{enumerate}
\end{assumption}
\begin{theorem}[Limit topology]\label{thm:lim_topology}
	Let $\C\in\CAlg(\PrL_{\st})$ and $V\in\C$. Let $\cdots\to V_2\to V_1\to V_0=V$ be a tower of $\bbE_1$-algebras in $\C$ satisfying \cref{asp:V_r}. Then the base change functors induce canonical equivalences
	\[
		L_V\C\xrightarrow{\simeq}\lim_rL_V\Mod_{V_r}(\C)\xleftarrow{\simeq}\lim_r\Mod_{V_r}(\C).
	\]
	Consequently, let $\M\in\PrL_{\C}$. Then the base change functors also induce canonical equivalences
	\[
		L_V\M\xrightarrow{\simeq}\lim_rL_V\Mod_{V_r}(\M)\xleftarrow{\simeq}\lim_r\Mod_{V_r}(\M).
	\]
\end{theorem}
\begin{proof}
	The proof follows the same strategy as \cite[Theorem 2.2.1]{LZ25}. We have an adjunction
	\[
		F:\C\rightleftarrows\lim_r\Mod_{V_r}(\C):G
	\]
	where $F$ sends $M\in\C$ to $(V_r\otimes M)_r\in\lim_r\Mod(V_r)$ with $V_r\otimes_{V_{r+1}}(V_{r+1}\otimes M)\simeq V_r\otimes M$, and $G$ sends $(M_r)_r$ to $\lim_rM_r\in\C$ with transition maps $M_{r+1}\simeq V_{r+1}\otimes_{V_{r+1}}M_{r+1}\to V_r\otimes_{V_{r+1}}M_{r+1}\simeq M_r$.

	We first show that $G$ is fully faithful, or equivalently, the counit map is an equivalence. Let $(M_r)_r\in\lim_r\Mod(V_r)$. Since $V_r\in\C^{\dbl}$, $V_r\otimes-$ preserves limits. It suffices to show that for every $r\geq0$, the canonical map $V_r\otimes\lim_sM_s\simeq V_r\otimes\lim_{s\geq r}M_s\simeq\lim_{s\geq r}V_r\otimes M_s\simeq\lim_{s\geq r}(V_r\otimes V_s)\otimes_{V_s}M_s\to M_r$ is an equivalence. By \cref{asp:V_r}(1), for every $s\geq r$, the transition map $V_r\otimes M_{s+1}\simeq(V_r\otimes V_{s+1})\otimes_{V_{s+1}}M_{s+1}\to(V_r\otimes V_s)\otimes_{V_{s+1}}M_{s+1}\simeq V_r\otimes M_s$ factors through the multiplication map $(V_r\otimes V_{s+1})\otimes_{V_{s+1}}M_{s+1}\to V_r\otimes_{V_{s+1}}M_{s+1}$. Then we obtain a factorization for every $s\geq r$
	\[\begin{tikzcd}
		{(V_r\otimes V_{s+1})\otimes_{V_{s+1}}M_{s+1}} & {(V_r\otimes V_s)\otimes_{V_{s+1}}M_{s+1}} \\
		{V_r\otimes_{V_{s+1}}M_{s+1}} & {V_r\otimes_{V_{s+1}}M_{s+1}}
		\arrow[from=1-1, to=1-2]
		\arrow[from=1-1, to=2-1]
		\arrow[from=1-2, to=2-2]
		\arrow[dashed, from=2-1, to=1-2]
		\arrow[equals, from=2-1, to=2-2]
	\end{tikzcd}\]
	It follows that $(V_r\otimes M_s)_{s\geq r}$ is pro-constant with value $M_r$. Therefore, the counit map is an equivalence $\lim_{s\geq r}V_r\otimes M_s\xrightarrow{\simeq}M_r$, and hence $G$ is fully faithful under \cref{asp:V_r}(1).

	By \cref{asp:V_r}(2), for every $r\geq0$, we have $L_V(V_r\otimes-)\simeq\lim_{s\geq r}(V_r\otimes V_s\otimes-)\simeq V_r\otimes-$. Then the $L_V$-localization functor induces an equivalence $L_V:\Mod_{V_r}(\C)\xrightarrow{\simeq}L_V\Mod_{V_r}(\C)$ for every $r\geq0$. This proves the second equivalence.

	Now we replace $\C$ by $L_V\C$ in the adjunction and show that it forms a pair of inverse equivalences. By \cref{lem:LC=LC(Mod(L1))}, we have $\Mod_{V_r}(L_V\C)\simeq L_V\Mod_{V_r}(\C)\simeq\Mod_{V_r}(\C)$. Then the unit map is given by $(-)\to\lim_r(V_r\otimes-)$, which is an equivalence on $L_V\C$ by \cref{asp:V_r}(2). Hence the base change functors induce the first equivalence.

	Consequently, let $\M\in\PrL_{\C}$. By \cref{cor:L_KM}, we have $L_V\M\simeq L_V\C\otimes_{\C}\M$ and $\Mod_{V_r}(L_V\M)\simeq L_V\Mod_{V_r}(\C)\otimes_{\C}\M\simeq\Mod_{V_r}(\M)$. We show that the $L_V$-localization functor of $\M$ is given by $\lim_r(V_r\otimes-)$. Since $(V\otimes V_r)_{r\geq0}$ is pro-constant with value $V$, we have $V\otimes-\simeq\lim_r(V\otimes V_r\otimes-)\simeq V\otimes\lim_r(V_r\otimes-)$. Then $(-)\to\lim_r(V_r\otimes-)$ is an $L_V$-equivalence. Note that for any $X\in\M$, $\iHom_{\C}(X,\lim_r(V_r\otimes-))\simeq\lim_r(V_r\otimes\iHom_{\C}(X,-))\simeq L_V\iHom_{\C}(X,-)$ is $L_V$-local. Then $\lim_r(V_r\otimes-)$ is $L_V$-local, and hence is the $L_V$-localization functor of $\M$. Therefore, the same argument applies, yielding the desired equivalences.
\end{proof}
\begin{remark}
	The proof above adapts the argument of Li--Zhang in \cite[Theorem 2.2.1]{LZ25}. In their work, Li--Zhang consider a tower of type $n$ generalized Moore spectra $\cdots\to V_2\to V_1\to V_0$ and prove the analogous statement for $L_{K(n)}\Mod(R)$ with $R\in\CAlg(L_{K(n)}\Sp)$. A key input is the splitting $V_r\otimes V_s\simeq V_r\oplus\bigoplus_iV_r[d_{s,i}]$, where $d_{s,i}>0$ and the transition map is the identity on the $V_r$-summand and nilpotent on $V_r[d_{s,i}]$-summands. After passing to a suitable cofinal subtower satisfying \cref{asp:V_r}, one can reduce the proof to the present situation. We will explain this reduction in \cref{exm:LZ}.
	
	Under \cref{asp:V_r}(1), the proof also identifies $\lim_r\Mod_{V_r}(\C)$ with the full subcategory of $\C$ spanned by objects $M\in\C$ for which $M\to\lim_r(V_r\otimes M)$ is an equivalence. This identification holds in the more general setting of \cite[Theorem 2.2.2]{BM26quotient}. We will revisit it in \cref{rmk:E_1/E_2_limit_topology}.
\end{remark}
Observe that in \cite{MW25refinedTC}, Meyer--Wagner replace \cref{asp:V_r}(2) with the assumption that each $V_r$ is contained in the thick tensor ideal generated by $V$. However, the only property of this assumption used in the subsequent arguments is that each $V_r$ belongs to $\C^{V\text{-}\tors}$. Therefore, their results carry over in our setting, yielding the following proposition.
\begin{proposition}\label{prop:MW_traceclass}
	Let $\C\in\CAlg(\PrL_{\st})$ and $V\in\C$. Let $\cdots\to V_2\to V_1\to V_0=V$ be a tower of $\bbE_1$-algebras in $\C$ satisfying \cref{asp:V_r}(1), and assume that for every $r\geq0$, $V_r\in\C^{V\text{-}\tors}$. Then \cref{asp:V_r}(2) is automatically satisfied. Moreover, the base change functors 
	\[
		V_r\otimes_{V_{r+1}}-:\Mod_{V_{r+1}}(\C)\to\Mod_{V_r}(\C)
	\]
	are trace-class in $\Pr^{\dbl}_{\C}$, and induce a canonical equivalence
	\[
		\colim_r\left(\Ind(\Mod_{V_r^{\op}}(\C^{\dbl}))\otimes_{\Ind(\C^{\dbl})}\C\right)\xrightarrow{\simeq}\C^{V\text{-}\tors}.
	\]
	Consequently, we have a pro-equivalence
	\[
		\left(\Mod_{V_r}(\C)\right)_{r\geq0}\xrightarrow{\simeq}\left(\iHom^{\cg}_{\Ind(\C^{\dbl})}(\Ind(\Mod_{V_r^{\op}}(\C^{\dbl})),\Ind(\C^{\dbl}))\otimes_{\Ind(\C^{\dbl})}\C\right)_{r\geq0}
	\]
	whose limit is given by $L_V\C$.
\end{proposition}
\begin{proof}
	As shown in the proof of \cref{thm:lim_topology}, we have $V\otimes-\simeq\lim_r(V\otimes V_r\otimes-)\simeq V\otimes\lim_r(V_r\otimes-)$. So it suffices to show that $\lim_r(V_r\otimes-)$ is $L_V$-local. Since $V_r\in\C^{V\text{-}\tors}$, we have $V_r\otimes-\simeq V_r\otimes L_V(-)$, and hence $V_r\otimes-$ is $L_V$-local as $V_r\in\C^{\dbl}$. Then $\lim_r(V_r\otimes-)$ is also $L_V$-local, and hence $(-)\to\lim_r(V_r\otimes-)$ exhibits the latter as the $L_V$-localization functor.

	We first assume that $\C\in\CAlg^{\rig}(\PrL_{\st})^{\cg}$ and use only \cref{asp:V_r}(1). After proving the trace-class and pro-equivalence assertions in this setting, we will base change along $\Ind(\C^{\dbl})\to\C$ for a general $\C$. It follows from \cite[Lemma 2.23]{MW25refinedTC} that $V_r\otimes_{V_{r+1}}-:\Mod_{V_{r+1}}(\C)\to\Mod_{V_r}(\C)$ is trace-class in $\Pr^{\cg}_{\C}$. The predual category of $\Mod_{V_{r+1}}(\C)$ in $\Pr^{\cg}_{\C}$ is given by
	\[
		\iHom^{\cg}_{\C}(\Mod_{V_{r+1}}(\C),\C)\simeq\Ind(\Fun_{\C^\omega}(\Mod_{V_{r+1}}(\C)^{\omega},\C^{\omega}))\simeq\Ind(\Mod_{V_{r+1}^{\op}}(\C^{\omega})),
	\]
	since $\Fun_{\C^{\omega}}(\Mod_{V_{r+1}}(\C)^{\omega},\C^{\omega})\subseteq\FunL_{\C}(\Mod_{V_{r+1}}(\C),\C)\simeq\Mod_{V_{r+1}^{\op}}(\C)$ is identified with $\Mod_{V_{r+1}^{\op}}(\C^{\omega})$. Therefore, the dual functors
	\[
		\Ind(\Mod_{V_r^{\op}}(\C^{\omega}))\to\Ind(\Mod_{V_{r+1}^{\op}}(\C^{\omega})),
	\]
	which are induced by the forgetful functors $\Mod_{V_r^{\op}}(\C^{\omega})\to\Mod_{V_{r+1}^{\op}}(\C^{\omega})$, are trace-class in $\Pr^{\cg}_{\C}$. Each forgetful functor $\Mod_{V_r^{\op}}(\C^\omega)\to\C^{\omega}$ takes values in $\C^{\{V_r\}_{r\geq0}\text{-}\tors}$. The argument in the proof of \cite[Lemma 2.25]{MW25refinedTC} shows that the dual functors induce a $\C$-linear equivalence
	\[
		\colim_r\Ind(\Mod_{V_r^{\op}}(\C^\omega))\xrightarrow{\simeq}\C^{\{V_r\}_{r\geq0}\text{-}\tors}.
	\]
	Consequently, since the base change functors are trace-class in $\Pr^{\cg}_{\C}$, they also induce a pro-equivalence
	\[
		\left(\Mod_{V_r}(\C)\right)_{r\geq0}\xrightarrow{\simeq}\left(\iHom^{\cg}_{\C}(\Ind(\Mod_{V_r^{\op}}(\C^{\omega})),\C)\right)_{r\geq0}.
	\]

	For the general case, base change along the $\Sp$-algebra map $\Ind(\C^{\dbl})\to\C$. It follows that both $V_r\otimes_{V_{r+1}}-:\Mod_{V_{r+1}}(\C)\to\Mod_{V_r}(\C)$ and $\Ind(\Mod_{V_r^{\op}}(\C^{\dbl}))\otimes_{\Ind(\C^{\dbl})}\C\to\Ind(\Mod_{V_{r+1}^{\op}}(\C^{\dbl}))\otimes_{\Ind(\C^{\dbl})}\C$ are trace-class in $\Pr^{\dbl}_{\C}$. By \cref{cor:L_KM}, $\Ind(\C^{\dbl})^{\{V_r\}_{r\geq0}\text{-}\tors}\otimes_{\Ind(\C^{\dbl})}\C\simeq\C^{\{V_r\}_{r\geq0}\text{-}\tors}\simeq\C^{V\text{-}\tors}$ as $V_0=V$ and $V_r\in\C^{V\text{-}\tors}$ for all $r\geq0$. Since we have shown that \cref{asp:V_r} holds, we obtain the desired pro-equivalence with limit $L_V\C$.
\end{proof}
This proposition provides a method for computing refined localizing invariants of $L_V\C$, such as the refined topological Hochschild homology mentioned in \cref{exm:R^BS^1}, whenever $\C$ is a proper and smooth rigid $\Sp$-algebra. In particular, the resulting pro-equivalence generalizes \cite[Proposition 5.25]{Efi25limit} as shown in \cref{exm:Nuc(R_I)} for discrete commutative rings.

To apply Efimov's continuity theorem, it remains to verify the strongly Mittag-Leffler conditions introduced by Efimov \cite{Efi25limit} for these towers.
\begin{definition}[{\cite[Definition 5.1]{Efi25limit}}]
	Let $\C\in\CAlg^{\rig}(\PrL_{\st})$. Let $\cdots\to\C_2\to\C_1\to\C_0$ be a tower in $\Pr^{\dbl}_{\C}$. We denote by $f_{mn}:\C_m\to\C_n$ the transition maps for $m\geq n\geq0$. We say that $(\C_r)_{r\geq0}$ satisfies the \emph{strongly Mittag-Leffler condition over $\C$} if
	\begin{enumerate}
		\item For every $n\geq0$, $(f_{mn}f^R_{mn})_{m\geq n}:\C_n\to\C_n$ is pro-constant.
		\item For every $n,k\geq0$, the functor $\lim_{m\geq n,k}f_{mk}f_{mn}^R:\C_n\to\C_k$ is an $\Sp$-internal left adjoint and admits a left adjoint.
	\end{enumerate}
\end{definition}
We simply say strongly Mittag-Leffler condition for strongly Mittag-Leffler condition over $\Sp$. In fact, we may verify the following stronger condition recently introduced by Ben-Moshe \cite{BM26quotient}.
\begin{definition}[{\cite[Definition 2.1.4]{BM26quotient}}]
	Let $\C\to\cdots\to\C_2\to\C_1\to\C_0$ be a tower in $\Pr^{\dbl}_{\st}$. We denote by $f_n:\C\to\C_n$ and $f_{mn}:\C_m\to\C_n$ the transition maps for $m\geq n\geq0$. We say that $\C\to(\C_r)_{r\geq0}$ satisfies the \emph{strongly Mittag-Leffler condition} if
	\begin{enumerate}
		\item For every $n\geq0$, $f_nf_n^R\to(f_{mn}f_{mn}^R)_{m\geq n}$ is a pro-equivalence.
		\item For every $n\geq0$, $f_n^R$ is an $\Sp$-internal left adjoint and $f_n$ admits a left adjoint.
	\end{enumerate}
\end{definition}
In particular, the strongly Mittag-Leffler condition of $\C\to\cdots\to\C_2\to\C_1\to\C_0$ implies the strongly Mittag-Leffler condition of $\cdots\to\C_2\to\C_1\to\C_0$.
\begin{proposition}\label{prop:ML}
	Let $\C\in\CAlg^{\rig}(\PrL_{\st})$ and $V\in\C$. Let $\cdots\to V_2\to V_1\to V_0=V$ be a tower of $\bbE_1$-algebras in $\C$ satisfying \cref{asp:V_r}(1). Then the tower
	\[
		\C\to\cdots\to\Mod_{V_2}(\C)\to\Mod_{V_1}(\C)\to\Mod_{V_0}(\C)
	\]
	satisfies the strongly Mittag-Leffler condition. Consequently, let $\M\in\Pr^{\dbl}_{\C}$. Then the tower
	\[
		\M\to\cdots\to\Mod_{V_2}(\M)\to\Mod_{V_1}(\M)\to\Mod_{V_0}(\M)
	\]
	also satisfies the strongly Mittag-Leffler condition.
\end{proposition}
\begin{proof}
	Let $n\geq0$. Since $V_n\in\C^{\dbl}$, $f_n\simeq V_n\otimes-$ admits a left adjoint, $f_n^{RR}\simeq\iHom_{\C}(V_n,-)$ admits a right adjoint, and hence $f_n^R$ is an $\Sp$-internal left adjoint. The unit map is given by $f_nf_n^R\simeq V_n\otimes-\simeq(V_n\otimes V_n)\otimes_{V_n}(-)\to(V_n\otimes_{V_m}V_n)\otimes_{V_n}(-)\simeq V_n\otimes_{V_m}(-)\simeq f_{mn}f^R_{mn}$ for $m\geq n$. Then $(f_{mn}f^R_{mn})_{m\geq n}$ is pro-constant with value $f_nf_n^R$ since $(V_n\otimes_{V_m}V_n)_{m\geq n}$ is pro-constant with value $V_n\otimes V_n$. This comes from \cite[Lemma 2.25]{MW25refinedTC}. Indeed, viewing $V_n\otimes V_n$ as a $V_m^{\op}\otimes V_m$-module, it forgets to a $V_{m+1}^{\op}\otimes V_{m+1}$-module. Then $\id:V_n\otimes V_n\to V_n\otimes V_n$ factors through $(V_m^{\op}\otimes V_m)\otimes_{V_{m+1}^{\op}\otimes V_{m+1}}(V_n\otimes V_n)\to V_n\otimes V_n$. By \cref{asp:V_r}(1), $V_{m+1}\otimes V_{m+1}\to V_m\otimes V_m$ factors through $V_{m+1}\otimes V_{m+1}\to V_{m+1}$ as a map of $V_{m+1}\otimes V_{m+1}^{\op}$-modules. Then $V_n\otimes V_n\to(V_m^{\op}\otimes V_m)\otimes_{V_{m+1}^{\op}\otimes V_{m+1}}(V_n\otimes V_n)$ factors through $V_{m+1}\otimes_{V_{m+1}^{\op}\otimes V_{m+1}}(V_n\otimes V_n)\simeq V_n\otimes_{V_{m+1}}V_n$. Therefore, we obtain a factorization for every $m>n$
	\[\begin{tikzcd}
		{V_n\otimes_{V_{m+1}}V_n} & {V_n\otimes_{V_m}V_n} \\
		{V_n\otimes V_n} & {V_n\otimes V_n}
		\arrow[from=1-1, to=1-2]
		\arrow[from=1-1, to=2-1]
		\arrow[from=1-2, to=2-2]
		\arrow[dashed, from=2-1, to=1-2]
		\arrow[equals, from=2-1, to=2-2]
	\end{tikzcd}\]
	It follows that $(V_n\otimes_{V_{m+1}}V_n)_{m\geq n}$ is pro-constant with value $V_n\otimes V_n$. More precisely,
	\[
		V_n\otimes V_n\to(V_n\otimes_{V_m}V_n)_{m\geq n}
	\]
	is a pro-equivalence.

	Therefore, $\C\to(\Mod_{V_r}(\C))_{r\geq0}$ satisfies the strongly Mittag-Leffler condition. In fact, the proof shows that it satisfies the strongly Mittag-Leffler condition over $\C$. Consequently, let $\M\in\Pr^{\dbl}_{\C}$. Applying $-\otimes_{\C}\M$, \cite[Proposition 2.1.6]{BM26quotient} shows that $\M\simeq\C\otimes_{\C}\M\to(\Mod_{V_r}(\C)\otimes_{\C}\M)_{r\geq0}\simeq(\Mod_{V_r}(\M))_{r\geq0}$ also satisfies the strongly Mittag-Leffler condition.
\end{proof}
\begin{remark}\label{rmk:E_1-rig}
	In this remark, we discuss $\bbE_1$-rigidifications and limits of rigid $\bbE_1$-monoidal stable categories, which will be used to identify the rigidification of the limit category with the dualizable limit.
	
	The $\bbE_1$-rigidification of a presentably $\bbE_1$-monoidal stable category $\C\in\Alg(\PrL_{\st})$ is given by
	\[
		\C^{\bbE_1\text{-}\rig}\simeq\Ind_{\I_{\mathrm{sd}}}(\C),
	\]
	where $\I$ denotes the idealoid of both left and right trace-class maps. Equivalently, $\C^{\bbE_1\text{-}\rig}\subseteq\Ind(\C)$ is the localizing stable subcategory generated by objects of the form $\colim_{\bbQ_{\geq0}}j(C_p)$, where $C_p\to C_q$ are both left and right trace-class for all $p<q$. Indeed, the proof of \cite[Theorem 4.4.16]{KNP24} or \cite[Corollary 7.20]{Aok25} adapts to this setting. It follows that $\C^{\bbE_1\text{-}\rig}$ is dualizable, every compact map is both left and right trace-class, and hence the multiplication functor is a bilinear internal left adjoint. The universal property follows in the same way.

	Moreover, the forgetful functor $\Alg^{\rig}(\PrL_{\st})\to\Pr^{\dbl}_{\st}$ preserves limits. Indeed, the proof of \cite[Corollary 4.89]{Ram26locallyrigid} carries over after replacing ``hom" with ``both left and right homs".
\end{remark}
Combining the strongly Mittag-Leffler condition with the description of $\bbE_1$-rigidifications, we obtain the following proposition.
\begin{proposition}\label{prop:E_1-rig}
	Let $\cdots\to\C_2\to\C_1\to\C_0$ be a tower of rigid $\bbE_1$-monoidal stable categories satisfying the strongly Mittag-Leffler condition. Suppose that $\C_r$ is a rigid $\bbE_{r+1}$-monoidal stable category and $\C_{r+1}\to\C_r$ is an $\bbE_{r+1}$-monoidal functor for every $r\geq1$. Then $\lim_r\C_r$ is a locally rigid $\Sp$-algebra and the canonical map induces an equivalence of rigid $\Sp$-algebras
	\[
		(\lim_r\C_r)^{\rig}\xrightarrow{\simeq}\lim^{\dbl}_r\C_r.
	\]
\end{proposition}
\begin{proof}
	By \cite[Corollary 1.3.5]{LZ25}, $\lim_r\C_r$ admits a natural symmetric monoidal structure. It follows from \cref{rmk:E_1-rig} that $(\lim_r\C_r)^{\rig}\simeq(\lim_r\C_r)^{\bbE_1\text{-}\rig}\simeq\lim^{\dbl}_r\C_r$.
	
	By \cite[Proposition 5.5]{Efi25limit}, $\lim_r\C_r$ is dualizable and each projection map $\lim_r\C_r\to\C_r$ is an $\Sp$-internal left adjoint. Then the projection maps induce a fully faithful functor $\lim_r\C_r\hookrightarrow\lim^{\dbl}_r\C_r$, which is the left adjoint to the canonical functor $\lim^{\dbl}_r\C_r\to\lim_r\C_r$ by \cite[Remark 5.15]{Efi25limit}. Therefore, $\lim_r\C_r$ is locally rigid with rigidification $\lim^{\dbl}_r\C_r$.
\end{proof}
Efimov proves a continuity theorem for continuous $K$-theory of towers satisfying the strongly Mittag-Leffler condition, which we will use to establish our continuity theorem.
\begin{proposition}[{\cite[Corollary 6.2]{Efi25limit}}]
	Let $\cdots\to\C_2\to\C_1\to\C_0$ be a tower in $\Pr^{\dbl}_{\st}$ satisfying the strongly Mittag-Leffler condition. Then the canonical maps induce an equivalence
	\[
		K^{\cont}(\lim^{\dbl}_r\C_r)\xrightarrow{\simeq}\lim_rK^{\cont}(\C_r).
	\]
\end{proposition}
\begin{theorem}[Continuity]\label{thm:continuity_K}
	Let $\C\in\CAlg^{\rig}(\PrL_{\st})$ and $V\in\C$. Let $\cdots\to V_2\to V_1\to V_0=V$ be a tower of $\bbE_1$-algebras in $\C$ satisfying \cref{asp:V_r}. Suppose that $V_r$ is an $\bbE_{r+1}$-algebra and $V_{r+1}\to V_r$ is an $\bbE_{r+1}$-algebra map for every $r\geq1$. Then, for every $\D\in\CAlg^{\rig}(\PrL_{\C})$, the tower $\D\to(\Mod_{V_r}(\D))_{r\geq0}$ satisfies the strongly Mittag-Leffler condition. Moreover, the base change functors induce equivalences of $\D$-algebras
	\[
		L_V\D\simeq\lim_r\Mod_{V_r}(\D),\quad\Nuc_V(\D)\simeq\lim^{\dbl}_r\Mod_{V_r}(\D),
	\]
	and the canonical maps induce an equivalence
	\begin{equation}\label{eq:continuity}
		K^{\cont}(\Nuc_V(\D))\xrightarrow{\simeq}\lim_rK^{\cont}(\Mod_{V_r}(\D)).
	\end{equation}
\end{theorem}
\begin{proof}
	By \cref{prop:ML}, the tower $\D\to(\Mod_{V_r}(\D))_{r\geq0}$ satisfies the strongly Mittag-Leffler condition. By \cref{thm:lim_topology}, the base change functors induce a $\D$-algebra equivalence $L_V\D\simeq\lim_r\Mod_{V_r}(\D)$. $L_V\C$ is locally rigid over $\C$, and $L_V\mathbbm{1}_{\C}\in(L_V\C)^{\omega_1}$ since $V\in\C^{\dbl}$ is compact. By \cref{thm:unst_completion}, we have $\Nuc_V(\D)\simeq\iHom^{\dbl}_{\C}(L_V\C,\D)\simeq(L_V\D)^{\rig}$. Then \cref{prop:E_1-rig} gives $\D$-algebra equivalences
	\[
		\Nuc_V(\D)\simeq(L_V\D)^{\rig}\simeq(\lim_r\Mod_{V_r}(\D))^{\rig}\simeq\lim_r^{\dbl}\Mod_{V_r}(\D).
	\]
	Therefore, the desired continuity follows from Efimov's continuity theorem.
\end{proof}
In particular, this theorem recovers \cite[Proposition 5.23, Corollary 6.3]{Efi25limit}, as shown in \cref{exm:Nuc(R_I)}.

\begin{remark}\label{rmk:continuity}
	To obtain only the equivalence \eqref{eq:continuity}, it suffices to assume that $V_r$ is an $\bbE_2$-algebra and $V_{r+1}\to V_r$ is an $\bbE_2$-algebra map for every $r\geq1$. Indeed, $\Mod_{V_r}(\D)$ is a tower of rigid $\bbE_1$-monoidal stable categories and the equivalence $L_V\D\simeq\lim_r\Mod_{V_r}(\D)$ is $\bbE_1$-monoidal. Hence \cref{rmk:E_1-rig} gives equivalences $\Nuc_V(\D)\simeq(L_V\D)^{\rig}\simeq(L_V\D)^{\bbE_1\text{-}\rig}\simeq(\lim_r\Mod_{V_r}(\D))^{\bbE_1\text{-}\rig}\simeq\lim^{\dbl}_r\Mod_{V_r}(\D)$.
\end{remark}

Finally, following \cite[Subsection 2.5]{MW25refinedTC}, we construct our main examples satisfying \cref{asp:V_r}, to which \cref{thm:continuity_K} applies. Burklund develops an obstruction theory for constructing $\bbE_1$-quotients and uses a deformation of the ambient category to establish the required obstruction vanishing \cite{Bur22}. Meyer--Wagner then construct a cofinal subtower satisfying \cref{asp:V_r} in \cite{MW25refinedTC}. We recall their results in the following form.
\begin{theorem}[{\cite[Theorem 1.5]{Bur22}, \cite[Corollary 2.30]{MW25refinedTC}}]
	Let $\C\in\CAlg(\PrL_{\st})$ and $v:I\to\mathbbm{1}$ such that $\mathbbm{1}/v$ admits a right unital multiplication. Then there exists a tower of $\bbE_1$-algebras
	\[
		\cdots\to\mathbbm{1}/v^4\to\mathbbm{1}/v^3\to\mathbbm{1}/v^2
	\]
	such that for every $n\geq2$, $\mathbbm{1}/v^n$ is an $\bbE_{n-1}$-algebra and $\mathbbm{1}/v^{n+1}\to\mathbbm{1}/v^n$ is an $\bbE_{n-1}$-algebra map.
	
	Suppose moreover that $\C\in\CAlg^{\rig}(\PrL_{\st})$ and $I\in\C^{\dbl}$. Then for every $k\geq3$ and $s\geq 2k+3$, $\mathbbm{1}/v^k\otimes\mathbbm{1}/v^s\to\mathbbm{1}/v^k\otimes\mathbbm{1}/v^k$ factors through $\mathbbm{1}/v^k$ as a map of $\mathbbm{1}/v^k\otimes(\mathbbm{1}/v^s)^{\op}$-modules. Therefore, there exists a cofinal subtower
	\[
		\cdots\to\mathbbm{1}/v^{k_2}\to\mathbbm{1}/v^{k_1}\to\mathbbm{1}/v^{k_0}
	\]
	satisfying \cref{asp:V_r}, and for every $r\geq0$, $\mathbbm{1}/v^{k_r}\in\C^{\mathbbm{1}/v\text{-}\tors}$.
\end{theorem}
\begin{construction}
	Let $\C\in\CAlg^{\rig}(\PrL_{\st})$ and $v_i:I_i\to\mathbbm{1}$ for $0\leq i\leq n$, such that each $I_i\in\C^{\dbl}$ and $\mathbbm{1}/v_i$ admits a right unital multiplication. We denote $V:=\mathbbm{1}/(v_0,\cdots,v_n)$. Then there exists a cofinal tower $(k_{0,r},\cdots,k_{n,r})_{r\geq0}$ such that the tower of $\bbE_1$-algebras
	\[
		V_r:=\mathbbm{1}/(v_0^{k_{0,r}},\cdots,v_n^{k_{n,r}})
	\]
	satisfies \cref{asp:V_r}, and for every $r\geq1$, $V_r\in\C^{V\text{-}\tors}$ and is an $\bbE_{r+1}$-algebra, and $V_{r+1}\to V_r$ is an $\bbE_{r+1}$-algebra map. Here $V_0$ need not equal the unpowered quotient $V$. Applying \cref{thm:continuity_K}, we obtain $\C$-algebra equivalences
	\[
		L_{\mathbbm{1}/(v_0,\cdots,v_n)}(\C)\simeq\lim_r\Mod_{\mathbbm{1}/(v_0^{k_{0,r}},\cdots,v_n^{k_{n,r}})}(\C),\quad
		\Nuc_{\mathbbm{1}/(v_0,\cdots,v_n)}(\C)\simeq\lim^{\dbl}_r\Mod_{\mathbbm{1}/(v_0^{k_{0,r}},\cdots,v_n^{k_{n,r}})}(\C),
	\]
	together with an equivalence of continuous $K$-theory spectra
	\[
		K^{\cont}(\Nuc_{\mathbbm{1}/(v_0,\cdots,v_n)}(\C))\simeq\lim_rK^{\cont}(\Mod_{\mathbbm{1}/(v_0^{k_{0,r}},\cdots,v_n^{k_{n,r}})}(\C)).
	\]
\end{construction}
\begin{example}\label{exm:Nuc(R_I)}
	In this example, we study even quotients of $\bbE_\infty$-rings. Let $R\in\CAlg(\Sp)$ and $v\in\pi_{2*}(R)$ be an even element. Then by \cite[Remark 5.5]{Bur22}, $R/v^2$ admits a unital multiplication. In general, let $I=(v_0,\cdots,v_n)\subseteq\pi_*(R)$ be an evenly finitely generated homogeneous ideal. Choosing a cofinal sequence $(k_{0,r},\cdots,k_{n,r})_{r\geq0}$ as above, we obtain equivalences of $\Sp$-algebras
	\[
		\Mod(R)_{I\text{-}\cplt}\simeq\lim_r\Mod(R/(v_0^{k_{0,r}},\cdots,v_n^{k_{n,r}})),\quad
		\Nuc(R^{\wedge}_I)\simeq\lim^{\dbl}_r\Mod(R/(v_0^{k_{0,r}},\cdots,v_n^{k_{n,r}})).
	\]
	Therefore, we have an equivalence
	\[
		K^{\cont}(\Nuc(R^{\wedge}_I))\simeq\lim_rK(R/(v_0^{k_{0,r}},\cdots,v_n^{k_{n,r}})).
	\]

	This recovers Efimov's result for discrete commutative rings with finitely generated ideals \cite{Efi25limit}. In this case, $R/(v_0^{k_0},\cdots,v_n^{k_n})\simeq R\otimes_{\bbZ[x_0,\cdots,x_n]}\bbZ$ naturally admits an $\bbE_\infty$-ring structure, where $x_i$ maps to $v_i^{k_i}$ in $R$ and to $0$ in $\bbZ$. More generally, the above equivalences apply to any adic $\bbE_\infty$-ring $R$ with a finitely generated ideal of definition $I\subseteq\pi_0(R)$.
\end{example}
\begin{remark}\label{rmk:E_1/E_2_limit_topology}
	In this remark we discuss the $\bbE_1$- and $\bbE_2$-monoidal variants of our results. In \cref{asp:V_r}, the assumption $\C\in\CAlg(\PrL_{\st})$ can be weakened to $\C\in\Alg(\PrL_{\st})$, provided that each $V_r$ is both left and right dualizable and that the factorizations are formulated in terms of bimodules. With this modification, \cref{thm:lim_topology,prop:ML} and the continuity statement for the dualizable limit in \cref{thm:continuity_K} remain valid for the underlying $\Sp$-modules. In particular, \cref{prop:ML} allows us to deduce \cref{thm:lim_topology} directly from \cite[Theorem 2.2.2]{BM26quotient}.
	
	If $\C\in\Alg_{\bbE_2}(\PrL_{\st})$, then $L_V\C$ naturally inherits an $\bbE_2$-monoidal structure over $\C$. If, moreover, $V_r$ is an $\bbE_2$-algebra and $V_{r+1}\to V_r$ is an $\bbE_2$-algebra map for every $r\geq1$, then for every rigid $\bbE_2$-monoidal stable category $\D$ over $\C$, we obtain an $\bbE_1$-monoidal equivalence $L_V\D\simeq\lim_r\Mod_{V_r}(\D)$. Hence \cref{rmk:E_1-rig} gives equivalences $(L_V\D)^{\bbE_1\text{-}\rig}\simeq(\lim_r\Mod_{V_r}(\D))^{\bbE_1\text{-}\rig}\simeq\lim^{\dbl}_r\Mod_{V_r}(\D)$.

	Burklund's obstruction theory for $\bbE_1$-quotients \cite[\Section2]{Bur22} only requires $\C\in\Alg(\PrL_{\st})$, whereas the tower construction of Burklund and Meyer--Wagner recalled above applies when $\C\in\Alg_{\bbE_2}(\PrL_{\st})$. In particular, in \cref{exm:Nuc(R_I)} we may also consider even quotients of $\bbE_3$-rings.
\end{remark}
Using this remark, we show that the main example in \cite{BM26quotient} satisfies our \cref{asp:V_r}.
\begin{example}\label{exm:E_2_quotient}
	Let $d\in\bbZ$. The free graded $\bbE_1$-ring $\bbS[x]:=\Free_{\bbE_1}(\bbS(1)[2d])$ has a canonical graded $\bbE_2$-ring structure by \cite{Rot}. The weight truncations $\bbS[x]/x^i:=w_{<i}(\bbS[x])$ and their transition maps inherit graded $\bbE_2$-ring structures. Thus, $(\bbS[x]/x^{i+1})_{i\geq0}$ forms a tower of $\bbE_1$-algebras in the rigid $\bbE_1$-monoidal stable category $\Mod_{\bbS[x]}(\Gr(\Sp))$. We show that this tower admits a cofinal subtower satisfying the $\bbE_1$-variant of \cref{asp:V_r}.

	By \cite[Proposition 3.1.13--3.1.15]{BM26quotient}, $\bbS[x]/x^i$ is the cofiber of $x^i:\bbS[x](i)[2di]\to\bbS[x]$ with any choice of $\bbS[x]\otimes\bbS[x]^{\op}$-module structure on this map. Moreover, it is both left and right dualizable and lies in the thick ideal of $\Mod_{\bbS[x]}(\Gr(\Sp))$ generated by $\bbS[x]/x\simeq\bbS$.
	
	For the factorization condition, we apply the strategy of \cite[Proposition 2.27, Corollary 2.30]{MW25refinedTC} to $x:\bbS[x](1)[2d]\to\bbS[x]$. Just as Burklund deforms the ambient category to a larger one to make the obstructions vanish, the argument for even $\bbE_2$-quotients views $\bbS[x]$ as a graded $\bbE_2$-ring in $\Gr(\Sp)$, where weight considerations force the required vanishing.
	
	Let $k\geq j\geq i\geq1$. Since $\bbS[x]/x^i$ is an $\bbE_1$-$\bbS[x]$-algebra, we obtain an equivalence $\bbS[x]/x^i\otimes_{\bbS[x]}\bbS[x]/x^j\simeq\bbS[x]/x^i\oplus(\bbS[x](1)[2d])^{\otimes_{\bbS[x]}j}/x^i[1]$ of $\bbS[x]/x^i$-modules. We can also identify $\bbS[x]/x^i$-modules $(\bbS[x](1)[2d])^{\otimes_{\bbS[x]}j}/x^i\simeq(\bbS[x]/x^i(1)[2d])^{\otimes_{\bbS[x]/x^i}j}$. For $l\geq2$, the mapping spectra occurring in \cite[Proposition 2.4]{Bur22} satisfy
	\begin{align*}
		\Hom_{\Mod_{\bbS[x]/x^i}(\Gr(\Sp))}&(((\bbS[x](1)[2d]/x^i)^{\otimes_{\bbS[x]/x^i}k}[2])^{\otimes_{\bbS[x]/x^i}l}[-3],\bbS[x]/x^i\otimes_{\bbS[x]}\bbS[x]/x^j)\\
		&\simeq\Hom_{\Mod_{\bbS[x]}(\Gr(\Sp))}(\bbS[x](kl)[2dkl+2l-3],\bbS[x]/x^i\oplus(\bbS[x](1)[2d])^{\otimes_{\bbS[x]}j}/x^i[1])\\
		&\simeq\Hom_{\Mod_{\bbS[x]}(\Gr(\Sp))}(\bbS[x](kl)[2dkl+2l-3],\bbS[x]/x^i\oplus\bbS[x]/x^i(j)[2dj+1])\simeq0.
	\end{align*}
	Indeed, $\bbS[x](kl)[2dkl+2l-3]$ is concentrated in weight $\geq kl$, whereas $\bbS[x]/x^i\oplus\bbS[x]/x^i(j)[2dj+1]$ is concentrated in weight $[0,i)\cup[j,j+i)$, and $j+i\leq kl$. In the case $k=j$, \cite[Proposition 2.4, Remark 2.5]{Bur22} shows that the $\bbE_1$-$\bbS[x]/x^i$-algebra structure on $\bbS[x]/x^i\otimes_{\bbS[x]}\bbS[x]/x^j$ is unique, so it coincides with the trivial square-zero structure. For general $k\geq j$, the argument in the proof of \cite[Theorem 5.2]{Bur22} shows that the $\bbE_1$-$\bbS[x]/x^i$-algebra map $\bbS[x]/x^i\otimes_{\bbS[x]}\bbS[x]/x^k\to\bbS[x]/x^i\otimes_{\bbS[x]}\bbS[x]/x^j$ is unique, and hence coincides with the map of trivial square-zero extensions induced by $x^{k-j}:(\bbS[x](1)[2d])^{\otimes_{\bbS[x]}k}/x^i\to(\bbS[x](1)[2d])^{\otimes_{\bbS[x]}j}/x^i$.

	Similarly, for $l\geq2$, we have
	\begin{align*}
		\Hom_{\Mod_{\bbS[x]/x^i}(\Gr(\Sp))}&(((\bbS[x](1)[2d]/x^i)^{\otimes_{\bbS[x]/x^i}k}[2])^{\otimes_{\bbS[x]/x^i}l}[-3],\bbS[x]/x^i)\\
		&\simeq\Hom_{\Mod_{\bbS[x]}(\Gr(\Sp))}(\bbS[x](kl)[2dkl+2l-3],\bbS[x]/x^i)\simeq0.
	\end{align*}
	Then the multiplication $\bbS[x]/x^i\otimes_{\bbS[x]}\bbS[x]/x^k\to\bbS[x]/x^i$ is unique, and hence coincides with the augmentation map $\bbS[x]/x^i\oplus(\bbS[x](1)[2d])^{\otimes_{\bbS[x]}k}/x^i[1]\to\bbS[x]/x^i$. Taking $k=2i$ and $j=i$, we obtain that $x^i:(\bbS[x](1)[2d])^{\otimes_{\bbS[x]}2i}/x^i\to(\bbS[x](1)[2d])^{\otimes_{\bbS[x]}i}/x^i$ is $0$, and hence the map
	\[
		\bbS[x]/x^i\otimes_{\bbS[x]}\bbS[x]/x^{2i}\to\bbS[x]/x^i\otimes_{\bbS[x]}\bbS[x]/x^i
	\]
	factors through $\bbS[x]/x^i$ as a map of $\bbE_1$-$\bbS[x]/x^i$-algebras. This gives the required factorization of $\bbS[x]/x^i\otimes_{\bbS[x]}(\bbS[x]/x^{2i})^{\op}$-modules. Therefore, the $\bbE_1$-variant of \cref{prop:MW_traceclass} shows that the cofinal subtower $(\bbS[x]/x^{2^r})_{r\geq0}$ satisfies the $\bbE_1$-variant of \cref{asp:V_r}.

	Now let $R$ be an even $\bbE_2$-ring and $v_0,\cdots,v_n\in\pi_{2*}(R)$ be homogeneous elements. By \cite[Proposition 4.0.2, Remark 4.0.3]{BM26quotient}, we may choose an $\bbE_2$-map $\bbS[x_0]\otimes\cdots\otimes\bbS[x_n]\to R$ sending $x_i$ to $v_i$ in $R$. Then we obtain the $\bbE_1$-$R$-algebras
	\[
		R/(v_0^{k_0},\cdots,v_n^{k_n})\simeq R\otimes_{\bbS[x_0]\otimes\cdots\otimes\bbS[x_n]}(\bbS[x_0]/x_0^{k_0}\otimes\cdots\otimes\bbS[x_n]/x_n^{k_n}).
	\]
	The resulting tower therefore admits a cofinal subtower satisfying the $\bbE_1$-variant of \cref{asp:V_r}, and hence satisfies Ben-Moshe's strongly Mittag-Leffler condition. Applying our \cref{thm:continuity_K,rmk:E_1/E_2_limit_topology}, we recover \cite[Theorem A]{BM26quotient}. The same argument applies in the general categorical setting and yields \cite[Theorem 4.3.6]{BM26quotient}.
\end{example}
\begin{remark}\label{rmk:Nuc_V(C)}
	Let $\C\in\CAlg^{\rig}(\PrL_{\st})$ and $V\in\C^{\dbl}$. Then the quotient of $L_V\C\hookrightarrow\Nuc_V(\C)$ is given by $\Nuc_V(\C)_{\eta}\simeq C_V\C\otimes_{\C}\Nuc_V(\C)$ by \cref{cor:generic_fiber,exm:local_duality}, which is $V$-acyclic. If moreover $\C\in\CAlg^{\rig}(\PrL_{\st})^{\cg}$, then \cref{cor:Ind_dbl} implies that the quotient of $L_V\C\hookrightarrow\Ind((L_V\C)^{\dbl})$ is also $V$-acyclic. In the symmetric monoidal setting, these observations, together with \cref{rmk:E_1/E_2_limit_topology,exm:E_2_quotient}, recover the corresponding comparisons in \cite[Theorem 4.4.8, Proposition 4.5.9]{BM26quotient} and, in particular, those in \cite[Theorem B]{BM26quotient}.
\end{remark}

We apply these results to $T(n)$-localizations and $T(n)$-nuclear module categories. To do so, we need to construct a tower of type $n$ generalized Moore spectra, to which the above construction does not directly apply. Instead, we use the following technical construction.
\begin{construction}[{\cite[Example 2.34]{MW25refinedTC}}]\label{constr:V_r}
	Let $n\geq1$. We can inductively construct a tower of type $n$ generalized Moore spectra
	\[
		V_r:=\bbS/(p^{k_{0,r}},v_1^{k_{1,r}},\cdots,v_{n-1}^{k_{n-1,r}}),
	\]
	such that for every $r\geq1$, $V_r$ is an $\bbE_{r+1}$-ring and $V_{r+1}\to V_r$ is an $\bbE_{r+1}$-ring map, and such that the tower of $\bbE_1$-rings satisfies \cref{asp:V_r}. Then we have a symmetric monoidal equivalence
	\[
		L_{F(n)}\Sp\simeq\lim_r\Mod(V_r).
	\]
	Applying \cref{prop:M_n^mf} and tensoring with $L_n^f\Sp$, we obtain equivalences of $\Sp$-algebras
	\[
		L_{T(n)}\Sp\simeq\lim_r\Mod(L_n^fV_r),\quad\Nuc(L_{T(n)}\bbS)\simeq\lim^{\dbl}_r\Mod(L_n^fV_r).
	\]
	In particular, we have $L_n^fV_r\simeq L_{T(n)}V_r$ and
	\[
		L_{T(n)}\simeq\lim_r(V_r\otimes L_n^f(-))\simeq\lim_r(L_{T(n)}V_r\otimes-),
	\]
	which can also be deduced from \cite[Proposition 7.10]{HS99}. Then for every $r\geq0$, we have $V_r\otimes L_{T(n)}\simeq L_{T(n)}V_r\otimes-\simeq L_{T(n)}(V_r\otimes-)$ and hence $\Mod_{V_r}(L_{T(n)}\Sp)\simeq L_{T(n)}\Mod(V_r)\simeq\Mod(L_{T(n)}V_r)$.
\end{construction}
\begin{theorem}[Nuclear limit topology]\label{thm:nuc_lim_topology}
	Let $n\geq1$, and let $\C\in\CAlg(\PrL_{\st})$. Then the base change functors induce canonical equivalences of $\Sp$-algebras
	\[
		L_{T(n)}\C\xrightarrow{\simeq}\lim_rL_{T(n)}\Mod_{V_r}(\C)\xleftarrow{\simeq}\lim_r\Mod_{V_r}(L_{T(n)}\C).
	\]
	If $\C\in\CAlg^{\rig}(\PrL_{\st})$, then the base change functors induce canonical equivalences of rigid $\Sp$-algebras
	\[
		\Nuc_{T(n)}\C\xrightarrow{\simeq}\lim_r^{\dbl}L_{T(n)}\Mod_{V_r}(\C)\xleftarrow{\simeq}\lim_r^{\dbl}\Mod_{V_r}(L_{T(n)}\C).
	\]

	Moreover, if $\C\in\CAlg^{\rig}(\PrL_{L_n^f})$, then the base change functors induce an equivalence
	\[
		K_{T(n+1)}^{\cont}(\C)\xrightarrow{\simeq}L_{T(n+1)}\left(\lim_rK^{\cont}(\Mod_{L_{T(n)}V_r}(\C))\right).
	\]
\end{theorem}
\begin{proof}
	Applying \cref{thm:lim_topology} to the tower of type $n$ generalized Moore spectra constructed in \cref{constr:V_r} and $L_n^f\C$, we obtain $L_{T(n)}\C\simeq\lim_rL_{T(n)}\Mod_{V_r}(\C)$ by \cref{prop:M_n^mf}. By the discussion in \cref{constr:V_r}, we also have $\Mod_{V_r}(L_{T(n)}\C)\simeq L_{T(n)}\Mod_{V_r}(\C)\simeq\Mod_{L_{T(n)}V_r}\C$. If $\C$ is rigid, the equivalences 
	\[
		\Nuc_{T(n)}(\C)\simeq\lim_r^{\dbl}L_{T(n)}\Mod_{V_r}(\C)\simeq\lim^{\dbl}_r\Mod_{V_r}(L_{T(n)}\C)\simeq\lim^{\dbl}_r\Mod_{L_{T(n)}V_r}(\C)
	\]
	follow from \cref{thm:continuity_K}. 

	Moreover, if $\C$ is $L_n^f$-local rigid, then \cref{cor:purity_T(n)} gives equivalences
	\[
		K_{T(n+1)}^{\cont}(\C)\simeq K_{T(n+1)}^{\cont}(\Nuc_{T(n)}\C)\simeq L_{T(n+1)}\left(\lim_rK^{\cont}(\Mod_{L_{T(n)}V_r}(\C))\right).
	\]
\end{proof}
In particular, taking $\C=L_n^f\Sp$ or $\C=\Nuc(L_{T(n)}\bbS)$ gives the equivalence
\[
	K_{T(n+1)}(L_{T(n)}\bbS)\simeq L_{T(n+1)}\left(\lim_rK(L_{T(n)}V_r)\right).
\]
\begin{example}\label{exm:LZ}
	Let $n\geq1$. Then \cref{thm:nuc_lim_topology} recovers \cite[Theorem 2.2.1]{LZ25} by taking $\C=L_n\Mod(R)$ with $R\in\CAlg(L_{K(n)}\Sp)$. Indeed, by \cref{prop:L_n}, we have $L_{F(n)}L_n\Mod(R)\simeq L_{K(n)}\Mod(R)$.
\end{example}
\begin{remark}
	Let $n\geq1$ and $r\geq1$. The generalized Moore spectrum $V_r$ provides a counterexample to \cref{thm:Hahn} in the $\bbE_r$-sense. Indeed, $V_r$ is an $\bbE_{r+1}$-ring with $L_{T(n)}V_r\not\simeq0$ and $L_{T(n-1)}V_r\simeq0$. It also gives a counterexample to \cref{prop:End(1_Nuc_T(n))} in the $\bbE_r$-sense. Indeed, $L_{T(n)}\Mod(V_r)\simeq\Mod(L_{T(n)}V_r)$ is a nonzero rigid $T(n)$-local $\bbE_r$-monoidal stable category.
\end{remark}
\begin{remark}\label{rmk:K(Nuc_V(C))}
	Ben-Moshe studies another form of chromatic continuity \cite{BM26quotient}. Let $\C\in\CAlg^{\rig}(\PrL_{\st})$ and $V\in\C^{\dbl}$, and suppose that $C_V\C$ is $L_{n-1}^f$-local. By \cref{thm:purity_dualizable}, $L_V\C\hookrightarrow\C$ induces an equivalence on $K^{\cont}_{T(n+1)}$. Since $\C\to\Nuc_V(\C)$ is a map of rigid envelopes of $L_V\C$, it induces an equivalence on $K^{\cont}_{T(n+1)}$ by \cref{thm:pullback_of_map_of_rigid_envelope,thm:nuc_purity}. If $\C\in\CAlg^{\rig}(\PrL_{\st})^{\cg}$, then $\C\to\Mod_{L_V\mathbbm{1}_{\C}}(\C)\to\Ind((L_V\C)^{\dbl})$ are maps of rigid envelopes of $L_V\C$ by \cref{lem:1_rig}, thereby inducing an equivalence on $K^{\cont}_{T(n+1)}$. In the symmetric monoidal setting, these observations recover \cite[Theorem 5.2.5, Theorem 5.2.9]{BM26quotient}, and in particular, \cite[Theorem C]{BM26quotient}.
\end{remark}

\subsection{Chromatic redshift for nuclear solid and gaseous modules}\label{sec:3.5}

Finally, we prove that both nuclear solid and nuclear gaseous module categories satisfy chromatic redshift. The key observation is the following ``downward closed" property of chromatic redshift.
\begin{lemma}\label{lem:redshift}
	Let $\C\to\D\in\CAlg(\Pr^{\dbl}_{\st})$. Suppose that $\operatorname{ht}(\C)=\operatorname{ht}(\D)\geq0$. For example, this applies when $\C\hookrightarrow\D$ is fully faithful or $\C\simeq\D^{\rig}\to\D$ is the canonical rigidification map. If $\D$ satisfies chromatic redshift, then so does $\C$.
\end{lemma}
\begin{proof}
	Since the $\Sp$-algebra map $\C\to\D$ induces an $\bbE_\infty$-ring map $K^{\cont}(\C)\to K^{\cont}(\D)$, $\operatorname{ht}(K^{\cont}(\D))\leq\operatorname{ht}(K^{\cont}(\C))$. By \cref{thm:redshift}, we have $\operatorname{ht}(K^{\cont}(\C))\leq\operatorname{ht}(\C)+1$. If $\operatorname{ht}(K^{\cont}(\D))=\operatorname{ht}(\D)+1$, then $\operatorname{ht}(K^{\cont}(\C))\geq\operatorname{ht}(K^{\cont}(\D))=\operatorname{ht}(\D)+1=\operatorname{ht}(\C)+1$. Therefore, $\operatorname{ht}(K^{\cont}(\C))=\operatorname{ht}(\C)+1$.
\end{proof}

We first apply this lemma to nuclear solid modules. In \cite{CMM21}, Clausen--Mathew--Morrow generalize Gabber's rigidity theorem \cite{Gab92}. As a consequence, they prove the following continuity theorems for connective algebraic $K$-theory. For discrete commutative rings, such continuity holds under some finiteness assumptions. Recall that we say a discrete commutative $\bbF_p$-algebra $R$ is \emph{$F$-finite} if the absolute Frobenius map $\Frob:R\to R$ is finite. 
\begin{theorem}[{\cite[Theorem F]{CMM21}}]
	Let $R$ be a Noetherian discrete commutative ring and $I\subseteq R$ an ideal. Suppose that $\pi_0(R^{\wedge}_I/p)$ is $F$-finite. Then the canonical maps induce an equivalence
	\[
		(\tau_{\geq0}K(R^{\wedge}_I))^{\wedge}_p\xrightarrow{\simeq}(\lim_n\tau_{\geq0}K(\pi_0(R/I^n)))^{\wedge}_p.
	\]
\end{theorem}
In the special case $I=(p)$, the Noetherian and $F$-finiteness assumptions can be removed, and the result can be extended to connective $\bbE_1$-$\bbZ$-algebras.
\begin{theorem}[{\cite[Theorem 5.21]{CMM21}}]
	Let $R$ be a connective $\bbE_1$-$\bbZ$-algebra. Suppose that $\pi_0(R)$ is commutative and henselian along $(p)$. Then the canonical maps induce an equivalence
	\[
		(\tau_{\geq0}K(R))^{\wedge}_p\xrightarrow{\simeq}(\lim_n\tau_{\geq0}K(R/p^n))^{\wedge}_p.
	\]
\end{theorem}
Combining the preceding continuity theorems with our continuity and redshift theorems yields the following results.
\begin{corollary}[Redshift for nuclear solid modules]\label{cor:redshift_nuc_solid}
	\begin{enumerate}
		\item Let $R$ be a Noetherian discrete commutative ring and $I\subseteq R$ an ideal. Suppose that $\pi_0(R^{\wedge}_I/p)$ is $F$-finite and $R^{\wedge}_I[1/p]$ is nonzero. Then both $\Nuc(R^{\wedge}_I)$ and $\Nuc(R^{\wedge,\solid}_I)$ satisfy chromatic redshift.
		\item Let $R$ be a connective $\bbE_\infty$-$\bbZ$-algebra. Suppose that $R^{\wedge}_p[1/p]$ is nonzero. Then both $\Nuc(R^{\wedge}_p)$ and $\Nuc(R^{\wedge,\solid}_p)$ satisfy chromatic redshift.
	\end{enumerate}
\end{corollary}
\begin{proof}
	We first prove (1). Since $R^{\wedge}_I$ is discrete, we have $L_{T(n)}R^{\wedge}_I\simeq0$ for every $n\geq1$. By \cref{exm:R_I^solid}, the fully faithful $\Sp$-algebra map $\Nuc(R^{\wedge,\solid}_I)\hookrightarrow\Nuc(R^{\wedge}_I)$ induces equivalences
	\[
		\End(\mathbbm{1}_{\Nuc(R^{\wedge,\solid}_I)})\simeq\End(\mathbbm{1}_{\Nuc(R^{\wedge}_I)})\simeq R^{\wedge}_I.
	\]
	Since $R^{\wedge}_I[1/p]$ is nonzero, it follows that $\operatorname{ht}(\Nuc(R^{\wedge}_I))=\operatorname{ht}(\Nuc(R^{\wedge,\solid}_I))=\operatorname{ht}(R^{\wedge}_I)=0$. Then by \cref{lem:redshift}, it suffices to show that $\operatorname{ht}(K^{\cont}(\Nuc(R^{\wedge}_I)))=1$. Since $R$ is Noetherian, \cite[Lemma 5.26]{Efi25limit} implies that the $\bbE_1$-ring tower $(R/I^n)_n$ is pro-equivalent to $(\pi_0(R/I^n))_n$. Recall that by \cite[Lemma 2.2]{LMMT24}, for every $n\geq1$, $\bbS/p$-acyclic spectra and bounded above spectra are $T(n)$-acyclic. In other words, we have $L_{T(n)}\simeq L_{T(n)}\circ(-)^{\wedge}_p$ and $L_{T(n)}\simeq L_{T(n)}\circ\tau_{\geq0}$. Therefore, combining the first continuity theorem above with \cref{exm:Nuc(R_I)}, we obtain
	\begin{align*}
		K^{\cont}_{T(1)}(\Nuc(R^{\wedge}_I))\simeq&\;L_{T(1)}\lim_nK(\pi_0(R/I^n))\simeq L_{T(1)}\tau_{\geq0}(\lim_nK(\pi_0(R/I^n)))\\
		\simeq&\;L_{T(1)}\tau_{\geq0}(\lim_n\tau_{\geq0}K(\pi_0(R/I^n)))\simeq L_{T(1)}\lim_n\tau_{\geq0}K(\pi_0(R/I^n))\\
		\simeq&\;L_{T(1)}(\lim_n\tau_{\geq0}K(\pi_0(R/I^n)))^{\wedge}_p\simeq L_{T(1)}(\tau_{\geq0}K(R^{\wedge}_I))^{\wedge}_p\\
		\simeq&\;L_{T(1)}K(R^{\wedge}_I).
	\end{align*}
	If $R^{\wedge}_I\otimes\bbQ\not\simeq0$, then $\operatorname{ht}(K(R^{\wedge}_I))=\operatorname{ht}(R^{\wedge}_I)+1=1$ by redshift of $\bbE_\infty$-rings. For general $R$, choose a map from the discrete commutative ring $R^{\wedge}_I[1/p]$ to an algebraically closed field $k$. Then $\operatorname{char}(k)\neq p$, and hence $K_{T(1)}(k)\not\simeq0$ by Suslin's theorem (see \cite[proof of Proposition 3.5]{BCM20}). It follows that $\operatorname{ht}(K^{\cont}(\Nuc(R^{\wedge}_I)))=\operatorname{ht}(K(R^{\wedge}_I))=\operatorname{ht}(K(k))=1$. Thus, both nuclear module categories satisfy chromatic redshift.
	
	For (2), note that $R^{\wedge}_p$ is a connective $\bbE_\infty$-$\bbZ$-algebra and that $\pi_0(R^{\wedge}_p)$ is henselian along $(p)$. The second continuity theorem above applies. Since $\operatorname{ht}(\bbZ)=0$ and $R^{\wedge}_p[1/p]$ is nonzero, we have $\operatorname{ht}(R^{\wedge}_p)=0$. Moreover, since $R^{\wedge}_p$ is $p$-local, the redshift theorem for $\bbE_\infty$-rings applies. Therefore, the same argument as in (1) shows that $\operatorname{ht}(K^{\cont}(\Nuc(R^{\wedge}_p)))=\operatorname{ht}(K^{\cont}(\Nuc(R^{\wedge,\solid}_p)))=1$. Thus, both nuclear module categories satisfy chromatic redshift.
\end{proof}
Note that for $\C\in\CAlg(\Pr^{\dbl}_{\st})$ with $\operatorname{ht}(\C)=0$, $\C$ satisfies chromatic redshift if and only if $\C[1/p]$ satisfies chromatic redshift. Indeed, we have $K^{\cont}_{T(1)}(\C)\simeq K^{\cont}_{T(1)}(L_{T(0)\oplus T(1)}\C)\simeq K^{\cont}_{T(1)}(\C[1/p])$.
\begin{corollary}
	Let $R$ be as in \cref{cor:redshift_nuc_solid}. In case (1), both $\Nuc(R^{\wedge}_I[1/p])$ and $\Nuc(R^{\wedge}_I[1/p]^{\solid})$ satisfy chromatic redshift. In case (2), both $\Nuc(R^{\wedge}_p[1/p])$ and $\Nuc(R^{\wedge}_p[1/p]^{\solid})$ satisfy chromatic redshift.
\end{corollary}
In particular,
\[
	\Nuc(\bbZ^{\wedge}_p),\quad\Nuc(\bbZ^{\wedge,\solid}_p),\quad\Nuc(\bbQ^{\wedge}_p),\quad\Nuc(\bbQ^{\wedge,\solid}_p)
\]
all satisfy chromatic redshift.

We now turn to nuclear gaseous modules. We begin by recalling a classical comparison between algebraic and topological $K$-theory, which will later reappear in the construction of nuclear gaseous modules.
\begin{definition}
	Let $\H:=\ell^2(\bbN)$ denote the standard separable $\infty$-dimensional Hilbert space. Let $B(\H)$ denote the ring of bounded operators on $\H$, and $C(\H)$ denote the ideal of compact operators on $\H$.
\end{definition}
By \cite[Theorem 4]{Wod94}, a two-sided ideal $I\subseteq B(\H)$ is $H$-unital over $\bbZ$ if and only if $I^2=I$. Then $C(\H)$ is $H$-unital over $\bbZ$, and hence $H$-unital by \cref{prop:fib_conn_Hunital} as it is connective. By \cite[Corollary 10.4]{SW92}, any nonunital $C^*$-algebra is $H$-unital over $\bbZ$, and hence $H$-unital.

\begin{theorem}[{\cite[Theorem 10.9]{SW92}, \cite[Theorem 2]{Wod94}}]
	Let $A$ be a nonunital $C^*$-algebra. Then the canonical maps induce equivalences
	\[
		K^{\cont}(A\widetilde{\otimes}C(\H))\xrightarrow{\simeq}K^{\mathrm{top}}(A\widetilde{\otimes}C(\H))\xleftarrow{\simeq}K^{\mathrm{top}}(A),
	\]
	where $\widetilde{\otimes}$ denotes the spatial tensor product. In particular, taking $A=\bbC$ we obtain an equivalence
	\[
		K^{\cont}(C(\H))\simeq\KU.
	\]
\end{theorem}
Thus $\Mod_H(C(\H))$ realizes $\KU$ as the continuous $K$-theory of a dualizable stable category, namely, the $H$-module category over the $H$-unital ring $C(\H)$. To restore the $\bbE_\infty$-ring structure, we introduce a more categorical construction of $\Mod_H(C(\H))$.
\begin{construction}[{\cite[2/3 00:38:05]{Cla23Efimov}}]
	Let $\mathrm{Hilb}_{\mathrm{sep}}$ denote the additive category of separable Hilbert spaces with bounded operators. Then the $\Ind$-completion of the bounded derived category of $\mathrm{Hilb}_{\mathrm{sep}}$ is $\P_{\Sigma}(\mathrm{Hilb}_{\mathrm{sep}};\Sp)\simeq\Mod(B(\H))$, and $\Mod_H(C(\H))$ is the localizing stable subcategory generated by $\colim_{\bbN}\H_n$ such that $\H_n\to\H_{n+1}$ are all compact operators in $\mathrm{Hilb}_{\mathrm{sep}}$. Then the Hilbert space tensor product induces a symmetric monoidal structure on $\Mod(B(\H))$, and hence on $\Mod_H(C(\H))$ since the tensor product of compact operators remains compact.
\end{construction}
\begin{proposition}[{\cite[2/3 00:41:58]{Cla23Efimov}}]
	We have $\Mod_H(C(\H))\in\CAlg(\Pr^{\dbl}_{\st})$ and $K^{\cont}(C(\H))\simeq\KU$ is an $\bbE_\infty$-ring equivalence.
\end{proposition}
One should keep in mind that the $\bbE_\infty$-ring structure on $K^{\cont}(C(\H))$ does not arise directly from the $H$-unital ring structure on $C(\H)$. Nevertheless, the chromatic height of an $\bbE_\infty$-ring depends only on its underlying spectrum.
\begin{corollary}
	Both $\Mod_H(C(\H))$ and $\Mod_H(C(\H))^{\rig}$ satisfy chromatic redshift.
\end{corollary}
\begin{proof}
	Applying \cref{lem:redshift} to $\Mod_H(C(\H))^{\rig}\to\Mod_H(C(\H))$, it suffices to show that $\Mod_H(C(\H))$ satisfies chromatic redshift. Since $\Mod_H(C(\H))$ is $T(0)$-local, we have $\operatorname{ht}(\Mod_H(C(\H)))=0$. It follows from $\operatorname{ht}(\KU)=1$ that $\Mod_H(C(\H))$ satisfies chromatic redshift.
\end{proof}
\begin{remark}[{\cite[3/3 01:02:26]{Cla23Efimov}, \cite[Lecture 2 00:34:19, Lecture 4 00:20:36]{Cla26realanalytic}}]
	Let $S_1$ denote the ideal of \emph{trace-class} operators on $\H$, namely, compact operators whose singular values $\lambda_1\geq\lambda_2\geq\cdots$ are summable: $\sum_{n=1}^{\infty}\lambda_n<\infty$. These are precisely the trace-class endomorphisms of $\H$ in $\Mod(B(\H))$.
	
	Let $S_{0^+}$ denote the ideal of compact operators whose singular values $\lambda_1\geq\lambda_2\geq\cdots$ exhibit \emph{rapid decay}, namely, for every $p>0$, $\sum_{n=1}^\infty\lambda_n^p<\infty$. These are precisely the very trace-class endomorphisms of $\H$ in $\Mod(B(\H))$. Then $\Mod(B(\H))^{\rig}$ is the localizing stable subcategory generated by $\colim_{\bbN}\H_n$ such that $\H_n\to\H_{n+1}$ are all compact operators in $\mathrm{Hilb}_{\mathrm{sep}}$ that exhibit rapid decay. $S_{0^+}$ is $H$-unital over $\bbZ$, and hence $H$-unital since it is connective. Moreover, Clausen--Scholze identify $\Mod_H(C(\H))^{\rig}\simeq\Mod(B(\H))^{\rig}$ with the category of $H$-modules over $S_{0^+}$
	\[
		\Mod_H(S_{0^+}):=\Mod(S_{0^+}^+,S_{0^+})\simeq\Mod(B(\H),S_{0^+}).
	\]
\end{remark}

We now work in the language of (light) condensed mathematics developed by Clausen--Scholze \cite{CS24analyticstack}, which provides a framework for recasting analytic geometry in algebraic terms. One important application is the reproof of classical results in complex analytic geometry \cite{CS26complex}. For a complex manifold $M$, the nuclear gaseous module category $\Nuc(M^{\gas})$ provides a reasonable categorical object for conjecturing that its continuous $K$-theory coincides with Deligne $K$-theory, a statement referred to as the \emph{modified Hodge conjecture}. Following \cite{Cla24deligne}, we review this construction.

\begin{construction}
	Let $\Mod(\bbC^{\gas})$ denote the derived category of gaseous $\bbC$-modules. Let $M$ be a complex manifold. By \cite[Theorem 7.3.4.9]{HTT}, we have an equivalence between the sheaf category and the $\K$-sheaf category
	\[
		\Shv(M;\Mod(\bbC^{\gas}))\simeq\Shv_{\K}(M;\Mod(\bbC^{\gas})).
	\]
	By \cite[Corollary 3.9]{Cla24deligne}, we get a structure sheaf $\O\in\Shv(M;\CAlg(\Mod(\bbC^{\gas})))$ whose restriction to closed polydisks coincides with the usual holomorphic function sheaf $\O^{\mathrm{hol}}$.
\end{construction}
\begin{definition}
	Let $M$ be a complex manifold. We define the category of \emph{gaseous quasi-coherent sheaves} on $M$ as
	\[
		\QCoh(M^{\gas}):=\Mod_{\O}(\Shv(M;\Mod(\bbC^{\gas}))).
	\]
\end{definition}
Thus $\QCoh(M^{\gas})$ is a $\QCoh(*^{\gas})\simeq\Mod(\bbC^{\gas})$-algebra.

However, the $\Sp$-algebra $\Mod(\bbC^{\gas})$ is not rigid. By \cite[p. 9]{Cla24deligne}, every trace-class map in $\Mod(\bbC^{\gas})$ factors as a composite of two trace-class maps. We define the category of \emph{nuclear gaseous $\bbC$-modules} as
\[
	\Nuc(\bbC^{\gas}):=\Nuc(\Mod(\bbC^{\gas}))\simeq\Mod(\bbC^{\gas})^{\rig},
\]
where the equivalence comes from \cref{thm:Nuc=rig}. Since $\mathbbm{1}_{\Mod(\bbC^{\gas})}\in\Mod(\bbC^{\gas})^{\omega}$, $\Nuc(\bbC^{\gas})\subseteq\Mod(\bbC^{\gas})$ is the localizing subcategory generated by basic nuclear objects.

Furthermore, $\Nuc(\bbC^{\gas})$ admits an explicit description closely related to $\Mod_H(C(\H))$.
\begin{construction}[Clausen--Scholze, {\cite[00:18:35]{Efi25analytic}}]
	Let $J\subseteq B(\H)$ denote the ideal of compact operators whose singular values $\lambda_1\geq\lambda_2\geq\cdots$ exhibit \emph{quasi-exponential decay}, namely, there exists some $C,\varepsilon>0$ such that $\lambda_n\leq C\cdot 2^{-n^{\varepsilon}}$ for all $n\geq1$. Then $J$ is also $H$-unital over $\bbZ$, and hence $H$-unital since it is connective. Moreover, Clausen--Scholze identify $\Nuc(\bbC^{\gas})$ with the category of $H$-modules over $J$
	\[
		\Mod_H(J):=\Mod(J^+,J)\simeq\Mod(B(\H),J).
	\]
\end{construction}
\begin{corollary}
	$\Nuc(\bbC^{\gas})$ satisfies chromatic redshift.
\end{corollary}
\begin{proof}
	By the above construction, the canonical map $\Nuc(\bbC^{\gas})\to\Mod_H(C(\H))$ is a fully faithful $\Sp$-algebra map. Then the result follows from \cref{lem:redshift}.
\end{proof}
In particular, we have a fully faithful map of rigid $\Sp$-algebras $\Nuc(\bbC^{\gas})\hookrightarrow\Mod_H(C(\H))^{\rig}$.
\begin{definition}
	Let $M$ be a complex manifold. We define the category of \emph{nuclear gaseous modules} on $M$ as 
	\[
		\Nuc(M^{\gas})\subseteq\QCoh(M^{\gas})
	\]
	the full subcategory spanned by those sheaves $\F$ such that $\F(K)\in\Nuc(\bbC^{\gas})$ for all compact $K$.
\end{definition}
Then $\Nuc(M^{\gas})$ is rigid whenever $M$ is compact. Indeed, we have
\begin{align*}
	\Nuc(M^{\gas})\simeq&\;\Mod_{\O}(\Shv_{\K}(M;\Nuc(\bbC^{\gas})))\simeq\Mod_{\O}(\Shv(M;\Nuc(\bbC^{\gas})))\\
	\simeq&\;\Mod_{\O}(\Shv(M;\Sp)\otimes\Nuc(\bbC^{\gas})).
\end{align*}
By \cite[Example 4.43]{Ram26locallyrigid} or \cite[Example 4.2.9]{KNP24}, $\Shv(M;\Sp)$ is rigid. Hence $\Shv(M;\Sp)\otimes\Nuc(\bbC^{\gas})$ is rigid, and so is $\Nuc(M^{\gas})$.
\begin{example}
	Let $X$ be a proper smooth variety over $\bbC$. We define the category of \emph{nuclear gaseous modules} on $X$ as
	\[
		\Nuc(X^{\gas}):=\Nuc(\bbC^{\gas})\otimes_{\bbC}\QCoh(X)\simeq\Nuc((X^{\an})^{\gas}),
	\]
	where the equivalence comes from GAGA \cite[proof of Theorem 4.10]{CS26complex,Cla24deligne}.
\end{example}
The rigidity of $\Nuc(M^{\gas})$ gives rise to the trace map 
\[
	\mathrm{tr}:K^{\cont}(\Nuc(M^{\gas}))\to\HH(\Nuc(M^{\gas})/\Nuc(\bbC^{\gas}))^{hS^1},
\]
which in turn induces the canonical map to Deligne cohomology. For our purposes, the only result we shall use is the following identification of the homotopification of continuous $K$-theory with $\KU$-coefficient sheaf cohomology.
\begin{definition}
	Let $\F\in\Shv(\mathrm{Man}_{\bbC};\Sp)$ be a sheaf on the site of complex manifolds. We view it as a $\K$-sheaf $\F\in\Shv_{\K}(\mathrm{Man}_{\bbC};\Sp)$, and define its \emph{homotopification} as
	\[
		\F^h:=\colim_{[n]\in\Delta^{\op}}\F(-\times\Delta^n)
	\]
	equipped with the canonical map $\F\to\F^h$.
\end{definition}
\begin{theorem}[{\cite[Theorem 4.10]{Cla24deligne}}]
	We have $K^{\cont}(\Nuc(-^{\gas}))^h\simeq\mathrm{R}\Gamma(-;\KU)$.
\end{theorem}
Thus, for every complex manifold $M$, we have a map of $\bbE_\infty$-rings $K^{\cont}(\Nuc(M^{\gas}))\to\mathrm{R}\Gamma(M;\KU)$.
\begin{corollary}[Redshift for nuclear gaseous modules]\label{cor:redshift_nuc_gas}
	Let $M$ be a nonempty compact complex manifold. Then $\Nuc(M^{\gas})$ satisfies chromatic redshift.
\end{corollary}
\begin{proof}
	Since $\Nuc(M^{\gas})$ is $\bbC$-linear, $\operatorname{ht}(\Nuc(M^{\gas}))=0$. By \cref{thm:redshift}, $\operatorname{ht}(K^{\cont}(\Nuc(M^{\gas})))\leq1$. Since $\mathrm{R}\Gamma(M;\KU)$ is a $\KU$-module, $\operatorname{ht}(\mathrm{R}\Gamma(M;\KU))\leq1$. Since evaluation at any point gives an $\bbE_\infty$-ring map $\mathrm{R}\Gamma(M;\KU)\to\KU$, $\operatorname{ht}(\mathrm{R}\Gamma(M;\KU))=1$. The $\bbE_\infty$-ring map $K^{\cont}(\Nuc(M^{\gas}))\to\mathrm{R}\Gamma(M;\KU)$ then implies that $\operatorname{ht}(K^{\cont}(\Nuc(M^{\gas})))\geq\operatorname{ht}(\mathrm{R}\Gamma(M;\KU))$. Therefore, $\operatorname{ht}(K^{\cont}(\Nuc(M^{\gas})))=1$.
\end{proof}

\renewcommand{\bibfont}{\small}
\setlength{\bibitemsep}{0pt}
\setlength{\bibparsep}{0pt}
\printbibliography[heading=bibintoc]

\end{document}